\documentclass[10pt, english]{smfart}
\usepackage{amssymb} 
\usepackage{amsmath}
\usepackage{amsthm} 
\usepackage{mathrsfs}
\usepackage{fontenc}
\usepackage{graphicx} 
\usepackage{color} 
\usepackage{manfnt}
\usepackage{enumitem} 
\usepackage[utf8]{inputenc}   
\usepackage{geometry}
\usepackage[francais]{babel}  

\usepackage[all]{xy} 
\usepackage{hyperref} 
\usepackage{vmargin}
\setmarginsrb{1cm}{1cm}{1cm}{1cm}{1cm}{1cm}{1cm}{1cm} 
\usepackage{pdfsync}
\newtheorem{thm}{Theorem}[section]
\newtheorem{cor}[thm]{Corollary}

\newtheorem{lem}[thm]{Lemma}
\newtheorem{prop}[thm]{Proposition} 
 
\newtheorem{propnots}[thm]{Proposition et notations}
\theoremstyle{definition} 
\newtheorem{rem}[thm]{Remark}

\newtheorem{defn}[thm]{Definition} 
\newtheorem{defn-lem}[thm]{Definition-Lemma}

\newtheorem{example}[thm]{Example}

\newtheorem{notation}[thm]{Notation}
\newtheorem{notations}[thm]{Notations}

\theoremstyle{remark}
 
\numberwithin{equation}{section}

\newcommand{\abs}[1]{\left\vert#1\right\vert}

\newcommand{\ord}{\text{ord}\:} 
\newcommand{\ac}{\overline{\text{ac}}\:}

\newcommand{\Z}{\mathbb Z} 
\newcommand{\IN}{\mathbb N}
\newcommand{\R}{\mathbb R} 
 
\renewcommand{\L}{\mathbb L}

\newcommand{\eps}{\varepsilon} 
\newcommand{\ra}{\rightarrow}

\newcommand{\Gm}{\mathbb G_{m}} 
\newcommand{\mgg}{\mathcal M_{\mathbb G_{m}}^{\mathbb G_{m}}}

\newcommand{\N}{\mathcal N}

\newcommand{\mes}{\text{mes}\:}
\newcommand{\supp}{\text{supp}}
\renewcommand{\k}{\mathbf k}

\newcommand\ag{{a}}\newcommand\bg{{b}}
 
\newcommand\mg{{\mu}}
\newcommand\sg{{\sigma}}

\newcommand{\m}{\mathfrak m}

\definecolor{vert}{rgb}{0,0.6,0.2}

\newcommand{\GN}{\overline{\N}}
\newcommand{\Eu}{\chi_c}
\newcommand{\disc}{\text{disc}\:}

\title[Newton transformations and motivic invariants at infinity of plane curves]{Newton transformations \\ and \\ motivic invariants at infinity of plane curves}
\author{Pierrette Cassou-Noguès}
\address{Institut de Math\'ematiques de Bordeaux, UMR 5251, Universit\'e de Bordeaux, 350, Cours de la Lib\'eration, 33405, Talence Cedex, France}
\email{Pierrette.Cassou-Nogues@math.u-bordeaux.fr}
\author{Michel Raibaut} 
\address{Laboratoire de Math\'ematiques, UMR 5127, Université Grenoble Alpes, Universit\'e Savoie Mont-Blanc, B\^atiment Chablais, Campus Scientifique, 
Le Bourget du Lac, 73376 Cedex, France} 
\email{Michel.Raibaut@univ-savoie.fr}
\urladdr{www.lama.univ-savoie.fr/$\sim$raibaut/}

\begin{document} 
\maketitle{} 
\begin{abstract}
	In this article we give an expression of the motivic Milnor fiber at infinity and the motivic nearby cycles at infinity of a polynomial $f$ in two variables with coefficients in an algebraic closed field of characteristic zero. This expression is given in terms of some motives associated to the faces of the Newton polygons appearing in the Newton algorithm at infinity of $f$ without any condition of convenience or non degeneracy. In the complex setting, we compute the Euler characteristic of the generic fiber of $f$ in terms of the area of the surfaces associated to faces of the Newton polygons. Furthermore, if $f$ has isolated singularities, we compute similarly the classical invariants at infinity $\lambda_{c}(f)$ which measures the non equisingularity at infinity of the fibers of $f$ in $\mathbb P^2$, 
	and  we prove the equality between the topological and the motivic bifurcation sets and give an algorithm to compute them.
\end{abstract}
\today 
\tableofcontents{}
\section*{Introduction} 
 Let $\k$ be an algebraic closed field of characteristic zero. Let $f$ be a polynomial with coefficients in $\k$. 
Using the motivic integration theory, introduced by Kontsevich in \cite{Kon95a}, and more precisely constructions of Denef -- Loeser in \cite{DenLoe98b, DenLoe99a,DenLoe02a} and  Guibert -- Loeser -- Merle in \cite{GuiLoeMer06a}, Matsui--Takeuchi in \cite{Matsui-Takeuchi-13, Matsui-Takeuchi-14} and independently the second author in \cite{Rai11} (see also \cite{Rai10a} and \cite{Rai12})
defined a \emph{motivic Milnor fiber at infinity} of $f$, denoted by $S_{f,\infty}$. It is an element of 
$\mathcal M_{\{\infty\} \times \Gm}^{\Gm}$, a modified Grothendieck ring of varieties over $\k$ endowed with an action of the multiplicative group $\Gm$ of $\k$. 
It follows from Denef--Loeser results that the motive $S_{f,\infty}$ is a ``motivic'' incarnation of the topological Milnor fiber at infinity of $f$, denoted by $F_\infty$ and endowed with its monodromy action $T_\infty$. 
For instance, when $\k$ is the field of complex numbers, the motive $S_{f,\infty}$ realizes on the Euler characteristic of $F_\infty$ and the monodromy zeta function or the Steenbrink's spectrum of $(F_\infty,T_\infty)$.

In the same way, in \cite[\S 4]{Rai11} (see also \cite{Matsui-Takeuchi-14}) a notion of \emph{motivic nearby cycles at infinity} 
$S_{f,a}^{\infty}$ of $f$ for a value $a$ is defined as an element of the Grothendieck ring $\mathcal M_{\{a\} \times \Gm}^{\Gm}$. 
This motive is constructed using a compactification $X$ of the graph of $f$ and arcs with origin in the closure of the fiber $f^{-1}(a)$ in $X$. 
It is shown in \cite{Rai11} that this motive does not depend on the chosen compactification. 
If $f$ has isolated singularities and $\k$ is the field of complex numbers, Fantini and the second author proved in \cite{FR} that the Euler characteristic of the motive $S_{f,a}^{\infty}$ is equal to $(-1)^{d-1}\lambda_{a}(f)$, where $d$ is the dimension of the ambient space and $\lambda_{a}(f)$ is the classical invariant which measures the lack of equisingularity at infinity of the usual compactification of $f$ in $\mathbb P_{\mathbb C}^{d} \times \mathbb P_{\mathbb C}^{1}$ (see for instance \cite{ArtLueMel00b} and \cite{Tibar}).
Furthermore, the second author introduced in \cite{Rai11} a \emph{motivic bifurcation set} $B_{f}^{\text{mot}}$ as the set of values $a$ which belong to the discriminant of $f$ or such that $S_{f,a}^{\infty}\neq 0$. It is shown in \cite{Rai11} that this set is finite, and for instance if $f$ has isolated singularities at infinity, it is proven in \cite{FR} that the usual topological bifurcation set $B_{f}^{\text{top}}$ of $f$ is included in $B_{f}^{\text{mot}}$.

In this article we investigate the case of polynomials in $\k[x,y]$ in full generality (namely without any assumptions of convenience or non degeneracy w.r.t any Newton polygon) using ideas of Guibert in \cite{Gui02a}, Guibert, Loeser and Merle in \cite{GuiLoeMer05a}, and the works of the first author and Veys in the case of an ideal of $\k[[x,y]]$ in \cite{CassouVeys13, Cassou-Nogues-Veys-15} (see also \cite{carai-antonio} for the equivariant case). 

More precisely, the first main results of this article are the expression, in Theorem \ref{thm:thmSfinfini} and Theorem \ref{thmcyclesprochesinfini}, of the 
\emph{motivic Milnor fiber at infinity} $S_{f,\infty}$ and the \emph{motivic nearby cycles at infinity} $S_{f,a}^{\infty}$ for a value $a$, in terms of some motives associated to faces of the Newton polygons appearing in the Newton algorithm at infinity presented in \cite{Cassou11} and recalled in Definition \ref{algoNewtoninfini}. These formulas use in particular the computation using the Newton algorithm (Theorem \ref{thmSfeps}) of some specific motivic nearby cycles
$\left(S_{h^{\eps},x\neq 0}\right)_{(0,0),0}$ , for $h$ any element of $\k[x^{-1},x,y]$ of the form $x^{-M}g(x,y)$ and where $h^{\epsilon}$ denotes $h$ or $h^{-1}$ if $\epsilon$ is $1$ or $-1$. Note that in the statement of these theorems, we give an explicit expression of the rational form of the motivic zeta functions defining 
$S_{f,a}^{\infty}$ and $S_{f,\infty}$. This is done to study in a future article the motivic monodromy conjecture for $f$ for the monodromy at infinity or the monodromy around a value of a bifurcation set.

Applying the realization of the Euler characteristic on the motives $S_{f,\infty}$ and $S_{f,a}^{\infty}$, we deduce in Corollary \ref{Kouchnirenkoformulagenfiber} and Corollary \ref{calcullambdaaire}, a Kouchnirenko type formula for the Euler characteristic of the generic fiber of $f$ and the invariant $\lambda_{a}(f)$ (in the case of isolated singularities), in terms of the area of surfaces associated to faces of Newton polygons appearing in the Newton algorithm at infinity. 

Finally, by area considerations, all these results allow to prove the fundamental Theorem \ref{Bftop=Bfmot} of this article, which states that in the isolated singularities case, we have the equality 
$$B_f^{\text{top}}=B_f^{\text{Newton}}=B_f^{\text{mot}}$$
where $B_f^{\text{Newton}}$ is the Newton bifurcation set of $f$ which is defined in an algorithmic way (Definition \ref{Newtonbifset}) using the Newton algorithm at infinity of $f$. {Up to factorisation of polynomials, this} gives in particular an algorithm to compute the usual topological bifurcation set for curves without assumptions of convenience or non degeneracy toward any Newton polygon, recovering and generalizing the analogous result for curves by Némethi and Zaharia in \cite{NemZah90a}. 

\section{Newton algorithms} \label{section1}
Let $\k$ be an algebraically closed field of characteristic zero, with multiplicative group denoted by $\Gm$.

\begin{defn}[Coefficients and support] 
	The \emph{support} of a polynomial $f(x,y) = \sum_{(a,b)\in \Z^2}c_{a,b}x^{a}y^{b}$ with coefficients in $\k$, is the set defined by 
	$\text{Supp}(f)=\{(a,b)\in \Z^2 \vert c_{a,b} \neq 0\}$. Sometimes, we will denote by $c_{a,b}(f)$ the coefficient $c_{a,b}$ of $f$.
\end{defn}

\subsection {Newton algorithm} 
\subsubsection{Newton polygons} \label{section:Newton-algorithm-local}

\begin{notation}  
	Let $E$ be a subset of $\Z_{\geq -n} \times \Z_{\geq -m}$ with $(n,m)$ in $\mathbb N^2$. We denote by $\Delta (E)$ the smallest convex set containing $E+\R ^2_+=\{a+b, a\in E, b \in \R ^2_+\}.$
\end{notation}
\begin{defn}[Newton diagram for a set, for a polynomial, vertices]
	A subset $\Delta$ of $\R^2$ is called \emph{Newton diagram} if $\Delta = \Delta(E)$ for some set $E$ in 
	$\Z_{\geq -n} \times \Z_{\geq -m}$
	with $(n,m)$ in $\mathbb N^2$. The smallest set $E_0$ of $\Z^2$ such that 
	$\Delta=\Delta (E_0)$ is called the \emph{set of vertices} of $\Delta$. 
	The \emph{Newton diagram} $\Delta(f)$ of a polynomial $f$ in $\k[x^{-1},x,y]$ is the diagram $\Delta(\text{Supp} f)$.
\end{defn} 

\begin{rem} The set of vertices of a Newton diagram is finite.\end{rem} 

\begin{defn}[Newton polygon, one dimensional faces, zero dimensional faces, horizontal and vertical faces] \label{def:Newton-polygon-height}
	Let $\Delta$ be a Newton diagram and $E_0 = \{v_0, \dots, v_d\}$ be its set of vertices with $v_i = (a_i,b_i)$ in $\Z^2$ satisfying $a_{i-1}<a_i$ and $b_{i-1}>b_i$, for any $i$ in $\{1,\dots,d\}$.
	For such $i$, we denote by $S_i$
	the segment $[v_{i-1},v_i]$ and by $l_{S_i}$ the line supporting $S_i$. 
	We define the \emph{Newton polygon} of $\Delta$ as the set 
	$$\mathcal{N}(\Delta)=\{S_i\}_{i \in \{1,\dots,d\}} {\cup \{v_i\}_{i \in \{0,\dots,d\}}},$$ 
	the \emph{height} of $\Delta$ as the integer $h(\Delta)=b _0-b_d$, the \emph{one dimensional faces} of $\mathcal{N}(\Delta)$ as the segments $S_i$, 
	the \emph{zero dimensional faces} of $\mathcal{N}(\Delta)$ as the vertices $v_i$ and among them the \emph{vertical face} $\gamma_v$ as $v_0$ and the \emph{horizontal face} $\gamma_h$ as $v_d$.
\end{defn}

\begin{defn}[Newton polygon at the origin, height of a polynomial] \label{def:Newton-polygon-origin}
The \emph{Newton polygon at the origin} $\mathcal{N}(f)$ of a polynomial $f$ in $\k[x^{-1},x,y]$ is the Newton polygon $\mathcal{N}(\Delta(f))$.
The \emph{height of $f$}, denoted by $h(f)$, is the height $h(\Delta(f))$.
\end{defn}

\begin{defn}[Face polynomials, roots and multiplicities, non degenerate case] \label{def:facepolynomial} 
	\label{def:zero-partie-initiale}\label{def:roots}
	 Let $f$ be a polynomial in $\k[x^{-1},x,y]$ and $\gamma$ be a face of $\N(f)$.
	If $\gamma$ has dimension zero, then $\gamma$ is a point $(a_0,b_0)$ and we denote by $f_{\gamma}$ the monomial 
	$c_{(a_0,b_0)}(f)x^{a_0}y^{b_0}$.
	If $\gamma$ has dimension one, then $\gamma$ is supported by a line $l$ and we define 
	$$f_\gamma(x,y):=\sum_{(\ag,\bg)\in l\cap \mathcal{N}(f)}c_{\ag,\bg}(f)x^{\ag}y^{\bg}.$$ 
	As the field $\k$ is algebraically closed, there exist $c$ in $\k$, $(a_\gamma,b_\gamma)$ in $\mathbb{Z}\times\mathbb{N}$,  $(p,q)$ in $\mathbb N^2$ and coprime, 
	$r$ in $\IN^*$, $\mu_i$ in $\k^*$ (all different) and $\nu_i$ in $\IN^*$ such that 
	$$f_\gamma(x,y) = cx^{a_\gamma}y^{b_\gamma}\prod _{1\leq i\leq r }(y^p-\mg_ix^q)^{\nu_i}.$$
	The polynomial $f_\gamma$ is called \emph{face polynomial} of $f$ associated to the face $\gamma$. In the one dimensional case, each $\mu_i$ is called \emph{root} with \emph{multiplicity} $\nu_i$ of the face polynomial $f_\gamma$. The set of roots is denoted by $R_\gamma$. The roots are said to be \emph{simple} if all the $\nu_i$ are equal to 1.
	Following \cite{Kou76a}, $f$ is said \emph{non degenerate} with respect to its Newton polygon $\N(f)$, if for each one dimensional face $\gamma$ in
	$\N(f)$, the face polynomial $f_\gamma$ has no critical points on the torus $\Gm^2$, in particular all its roots are simple.
\end{defn} 

\begin{defn}[Rational polyhedral convex cone]
Let $I$ be a finite set. A \emph{rational polyhedral convex cone} of {$\mathbb R^{\abs{I}}\setminus \{0\}$}
is a convex part of {$\mathbb R^{\abs{I}}\setminus \{0\}$} defined by a finite number of linear
inequalities with integer coefficients of type $a\leq 0$ and $b>0$ and stable
under multiplication by elements of $\mathbb R_{>0}$.
\end{defn}

In the well-known following proposition, we introduce notations used throughout this article.

\begin{propnots}[Function $\m$, dual cone $C_\gamma$ and normal vector $\vec{\eta}_\gamma$] \label{prop:dualfan-Newton-local}
	Let $E$ be a subset of some $\Z_{\geq -n}\times \Z_{\geq -m}$ with $(n,m)$ in $\mathbb N^2$.
	Let $(p,q)$ be in $\mathbb N^2$ with $\gcd(p,q)=1$ and $$l_{(p,q)}:(a,b) \in \mathbb R^2 \mapsto ap+bq.$$
	\begin{enumerate}
		\item The minimum of the restriction $l_{(p,q)\mid \Delta(E)}$, denoted by $\m(p,q)$, is reached on a face denoted by $\gamma(p,q)$ of $\Delta(E)$. 
			Furthermore, the linear map $l_{(p,q)}$ is constant on the face $\gamma(p,q)$.  
		\item For any face $\gamma$ of $\Delta(E)$, we denote by $C_\gamma$ the interior in its own generated vector space in $\mathbb R^2$, of the positive cone generated by the set $\{(p,q) \in \mathbb N^2 \mid \gamma(p,q) = \gamma\}$. This set is called \emph{dual cone} to the face $\gamma$ and is a relatively open rational polyhedral convex cone of $(\mathbb R_{\geq 0})^{2}$.
	\end{enumerate}
	For a one dimensional face $\gamma$, we denote by $\vec{n}_\gamma$ the normal vector to the face $\gamma$ with integral non negative coordinates and the smallest norm. With these notations we have the following properties.
	\begin{enumerate}[resume] 	
		       \item The dual cone $C_\gamma$ of any one dimensional face $\gamma$ of $\Delta(E)$ is the cone 
			       $\mathbb R_{>0} \vec{n}_\gamma$.
		       \item Any zero dimensional face $\gamma$ of $\Delta(E)$ is an intersection of two one dimensional faces $\gamma_1$ and $\gamma_2$ of $\Delta(E)$ (may be not compact) and its dual cone $C_\gamma$ is the cone $\mathbb R_{>0}\vec{n}_{\gamma_1} + \mathbb R_{>0} \vec{n}_{\gamma_2}$.
		       \item The set of dual cones $(C_\gamma)_{\gamma \in \N(\Delta(E))}$ is a fan of $(\mathbb R_{\geq 0})^{2}$, called \emph{dual fan} of $\Delta(E)$.
	\end{enumerate}
\end{propnots}

\subsubsection{Newton algorithm}

\begin{defn}[Newton transformations, Newton transforms and compositions] \label{defn:Newton-map-local} 
	Let $(p,q)$ be in $\IN^2$ with $\gcd(p,q)=1$. Let $(p',q')$ be in $\IN^2$ such that
	$pp'-qq'=1$. Let $\mu$ be in $\Gm$. We define the \emph{Newton transformation} associated to $(p,q,\mu)$ as the application
	\begin{equation} \label{formule:Newtontransform} 
		\begin{matrix} 
		&\sigma_{(p,q,\mu)}&:&  \k[x^{-1},x,y] &\longrightarrow& \k[x_1^{-1},x_1, y_1]\\ 
		&&&f(x,y)&\mapsto &f(\mu^{q'}x_1^p, x_1^q(y_1+\mu ^{p'}))\\ 
	\end{matrix}.
        \end{equation}
	We call $\sigma _{(p,q,\mu)}(f)$ a Newton transform of $f$ and denote it by $f_{{\sigma_{(p,q,\mu)}}}$ or simply $f_\sigma$. More generally, let $\Sigma _n=(\sigma _1,\cdots, \sigma _n)$ be a finite sequence of Newton maps $\sigma_i$, we define the \emph{composition} $f_{\Sigma _n}$ by induction: 
	$f_{\Sigma _1}=f_{\sigma_1}$, $ f_{\Sigma _i}=(f_{\Sigma _{i-1}})_{\sigma _i}$ for any $i$.
\end{defn}

\begin{rem} 
	The Newton map $\sigma _{(p,q,\mu)}$ depends on $(p',q')$, nevertheless if $(p'+lq,q'+lp)$ is another pair, then
	$$f(\mu^{q'+lp}x_1^p, x_1^q(y_1+\mu^{p'+lq}))=
	f(\mu^{q'}(x_1\mu ^l)^p, (x_1\mu^l)^q(y_1\mu^{-lq}+\mu ^{p'})),$$
	for any $f$ in $\k[x^{-1},x,y]$. Furthermore, there is exactly one choice of $(p',q')$ satisfying $pp'-qq'=1$ and 
	$p'\leq q$ and $q'<p$. In the sequel we will always assume these inequalities.
	This will make procedures canonical.
\end{rem}

\begin{rem} Sometimes, for instance in subsection \ref{sec:cashorizontalinfini}, we will work with polynomials in $\k[x,y,y^{-1}]$.
	In this case, the Newton polygon $\N(f)$ is defined as in Definition \ref{def:Newton-polygon-origin} and we go back to the case $\k[x^{-1},x,y]$, using a Newton maps as
$$\begin{matrix} 
	&\sigma_{(p,q,\mu)}&:& \k[x,y,y^{-1}] &\longrightarrow& \k[x_1^{-1},x_1,y_1]\\ 
	&&&f(x,y)&\mapsto &f(x_1^p(y_1+\mu ^{q'}), x_1^q\mu ^{p'})\\ 
\end{matrix}
$$
\end{rem}

\begin{lem}[Newton lemma] \label{lem:Newton-alg} 
	Let $(p,q)$ be in $\IN^2$ with $\gcd(p,q)=1$. 
	Let $\mu$ be in $\Gm$.
	Let $f$ be a non zero element in $\k[x^{-1},x,y]$ and $f_1$ be its Newton transform 
	$\sg_{(p,q,\mu)}(f)$ in $\k[x_1^{-1},x_1,y_1]$ and $\m$ as above, defined relatively to $\N(f)$.
	\begin{enumerate}
		\item If there does not exist a one dimensional face $\gamma$ of $\mathcal{N}(f)$ whose
			supporting line has equation $p\ag +q\bg=N$, for some $N$,
			then there is a polynomial $u(x_1,y_1)$ in $\k[x_1,y_1]$ with $u(0,0)\neq 0$ such that
			$f_1(x_1,y_1)=x_1^{\m(p,q)} u(x_1,y_1)$.
		\item If there exists a one dimensional face $\gamma$ of $\mathcal{N}(f)$ whose supporting line
			has equation $p\ag +q\bg = N$, if $\mu$ is not a root of $f_\gamma$, then $\m(p,q)=N$ and there is a polynomial $u(x_1,y_1)$ in $\k[x_1,y_1]$ with $u(0,0)\neq 0$ such that $f_1(x_1,y_1)=x_1^{N} u(x_1,y_1)$.
			
		\item If there exists a one dimensional face $\gamma$ of $\mathcal{N}(f)$ whose supporting line has equation $p\ag +q\bg = N$, if $\mu$ is a root of $f_\gamma$ of multiplicity $\nu$ then $\m(p,q)=N$ and there is a polynomial $g_1(x_1,y_1)$ in $\k[x_1,y_1]$ with  
		$g_1(0,0)=0$ and $g_1(0,y_1)$ of valuation $\nu$, such that $f_1(x_1,y_1)=x_1^{N} g_1(x_1,y_1).$
		In that case we have in particular the inequality $h(f)\geq \nu \geq h(f_1)$.
	\end{enumerate} 
\end{lem}
\begin{proof} This lemma is proved by a simple computation, see \cite{CassouVeys13}, \cite{Cassou-Nogues-Veys-15} 
	(or \cite[Lemma 2]{carai-antonio}). 
\end{proof}

\begin{rem}
	If $f_1(x_1,y_1)$ is equal to $x_1^{n_1}y_1^{m_1}u(x_1,y_1)$, where
	$(n_1,m_1)$ belongs to $\Z \times \IN^2$ and $u\in \k[x_1,y_1]$ is a unit in $\k[[x_1,y_1]]$, we say for
	short that $f_1$ is a \emph{monomial times a unit}. From this lemma, we see that there is a finite number of triples $(p,q,\mu)$ such that $\sigma _{(p,q,\mu)}(f)$ is
	eventually not a monomial times a unit in $\k[[x_1,y_1]]$. These triples are given by the equations of the faces of the Newton polygon and the roots of the
	corresponding face polynomials.
\end{rem}
We recall the notion of Newton algorithm based on Lemma \ref{lem:Newton-alg} and refer to \cite{ArtCasLue05a}, \cite{CassouVeys13}, \cite{Cassou-Nogues-Veys-15} for more details.

\begin{defn}[Newton algorithm] \label{def:algo-Newton}
	Let $f$ be a polynomial in $\k[x^{-1},x,y]$. The \emph{Newton algorithm} of $f$ is defined by induction.
	It starts by applying Newton transformations given by the equations of the faces of the Newton polygon $\N(f)$ and the roots of the corresponding face polynomials. 
	Then, this process is applied on each Newton transform until a \emph{base case} of the form $u(x,y)x^{-M}y^m$ or 
	$u(x,y)x^{-M}(y-\mu x^q+g(x,y))^{m}$ is obtained with $\mu$ in $\Gm$, $(M,m)$ in $\mathbb{Z}\times\mathbb{N}$, $q$ in $\mathbb{N}$, $g(x,y)=\sum_{a+bq>q} c_{a,b}x^ay^b$ in $\k[x,y]$, and $u(x,y)$ in $\k[[x,y]]$
	with $u(0,0)\neq 0$. The output of the algorithm is the sequence of the Newton transform polynomials produced.
\end{defn}
This definition is a consequence of the following Lemma \ref{lem:stabilite} and Theorem \ref{thm:algo-Newton}. 
\begin{defn-lem}[Stability lemma]\label{lem:stabilite} Let $f$ be a polynomial in $\k[x^{-1},x,y]$.
	\begin{enumerate} 
		\item \label{lem2.11} Let $\sigma$ be a Newton transformation. If the height of the Newton transform $f_\sigma$ is equal to the height of $f$, then the Newton polygon of $f$ has a unique face $\gamma$ with face polynomial 
			$f_\gamma(x,y)=x^ky^l(y-\mu x^q)^{\nu}$, with $(k,l,\nu)$ in $\mathbb{Z} \times \mathbb{N}^2$, $\mu$ in $\Gm$ and $q$ in 
			$\mathbb{N}^*$. If the height in the Newton process remains constant, we say that the Newton algorithm \emph{stabilizes}.
		\item  \label{stabilisation} If the Newton algorithm of $f$ stabilizes with height $m$ then $f$ can be written as
			\begin{equation} \label{eqcasdebase}
				f(x,y)=U(x,y)x^{M}(y{-\mu}x^q+g(x,y))^{m}
			\end{equation}
			with $\mu$ in $\Gm$, $(M,q)\in \mathbb{Z}\times\mathbb{N}$, $g(x,y)=\sum_{a+bq>q} c_{a,b}x^ay^b$ in $\k[x,y]$ and $U(x,y)\in \k[[x,y]]$ 
			with $U(0,0)\neq 0$.
	\end{enumerate}
\end{defn-lem}
\begin{proof} The proof of point \ref{lem2.11} is Lemma 2.11 in \cite{CassouVeys13}. The proof of point \ref{stabilisation} is similar to that of \cite[Lemma 4]{carai-antonio}. 
\end{proof}

\begin{thm}\label{thm:algo-Newton}
	For all $f(x,y)$ in $\k[x^{-1},x,y]$, there exists a natural integer $n_0$,  such
	that for any sequence of Newton maps  $\Sigma _n=(\sigma _1,\cdots, \sigma _n)$ 
	with $n\geq n_0$, 
	$f_{\Sigma _n}$ is of the form $u(x,y) cx^{-M}y^{m}$ or $u(x,y) cx^{-M}(y -\mu x^q+g(x,y))^{m}$, with $\mu$ in $\Gm$,
	$(M,m,q)$ in $\mathbb{Z} \times \mathbb{N}^2$, $g(x,y)=\sum_{a+bq>q} c_{a,b}x^ay^b$ in  $\k[x,y]$, and $u(x,y)$ belongs to $\k[[x,y]]$ with $u(0,0)\neq 0$.
\end{thm}
\begin{proof} The proof is similar to  \cite[Theorem 1]{carai-antonio}. 
\end{proof}

\subsubsection{Local dicritical faces, Newton generic and non generic values and Newton bifurcation set}
\begin{defn}[Local dicritical face, discriminant] \label{def:face-dicritique-discriminant}
	Let $f$ be a polynomial in $\k[x^{-1},x,y]$. 
	\begin{itemize}
		\item  By definition, if $f$ belongs to $\k[x,y]$ then it does not have a local dicritical face. Otherwise, a \emph{local dicritical face} is defined as a one dimensional face of the Newton polygon of $\Delta (\text{Supp}(f) \cup \{(0,0)\})$ which contains $(0,0)$.
		\item A local dicritical face is said to be \emph{smooth} if an equation of its underlying line is $\alpha+q\beta=0$ with $q$ in $\mathbb N$.
		\item If $\gamma$ is a local dicritical face, we define its \emph{associated polynomial (relatively to $f$)} as the polynomial 
			$\sum_{(a,b) \in \gamma}c_{(a,b)}(f)x^ay^b$. In particular if $\gamma$ is a face of $\N(f)$, then this polynomial is the face polynomial $f_{\gamma}$. The associated polynomial to $\gamma$ can be written under the form $P_\gamma(x^{-q}y^{p})$ with $(p,q)$ in $(\mathbb N^*)^2$ and coprime, and $P_{\gamma}(s)$ a polynomial in $\k[s]$. The discriminant of $P_{\gamma}(s)-c$ with respect to $s$, element of $\k[c]$, is called \emph{discriminant of the face $\gamma$ (relatively to $f$)}.
	\end{itemize}
\end{defn}

\begin{rem} A value $c_0$ is a root of the discriminant of $\gamma$ if and only if the polynomial $P_{\gamma}(s)-c_0$ has a multiple root. \end{rem}

\begin{defn}[Newton generic and non generic values] \label{Newtongenericvalues-local}
Let $f$ be a polynomial in $\k[x^{-1},x,y]$. 
	\begin{itemize}
		\item  A value $c_0$ is \emph{Newton non generic} for $f$, if $f$ admits a local dicritical face $\gamma$ and $c_0$ satisfies one of the two conditions: 
		\begin{itemize}
			\item  $c_0 \neq c_{(0,0)}(f)$ is a root of the discriminant of the face $\gamma$,
			\item  $c_0=c_{(0,0)}(f)$ is a root of the discriminant of the face $\gamma$ or $\gamma$ is not smooth.
		 \end{itemize}
		\item  A value $c_0$ is \emph{Newton generic} for $f$ if it is not Newton non generic.
	\end{itemize}
\end{defn}

\begin{defn}[Local Newton bifurcation set]	
	Let $f$ be in $\k[x^{-1},x,y]$.  The \emph{local Newton bifurcation set} of $f$, denoted by 
	$B_{f,loc}^{\text{Newton}}$, is the set of Newton non generic values of either $f$ or a Newton transform $f_\Sigma$ where $\Sigma$ is a composition of Newton maps in the Newton algorithm of $f$.
\end{defn}

\begin{example}[Local Newton bifurcation set of the base cases] \label{Bf^Newton-base-case}
	The local Newton bifurcation set of a base case (Definition \ref{def:algo-Newton}) is contained in $\{0\}$.
	It is empty if the base case is assumed with isolated singularities.
\end{example}
\begin{proof} Using definitions, the value 0 is the unique Newton non generic value candidate of a base case polynomial.
\end{proof}

\begin{prop}[Finiteness of the local Newton bifurcation set]
	The local Newton bifurcation set of an element of $\k[x^{-1},x,y]$ is finite.
\end{prop}

\begin{proof} 
	Indeed a polynomial $f$ in $\k[x^{-1},x,y]$ has at most one local dicritical face and its discriminant has finitely many roots. Thus, the number of Newton non generic values of $f$ is finite. We conclude by Definition \ref{def:algo-Newton} and Example \ref{Bf^Newton-base-case}.
\end{proof}
	
\newpage
\subsection{Newton algorithm at infinity}
\subsubsection{Newton polygon at infinity}
\begin{defn}[Newton polygon at infinity] \label{sec:polygones-Newton-infini} 
	Let $E$ be a nonempty finite subset of $\mathbb{N}^2$. 
	We consider
	$\Delta_{0,\infty}(E)$, $\Delta _{\infty,\infty}(E)$ and $\Delta_{\infty,0}(E)$ the smallest convex sets containing respectively
	$E+(\R_+\times \R_-)$, $E+(\R_-\times \R_-)$ and $E+(\R_-\times \R_+)$.
For any $(i,j)$ equal to $(0,\infty)$, $(\infty,\infty)$ or $(\infty,0)$, we denote by $V_{i,j} (E)$ the set of vertices of $\Delta_{i,j}(E)$. 
We define the sets
$$V_{\infty}(E)=V_{\infty,0}(E)\cup V_{\infty,\infty}(E)\cup V_{0,\infty}(E) 
\:\text{and}\: \Delta_{\infty}(E)=\text{convex hull}(V_{\infty}(E)).$$
The set $V_{\infty}(E)$ is called set of \emph{vertices at infinity} of $E$ and we order its elements in the following way: 
\begin{itemize}
	\item  on  $V_{\infty,\infty}(E)\cup V_{0,\infty}(E)$, we define $(\alpha,\beta)<(\alpha',\beta')$ if and only if $\alpha<\alpha'$,
	\item  on  $V_{\infty,0}(E)\cup V_{\infty,\infty}(E)$, we define $(\alpha,\beta)<(\alpha',\beta')$ if and only if $\beta>\beta'$. 
\end{itemize}
Write $V_{\infty}(E)=\{v_0,\cdots,v_m\}$, with $v_i=(\alpha_i,\beta_i)$ and $v_0<v_1<\cdots<v_m$.
Let $S_i$ be the line segment whose endpoints are $v_{i-1}$ and $v_i$. We denote by $\mathcal S$ the set of these segments.
The \emph{Newton polygon at infinity of $E$} is defined as the set
$$\mathcal{N}_{\infty}(E)=\{S_1,\cdots,S_m\}\cup V_{\infty}(E).$$
For any $(i,j)$ equal to $(0,\infty)$, $(\infty,\infty)$ or $(\infty,0)$, we define 
$$\mathcal{N}_{i,j}(E)=\{S \in \mathcal S \vert S \ \text{has both end points in} \ V_{i,j}(E)\} \cup V_{i,j}(E).$$
\end{defn}
\begin{rem} The vertical and horizontal faces of $\N_\infty(E)$ are not contained in the union $\N_{(\infty,\infty)}\cup \N_{(0,\infty)} \cup \N_{(\infty,0)}$.
\end{rem}

\begin{lem}[Function $\m$, dual cone $C_\gamma$ and normal vector $\vec{\eta}_\gamma$] \label{prop:dualfan}
	Let $E$ be a finite set of $\mathbb Z^2$.
	Let $(p,q)$ be in $\mathbb Z^2$ with $\gcd(p,q)=1$ and $l_{(p,q)}:(a,b) \in \mathbb R^2 \mapsto ap+bq$.
	Let $\Delta$ be the convex hull of $E$ and $\mathcal F(\Delta)$ its set of faces.
	\begin{enumerate}
		\item The maximum of the restriction $l_{(p,q)\mid \Delta}$, denoted by $\m(p,q)$, is 
			reached on a face denoted by $\gamma(p,q)$ of $\Delta$. 
			Furthermore, the linear map $l_{(p,q)}$ is constant on the face $\gamma(p,q)$.  
		\item For any face $\gamma$ of $\Delta$, we denote by $C_\gamma$ the interior, in its own generated vector
			space in $\mathbb R^2$, of the positive cone generated by the set 
			$\{(\alpha,\beta) \in \mathbb Z^2 \mid \gamma(\alpha,\beta) = \gamma\}$.
			This set is called \emph{dual cone} of the face $\gamma$ and is a relatively open rational polyhedral convex cone of {$\mathbb R^2$}.
	\end{enumerate}

               For a one dimensional face $\gamma$, we denote by $\vec{n}_\gamma$ the normal vector to $\gamma$, exterior to $\Delta$, with integral coordinates and the smallest norm. With these notations we have:

	       \begin{enumerate}[resume] 	
		       \item The dual cone $C_\gamma$ of any one dimensional face $\gamma$ of $\Delta$ is the cone $\mathbb R_{>0} \vec{n}_\gamma$.
         	       \item Any zero dimensional face $\gamma$ of $\Delta$ is an intersection of two one dimensional faces $\gamma_1$ and $\gamma_2$ of $\Delta$ and its dual cone $C_\gamma$ is the cone $\mathbb R_{>0}\vec{n}_{\gamma_1} + \mathbb R_{>0} \vec{n}_{\gamma_2}$.
		       \item The set of dual cones $(C_\gamma)_{\gamma \in \mathcal F(\Delta)}$ is a fan of 
		       $\mathbb R^{2}$, called \emph{dual fan} of $\Delta$.
	       \end{enumerate}
\end{lem}

\begin{defn}[Newton polygon at infinity and global Newton polygon] \label{def:Ninfinif} \label{def:polygon-global}
	Let $f$ be a polynomial in $\k[x,y]$.
	We define 
	$\mathcal{N}_{\infty}(f)=\mathcal{N}_{\infty}(\text{Supp} f \cup \{(0,0)\}),\:
	\Delta_{\infty}(f)=\Delta_{\infty}(\text{Supp} f \cup \{(0,0)\})\:\:\text{and}\:\:\overline{\Delta}(f)=\text{convex hull}(\text{Supp} f).$
	The set $\N_{\infty}(f)$ is called \emph{Newton polygon at infinity} of $f$. Similarly to Definition \ref{sec:polygones-Newton-infini}, 
	the \emph{global Newton polygon} $\GN(f)$ is defined as the set of vertices and segments of $\overline{\Delta}(f)$.	
	We define also $\N_\infty(f)^{o}$ as the set of faces of $\N_\infty(f)$ which do not contain the origin
	and we simply denote $\N_{\infty,\infty}(\text{Supp} f \cup \{(0,0)\})$ by $\N_{\infty,\infty}(f)$ and define similarly 
	$\N_{0,\infty}(f)$ and $\N_{\infty,0}(f)$. 
	We introduce $\widetilde{\N}(f)$ defined as the set $(\overline{\N}(f)\setminus \N(f)) \cup \{\gamma_h,\gamma_v\}.$
\end{defn}

\begin{rem} \label{rem:egalitepolygones}
	If $c\neq c_{(0,0)}(f)$, we have $\mathcal{N}_{\infty}(f)= \mathcal{N}_{\infty}(\text{Supp}(f -c))$. If $f(0,0)\neq 0$ then $\overline{\N}(f)=\N_{\infty}(f)$.
\end{rem}

\subsubsection{Newton algorithm at infinity}
\begin{defn}[Newton transformation at infinity] \label{def:lestransformationsdeNewton}
	Let $(p,q)$ be a primitive vector of $\mathbb Z^2$ and $\mu$ be an element of $\k$.
	We define the \emph{Newton transformation at infinity} ${\sigma_{(p,q,\mu)}}$ by
		\begin{equation}
				\begin{matrix} 
					&\sigma_{(p,q,\mu)}&:& \k[x,y] &\longrightarrow & \k[v^{-1},v,w]\\ 
					&&& f(x,y)&\mapsto &  f(\mu^{q'}v^{-p}, v^{-q}(w+\mu ^{p'}))\\
				\end{matrix}
		\end{equation}
	        in the case $p>0$ and $q>0$, choosing $(p',q')$ in $\mathbb Z^2$ with $qq'-pp'=1$;
		
		\begin{equation} \label{transformation-Newton-infini-zero}
				\begin{matrix} 
					&\sigma_{(p,q,\mu)}&:& \k[x,y] &\longrightarrow& \k[v^{-1},v,w]\\
					&&&f(x,y)&\mapsto &f(\mu^{q'}v^{-p}, v^{-q}(w+\mu ^{p'}))\\
				\end{matrix}
			\end{equation}
		in the case $p>0$ and $q<0$, choosing $(p',q')$ in $\mathbb Z^2$ with $pp'-qq'=1$;

		\begin{equation}\label{transformation-Newton-zero-infini}
				\begin{matrix} 
					&\sigma _{(p,q,\mu)}&:& \k[x,y] &\longrightarrow& \k[v^{-1},v,w]\\ 
					&&&f(x,y)&\mapsto &f(v^{-p}(w+\mu ^{q'}),\mu^{p'}v^{-q})\\
				\end{matrix}
		\end{equation}
		in the case $p<0$ and $q>0$ choosing $(p',q')$ in $\mathbb Z^2$ with $pp'-qq'=1$; 
		
		\begin{equation}
				\begin{matrix} 
					&\sigma _{(p,q,\mu)}&:& \k[x,y] &\longrightarrow& \k[v^{-1},v,w]\\ 
					&&&f(x,y)&\mapsto &f((w+\mu ),v^{-1})\\
				\end{matrix}
		\end{equation}
                        in the case $p=0,q=1$;
			\begin{equation}
				\begin{matrix} 
					&\sigma_{(p,q,\mu)}&:& \k[x,y] &\longrightarrow& \k[v^{-1},v,w]\\
					&&&f(x,y)&\mapsto &f(v^{-1}, (w+\mu ))\\
				\end{matrix}
			\end{equation}
			in the case $p=1,q=0$.
\end{defn}

\begin{rem} \label{def:Newton-process-infinity}
	In the following we will apply these Newton maps, for each face of $\N_{\infty}(f)$ or $\GN{(f)}\setminus \N(f)$ and
	each root of the face polynomials, and we will get elements in $\k[v^{-1},v,w]$ that we have studied in the previous section. 
\end{rem}
	
	\begin{notations} \label{rem:factorisationinfini}
	Let $f$ be a polynomial in $\k[x,y]$. If $\gamma$ is a one dimensional face of ${ \GN(f)\setminus \N(f)}$, then its primitive exterior normal vector $(p,q)$ belongs to $\mathbb Z^2 \setminus (\mathbb Z_{\leq 0})^2$ and the face $\gamma$ is supported by a line with equation $pa+qb=N$. Furthermore, 
\begin{itemize} 
	\item the face $\gamma$ belongs to $\mathcal{N}_{\infty,\infty}(f)$ if and only if $p>0$ and $q>0$, and in that case we have 
		the factorisation
		$$f_\gamma(x,y)=x^{a_S}y^{b_S}\prod_{\mu_i \in R_\gamma}(x^q-\mu_iy^p)^{\nu_i}\:\:\text{and}\:\: 
		f_\gamma(v^{-p},v^{-q}w)=v^{-N}w^{b_S}\prod_{\mu_i \in R_\gamma}(1-\mu_iw^p)^{\nu_i},$$
	\item the face $\gamma$ belongs to ${\mathcal{N} _{0,\infty}(f)}$ if and only if $p<0$ and $q>0$, and in that case we have 
		the factorisation		
		$$f_\gamma(x,y)=x^{a_S}y^{b_S}\prod_{\mu_i \in R_\gamma}(\mu_ix^qy^{-p}-1)^{\nu_i} \:\:\text{and}\:\: 
		f_\gamma(v^{-p}w,v^{-q})=v^{-N}w^{a_S}\prod_{\mu_i \in R_\gamma}(\mu_i w^q-1)^{\nu_i},$$
	\item the face  $\gamma$ is horizontal if and only if $(p,q)=(0,1)$, and in that case we have the factorisation
		$$f_\gamma(x,y)=x^{a_S}y^{b_S}\prod_{\mu_i \in R_\gamma}(x-\mu_i)^{\nu_i} \:\:\text{and}\:\: 
		f_\gamma(w,v^{-1})=v^{-N}w^{a_S}\prod_{\mu_i \in R_\gamma}(w-\mu_i)^{\nu_i},$$
	\item the face $\gamma$ belongs to $\mathcal{N} _{\infty,0}(f)$, if and only if $p>0$ and $q<0$, and in that case we have 
		the factorisation
		$$f_\gamma(x,y)=x^{a_S}y^{b_S}\prod_{\mu_i \in R_\gamma}(x^{-q}y^p-\mu_i)^{\nu_i} \:\:\text{and}\:\: 
		f_\gamma(v^{-p},v^{-q}w)=v^{-N}w^{b_S}\prod_{\mu_i \in R_\gamma}(w^p-\mu_i)^{\nu_i},$$
	\item the face $\gamma$ is vertical, if and only if $(p,q)=(1,0)$, and in that case we have the factorisation
		$$f_\gamma(x,y)=x^{a_S}y^{b_S}\prod_{\mu_i \in R_\gamma}(y-\mu_i)^{\nu_i} \:\:\text{and}\:\:  
		f_\gamma(v^{-1},w)=v^{-N}w^{b_S}\prod_{\mu_i \in R_\gamma}(w-\mu_i)^{\nu_i}.$$
\end{itemize} 
Each $\mu_i$ is called root of the face polynomial $f_\gamma$ and $R_\gamma$ is the set of these roots.
\end{notations}

\begin{defn}[Newton algorithm at infinity] \label{algoNewtoninfini} 
	The \emph{Newton algorithm at infinity} of a polynomial $f$ in $\k[x,y]$ consists in applying Newton maps at infinity associated to the one dimensional faces of $\GN(f)\setminus \mathcal{N}(f)$, and then to apply the Newton algorithm to each Newton transform. 
	The output of the algorithm is the sequence of the Newton transform polynomials produced.
\end{defn}
\begin{rem}
The Newton algorithm at infinity of $f-c$ with $c$ an element of $\k$, depends on $c$.
\end{rem}

\begin{lem} \label{compositionNT}
		A composition $\Sigma$ of Newton transformations in a Newton algorithm (at infinity) has the form 
		$$
		\begin{array}{ccc}
		\begin{array}{ccc}
			\k[x,y] & \to & \k[v^{-1},v,w] \\
			P(x,y) & \mapsto & P(av^A,bv^Bw+Q(v))
		\end{array}
		& \text{or} &   
                \begin{array}{ccc}
			\k[x,y] & \to & \k[v^{-1},v,w] \\
			P(x,y) & \mapsto & P(bv^Bw+Q(v),av^A)
		\end{array}
	        \end{array}$$
		with $A$ and $B$ in $\mathbb Z$, $a$ and $b$ in $\k$, and $Q$ in {$\k[v,v^{-1}]$}.
\end{lem}
\begin{proof} 
		The proof is done by induction on the length of $\Sigma$ using the definition of a Newton transformation in Definition 
		\ref{defn:Newton-map-local} and Definition \ref{def:lestransformationsdeNewton}.
\end{proof}

\begin{lem} \label{lemme:singularitesisolees}
	Let $f$ be in $\mathbb C[x,y]$ (resp in $\mathbb C[x^{-1},x,y]$) with isolated singularities in $\mathbb C^2$ (resp in $\mathbb C^* \times \mathbb C$). Then for any composition $\Sigma$ of Newton transformations, the polynomial 
	$f_\Sigma$ in $\mathbb C[x_1^{-1},x_1,y_1]$ has isolated singularities in $\mathbb C^* \times \mathbb C$.
\end{lem}
\begin{proof}
	Let $\Sigma$ be a composition of Newton transformations. It follows from Lemma \ref{compositionNT} that we can assume 
	$$f_\Sigma(x_1,y_1) = f(x_1^A,x_1^B y_1 + P(x_1))$$
	with $A$ and $B$ in $\mathbb Z\setminus \{0\}$ (the case $B=0$ is similar) and $P$ in {$\mathbb C[x_1,x_1^{-1}]$}.
	We have 
	$$
	\left\{
		\begin{array}{ccl}
			\frac{\partial f_\Sigma}{\partial x_1}(x_1,y_1) & = & 
			\frac{\partial f}{\partial x}\left(x_1^A,x_1^B y_1 + P(x_1)\right)Ax_1^{A-1} + 
			\frac{\partial f}{\partial y}\left(x_1^A,x_1^B y_1 + P(x_1)\right)(B x_1^{B-1}y_1 + P'(x_1)) \\ \\
			\frac{\partial f_\Sigma}{\partial y_1}(x_1,y_1) & = & 
			\frac{\partial f}{\partial y}\left(x_1^A,x_1^B y_1 + P(x_1)\right) x_1^{B} 
		\end{array}
		\right.
		$$
		If $f_\Sigma$ has non isolated singularities in $\mathbb C^*\times \mathbb C$, then there is a point $M=(a,b)$ in $\mathbb C^* \times \mathbb C$ and an infinite sequence of distinct points $M_n = (a_n,b_n)$ in $\mathbb C^* \times \mathbb C$ which converges to $M$ and such that for any $n$, we have 
		$\frac{\partial {f_{\Sigma}}}{\partial x_1}(M_n) =  \frac{\partial {f_{\Sigma}}}{\partial y_1}(M_n) = 0.$
		As for any $n$ we have $a_n\neq 0$, we conclude that
		$\frac{\partial f}{\partial y}(a_n^A,a_n^B b_n+P(a_n)) = \frac{\partial f}{\partial x}(a_n^A,a_n^B b_n+P(a_n)) = 0$
		which proves that all the points $(a_n^A,a_n^B b_n+P(a_n))$ are critical points of $f$. We can assume that all these points are distinct assuming for instance that infinitely many $a_n$ are distinct (otherwise there are finitely many $a_n$ but infinitely distinct $b_n$). Then we deduce that the point $(a^A,a^B b + P(a))$ is a non isolated critical point of $f$. Contradiction.
\end{proof}
\subsubsection{Newton generic or non generic values and Newton bifurcation set}
\begin{defn}[Dicritical faces at infinity, discriminant]
	Let $f$ be a polynomial in $\k[x,y]$.
	\begin{itemize}
		\item  	A \emph{dicritical face at infinity} of $f$ is a 
			one dimensional face of the Newton polygon $\N_{\infty}(f)$ which contains the origin. 
		\item  
			A dicritical face at infinity is said \emph{smooth} if its underlying line has an equation of the form $p\alpha+q\beta=0$ with $(p,q)$ in $\mathbb Z^2$ and $q=1$ (resp $p=1$) if the face belongs to $\N_{0,\infty}(f)$ (resp $\N_{\infty,0}(f)$).
		\item   If $\gamma$ is a dicritical face at infinity of $f$, then the face polynomial $f_{\gamma}$ can be written as $P_{\gamma}(x^ay^b)$ with $P_{\gamma}$ a polynomial in $\k[s]$ and $(a,b)$ coprime integers in $\mathbb Z^2$. We call \emph{discriminant of the face $\gamma$ (relatively to $f$)}, the discriminant of $P_{\gamma}(s)-c$ with respect to $s$, element of $\k[c]$.
        \end{itemize}        
\end{defn}
\begin{defn}[Newton generic and non generic values] \label{Newtongenericvalues}
Let $f$ be a polynomial in $\k[x,y]$. A value $c_0$ is said \emph{Newton non generic} for $f$ if $f$ has a dicritical face at infinity, denoted by  $\gamma$ and
{$c_0$ satisfies one of the two conditions:}
\begin{itemize}
	\item $c_0 \neq f(0,0)$ is a root of the discriminant of $f_\gamma$,
	\item $c_0 = f(0,0)$ is a root of the discriminant of $f_\gamma$ or the face $\gamma$ is not smooth.
\end{itemize}
\end{defn}
\begin{defn}[Newton bifurcation set] \label{Newtonbifset}
	Let $f$ be a polynomial in $\k[x,y]$.
	The \emph{Newton bifurcation set} of $f$, denoted by $B_f^{\text{Newton}}$, is the union of the discriminant of $f$ (formed by the critical values of $f$), the set of Newton non generic values of $f$ and the set of Newton non generic values of the Newton transforms $f_\Sigma$, where $\Sigma$ is a composition of Newton transforms during the Newton algorithm at infinity of $f$. 
\end{defn}

\begin{prop}[Finitness of the Newton bifurcation set] \label{BfNewtonfini} 
	The \emph{Newton bifurcation set} of a polynomial $f$ in $\k[x,y]$ is finite. 
\end{prop}
	
\begin{proof} The discriminant of $f$ is a finite set. There is a finite number of dicritical faces of $f$ and for each face the set of roots of the discriminant of the face polynomial is finite. We conclude by the fact that the Newton algorithm at infinity is finite and the set of Newton non generic values of each Newton transform $f_\Sigma$ occurring in the algorithm is finite.
\end{proof}
\begin{example} \label{example:example1}
	All along this article, we will consider the following example (see also Examples \ref{example:example1Sfinfini} and \ref{example:example1Sfcinfini}).
	\begin{equation}\label{exemplef}
		f(x,y)=x^6y^4+(4x^5+3x^4)y^3+(6x^4+11x^3+3x^2)y^2+(4x^3+13x^2+2x+1)y+x^2+5x+1.
	\end{equation}
	Using for instance Grobner basis of the Jacobian ideal of $f$, we observe that $f$ only has two critical points and their Milnor number is equal to one.
	We apply the Newton algorithm at infinity of $f$ and compute the Newton bifurcation values of $f$ given by the algorithm.
	The polynomial $f$ does not have any dicritical face at infinity and $f(0,0)=1$. Let $c$ be in $\k$. The polygon $\GN(f-c) \setminus \N(f-c)$ only has two one dimensional faces: 
	\begin{itemize}
		\item $\gamma_1^{(0)}$ supported by the line of equation $-x+2y=2$ and the face polynomial of $f-c$ is $y(x^2y+1)^3$,
		\item $\gamma_2^{(0)}$ supported by the line of equation $x-y=2$ and the face polynomial of $f-c$ is $x^2(xy+1)^4$.
	\end{itemize}
	We apply the Newton algorithm at infinity:
	\begin{itemize}
		\item For the face $\gamma_1^{(0)}$, we obtain
			\begin{equation}\label{exemplef1} 
				\sigma_{\gamma_1^{(0)}} : x=v(w+1), y=-v^{-2} \:\text{and}\:
				f_1(v,w)-c:=f_{\sigma_{\gamma_1^{(0)}}}(v,w)-c = v^{-2}(5v+8w^3+\dots).
			\end{equation}
			It does not have a local dicritical face and its Newton polygon has only one one dimensional face, denoted by $\gamma_1^{(1)}$, with face polynomial $v^{-2}(5v+8w^3)$.
			We continue the algorithm and get 
			\begin{equation}\label{exemplef1sigma}
				\sigma_{\gamma_1^{(1)}}:v=-8/5v_1^3,\:w=v_1(w_1+1) \:\text{and}\:
				(f_1)_{\sigma_{\gamma_1^{(1)}}}(v_1,w_1)-c=1500 v_1^{-3}(w_1-7/30v_1+\dots),
			\end{equation}
			which is a base case (Theorem \ref{thm:algo-Newton}).
		We conclude that the face $\gamma_1^{(0)}$ does not produce any Newton bifurcation value.
		\item For the face $\gamma_2^{(0)}$, we obtain
			\begin{equation} \label{exemplef2}
				\sigma_{\gamma_2^{(0)}}:x=v^{-1},\:y=v(w-1) \: \text{and} \:
				f_2(v,w)-c:=f_{\sigma_{\gamma_2^{(0)}}}(v,w)-c=v^{-2}(w^4+(2-c)v^2+2vw^2-4v^2w-v^3+\dots).
			\end{equation}
			The polynomial $f_2$ has a local dicritical face $\gamma_0^{(2)}$ supported by a line of equation $2x+y=0$.
			The associated polynomial is $P_{\gamma_0^{(2)}}(s)=s^2+2s+2$. 	The discriminant of the polynomial $P_{\gamma_0^{(2)}}(s)-c$ is the polynomial $4(-1+c)$ in the variable $c$. 
			\begin{itemize}
				\item Assume $c=1$. Thus, $c$ is the root of the discriminant of $\gamma_{0}^{(2)}$ 
					then $c$ is a Newton bifurcation value of $f$ as a local dicritical value of $f_2$. 
					We continue the Newton algorithm, this will be used in Examples \ref{example:example1Sfinfini} and 
					\ref{example:example1Sfcinfini}.
					The Newton polygon $\N(f_2-1)$ has only one one dimensional face, denoted by $\gamma_1^{(2)}$. It is supported by the line of equation $2x+y=0$. The face polynomial is $(v^{-1}w^2+1)^2$. We get 
					\begin{equation} \label{exemplef3}
						\sigma_{\gamma_1^{(2)}}:v=-v_1^2, w=v_1(w_1+1)\:\text{and}\:
						f_3(v_1,w_1)=(f_2)_{\sigma_{\gamma_1^{(2)}}}(v_1,w_1)-1=4w_1^2-7v_1+\dots.
					\end{equation}
					The Newton polygon $\N(f_3)$ has only one one dimensional face $\gamma^{(3)}$, it is supported by the line of equation $2x+y=2$. We continue the Newton algorithm and get a base case of Theorem \ref{thm:algo-Newton}
                                        \begin{equation} \label{exemplef3sigma}
						\sigma_{\gamma^{(3)}}:v_1=4/7v_2^2, w_1=v_2(w_2+1)\:\text{and}\:
						(f_3)_{\sigma_{\gamma^{(3)}}}(v_2,w_2)=8v_2 ^2(w_2+\dots)
					\end{equation}
					
				\item Assume $c=c_{(0,0)}(f_2)=2$. Thus, as the dicritical face $\gamma_0^{(2)}$ is not smooth, $c$ is a Newton bifurcation value of $f$. 
					We continue the Newton algorithm.
					The Newton polygon $\N(f_2-2)$ has two one dimensional faces:
						$\gamma_1^{(2)}$ supported by the line of equation $2x+y=0$, with a face polynomial equal to $v^{-2}w^2(w^2+2v)$ and
						$\gamma_2^{(2)}$ supported by the line of equation $x+y=1$, with a face polynomial equal to $v(2w^2v^{-2}-4v^{-1}w-1)$ which has two simple roots $r_1$ and $r_2$.
					\begin{itemize}
						\item For the face $\gamma_1^{(2)}$ we immediately get a base case of 
							Theorem \ref{thm:algo-Newton}
							\begin{equation} \label{exemplef2sigmagamma1}
								\sigma_{\gamma_1^{(2)}}:v=-1/2v_1^2,w=v_1(w_1+1) \:\text{and}\:
								(f_2)_{\sigma_{\gamma_1^{(2)}}}(v_1,w_1)-2=8w_1-10v_1+\dots
							\end{equation}
						\item For the face $\gamma_2^{(2)}$ and one root of the face polynomial, denoted by $\mu$, we also get a base case 							\begin{equation} \label{exemplef2sigmagamma2}
								\sigma_{\gamma_2^{(2)},\mu}:v=v_1,w=v_1(w_1+\mu) \: \text{and} \:
								(f_2)_{\sigma_{\gamma_2^{(2)},\mu}}(v_1,w_1)-2=v_1(w_1+\dots)
							\end{equation}
						\end{itemize}
				\item We consider $c\notin \{1,2\} \cup \disc f$. The Newton polygon $\N(f_2-c)$ has only one one dimensional face, denoted by $\gamma_1^{(2)}$. It is supported by the line of equation $2x+y=0$. Its face polynomial is $P_{\gamma_0^{(2)}}(x^{-1}y^2)-c$ which has two simple roots $\{\mu_1,\mu_2\}$. Let $\mu$ be one of such roots. Applying the Newton algorithm we obtain
					\begin{equation} \label{exemplef2sigma}
						\sigma_{\gamma_1^{(2)},\mu}:v=\mu^{-1}v_1^2,\:w=v_1(w_1+1) \: \text{and} \:
						(f_{2})_{\sigma_{\gamma_1^{(2)}},\mu}(v_1,w_1) =  w_1 u(v_1,w_1) \: \text{or} \:  
						(f_{2})_{\sigma_{\gamma_1^{(2)}},\mu}(v_1,w_1) = *w_1 + *v_1^{n} + \dots
					\end{equation}
					where $u$ is a unit and ``$*$'' are constants. In particular, the Newton transform 
					$(f_{2})_{\sigma_{\gamma_1^{(2)}},\mu}(v_1,w_1)$ is a base case of Theorem \ref{thm:algo-Newton}, then $c$ is not a local Newton bifurcation value of $f_2$. Furthermore $c$ does not belong to 
					$\disc f$ by assumption, then $c$ is not a Newton bifurcation value of $f$.
			\end{itemize}

	\end{itemize}
\end{example}

\section{Motivic Milnor fibers} \label{chapitrerappel}

We give below an introduction to motivic Milnor fibers and refer to\cite{DenLoe99a}, \cite{DenLoe01b}, \cite{Loo02a}, \cite{GuiLoeMer06a}, \cite{GuiLoeMer05a} and \cite{CNS} for further discussion.

\subsection{Motivic setting}

\subsubsection{Grothendieck rings} \label{inverse-direct}

Let $\k$ be a field of characteristic 0 and $\Gm$ its multiplicative
group.  We call \emph{$\k$-variety}, a separated reduced scheme of finite type
over $\k$.  We denote by $\text{Var}_{\k}$ the category of $\k$-varieties and for any
$\k$-variety $S$, by $\text{Var}_{S}$ the category of $S$-varieties, where objects are
morphisms $X \ra S$ in $\text{Var}_{\k}$.
We denote by $\mathcal M_S$ the localization of the Grothendieck ring
of $S$-varieties with respect to the class $[\mathbb A^1_{\k} \times S \to S]$. We will use also the
$\Gm$-equivariant variant $\mathcal M_{S\times \Gm}^{\Gm}$ introduced in
\cite[\S 2]{GuiLoeMer06a} or \cite[\S 2]{GuiLoeMer05a}, which is generated by isomorphism
classes of objects, $Y\ra S\times \Gm$ endowed with a monomial $\Gm$-action, of the category $\text{Var}_{S\times \Gm}^{\Gm}$.
In this context the class of the projection from $\mathbb A^1_{\k} \times (S\times \Gm) \to S\times \Gm$ endowed with the trivial action is denoted by $\mathbb L$. 
Let $f:S\ra S'$ be a morphism of varieties.  The composition by $f$ induces the
\emph{direct image} group morphism $f_{!}$ and the fibred product over $S'$ induces the \emph{inverse image} ring morphism $f^*$
$$f_{!}:\mathcal M_{S\times \Gm}^{\Gm} \ra \mathcal M_{S'\times \Gm}^{\Gm},\:\:
f^{*}:\mathcal M_{S'\times \Gm}^{\Gm} \ra \mathcal M_{S\times \Gm}^{\Gm}.$$
For a variety $(X\to S\times \Gm, \sigma)$ where $\sigma$ is a monomial action of $\Gm$, we consider its fiber in 1 denoted by $X^{(1)}\to S$ and endowed with an induced action $\sigma^{(1)}$ of the group of roots of unity $\hat{\mu}$. The corresponding Grothendieck ring to its operation is denoted by $\mathcal M_{S}^{\hat{\mu}}$ and isomorphic to $\mathcal M_{S\times \Gm}^{\Gm}$ (see \cite[Proposition 2.6]{GuiLoeMer06a}).

\subsubsection{Rational series} 

Let $A$ be one of the rings $\mathbb Z[\mathbb L, \mathbb L^{-1}]$, 
$\mathbb Z[\mathbb L, \mathbb L^{-1}, (1/(1-\mathbb L^{-i}))_{i>0}]$ and $\mathcal M_{S\times \Gm}^{\Gm}$. 
We denote by $A[[T]]_{sr}$ the $A$-submodule of $A[[T]]$ generated by 1 and finite products
of terms $p_{e,i}(T)=\mathbb L^{e}T^{i}/(1-\mathbb L^{e}T^{i})$ with $e$ in
$\mathbb Z$ and $i$ in $\mathbb N_{>0}$. There is a unique $A$-linear morphism
$\lim_{T \ra \infty}:A[[T]]_{sr}\ra A$ such that for any subset
$(e_{i},j_{i})_{i\in I}$ of $\mathbb Z \times \mathbb N_{>0}$ with $I$ finite
or empty, $\lim_{T \ra \infty}(\prod_{i\in I} p_{e_{i},j_{i}}(T))$ is equal to
$(-1)^{\abs{I}}$.

\subsubsection{Polyhedral convex cone} 
We will use the following lemma similar to \cite[Lemme 2.1.5]{Gui02a} and \cite[\S 2.9]{GuiLoeMer06a}.

\begin{lem}[Rational summation on a rational polyhedral convex cone] \label{lemmerationalitedescones} \label{lemmedescones}
	Let $\phi$ and $\eta$ be two $\mathbb Z$-linear forms defined on $\mathbb Z^2$.
	Let $C$ be a rational polyhedral convex cone of $\mathbb R^2\setminus\{(0,0)\}$, such that $\phi(C)$ and $\eta(C)$ are contained in $\mathbb N$. We assume that for any $n\geq 1$, the set $C_n$ defined as {$\phi^{-1}(n)\cap C \cap \mathbb Z^2$} is finite. We consider the formal series in $\mathbb Z\left[ \mathbb L, \mathbb L^{-1} \right][[T]]$
	$$S_{\phi, \eta, C}(T) = \sum_{n\geq 1} \sum_{(k,l)\in C_n} \mathbb L^{-\eta(k,l)} T^n.$$
	\begin{enumerate}
		\item \label{cone1} If $C$ is equal to $\mathbb R_{>0} \omega_1 + \mathbb R_{>0} \omega_2 $ where $\omega_1$ and $\omega_2$ are two non colinear primitive vectors in $\mathbb Z^2$ with $\phi(\omega_1)>0$ and $\phi(\omega_2)> 0$ then, 
		denoting $\mathcal P =( ]0,1]\omega_1 + ]0,1]\omega_2 )\cap \mathbb Z^2$, we have
			\begin{equation} \label{casconedim2}
				S_{\phi,\eta,C}(T)=
				\sum_{(k_0,l_0) \in \mathcal P}
				\frac{\mathbb L^{-\eta(k_0,l_0)}T^{\phi(k_0,l_0)}}
				{(1-\mathbb L^{-\eta(\omega_1)}T^{\phi(\omega_1)})(1-\mathbb L^{-\eta(\omega_2)}T^{\phi(\omega_2)})}
			\end{equation}
			and $\lim_{T \ra \infty} S_{\phi,\eta,C}(T)=1 = \Eu(C)$, {where $\Eu$ is the Euler characteristic with compact support morphism.}

     \item \label{cone2} If $C$ is equal to $\mathbb R_{>0}\omega$ where $\omega$ is a primitive vector in $\mathbb Z^2$ with $\phi(\omega)> 0$, then we have
			\begin{equation} \label{casconedim1} 
				S_{\phi,\eta,C}(T)= \frac{\mathbb L^{-\eta(\omega)}T^{\phi(\omega)}}{1-\mathbb L^{-\eta(\omega)}T^{\phi(\omega)}}
				\:\:\text{and}\:\:\lim_{T \ra \infty} S_{\phi,\eta,C}(T)=-1= \chi_{c}(C).
			\end{equation}
		\item \label{cone3} If $C$ is equal to $\mathbb R_{\geq 0} \omega_1 + \mathbb R_{>0} \omega_2 $ where $\omega_1$ and $\omega_2$ are two non colinear primitive vectors in $\mathbb N^2$ with $\phi(\omega_1)>0$ and $\phi(\omega_2)>0$ then, denoting
		$\mathcal P =( ]0,1]\omega_1 + ]0,1]\omega_2 )\cap \mathbb Z^2$, we have
			\begin{equation} \label{casconedim2-bis}
				S_{\phi,\eta,C}(T)=
				\sum_{(k_0,l_0) \in \mathcal P}
				\frac{\mathbb L^{-\eta(k_0,l_0)}T^{\phi(k_0,l_0)}}
				{(1-\mathbb L^{-\eta(\omega_1)}T^{\phi(\omega_1)})(1-\mathbb L^{-\eta(\omega_2)}T^{\phi(\omega_2)})} 
				+ \frac{\mathbb L^{-\eta(\omega_2)}T^{\phi(\omega_2)}}{1-\mathbb L^{-\eta(\omega_2)}T^{\phi(\omega_2)}}
			\end{equation}
             and $\lim_{T \ra \infty} S_{\phi,\eta,C}(T)=0= \chi_{c}(C)$.
	\end{enumerate}
\end{lem}
\begin{proof} The points \ref{cone1} and \ref{cone2} are similar to \cite[Lemma 1]{carai-antonio}. The point \ref{cone3} follows from the points \ref{cone1} and \ref{cone2}.
\end{proof}

\subsection{Euler characteristic and area} \label{subsection:area}
We recall the following definition introduced in \cite{carai-antonio}. We assume here $\k=\mathbb C$.
\begin{defn}[Area of a Newton polygon with respect to a polynomial]\label{area-f} 
	Let $\N$ be a Newton polygon or a Newton polygon at infinity. {We denote the set of one dimensional faces of $\N$ by $(S_i)$ with $i$ in $\{1,\dots,d\}$}. 
	Let $f$ be a polynomial such that $\N (f)=\N$ or $\N_{\infty}(f)=\N$. For each face $S_i$,  we denote by $r_i$ the number of roots of $f_{S_i}$, we denote by $s_i$ the number of points with integer coordinates on the face $S_i$ without its vertices and by $\mathcal S_i$ 
	its area defined by $\mathcal S_i = \abs{\det(v_i,w_i)}/2$ if $S_i$ is the segment $[v_i,w_i]$.
	We define the {\it area of  $\N$ with respect to $f$} as
	$$\mathcal{S}_{\N,f}= \sum_{i=1}^{d} \frac{r_i \mathcal{S}_i}{s_i+1}.$$ 
\end{defn}
\begin{rem}
	Note that if $f$ is non degenerate with respect to $\N$, then $\mathcal{S}_{\N,f}=\mathcal{S}_{\N}$ where $\mathcal{S}_{\N}$ is the area of $\N$. 
\end{rem}

\begin{prop}[Quasi-homogeneous case] \label{prop:caracteristiceuleraire}
	Let $f$ be a quasi homogeneous polynomial in $\mathbb C[x^{-1},x,y]$.
	\begin{itemize}
		\item if $f(x,y)=x^ay^b$ with $a$ in $\mathbb Z^*$ and $b$ in $\mathbb N^*$, then we have 
			$$\Eu(f^{-1}(1)\cap \Gm^2)=\Eu(f^{-1}(1))=0.$$
		\item if $f(x,y)=x^ay^b\prod_{i=1}^{r}(x^q-\mu_i y^p)^{\nu_i}$ (respectively $f(x,y)=x^ay^b\prod_{i=1}^{r}(x^qy^p-\mu_i)^{\nu_i}$), with $(p,q)$ a primitive vector of $(\mathbb N^*)^2$ and all the $\mu_i$'s are different, 
			then we have
			$$\Eu\left(f^{-1}(1)\cap \Gm^2\right) = -\frac{2r\mathcal S}{\sum_{i=1}^{r}\nu_i} = -2\mathcal S_{\N(f),f}$$
			with $\mathcal S$ the area of the associated triangle with vertices $(0,0)$, $(a+q\sum_{i=1}^{r}\nu_i,b)$, $(a,b+p\sum_{i=1}^{r}\nu_i)$ 
			(respectively of the triangle with vertices $(0,0)$, $(a,b)$, $(a+q\sum_{i=1}^{r}\nu_i,b+p\sum_{i=1}^{r}\nu_i)$).
			Furthermore, if $f$ belongs to $\mathbb C[x,y]$ then we have
			$$\Eu\left(f^{-1}(1)\cap \Gm^2\right) = -2\mathcal S_{\N_{\infty}(f),f}.$$
	\end{itemize}
\end{prop}
\begin{proof}
	Assume $f(x,y)=x^ay^b$ with $a\neq 0$ and $b>0$. Let $d$ be the greatest common divisor of $a$ and $b$. We denote $a'=a/d$ and 
	$b'=b/d$. Let $u$ and $v$ integers such that $a'u+b'v=1$. Considering the change of variables in the torus $x=x_1^{a'}y_1^{b'}$, {$y=x_1^{-v}y_1^{u}$} we obtain an isomorphism between the algebraic varieties $f^{-1}(1)$ and $\Gm\times \mu_d$ with $\mu_d$ the group of $d$-roots of unity. In particular, we deduce that $\Eu(f^{-1}(1)\cap \Gm^2)$ is zero. 
	Now, we consider the case 
	$$f(x,y)=x^ay^b\prod_{i=1}^{r}(x^q-\mu_i y^p)^{\nu_i} = x^ay^{b+p\sum_{i=1}^{r}\nu_i}\prod_{i=1}^{r}(x^qy^{-p}-\mu_i)^{\nu_i}$$
	with $(p,q)$ a primitive vector of $(\mathbb N^*)^2$ and all $\mu_i$ are different. 
	The case $f(x,y)=x^ay^b\prod_{i=1}^{r}(x^qy^p-\mu_i)^{\nu_i}$ is similar.
	We denote by $N=ap+bq+pq\sum_{i=1}^{r}\nu_i$.
	We consider $u$ and $v$ integers such that $uq-vp=1$.
	On the torus, we use the change of variables $x=x_1^uy_1^{sp}$, $y=x_1^vy_1^{sq}$ with $s=1$ if $N\geq 0$ and $s=-1$ if $N\leq 0$. Remark that if $N=0$, the two changes of variables are convenient. 
	We obtain the equality 
	$$f(x,y)=g(x_1,y_1)=x_1^{au+(b+p\sum_{i=1}^{r}\nu_i)v}y_1^{\abs{N}}\prod_{i=1}^{r}(x_1-\mu_i)^{\nu_i}.$$
	In particular the variety $f^{-1}(1)\cap \Gm^2$ is isomorphic to the variety $g^{-1}(1)\cap \Gm^2$ which Euler characteristic 
	{is} equal to $-r\abs{N}$. Indeed if $N\neq 0$, the variety is $\abs{N}$ copies of the graph of the function 
	$$x_1 \mapsto \left(x_1^{au+(b+p\sum_{i=1}^{r}\nu_i)v}\prod_{i=1}^{r}(x_1-\mu_i)^{\nu_i}\right)^{-1}$$
	defined on $\Gm\setminus\{\mu_i\mid 1\leq i \leq r\}$. 
	If $N=0$, then the variety is isomorphic to {$(g^{-1}(x_1,1)=1) \times \Gm$} which has Euler characteristic zero. To conclude, it is enough to use Definition \ref{area-f} and remark that 
	$$2\mathcal S = \abs{\det\left(\left(a+q\sum_{i=1}^{r}\nu_i,b\right), \left(a,b+p\sum_{i=1}^{r}\nu_i\right)\right)} = \abs{N}\sum_{i=1}^{r}\nu_i.$$
	This formula is correct in the case of $N=0$ because the area of the triangle is zero.
%
\end{proof}

\subsection{Arcs} \label{action}
Let $X$ be a $\k$-variety. For any integer $n$, we denote by $\mathcal L_{n}(X)$ the \emph{space of $n$-jets}
of $X$. This set is a $\k$-scheme of finite type and its $K$-rational points are morphisms 
$Spec \:K[t]/t^{n+1}\ra X$, for any extension $K$ of $\k$. There are canonical
morphisms $\mathcal L_{n+1}(X)\ra \mathcal L_{n}(X)$. These morphisms are
$\mathbb A_{\k}^{d}$-bundles when $X$ is smooth with pure dimension $d$.  The
\emph{arc space} of $X$, denoted by  $\mathcal L(X)$, is the projective limit
of this system. This set is a $\k$-scheme and we denote by
$\pi_{n}:\mathcal L(X)\ra \mathcal L_{n}(X)$ the canonical morphisms {called \emph{truncation maps}}. For more details
we refer for instance to \cite{DenLoe99a,Loo02a, CNS}.

For a non zero element $\varphi$ in $K[[t]]$ or in $K[t]/t^{n+1}$, we denote by
$\ord (\varphi)$ the valuation of $\varphi$ and by $\ac (\varphi)$ 
the coefficient  of $t^{\ord \varphi}$ in $\varphi$ called the \emph{angular component} of $\varphi$. By convention $\ac (0)$ is zero. 
The multiplicative group $\Gm$ acts canonically on 
$\mathcal L_{n}(X)$ by $\lambda.\varphi(t):=\varphi(\lambda t).$ 
We consider the application \emph{origin} $\pi_0:\varphi\mapsto \varphi(0)=\varphi \mod t$.

Let $F$ be a closed
subscheme of $X$ and $\mathcal I_{F}$ be the ideal of regular functions on $X$ which
vanish on $F$. We denote by $\ord F$ the function which assigns to each arc
$\varphi$ in $\mathcal L(X)$ the bound $\inf \ord g(\varphi)$ where $g$ runs on $\mathcal I_{F,\varphi(0)}$.

\subsection{The motivic Milnor fiber morphism} \label{section:Sf} 
\subsubsection{Motivic nearby cycles, motivic Milnor fiber}
Let $X$ be a
smooth $\k$-variety of pure dimension $d$ and $f:X\ra \mathbb A^{1}_{\k}$
be a morphism.  We set $X_{0}(f)$ for the zero locus of $f$ and for $n\geq 1$, we consider the scheme 
$$ X_{n}(f) = \{\varphi \in \mathcal L(X) \mid \ord f(\varphi)=n \} $$ 
endowed with the arrow 
$(\pi_{0},\ac \circ f)$ to $X_{0}(f)\times \Gm$.
In particular for any $m\geq n$, the truncation $\pi_m(X_{n}(f))$ denoted by $X_{n}^{(m)}(f)$, is a $(X_0(f)\times \Gm)$-variety endowed with the standard action of $\Gm$.
By smoothness of $X$, we have the equality 
$$
[X_{n}^{(m)}(f)]\mathbb L^{-md} =
[X_{n}^{(n)}(f)]\mathbb L^{-nd} \in \mathcal M_{X_0(f)\times \Gm}^{\Gm}
$$ 
this element is the \emph{motivic measure} of $X_{n}(f)$ denoted
by $\mes(X_{n}(f))$ and introduced by Kontsevich in \cite{Kon95a}.
In \cite{DenLoe98b, DenLoe02a}, Denef and Loeser introduce and prove the rationality of the following \emph{motivic zeta function}
$$Z_{f}(T) = \sum_{n\geq 1} \mes(X_n(f))T^n \in \mathcal M_{X_{0}(f)\times \Gm}^{\Gm}[[T]].$$
They define the \emph{motivic nearby cycles} $S_f$ and the \emph{motivic Milnor fiber} $S_{f,x_0}$ at any point $x_0$ in $X_0(f)$ as 
$$S_f = - \lim_{T\ra \infty} Z_f(T) \in \mathcal M_{X_0(f)\times \Gm}^{\Gm}\:\:\text{and}\:\:S_{f,x_0} := i^{*}_{x_0}(S_f) 
\in \mathcal M_{\{x_0\}\times \Gm}^{\Gm}$$ 
 where $i^{*}_{x_0}$ the pull-back morphism induced by the inclusion 
 $i_{x_0}:{\{x_0\}} \ra X_0(f)$. The motive $S_{f,x_0}$ realizes on classical invariants of the topological Milnor fiber of $f$ at $x_0$ as the Milnor number (in the isolated singularity case) or the Hodge spectrum (\cite[Theorem 4.2.1]{DenLoe98b} and \cite[\S 3.5]{DenLoe02a}). 

\subsubsection{Motivic nearby cycles morphism}
Using the weak factorisation theorem, Bittner \cite{Bit05a} extends the motivic
nearby cycles as a morphism defined over all the  Grothendieck ring $\mathcal M_{X}$.
Guibert, Loeser and Merle \cite{GuiLoeMer06a} give a different
construction using the motivic integration theory, we explain it below and we will use it in the following.

\begin{defn} \label{deffctzetamodifiee} Let $X$ be a smooth $\k$-variety with pure
	dimension $d$, let $U$ be an open and dense subset of $X$, let $F$ be
	its complement and let $f:X\ra \mathbb A^{1}_{\k}$ be a morphism. Let
	$n$ be in $\mathbb N^*$ and $\delta>0$, we consider the arc space 
	$$ X_{n}^{\delta}(f):=\left\{ 
		\varphi \in \mathcal L(X)\mid \ord f(\varphi)=n, 
		\ord \varphi^{*}F \leq  n \delta 
	\right\}
	$$
	endowed with the arrow 
	$(\pi_{0},\ac \circ f)$ to $X_{0}(f)\times \Gm$.
	Then, we consider the \emph{modified motivic zeta function} 
	$$ Z_{f,U}^{\delta}(T):=\sum_{n \geq 1} \mes( X_{n}^{\delta}(f))T^{n} 
	\in \mathcal M_{X_{0}(f)\times 	\Gm}^{\Gm}[[T]].$$ 
\end{defn}
\begin{prop}[\cite{GuiLoeMer06a}(\S 3.8)] \label{thmrationaliteopen} 
	Let $U$ be an open and dense subset of a smooth $\k$-variety $X$ with pure 
	dimension. Let $f:X \ra \mathbb A_{k}^{1}$ be a morphism. There is 
	$\delta_{0}>0$ such that for any $\delta\geq \delta_{0}$, the series $Z_{f,U}^{\delta}(T)$ is rational and its
	limit is independent of $\delta$. We will denote by $\mathcal S_{f,U}$ 
	the limit $- \lim_{T\ra \infty} Z_{f,U}^{\delta}(T)$.
\end{prop}
\begin{rem} \label{continuity-pushforward}
	Some remarks:
	\begin{itemize}
		\item With above notations, we deduce immediately from the proof of that proposition, the following equality
			$$f_{!}(S_{f,U}) = \lim_{T\ra \infty} \sum_{n \geq 1} \mes(X_n^{\delta}(f)\ra \{0\}\times \Gm) T^n \in \mathcal M_{\{0\}\times \Gm}^{\Gm}.$$
			Indeed it is enough to compare the computation on a log-resolution of $(X,X\setminus f^{-1}(0))$ of both sides of the equality.
			The computation of the left hand side term of the equality is done taking track of the origin of arcs above $f^{-1}(0)$ and then forgetting it 
			after application of the pushforward morphism $f_{!}$, whereas the computation of the right hand side term is directly done on the resolution. In the following sections, we will identify $\mathcal M_{\{0\}\times \Gm}^{\Gm}$ to $\mgg$ and we will simply write
			$$f_{!}(S_{f,U}) = \lim_{T\ra \infty} \sum_{n \geq 1} \mes(X_n^{\delta}(f)) T^n \in \mathcal M_{\Gm}^{\Gm},$$
			considering for any $n\geq 1$, $X_n^{\delta}(f)$ endowed with its structural map to $\Gm$.
		\item  With the same proof, Proposition \ref{thmrationaliteopen} can be extended to the case of $X$ not necessary smooth but $U$ smooth. 
			Indeed, the singular locus of $X$ is contained in $F=X\setminus U$ and the first step of the proof is a resolution of $(X,F\cup X_0(f))$.
	\end{itemize}
\end{rem}

\begin{thm}[\cite{Bit05a}, \cite{GuiLoeMer06a}(\S 3.9)] \label{extensionaugroup} 
	Let $f:X \ra \mathbb A_{k}^{1}$ be a morphism on a $\k$-variety not necessary smooth. 
	There is a unique $\mathcal M_{k}$-linear group morphism $\mathcal S_{f}:\mathcal M_{X} \ra \mathcal M_{X_{0}(f)\times \Gm}^{\Gm}$
	such that for all proper morphism $p:Z\ra X$ with $Z$ smooth, 
	and for all open and dense subset $U$ in $Z$, $\mathcal S_{f}([p:U \ra X])$ 
	is defined as $p_{!}(\mathcal S_{f\circ p,U}).$ 
\end{thm} 

\subsection{Motivic zeta function and differential form} \label{section:mot-zeta-omega} 
{In the following, we will need to use motivic zeta functions of a function and a differential. These are induced by the change of variables associated to Newton transformations and studied for instance in \cite{Veys01}, \cite{ArtCasLue05a}, \cite{CassouVeys13}. 
More precisely, we will need to consider a modified version taking account of a closed subset.}

{\begin{defn}\label{def:zetafomega} Let $X$ be a $\k$-variety of pure dimension $d$ and $g : X\ra \mathbb A^1_{\k}$ be a regular map. 
Let $U$ be a smooth open subvariety of $X$. The singular locus of $X$ is contained in the closed subset $X\setminus U$ denoted by $F$.
We assume $U$ to be dense in $X$ and $X$ to be endowed with a differential form $\omega$ of degree $d$ without poles and which zero locus is a divisor denoted by $D$ and included in $F$.
For any $\delta>0$, $n$ in $\mathbb N^*$ and $l$ in $\mathbb N^{*}$, we define
$$X_{n,l}^{\delta}(g,\omega,U) = 
\left\{ \varphi \in \mathcal L(X) \:
		\mid \ord g(\varphi) = n, \: 
		\ord \varphi^{*}(\mathcal I_F) \leq n\delta,\:
		\ord \omega(\varphi) = l 
\right\}
$$ 
endowed with its structural map $(\pi_{0},\ac \circ g)$ to $(D \cap g^{-1}(0))\times \Gm$. 
For any $m\geq n$ the truncation $\pi_m(X_{n,l}^{\delta}(g,\omega,U))$ is endowed with the standard action of $\Gm$.
We define the motivic zeta function associated to $(g,\omega,U)$ as 
$$\mathcal Z_{g,\omega,U}^{\delta}(S,T) = 
\sum_{n\geq 1} \sum_{l\geq 1} \mes(X_{n,l}^{\delta}(g,\omega,U)) S^l T^n 
\in \mathcal M_{(D\cap g^{-1}(0))\times	\Gm}^{\Gm}[[S,T]].$$
\end{defn}
}
\begin{lem} \label{lem:rationalite-zeta-omega}
	For any $\delta>0$, the motivic zeta function $\mathcal Z_{g,\omega,U}^{\delta}(S,T)$ is
	rational in variables $S$ and $T$.  The evaluation
	$Z_{g,\omega,U}^{\delta}(\mathbb L^{-1},T)$ is well-defined, and when $T$ goes to
	infinity this series has a limit independent from $\delta$ large enough.  
\end{lem}

\begin{proof} The differential form $\omega$ defines a divisor on $X$ denoted
	by $D$. We consider a log-resolution $(h,Y,E)$ of the couple $(X,D\cup g^{-1}(0) \cup F)$ 
	such that $h^{-1}(D)$, $h^{-1}(g^{-1}(0))$ and $h^{-1}(F)$ are normal crossing divisors as union of irreducible components of $E$.
	We denote by $(E_i)_{i\in A}$ the set of irreducible components of $E$. 
	We consider the following divisors
	$$\text{Jac}(h)=\sum_{i\in A}(\nu_i-1)E_i,\:
	\text{div}(g\circ h)=\sum_{i\in A}N_{i}(g)E_i,
	\text{div}(h^{*}\omega)=h^{-1}(D)=\sum_{i\in A}N_i(\omega)E_i \:\:\text{and}\:\:
	h^{-1}(F) = \sum_{i \in A}N_{i}(\mathcal I_F)E_i.
	$$ 
	Following the proof of Denef-Loeser \cite[Theorem 2.2.1]{DenLoe98b} and
	Guibert-Loeser-Merle\cite[Proposition 3.8]{GuiLoeMer06a}, we have 
	$$\mathcal Z_{g,\omega,U}^{\delta}(S,T) = \sum_{I \in \mathfrak I}\: [U_I \to (D\cap g^{-1}(0))\times \Gm,\sigma]S_{I}(S,T)$$ where
	$\mathfrak I = \{ I \subset A \mid  I \cap C_\omega \neq \emptyset\: \text{and}\: I \cap C_g \neq \emptyset \},\:
	C_g=\{i\in A \mid N_i(g)\neq 0\}, C_\omega=\{i \in A \mid N_i(\omega) \neq 0\};$ 
	for any $I$ in $\mathfrak I$, $U_I$ is a variety defined in \cite[\S 3.4]{GuiLoeMer06a}
	endowed with a structural map to the stratum $E_I^0 = \cap_{i\in I}E_i \setminus \cup_{j\notin I} E_j$ 
	composed with $g$ and a structural map to $\Gm$, and
	$$S_{I}(S,T)= \sum_{(k_i)\in C_I^{\delta}}\prod_{i\in I} 
	\left(\mathbb L^{-\nu_i} S^{N_{i}(\omega)} T^{N_i(g)}\right)^{k_i}\:\:
	\text{with}\:\:
	C_{I}^{\delta} = \left\{ 
		\begin{array}{c|c} 
			(k_i)\in \mathbb N_{\geq 1}^{\abs{I}} & 
			\sum_{i\in I} k_i N_{i}(\mathcal I_F) \leq \delta \sum_{i \in I} k_i N_i(g) 
		\end{array} 
	\right\}.
	$$
	For instance, if $C_{I}^{\delta} = \mathbb N_{\geq 1}^{\abs{I}}$ then we have
		$S_{I}(S,T) = 
		\prod_{i\in I} \frac{\mathbb L^{-\nu_i}  S^{N_{i}(\omega)} T^{N_i(g)}}
		{1-\mathbb L^{-\nu_i} S^{N_{i}(\omega)} T^{N_i(g)}}.
		$ 
	More generally, the rationality of $S_I$ is shown using a partition of the cone $C_{I}^{\delta}$ 
	in subcones (generated by a basis of primitive vectors of $\mathbb Z^{\abs{I}}$) and a toric change 
	of variables. This implies the rationality of the zeta function.
	Furthermore, we remark that 
		$S_{I}(\mathbb L^{-1}, T)$ and $\mathcal Z_{g,\omega,U}^{\delta}(\mathbb L^{-1}, T)$ are well defined.
	Finally, as in the proof of \cite[Proposition 3.8]{GuiLoeMer06a} 
	if $I\setminus C_g$ is not empty then $\lim_{T\ra \infty} S_{I}(\mathbb L^{-1},T)=0$.	
	If $I$ is included in $C_g$ then the limit $\lim_{T\ra \infty} S_{I}(\mathbb L^{-1},T)$ is equal to 
	$(-1)^{\abs{I}}$ if $\delta \geq \sup_{i\in I} \frac{N_{i}(\mathcal I_F)}{N_{i}(g)}$. 
\end{proof}

\begin{defn}[Motivic nearby cycles and Milnor fiber relatively to an open set and a differential form]Let $X$ be a $\k$-variety of pure dimension $d$ and $g : X\ra \mathbb A^1_{\k}$
be a regular map. Let $U$ be a smooth open subvariety of $X$ and $F$ be the
closed subset $X\setminus U$. Assume $U$ to be dense in $X$.
Let $\omega$ be a differential form of degree $d$ without poles and which zero locus is a divisor $D$ included in $F$.
We call the zeta function $\mathcal Z_{g,\omega,U}^{\delta}(\mathbb L^{-1},T)$, 
	\emph{motivic zeta function of $g$ relatively to the open set $U$ and
	the differential form $\omega$} and we denote it by 
	$$Z_{g,\omega,U}^{\delta}(T) = \sum_{n\geq 1} (\: \sum_{l \geq 1} \mes(X_{n,l}^{\delta}(g,\omega,U)) \mathbb L^{-l} \:) T^{n} \in 
	\mathcal M_{{(D \cap g^{-1}(0))}\times \Gm}^{\Gm}\left[ \left[ T \right] \right].$$ 
	We consider also its limit, still called \emph{motivic nearby cycles}, which does not depend on $\delta >>1$,
	$$ S_{g,\omega, U} = -\lim_{T\ra \infty} Z_{g,\omega, U}^{\delta}(T) \in 
	\mathcal M_{{(D \cap g^{-1}(0))}\times \Gm}^{\Gm}.$$ 
	For any point $x_0$ in $X_0(g)$, we consider the \emph{motivic Milnor fiber} of $g$ at $x_0$ and relatively to $U$ and $\omega$
	$$\left(S_{g,\omega,U}\right)_{x_0} := 
	i_{\{x_0\}}^{*}\left(S_{g,\omega, U}\right) \in \mathcal M_{\{x_0\}\times \Gm}^{\Gm}.$$

\end{defn} 

\begin{rem} \label{rem:independanceenomega}
	Some remarks:
	\begin{itemize}
		\item By Lemma \ref{lem:rationalite-zeta-omega} the motivic nearby cycles 
	$ S_{g,\omega, U} $ {depends} only on the irreducible {components} of the divisor of $\omega$ and 
	not on its multiplicities. Thus, if the divisor of $\omega$ is {equal to $F$} then $S_{g,\omega, U}$ does not depend on $\omega$.
\item 	The point $x_0$ is the origin of arcs defining $\left(S_{g,\omega,U}\right)_{x_0}$. As in Remark \ref{continuity-pushforward} we can identify 
	$\mathcal M_{\{x_0\}\times \Gm}^{\Gm}$ with $\mathcal M_{\Gm}^{\Gm}$. 
	\end{itemize}
\end{rem}

\subsection{The motivic Milnor fiber $\left(S_{f^\eps,x\neq 0}\right)_{((0,0),0)}$ with $f(x,y)=x^{-M}g(x,y)$ and $\eps=\pm$}
\label{Sfeps}  In section \ref{section:Sfinfini}, we will compute the motivic Milnor fiber at infinity (Theorem \ref{thm:thmSfinfini}) and the motivic nearby cycles at infinity
(Theorem \ref{thmcyclesprochesinfini}) of a polynomial in $\k[x,y]$. For this computation, we will need to compute motivic Milnor fibers of $1/f$ or $f$ along the open set $x\neq 0$ with $f$ an element of $\k[x^{-1},x,y]$. We do it in this section in Theorem \ref{thmSfeps}. 

\subsubsection{Setting}  \label{setting}

\begin{notations}\label{Neps}
In this section we consider an integer $M$ in $\mathbb Z$ and a polynomial $f$ in $\k[x^{-1},x,y]$ equal to 
$$f(x,y) = x^{-M}g(x,y) = \sum_{(a,b) \in \mathbb Z \times \mathbb N} c_{a,b}{(f)}x^a y^{b}$$ 
where $g$ is a polynomial in $\k[x,y]$ not divisible by $x$. 
Let $\N(f)$ be the Newton polygon at the origin of $f$ defined in Definition \ref{def:Newton-polygon-origin} and $\m$ the associated function defined in Proposition \ref{prop:dualfan-Newton-local} relatively to $\Delta(f)$.
We consider a sign $\eps$ in $\{\pm\}$. If ``$\eps=+$'' then $f^\eps$ denotes $f$ and if ``$\eps=-$'' then $f^\eps$ denotes $1/f$.
We denote by $\N(f)^\eps$ the set of compact faces $\gamma$ {in $\N(f)$} such that
	\begin{itemize} 
		\item if $\gamma=(a,b)$ is the horizontal face $\gamma_h$ (Definition \ref{def:Newton-polygon-height}) with dual cone 
			$\mathbb R_{>0}(0,1)+\mathbb R_{>0} \omega$, then 
			\begin{itemize}
				\item if $b=0$ then {$\gamma_h$ belongs to $\N(f)^\eps$ if and only if $\eps a>0$},
				\item if $b\neq 0$ then {$\gamma_h$ belongs to $\N(f)^\eps$ if and only if $\eps((a,b)\mid \omega)>0$}.
			\end{itemize}
		\item if $\gamma$ is a non horizontal zero-dimensional face with dual cone $\mathbb R_{>0}\omega_1 + \mathbb R_{>0} \omega_2$ then {$\gamma $ belongs to $\N(f)^\eps$ if and only if} $\eps\m(\omega_1)>0$ and $\eps \m(\omega_2)>0$.
		\item  if $\gamma$ is a one-dimensional face with dual cone $\mathbb R_{>0}\omega$ then {$\gamma $ belongs to $\N(f)^\eps$ if and only if} $\eps\m(\omega)>0$.
	\end{itemize}
\end{notations}

\begin{rem} \label{Nfepsfacedim1}
	The one dimensional faces of $\N(f)^\eps$ are the faces supported by lines with equation of  type $ap+bq=N$ with
	$(p,q)$ in $\mathbb N^2$ and $\eps N >0$. Remark that the intersection of two one dimensional faces of $\N(f)^\eps$ belongs to $\N(f)^\eps$.
\end{rem}

\begin{notations} \label{notations:Xeps-omega}
In this section we study the function $f^\eps$ at the (indeterminacy) point $(0,0)$ with value 0. In order to do that, we denote by $F_{\eps}$ the closed subset of $\mathbb A_\k^2$ where $f^\eps$ is not defined and we consider 
$$X_\eps = \overline{\{(x,y,z)\in (\mathbb A^2_{\k} \setminus F_\eps) \times \mathbb A^1_{\k} \mid z=f^{\eps}(x,y)\}} \subset \mathbb A^3_{\k}.$$
We have the following commutative diagram 
$$ \xymatrix{ 
\mathbb A_\k^2 \setminus F_{\eps} \ar^-{i_\eps}[r] \ar_{f^\eps}[d] & X_\eps \ar^{\pi_\eps}[ld] \\ 
   \mathbb A_\k^1 } 
$$ 
\noindent where $i_\eps$ is the open immersion
given by $i_\eps(x,y)=\left(x,y,f^{\eps}(x,y)\right)$ and $\pi_\eps$ is the projection along the last coordinate which extends the application $f^\eps$ to $X_\eps$.
We will work relatively to the open set $U_\eps = \left\{(x,y,z) \in X_\eps \mid x\neq 0\right\}$
and we denote by $\mathcal I_\eps$ the sheaf ideal defining $X_\eps\setminus U_\eps$. We consider an integer $\nu$ in $\mathbb N_{\geq 1}$ and a differential $\omega$ on $X_\eps$, with zero locus contained in the divisor $``x=0"$ of $X_\eps$ and the restriction to the open set $i(\mathbb A_\k^2 \setminus F_\eps)$ is equal to $x^{\nu-1}dx \wedge dy$. 

\end{notations}

\begin{rem} \label{rem:ferme}
	The open set $U_\eps$ is smooth because it is contained in the image of $i_\eps$. The singular locus of $X_\eps$ is contained in the divisor $``x=0"$ of $X_\eps$.
	For any arc $\varphi$ in $\mathcal L(X_\eps)$, the condition $\ord \varphi^*\mathcal I_{\eps} \leq \delta \ord f^{\eps}(\varphi)$ is $\ord x(\varphi) \leq \delta \ord f^{\eps}(\varphi)$.
\end{rem}

\subsubsection{The motive $\left(S_{f^\eps, \omega, x\neq 0}\right)_{((0,0),0)}$}
We use notations of subsection \ref{setting}.
We fix $\delta>0$. We have
\begin{equation} \label{zeta} 
	\left(Z_{\pi_\eps,\omega,U_\eps}^{\delta}(T)\right)_{((0,0),0)} = 
	\sum_{n\geq 1} (
	                     \sum_{n\delta \geq k \geq 1} \mathbb L^{-(\nu-1)k}	\mes(X_{\eps,n,k})
		       \:)
		       T^n 
\end{equation}
by {Definition \ref{def:zetafomega}} and Remark \ref{rem:ferme}, where for any integers $n\geq 1$ and $k\geq 1$, we define
\begin{equation}\label{Xepsnk}
	X_{\eps,n,k} = 
        \left\{ 
			\varphi \in \mathcal L(X) | \varphi(0)=((0,0),0),\: 
			\ord x(\varphi) = k,\: \ord \omega(\varphi) = (\nu-1)k,\:
			\ord \pi_\eps(\varphi) = \ord z(\varphi) = n 
        \right\} 
\end{equation}
endowed with the structural map
$\ac \circ z : X_{\eps,n,k} \ra \Gm,\: \varphi \mapsto \ac(z(\varphi))$.

\begin{rem}\label{rem:identification} 
	The origin of each arc of $X_{\eps,n,k}$ is $((0,0),0)$ and the generic point belongs to $U_\eps$, so from the definition of $X_\eps$,
	for any integers $n\geq 1$ and $k\geq 1$,
	there is an isomorphism between $X_{\eps,n,k}$ and 
	$$ \left\{
		\begin{array}{c|l} 
			(x(t),y(t))\in \mathcal L(\mathbb A^2_\k) &
			\begin{array}{l} 
				\ord x(t)= k,\:\: 
				\ord \omega(x(t),y(t))=(\nu-1)k,\:
				\ord y(t)>0,\:
				\ord f^\eps(x(t),y(t))=n
		\end{array} 
		\end{array}
	\right\},$$ 
	endowed with the map 
	$\ac \circ f^{\eps} : (x(t),y(t)) \mapsto \ac(f^{\eps}(x(t),y(t)))$.
	In this section, we will identify these arc spaces.
\end{rem}

\begin{rem}
        We will use the following notation 
        $$\left(Z_{f^\eps,\omega,x\neq 0}^{\delta}(T)\right)_{((0,0),0)}:=\left(Z_{\pi_\eps,\omega,U_\eps}^{\delta}(T)\right)_{((0,0),0)}.$$
	It follows from the definition of $\omega$ and Remark \ref{rem:independanceenomega} that the limit 
	$$\left(S_{f^\eps, \omega, x\neq 0}\right)_{((0,0),0)} : = -\lim_{T \to \infty} \left(Z_{f^\eps,\omega,x\neq 0}^{\delta}(T)\right)_{((0,0),0)}$$
	does not depend on the chosen differential form $\omega$ with zero locus contained in the divisor ``$x=0$" of $X_\eps$. In the following, we will simply denote it by 
	$\left(S_{f^\eps, x\neq 0}\right)_{((0,0),0)}$. This motive belongs to $\mathcal M_{\{((0,0),0)\}\times \Gm}^{\Gm}$ considered as $\mgg$ as in 
	Remark \ref{continuity-pushforward}. 
\end{rem}

\subsubsection{Computation of $\left(S_{f^\eps, x\neq 0}\right)_{((0,0),0)}$ in the case $g(0,0)\neq 0$}
\begin{prop} \label{casgunite}
	Let $\eps$ be in $\{\pm\}$ and $f(x,y)=x^{-M}g(x,y)$ be in $\k[x,x^{-1},y]$ with $g$ be in $\k[x,y]$ satisfying $g(0,0)\neq 0$.\\
	If $M\eps \geq 0 $ then we have 
	$$\left(Z_{f^\eps, x\neq 0}(T)\right)_{((0,0),0)} = 0\: and \:\left(S_{f^\eps, x\neq 0}\right)_{((0,0),0)} = 0.$$ 
	If $M\eps < 0$ then we have 
			$$\left(Z_{f^\eps, x\neq 0}(T)\right)_{((0,0),0)} = \left[x^{-\eps M}:\Gm \to \Gm, \sigma_{\Gm}\right] 
			\frac{\mathbb L^{-1}T^{-\eps M}}{1-\mathbb L^{-1}T^{-\eps M}}\:\text{and}\:
			\left(S_{f^\eps, x\neq 0}\right)_{((0,0),0)} = \left[x^{-\eps M}:\Gm \to \Gm, \sigma_{\Gm}\right].$$
\end{prop}

\begin{proof}
	Let $(n,k)$ be in $(\mathbb N^*)^2$. If $M\eps \geq 0$, then $\left(Z_{f^\eps, x\neq 0}(T)\right)_{((0,0),0)}$ is equal to zero because each arc space 
	$X_{\eps,n,k}$ is empty. If $M\eps <0$, then by Remark \ref{rem:identification} the arc space {$X_{\eps,n,k}$} is non empty if and only if $n=-\eps Mk$.
	The $n$-jet space {$\pi_n(X_{\eps,n,k})$} endowed with the canonical $\Gm$-action on jets is a bundle over $(x^{-\eps M}:\Gm  \to \Gm,\sigma_{\Gm})$ with fiber $\mathbb A^{2n-k}$ and $\sigma_{\Gm}$ is the action by translation of $\Gm$ defined by $\sigma_{\Gm}(\lambda,x)=\lambda.x$ for any $(\lambda,x)$ in $\Gm^2$. 
	Then by definition, the motivic measure of {$X_{\eps,n,k}$} is equal to $[x^{-\eps M}:\Gm \ra \Gm,\sigma_{\Gm}]\mathbb L^{-k}$ and the result follows by summation and limit.
\end{proof}

\subsubsection{Computation of $\left(S_{f^\eps, x\neq 0}\right)_{((0,0),0)}$ in the case $g(0,0)= 0$}
\label{sec:decompZgamma} \label{casg(0,0)=0}

Using notations of subsection \ref{setting}, the Newton algorithm and the strategy of the first author and Veys in \cite{CassouVeys13} (see also \cite{carai-antonio}), we express the motivic zeta function $(Z_{f^\eps,\omega,x\neq 0}^{\delta}(T))_{((0,0),0)}$ and the motivic Milnor fiber 
$(S_{f^\eps, x\neq 0})_{((0,0),0)}$ in terms of Newton polygons of $f$ and its Newton transforms.
\begin{notations} \label{notationsactionsformdiff} \label{actionGmGm2}
We use notations $R_\gamma$ and ${\sigma_{(p,q,\mu)}}$ introduced in Remark \ref{def:roots} and Definition \ref{defn:Newton-map-local}.
We define $\sigma_{\Gm}$ and $\sigma_{\Gm^2}$ the actions of $\Gm$ on $\Gm$ and $\Gm^2$ by 
$\sigma_{\Gm}(\lambda,x)=\lambda x\: \text{and} \: \sigma_{\Gm^2}(\lambda,(x,y))=(\lambda x, \lambda y)$.
For any $(p,q)$ in $\mathbb N^2$, we consider the differential form 
$\omega_{p,q}(v,w)=v^{\nu p + q - 1}dv \wedge dw$.
\end{notations}

\begin{rem} \label{notation-identification-action} {Let $\gamma$ be a face in $\N(f)$.}
	Let $(\alpha,\beta)$ be in $C_\gamma \cap (\mathbb N^*)^2$. By construction of the Grothendieck ring of varieties in 
	\cite[\S 2]{GuiLoeMer05a} and \cite[\S 2]{GuiLoeMer06a} (see \cite[Proposition 3.13]{Rai11}), the class 
	$[f_\gamma:\Gm^2 \setminus (f_\gamma=0) \to \Gm, \sigma_{\alpha,\beta}]$ with $\sigma_{\alpha,\beta}$ the action of $\Gm$ defined by
	$\sigma(\alpha,\beta)(\lambda,(x,y))=(\lambda^\alpha x,\lambda^\beta y)$, does not depend on $(\alpha,\beta)$ 
        in $C_\gamma \cap (\mathbb N^*)^2$. We replace $\sigma_{\alpha,\beta}$ by $\sigma_\gamma$.
\end{rem}

\begin{thm}[Computation of $\left(S_{f^{\eps},x\neq 0}\right)_{((0,0),0)}$] \label{thmSfeps} 
	Let $f(x,y)=x^{-M}g(x,y)$ be in $\k[x,x^{-1},y]$ with $g$ be in $\k[x,y]$ not divisible by $x$ and satisfying $g(0,0)=0$.
	Let $\nu$ be in $\mathbb N_{\geq 1}$ and $\omega$ be the associated differential form in \emph{Notations \ref{notations:Xeps-omega}}.
	Let $\eps$ be in $\{\pm\}$.
	Then, writing $(a,b)$ the horizontal face $\gamma_h$ of $\N(f)$ defined in \emph{Definition \ref{def:Newton-polygon-height}}, we have for any $\delta>0$
	\begin{equation}
		\begin{array}{ccl}
		                (Z_{f^\eps,\omega,x\neq 0}^{\delta}(T))_{((0,0),0)} & = & 
				[(x^ay^b)^\eps:\Gm^r \to \Gm, \sigma_{\Gm^r}]R_{(a,b),\eps,\omega}^{\delta}(T) \\
				& & +  \sum_{\gamma \in \N(f)\setminus \{\gamma_h\}} [f_{\gamma}^\eps:\Gm^2\setminus (f_\gamma=0) \to \Gm, \sigma_{\gamma}] 
				R_{\gamma,\eps,\omega}^{\delta}(T)\\
				&  & +\sum_{\gamma \in \N(f),\: \dim \gamma = 1} \sum_{\mu \in R_\gamma} 
				(Z_{f_{\sigma_{(p,q,\mu)}}^{\eps},\omega_{p,q,\mu},v\neq 0}^{\delta/p}(T))_{((0,0),0)}.
	        \end{array}
	\end{equation}
	with $\sigma_{(p,q,\mu)}$ the Newton transformation defined in Definition \ref{defn:Newton-map-local}.
	\begin{itemize}
		\item Furthermore, if $b=0$ then,
			\begin{itemize}
				\item in the case $``\eps = + "$, the motivic Milnor fiber $(S_{f,x\neq 0})_{((0,0),0)}$ is 
					\begin{equation} \label{casb=0+} \begin{array}{ccl} 
						(S_{f,x\neq 0})_{((0,0),0)} & = & s^{(+)} [x^{a}:\Gm \to \Gm, \sigma_{\Gm}] 
						+ \sum_{\gamma \in \N(f)^{+}\setminus \{\gamma_h\}}(-1)^{\dim \gamma+1} 
						[f_{\gamma}:\Gm^2\setminus (f_\gamma=0) \to \Gm, \sigma_{\gamma}] \\
						&   & + \sum_{\gamma \in \N(f),\: \dim \gamma = 1} \sum_{\mu \in R_\gamma} 
						(S_{f_{\sigma_{(p,q,\mu)}},v\neq 0})_{((0,0),0)}
					\end{array}
				        \end{equation} 
				\item in the case $``\eps = -"$, the motivic Milnor fiber $(S_{1/f,x\neq 0})_{((0,0),0)}$ is 0 if 
					{$f$ belongs to $\k[x,y]$}, otherwise
					\begin{equation} \begin{array}{ccl} \label{resS1/fb0}
						(S_{1/f,x\neq 0})_{((0,0),0)} & = & s^{(-)} [1/x^{a}:\Gm \to \Gm, \sigma_{\Gm}] 
						+ \sum_{\gamma \in \N(f)^{-}\setminus \{\gamma_h\}}(-1)^{\dim \gamma+1} 
						[1/f_{\gamma}:\Gm^2\setminus (f_\gamma=0) \to \Gm, \sigma_{\gamma}] \\
						&   & + \sum_{\gamma \in \N(f)^{-},\: \dim \gamma = 1} \sum_{\mu \in R_\gamma} 
						(S_{1/f_{\sigma_{(p,q,\mu)}},v\neq 0})_{((0,0),0)}
					\end{array}
				\end{equation}
			\end{itemize}
		\item if $b\neq 0$ then,
			\begin{itemize}
				\item
					in the case $``\eps = +"$, the motivic Milnor fiber $(S_{f,x\neq 0})_{((0,0),0)}$ is 
					\begin{equation} 
					\begin{array}{ccl} 
						(S_{f,x\neq 0})_{((0,0),0)} & = & - s^{(+)} [x^ay^b : \Gm^2 \to \Gm, \sigma_{\Gm^2}] 
						+ \sum_{\gamma \in \N(f)^{+}\setminus \{\gamma_h\}}(-1)^{\dim \gamma+1} 
						[f_{\gamma}:\Gm^2\setminus (f_\gamma=0) \to \Gm, \sigma_{\gamma}] \\
						&   & + \sum_{\gamma \in \N(f),\: \dim \gamma = 1} \sum_{\mu \in R_\gamma} 
						(S_{f_{\sigma_{(p,q,\mu)}},v\neq 0})_{((0,0),0)}
					\end{array}
				        \end{equation} 
				\item in the case $``\eps = -"$, the motivic Milnor fiber $(S_{1/f,x\neq 0})_{((0,0),0)}$ is 0 if 
					{$f$ belongs to $\k[x,y]$}, otherwise
					\begin{equation} \label{res1/f}
					\begin{array}{ccl}
						(S_{1/f,x\neq 0})_{((0,0),0)} & = &  
						 \sum_{\gamma \in \N(f)^{-}\setminus \{\gamma_h\}}(-1)^{\dim \gamma+1} 
						 [1/f_{\gamma}:\Gm^2\setminus (f_\gamma=0) \to \Gm, \sigma_{\gamma}] \\
						&   & + \sum_{\gamma \in \N(f)^{-},\: \dim \gamma = 1} \sum_{\mu \in R_\gamma} 
						(S_{1/f_{\sigma_{(p,q,\mu)}},v\neq 0})_{((0,0),0)}
					\end{array}
				        \end{equation}
			\end{itemize}
	\end{itemize}
	{with $r=1$ if $b=0$ otherwise $r=2$, $s^{(\eps)} = 1$ if $\gamma_h$ is in $\N(f)^\eps$ otherwise $s^{(\eps)} = 0$, and for any face $\gamma$ in $\N(f)$, $R_{\gamma,\eps,\omega}^{\delta}(T)$ are rational functions defined in \emph{(Propositions \ref{ex:cas-x^N} and \ref{lem:zeta=})}.}
\end{thm} 

 \begin{proof}
	 In the case ``$\eps=-$'', if $M\leq 0$, namely $f$ belongs to $\k[x,y]$, then $1/f$ does not vanish, so for any integers $n$ and $k$, the arc space $X_{-,n,k}$ (defined in formula (\ref{Xepsnk})) is empty and the motivic zeta function
	$\left(Z_{1/f,\omega,x\neq 0}(T) \right)_{((0,0),0)}$ is equal to 0. We give the ideas of the general proof and refer to subsection \ref{proofthmSfeps} for details.
	In subsection \ref{sec:decompositionzeta}, formula (\ref{formula:decomposition}), {for $\eps$ in $\{\pm\}$,} we consider the decomposition of the motivic zeta function
	   \begin{equation} 
	        (Z_{f^\eps,\omega,x\neq 0}^{\delta}(T))_{((0,0),0)}  = \sum_{\gamma \in \N(f)} Z^\delta_{\eps,\gamma,\omega}(T).
           \end{equation}
	   In subsection  \ref{sec:Z_(N,0)}, Proposition \ref{ex:cas-x^N}, in the case of the horizontal face $\gamma_h$ defined in Definition 
	   \ref{def:Newton-polygon-height}, we show the rationality and compute the limit of $Z^\delta_{\eps,\gamma_h,\omega}(T)$.
	   This limit does not depend on $\omega$.
	   In subsections \ref{sec:rationalite-limit-Zgamma=} and \ref{sec:rationalite-limite-Zgamma<} we consider the case of a non horizontal face $\gamma$. Depending on the fact that the face polynomial $f_\gamma$ vanishes or not on the angular components of the coordinates of an arc 
	   (Remark \ref{rem:n-inferieur-m(ord(x),ord(y)}), we decompose in formula (\ref{eq:decomposition-zeta-function}) the zeta function $Z^\delta_{\eps,\gamma,\omega}(T)$ as a sum of $Z_{\eps,\gamma,\omega}^{\delta,=}(T)$ and $Z_{\eps,\gamma,\omega}^{\delta,<}(T)$. 
	   In particular if the face $\gamma$ is zero dimensional then $Z_{\eps,\gamma,\omega}^{\delta,<}(T)$ is zero. {See also Remark \ref{nulliteZf}.} 
	   In Proposition \ref{lem:zeta=} we show the rationality and compute the limit of the zeta function $Z_{\eps,\gamma,\omega}^{\delta,=}(T)$. This limit does not depend on $\omega$.
	   In subsection \ref{sec:rationalite-limite-Zgamma<}, Proposition \ref{prop:rationalityZ<}
	   we prove the decomposition 
	   $$Z_{\eps,\gamma,\omega}^{\delta,<}(T) = \sum_{\mu \in R_\gamma} 
	   (Z_{f_{\sigma_{(p,q,\mu)}}^{\eps}, \omega_{p,q,\mu}, v\neq 0}^{\delta/p})_{((0,0),0)}.$$
	   Applying the Newton algorithm (Lemma \ref{lem:Newton-alg}) inductively, using the base cases (Examples \ref{ex:monomial} and \ref{ex:casdebase-f}), we recover the rationality of $(Z_{f^\eps,\omega,x\neq 0}^{\delta}(T))_{((0,0),0)}$, compute $(S_{f^\eps,\omega,x\neq 0})_{((0,0),0)}$ and check its independence on $\omega$.
   \end{proof}

Using Proposition \ref{prop:caracteristiceuleraire}, we deduce from Theorem \ref{thmSfeps} {a Kouchnirenko type formula computing the Euler characteristic of the motivic Milnor fiber $\left(S_{1/f,x\neq 0} \right)_{((0,0),0)}$. This formula will be used in Corollary \ref{calcullambdaaire}.}   

\begin{cor}[Kouchnirenko type formula for $\tilde{\chi}_{c}((S_{f^{\eps},x\neq 0})^{(1)}_{((0,0),0)})$] \label{KFSfeps} 
Let $f(x,y)=x^{-M}g(x,y)$ with $M$ in $\mathbb Z$ and $g$ in $\k[x,y]$ not divisible by $x$ and satisfying $g(0,0)=0$.
We denote $(a,b)=\gamma_h$ the horizontal face in Definition \ref{def:Newton-polygon-height}.
		\begin{itemize}
		\item If $b=0$ then we have,
		   \begin{itemize}
		     \item in the case $``\eps = + "$, we have 
			     \begin{equation} \label{KFSf+}
			    \begin{array}{ccl} 
				    \tilde{\chi_{c}}\left((S_{f,x\neq 0})^{(1)}_{((0,0),0)}\right) & = & 
				    s^{(+)}a -2 \sum_{\gamma \in \N(f)^{+},\dim \gamma=1} \mathcal S_{\N(f_\gamma),f_\gamma}
			          + \sum_{\gamma \in \N(f),\: \dim \gamma = 1} \sum_{\mu \in R_\gamma} 
				  \tilde{\chi}_{c}\left((S_{f_{\sigma_{(p,q,\mu)}},v\neq 0})^{(1)}_{((0,0),0)}\right)
			    \end{array}
			 \end{equation} 
		      \item in the case $``\eps = -"$, $\tilde{\chi}_{c}\left((S_{1/f,x\neq 0})^{(1)}_{((0,0),0)}\right)$ is 0 if 
			      {$f$ belongs to $\k[x,y]$}, otherwise we have
			  \begin{equation} 
		             \begin{array}{ccl}
				     \tilde{\chi}_{c}\left((S_{1/f,x\neq 0})^{(1)}_{((0,0),0)}\right) & = & 
				     s^{(-)}a -2 \sum_{\gamma \in \N(f)^{-},\dim \gamma=1} \mathcal S_{\N(f_\gamma), f_\gamma}
				     + \sum_{\gamma \in \N(f)^{-},\: \dim \gamma = 1} \sum_{\mu \in R_\gamma} 
				     \tilde{\chi}_{c}\left((S_{1/f_{\sigma_{(p,q,\mu)}},v\neq 0})^{(1)}_{((0,0),0)}\right)
					\end{array}
				\end{equation}
			\end{itemize}
		\item If $b\neq 0$ then we have,
			\begin{itemize}
				\item in the case $``\eps = +"$, we have 
					\begin{equation} \label{KFSf+b}
					\begin{array}{ccl} 
						\tilde{\chi}_{c}\left((S_{f,x\neq 0})^{(1)}_{((0,0),0)}\right) & = & 
						-2\sum_{\gamma \in \N(f)^{+},\dim \gamma=1} \mathcal S_{\N(f_\gamma), f_\gamma}
						 + \sum_{\gamma \in \N(f),\: \dim \gamma = 1} \sum_{\mu \in R_\gamma} 
						 \tilde{\chi}_{c}\left((S_{f_{\sigma_{(p,q,\mu)}},v\neq 0})^{(1)}_{((0,0),0)}\right)
					\end{array}
				        \end{equation} 
				\item in the case $``\eps = -"$, $\tilde{\chi}_{c}\left((S_{1/f,x\neq 0})^{(1)}_{((0,0),0)}\right)$ is 0 if {$f$ belongs to $\k[x,y]$}, otherwise we have
					\begin{equation} 
					\begin{array}{ccl}
						\tilde{\chi}_{c}\left((S_{1/f,x\neq 0})^{(1)}_{((0,0),0)}\right) & = &  
						 -2\sum_{\gamma \in \N(f)^{-},\dim \gamma=1} \mathcal S_{\N(f_\gamma), f_\gamma}
						  + \sum_{\gamma \in \N(f)^{-},\: \dim \gamma = 1} \sum_{\mu \in R_\gamma} 
						  \tilde{\chi}_{c}\left((S_{1/f_{\sigma_{(p,q,\mu)}},v\neq 0})^{(1)}_{((0,0),0)}\right)
					\end{array}
				        \end{equation}
			\end{itemize}
	\end{itemize}
	{where  $s^{(\eps)} = 1$ if $\gamma_h$ is in $\N(f)^\eps$ otherwise $s^{(\eps)} = 0$.}

\end{cor}

\subsubsection{Proof of Theorem \ref{thmSfeps}} \label{proofthmSfeps}
{In all this subsection we consider {the rational function}
$f(x,y)$ equal to $x^{-M}g(x,y)$ with $M$ in $\mathbb Z$ and $g$ in $\k[x,y]$ not divisible by $x$ and satisfying $g(0,0)=0$. We fix also an integer $\nu$ in $\mathbb N_{\geq 1}$ and denote by $\omega$ the associated differential form in Notations \ref{notations:Xeps-omega}.
We consider  $\eps$ in $\{\pm\}$ and $\delta \geq 1$.}

\paragraph{Decomposition of the zeta function along $\N(f)$} \label{sec:decompositionzeta}

\begin{notation}
	For any face $\gamma$ of $\N(f)$ {with dual cone $C_\gamma$}, for any integers $n\geq 1$ and $k\geq 1$, using Remark \ref{rem:identification} we consider 
$$X_{\eps,n,k}^{\gamma} = \left\{(x(t),y(t)) \in  X_{\eps,n,k} \mid (\ord x(t), \ord y(t))\in C_\gamma\right\},$$
endowed with its structural map to $(\ac \circ f^\eps)$ to $\Gm$ and we decompose  
\begin{equation} \label{decomparcs}
	X_{\eps,n,k} = \bigsqcup_{\gamma \in \N(f)} X_{\eps,n,k}^{\gamma}.
\end{equation}

\end{notation}

\begin{rem}\label{rem:n-inferieur-m(ord(x),ord(y)} 
   For any arc $(x(t),y(t))$ in $X^{\gamma}_{\eps,n,k}$, we can write
   $f(x(t),y(t))=t^{\m(\ord x(t),\ord y(t))}\tilde{f}(x(t),y(t),t)$
   with $\tilde{f}$ in $\k[x,y,u]$ and $\m$ the function defined in Proposition \ref{prop:dualfan-Newton-local} relatively to $\Delta(f)$.
   As $\ord f^\eps(x(t),y(t))=n$ we have $$\m(\ord x(t),\ord y(t)) \leq n/\eps=n\eps.$$
   {Two cases occur:} 
   \begin{itemize}
	   \item[$\bullet$]  $\eps n=\m(\ord(x),\ord(y))$, namely $n=\eps \m(\ord x, \ord y)$, if and only if $f_\gamma(\ac x, \ac y)\neq 0$.
	   \item[$\bullet$]  $\m(\ord(x),\ord(y))<n\eps$ if and only if $\dim \gamma =1$ and $f_\gamma(\ac x,\ac y) = 0$.
   \end{itemize}

\end{rem}

\begin{notations} \label{notationsconezeta}
	Let $\eps$ be in $\{\pm\}$ and $\gamma$ be a face of $\N(f)$. We introduce some notations.
	\begin{itemize}
		\item We consider the cones of $\mathbb R^2$ and $\mathbb R^3$
			$$ 
			C_{\eps,\gamma}^{\delta,=} = \left\{ \begin{array}{c|l} 
				                                 (\alpha,\beta) \in C_\gamma & 
				                                 \begin{array}{l} 
					                             0 < \alpha \leq \eps\m(\alpha,\beta) \delta,\: 0 < \beta
				                                 \end{array}
			                                     \end{array} 
						     \right\}  
			\:\:
			\text{and}
			\:\:
			C_{\eps,\gamma}^{\delta,<} = 
			\left\{
				\begin{array}{c|l} (n,\alpha,\beta) \in \mathbb R_{>0} \times
					C_\gamma & \begin{array}{l} 
						      \m(\alpha,\beta)<\eps n \\ 
						      0 < \alpha \leq n\delta,\: 0<\beta 
					           \end{array} 
				\end{array} 
			\right\}.
			$$ 

		\item For any $(\alpha,\beta)$ be in $C_{\eps,\gamma}^{\delta,=}$ and any $(n,\alpha,\beta)$ in $C_{\eps,\gamma}^{\delta,<}$, we consider the arc spaces
			$$ X_{\eps,\alpha,\beta}^{=} = 
			\left\{ 
				\begin{array}{c|c}
					(x(t),y(t))\in \mathcal L(\mathbb A^{2}_{\k}) & 
					\ord x(t) = \alpha,\: \ord y(t) = \beta,\:f_\gamma(\ac x(t),\ac y(t))\neq 0,\:\ord f^{\eps}(x(t),y(t))= \eps\m(\alpha,\beta)
				\end{array} 
			\right\} 
			$$ and 
			$$
			X_{\eps, n,\alpha,\beta}^{<} = 
			\left\{ 
				\begin{array}{c|c}
					(x(t),y(t))\in \mathcal L(\mathbb A^{2}_{\k}) 
					&
					{\ord x(t) = \alpha,\: \ord y(t) = \beta},\:f_\gamma(\ac x(t),\ac y(t))=0,\: \ord f^{\eps}(x(t),y(t))= n
				\end{array}
			\right\}
			$$
			endowed with $\ac f^{\eps}$, their structural map to $\Gm$

		\item  For any $\delta>0$ and any face $\gamma$ in $\N(f)$, we consider the motivic zeta function
			\begin{equation} \label{defZepsgamma} 
				Z^\delta_{\eps,\gamma,\omega}(T) = 
				\sum_{n\geq 1} ( \sum_{n\delta \geq k \geq 1} \mathbb L^{-(\nu-1)k} \mes(X_{\eps,n,k}^{\gamma})) T^n.
			\end{equation}
			Furthermore, if $\gamma$ is not horizontal we consider 
			\begin{equation} 
				Z_{\eps,\gamma,\omega}^{\delta,=}(T)=
				\sum_{n\geq 1}
				\sum_{\text{\tiny{
					$\begin{array}{c}
					(\alpha,\beta)\in C_{\eps,\gamma}^{\delta,=}\cap 
					\left(\mathbb N^{*}\right)^2 \\ n=\eps \m(\alpha,\beta)
				        \end{array}$
				      }}} 
				\mathbb L^{-(\nu-1)\alpha} \mes\left( X_{\eps, \alpha,\beta}^{=} 	\right)
				T^{n}
			\end{equation}
			and
			\begin{equation} \label{sommme<}
				Z_{\eps,\gamma,\omega}^{\delta,<}(T)=
				\sum_{n\geq 1}
				\sum_{\text{\tiny{
					          $\begin{array}{c} 
								  (\alpha,\beta) \in \left(\mathbb N^{*}\right)^2 \: s.t \\ 
								  (n,\alpha,\beta)\in C_{\eps,\gamma}^{\delta,<}\cap \mathbb N^3 
							  \end{array}$}
						  }}
				\mathbb L^{-(\nu-1)\alpha} \mes( X_{\eps,n,\alpha,\beta}^{<})T^{n}.
			\end{equation}
	\end{itemize}
\end{notations}

\begin{rem} If $\dim \gamma = 0$ then $Z_{\eps,\gamma, {\omega}}^{\delta,<}(T)=0$. \end{rem}
\begin{rem} \label{rem:sommefinie} As $\gamma$ {is not the} horizontal face, we observe that for any $n\geq 1$, the following sets are finite
	$$\{(\alpha,\beta) \in C_{\eps,\gamma}^{\delta,=} \mid n=\eps \m(\alpha,\beta)\}\:\text{and}\: \{(\alpha,\beta) \in C_{\gamma} \mid (n,\alpha,\beta) \in C_{\eps,\gamma}^{\delta,<}\}.$$ 
\end{rem}

\begin{prop} \label{prop:decomposition}
	For any $\delta>0$, for any $\eps$ in $\{\pm\}$, we have the decomposition 
\begin{equation}\label{formula:decomposition} 
	(Z_{f^\eps,\omega,x\neq 0}^{\delta}(T))_{((0,0),0)}  = \sum_{\gamma \in \N(f)} Z^\delta_{\eps,\gamma,\omega}(T) 
\end{equation}
with the following equality for any non horizontal face $\gamma$ in $\N(f)$
\begin{equation} \label{eq:decomposition-zeta-function}
	Z^{\delta}_{\eps,\gamma,\omega}(T) = Z_{\eps,\gamma,\omega}^{\delta,=}(T) + Z_{\eps,\gamma,\omega}^{\delta,<}(T).
\end{equation}
\end{prop}

\begin{proof} The proof follows from the additivity of the measure, using equality (\ref{decomparcs}) for (\ref{formula:decomposition}) and 
	Remark \ref{rem:n-inferieur-m(ord(x),ord(y)} for (\ref{eq:decomposition-zeta-function}).
\end{proof}

\paragraph{Rationality and limit of $Z^\delta_{\eps,\gamma_h, \omega}(T)$} \label{sec:Z_(N,0)}

In this subsection we study the case of the horizontal face $\gamma_h$. 
\begin{notation} If $x$ is a real number, we will denote by $[x]$ its integral part. We use $\N(f)^{\eps}$ defined in Notations \ref{Neps}.\end{notation}

\begin{prop}[Case of the horizontal face] \label{ex:cas-x^N} \label{ex:cas-x^N-cas2} 
	{Write $(a,b)$} the horizontal face $\gamma_h$ of $\mathcal{N}(f)$ with $a$ in $\mathbb Z$ and $b$ in $\mathbb N$. The dual cone of $\gamma_h$ is $C_{\gamma_h} = \mathbb R_{>0}(0,1)+\mathbb R_{>0}(p,q)$ with $(p,q)$ a primitive vector and $p\neq 0$. Let $N=ap+bq= (\gamma_h \mid (p,q))$.
	\begin{itemize}
		\item If $b=0$, then 
			\begin{itemize}
				\item if $\eps a >0$ then, there is $\delta_0>0$ such that, for any $\delta\geq \delta_0$,
					the motivic zeta function $Z^\delta_{\eps,\gamma_h, \omega}(T)$ is rational equal to 
					\begin{equation} \label{cas-face-horizontale}
						Z^\delta_{\eps,\gamma_h, \omega}(T) = [x^{\eps a}:\Gm\to \Gm,\sigma_{\Gm}]  R_{\gamma_h,\eps,\omega}^{\delta}(T) \:
					\end{equation}
					with $R_{\gamma_h,\eps,\omega}^{\delta}(T)$ computed in formula (\ref{gamma=(a,0)}) and converges to $-1$, when $T$ goes to $\infty$.
				\item if $\eps a \leq 0$, namely $\gamma_h \notin \N(f)^{\eps}$, then for any $\delta \geq 1$, 
					we have $Z^\delta_{\eps,\gamma_h, \omega}(T)=0$.

			\end{itemize}
		\item If $b\neq 0$, then we have
			\begin{equation}\label{cas-face-quelconcque}
				Z^\delta_{\eps,\gamma_h, \omega}(T) = [(x^ay^b)^{\eps}:\Gm^2\to \Gm,{\sigma_{\Gm^2}}]  R_{\gamma_h,\eps,\omega}^{\delta}(T),
			\end{equation}
			\begin{itemize}
				\item if ``$\eps=+$'' and $N>0$ then $R_{\gamma_h,\eps,\omega}^{\delta}(T)$ is computed in formula (\ref{Rgamma1}) and {converges to} 1.
				\item if ``$\eps=+$'' and $N\leq 0$, {namely $\gamma_h \notin \N(f)^{+}$}, then $R_{\gamma_h,\eps,\omega}^{\delta}(T)$ is computed in formula (\ref{Rgamma2}) and {converges to} 0.
			        \item if ``$\eps=-$'' and $N\geq 0$ then $R_{\gamma_h,\eps,\omega}^{\delta}(T)=0$.
				\item if ``$\eps=-$'' and $N<0$ then $R_{\gamma_h,\eps,\omega}^{\delta}(T)$ is computed in formula (\ref{Rgamma3}) and {converges to} 0.
			\end{itemize}
	\end{itemize}

\end{prop}

\begin{proof} 

	Assume first $b=0$. Then, by Remark \ref{rem:n-inferieur-m(ord(x),ord(y)}, for any arc $(x(t),y(t))$ with $(\ord x(t),\ord y(t))$ in $C_{\gamma_h}$ we have
	$$\ord f^{\eps}(x(t),y(t)) = \eps \m(x(t),y(t))= \eps a\ord x(t).$$ 
	\begin{itemize}
		\item[$\bullet$]  Assume $\eps a \leq 0$ then for any $(n,k)$ in 
			$\mathbb N^{*2}$, the set $X_{n,k,\gamma_h}^\delta$ is empty and $Z_{\eps,\gamma_h}^{\delta}(T)=0$.
		\item[$\bullet$]  Assume $\eps a>0$ and $\delta>1$. As $a\eps$ is an integer, the condition $\delta \ord f^{\eps}(x(t),y(t)) \geq \ord x(t)$ is satisfied.
			Then applying {formula (\ref{defZepsgamma})}, we have 
			$Z^\delta_{\eps,\gamma_h, \omega}(T) = \sum_{k\geq 1}\mathbb L^{-(\nu-1)k} \mes( X_{\eps,\eps ak,k}^{\gamma_h})T^{\eps a k}$.
			We prove now formula (\ref{cas-face-horizontale})
			with
			\begin{equation} \label{gamma=(a,0)}
				R_{\gamma_h,\eps,\omega}^{\delta}(T) = 
				-1+\frac{1}{1-\mathbb L^{-\nu p - q}T^{p\eps a}} \sum_{r=0}^{p-1} \mathbb L^{-\nu r - [qr/p]}T^{r\eps a}.
			\end{equation}
			A couple $(k,l)$ in $(\mathbb N^{*})^2$ belongs to $C_{\gamma_h}$ if and only if $pl>qk$. Then, for any $k\geq 1$ we have
			$$X_{\eps,\eps ak,k}^{\gamma_h} = \{(x(t),y(t))\in \mathcal L(\mathbb A^2_{\k}) \mid \ord x(t)=k,\: p.\ord y(t) >qk, \text{namely}\: \ord y(t) \geq [qk/p]+1\}$$
		and by definition of the motivic measure, we get
		$ \mes(X_{\eps, \eps ak,k}^{\gamma_h}) = [x^{\eps a}:\Gm \ra \Gm, \sigma_{\Gm}]\mathbb L^{-[qk/p]-k}$ and
		$$Z_{\eps,\gamma_h, \omega}^{\delta}(T)=[x^{\eps a} : \Gm \ra \Gm, \sigma_{\Gm}] \sum_{k\geq 1} \mathbb L^{-\nu k - [qk/p]}T^{k\eps a}.$$
			Let $k\geq 1$, there is a unique integer $r$ {in} $\{0,\dots,p-1\}$ such that $k=[k/p]p+r$. 
			There is also a unique integer $\beta_r$ in $\{0,\dots,p-1\}$ such that 
			$qr=[qr/p]p+\beta_r$ implying $[qk/p]=[k/p]q+[qr/p]$. We obtain equality (\ref{gamma=(a,0)}), writing $k$ as $lp+r$ with $l$ {in} $\mathbb N$ and $r$ {in} $\{0,\dots,p-1\}$, and  by decomposition of $Z_{\eps,\gamma_h, \omega}^{\delta}(T)$ as the sum of $p$ formal series 
			$$Z_{\eps,\gamma_h, \omega}^{\delta}(T)= 
			[x^{\eps a} : \Gm \ra \Gm, \sigma_{\Gm}] \left(-1+\sum_{r=0}^{p-1} \sum_{l\geq 0} \mathbb L^{-\nu(lp+r)-lq-[qr/p]}T^{a\eps(lp+r)}\right).$$
	\end{itemize}
        Assume $b\neq 0$, by Remark \ref{rem:n-inferieur-m(ord(x),ord(y)}, for any arc $(x(t),y(t))$ with $(\ord x(t),\ord y(t))$ in $C_{\gamma_h}$, we have
	$$\ord f^{\eps}(x(t),y(t)) = \eps \m(x(t),y(t))= \eps (a\ord x(t)+b\ord y(t)).$$
	In particular, the motivic zeta function can be written as
	$$Z^\delta_{\eps,\gamma_h, \omega}(T) = 
	\sum_{n\geq 1} \sum_{\text{\tiny{$\begin{array}{c}(k,l)\in C_{\eps,\gamma_h}^{\delta,=} \cap (\mathbb N^{*})^2 \\ n=\eps(ak+bl) \end{array}$}}} 
	\mathbb L^{-(\nu-1)k} \mes\left( X_{\eps,k,l}\right)T^{n}$$
	with $X_{\eps,k,l}=\{(x(t),y(t))\in \mathcal L(\mathbb A^2_k)\mid \ord x(t)=k, \ord y(t)=l\}$
	with its structural map, $(x(t),y(t)) \mapsto \ac(x(t)^{\eps a}y(t)^{\eps b})$,
	and 
	\begin{equation}\label{defCepsgammahdelta} C_{\eps,\gamma_h}^{\delta,=}=\{(k,l)\in (\mathbb R_{>0})^2 \mid l>kq/p,\:\eps(ak+bl)\delta \geq k \} \subseteq C_{\gamma_h}.\end{equation}
	For any $(k,l)$ in $C_{\gamma_h} \cap (\mathbb N^*)^2$, for any integer $m\geq \m(k,l)$, the $m$-jet space 
	$\pi_m(X_{\eps,k,l})$ with the canonical $\Gm$-action is a bundle over $((x^ay^b)^\eps:\Gm^2\to \Gm,\sigma_{k,l})$ with fiber $\mathbb A^{2m-k-l}$ and for any $(\lambda,x,y)$ in $\Gm^3$, $\sigma_{k,l}(\lambda,(x,y))=(\lambda^{k} x,\lambda^{l} y)$.
	By definition of the motivic measure and Remark \ref{notation-identification-action}, we get
	$\mes\left( X_{\eps,k,l}\right) = \mathbb L^{-k-l}[(x^ay^b)^\eps:\Gm^2\to \Gm,\sigma_{\Gm^2}]$ and formula (\ref{cas-face-quelconcque}) with
	$$
	R_{\gamma_h,\eps,\omega}^{\delta}(T) = 
	\sum_{n\geq 1} \sum_{\text{\tiny{$\begin{array}{c}(k,l)\in C_{\eps,\gamma_h}^{\delta,=} \cap (\mathbb N^*)^2 \\ n=\eps(ak+bl) \end{array}$}}} 
	\mathbb L^{-\nu k-l}T^{n}.$$

	\begin{itemize}
		\item Assume ``$\eps=+$'' and $N=ap+bq>0$, namely $\frac{q}{p}>-\frac{a}{b}$, then there is $\delta_0>0$ such that for any $\delta \geq \delta_0$,
			$\frac{q}{p} \geq \frac{1-a\delta}{b\delta}$ then, for any $(k,l)$ in $C_{\gamma_h}$, we have 
			$\delta(ak+bl)>\delta k (a+bq/p) \geq \delta k(a+(1-a\delta)/\delta) = k$
			inducing the equality
			$C_{\eps,\gamma_h}^{\delta,=}=C_{\gamma_h} = \mathbb R_{>0}(0,1) + \mathbb R_{>0}(p,q).$
			Then, applying Lemma \ref{lemmerationalitedescones} formula (\ref{casconedim2}), we obtain
			\begin{equation} \label{Rgamma1} 
				R_{\gamma_h,\eps,\omega}^{\delta}(T) =  \sum_{(k_0,l_0) \in \mathcal P}
				\frac{\mathbb L^{-\nu k_0-l_0}T^{a k_0 + b l_0}}
				{(1-\mathbb L^{-1}T^{b})(1-\mathbb L^{-\nu p - q}T^{ap+bq})}
			\end{equation}
		with $\mathcal P =( ]0,1](0,1) + ]0,1](p,q) )\cap \mathbb N^2$. We have $\lim_{T \ra \infty} R_{\gamma_h,\eps,\omega}^{\delta}(T) = 1$.

	\item Assume ``$\eps=+$'' and $N=ap+bq \leq 0$, namely $\frac{q}{p} \leq -\frac{a}{b}$ (this case occurs only for $a<0$), then for any $\delta > 0$,
		$\frac{1-a\delta}{b\delta} > \frac{q}{p}$ and in that case, we remark that 
		$C_{\eps,\gamma_h}^{\delta,=} = \left\{ 
			                                \begin{array}{c|c}
								 (k,l)\in (\mathbb R_{>0})^2 &
								 l \geq k \left(\frac{1-\delta a}{\delta b} \right)
							 \end{array}
					          \right\}  
						  = \mathbb R_{\geq 0}(0,1) + \mathbb R_{>0}{\omega^\delta}
		$ 
		{with $\omega^\delta = (b\delta,1-a\delta)$}. Then, applying Lemma \ref{lemmerationalitedescones} {formula (\ref{casconedim2-bis})}, 
		we obtain
		\begin{equation}\label{Rgamma2}
			R_{\gamma_h,\eps,\omega}^{\delta}(T)  =  
			\sum_{(k_0,l_0) \in \mathcal P}
			\frac{\mathbb L^{-\nu k_0-l_0}T^{a k_0 + b l_0}}
			{(1-\mathbb L^{-1}T^{b})(1-\mathbb L^{-((\nu,1)\mid \omega^\delta)}T^{((a,b)\mid \omega^\delta)})} 
			+  
			\frac{\mathbb L^{-((\nu,1)\mid(\omega^\delta))}T^{((a,b)\mid \omega^\delta))}}
			{1-\mathbb L^{-((\nu,1)\mid(\omega^\delta))}T^{((a,b)\mid \omega^\delta))}}
		\end{equation} 
	with $\mathcal P =( ]0,1](0,1) + ]0,1] \omega^\delta )\cap \mathbb N^2$. We have $\lim_{T \ra \infty} R_{\gamma_h,\eps,\omega}^{\delta}(T)=0$.

\item Assume ``$\eps=-$''. Then, by formula (\ref{defCepsgammahdelta}) we have,
	$C_{\eps,\gamma_h}^{\delta,=}=\{(k,l)\in (\mathbb R_{>0})^2 \mid -k(1+a \delta)/(b\delta) \geq l > kq/p \}.$
	Remark that for any $\delta>0$, we have the inequality $-a/b>-1/(b\delta)-a/b.$
	\begin{itemize}
		\item If $N=qb+ap\geq 0$, namely $\frac{q}{p}\geq -\frac{a}{b}$, then the cone $C_{\eps,\gamma_h}^{\delta,=}$ is empty for any $\delta>0$, and 
			$R_{\gamma_h,\eps,\omega}^{\delta}(T)  = 0$.
		\item  If $N=qb+ap<0$, namely $\frac{q}{p}<-\frac{a}{b}$, then there is $\delta_0>0$ such that for any $\delta\geq \delta_0$, we have 
			$-(1+a \delta)/(b\delta) > q/p$ and in that case 
			$C_{\eps,\gamma_h}^{\delta,=}=\mathbb R_{\geq 0}(p,q) + \mathbb R_{>0}{\omega_\delta}$
			{with $\omega^\delta = (b\delta,-(1+a\delta))$} and by Lemma \ref{lemmerationalitedescones} {formula (\ref{casconedim2-bis})} we have 
			\begin{equation} \label{Rgamma3}
				R_{\gamma_h,\eps,\omega}^{\delta}(T)  =  
				\sum_{(k_0,l_0) \in \mathcal P}
				\frac{\mathbb L^{-\nu k_0-l_0}T^{a k_0 + b l_0}}
				{(1-{\mathbb L^{-((\nu,1)\mid (p,q))}T^{((a,b)\mid (p,q))})(1-\mathbb L^{-((\nu,1)\mid \omega^\delta)}}T^{((a,b)\mid \omega^\delta})} 
				+  
				\frac{\mathbb L^{-((\nu,1)\mid(\omega^\delta))}T^{((a,b)\mid \omega^\delta))}}
				{1-\mathbb L^{-((\nu,1)\mid(\omega^\delta))}T^{((a,b)\mid \omega^\delta))}}
			\end{equation} 
		with $\mathcal P =( ]0,1]{(p,q)} + ]0,1] \omega^\delta )\cap \mathbb N^2$. We have $\lim_{T \ra \infty} R_{\gamma_h,\eps,\omega}^{\delta}(T)=0$.
\end{itemize}
              \end{itemize}
\end{proof}

\paragraph{Rationality and limit of $Z_{\eps,\gamma}^{\delta,=}(T)$ for $\gamma$ non horizontal} \label{sec:rationalite-limit-Zgamma=}
		
We use the subset $\N(f)^\eps$ of $\N(f)$ defined in Notations \ref{Neps} and the {actions} $\sigma_\gamma$ defined in {Remark}
\ref{notation-identification-action}. Similarly to \cite{Gui02a}, \cite{GuiLoeMer05a} and \cite{Rai11} we have

\begin{prop} \label{lem:zeta=} 
	        Let $\gamma$ be a non horizontal face of $\N(f)$. The motivic zeta function $Z_{\eps,\gamma,\omega}^{\delta,=}(T)$ is rational and for $\delta$ large enough,
		we have the convergence
		$$\lim_{T \ra \infty} Z_{\eps,\gamma, \omega}^{\delta,=}(T) =
				s_\gamma^{\eps}(-1)^{\dim \gamma} \left[f_{\gamma}^{\eps} : \Gm^{2}\setminus (f_\gamma = 0) \ra \Gm, \sigma_{\gamma}\right]  \in \mgg$$
				with $s_\gamma^\eps=1$ if $\gamma$ is a face in $\N(f)^\eps$, otherwise $s_\gamma^\eps=0$.
	More precisely, there is a rational function $R_{\gamma,\eps,\omega}^{\delta}(T)$ such that
	\begin{equation} \label{eqa}
		Z_{\eps,\gamma,\omega}^{\delta,=}(T) = 
		[f_\gamma^{\eps}:\Gm^2 \setminus (f_\gamma=0) \to \Gm, \sigma_{\gamma}]
		R_{\gamma,\eps,\omega}^{\delta}(T).
	\end{equation}
	\begin{itemize}
		\item If $\gamma$ is a zero dimensional face $(a,b)$ in $\mathbb Z \times \mathbb N$, with  
			$C_\gamma = \mathbb R_{>0} \omega_1 + \mathbb R_{>0} \omega_2$ with $\omega_1$ and $\omega_2$ primitive vectors in $\mathbb N^*\times \mathbb N$, 
				\begin{itemize}
				\item if $\gamma$ {belongs to} $\N(f)^\eps$, namely $\eps((a,b)\mid \omega_1)>0$ and $\eps((a,b)\mid \omega_2)>0$,
					$R_{\gamma,\eps,\omega}^{\delta}(T)$ is a rational function computed in {formula (\ref{Rgammacas2})} which does not depend on $\delta$ large enough and converges to 1,	
				\item if $\eps((a,b)\mid \omega_1)\leq 0$ and $\eps((a,b)\mid \omega_2) \leq 0$, then for any 
					$\delta>0$, $R_{\gamma,\eps}^{\delta}(T)=0$,
				\item otherwise, for any $\delta>0$, $R_{\gamma,\eps,\omega}^{\delta}(T)$ is a rational function {as in formula (\ref{Rgammacas3})} 
				and converges to $0$.
			\end{itemize}

		\item If $\gamma$ is a one dimensional face supported by a line of equation $ap+bq=N$ with dual cone $C_\gamma=\mathbb R_{>0}(p,q)$ then,
			\begin{itemize}
				\item if $\gamma$ {does not belong to} $\N(f)^\eps$, namely $\eps N\leq 0$, then $R_{\gamma,\eps,\omega}^{\delta}(T)=0$.
				\item if $\gamma$ {belongs to} $\N(f)^\eps$, namely $\eps N>0$, then for $\delta \geq \frac{p}{\eps N}$, 
					$R_{\gamma,\eps,\omega}^{\delta}(T)$ is computed in formula (\ref{Rgammacas1}), does not depend on $\delta$ and converges to $-1$.
			\end{itemize}
	\end{itemize}
\end{prop}

\begin{proof} 
	Let $\gamma$ be a non horizontal face of $\N(f)$. We use Notations \ref{notationsconezeta}. 
	For any $(\alpha,\beta)$ in $C_\gamma \cap (\mathbb N^*)^2$, for any integer $m\geq \m(\alpha,\beta)$, the $m$-jet space 
	$\pi_m(X_{\eps,\alpha,\beta}^{=})$ with the canonical $\Gm$-action is a bundle over $(f_\gamma:\Gm ^2\setminus (f_\gamma=0) \to \Gm,\sigma_{\alpha,\beta})$ with fiber $\mathbb A^{2m-\alpha-\beta}$ and for any $(\lambda,x,y)$ in $\Gm^3$, $\sigma_{\alpha,\beta}(\lambda,(x,y))=(\lambda^\alpha x,\lambda^\beta y)$. Then, by definition of the motivic measure
	$\mes(X_{\eps,\alpha,\beta}^{=}) = \mathbb L^{-\alpha -\beta}[(f_\gamma:\Gm ^2\setminus (f_\gamma=0) \to \Gm,\sigma_{\alpha,\beta})]$ in $\mgg$.
	Then, using Remark \ref{notation-identification-action} we have formula \ref{eqa}, with 
	\begin{equation} 
		R_{\gamma,\eps,\omega}^{\delta}(T)=\sum_{n\geq 1}
		\sum_{\text{\tiny{$\begin{array}{c}(\alpha,\beta)\in C_{\eps,\gamma}^{\delta,=}\cap \left(\mathbb N^{*}\right)^2 \\ n=\eps \m(\alpha,\beta)\end{array}$}}} 
		\mathbb L^{-\nu\alpha -\beta} T^{n}.
	\end{equation}
	The {proof of Proposition \ref{lem:zeta=}} follows from Lemma \ref{lemmerationalitedescones} using the following description of the cone $C_{\eps,\gamma}^{\delta}$. 
	\begin{itemize}
		\item  {Assume $\gamma$ is a zero dimensional face equal to $(a,b)$ with $b>0$.}
			The associated cone $C_\gamma$ has dimension 2. It can be
			described as $C_\gamma =\mathbb R_{>0}\omega_1 + \mathbb R_{>0} \omega_2$ where $\omega_1$ and $\omega_2$ are the primitive
			normal vectors of adjacent faces $\gamma_1$ and $\gamma_2$ of $\gamma$, elements of $\mathbb N^* \times \mathbb N$, because $\gamma$ is not horizontal. {For any $(\alpha,\beta)$ in $C_\gamma$, we have $\m(\alpha,\beta)=(\gamma \mid (\alpha,\beta))$.}
			\begin{itemize}

				\item  Assume $\eps(\omega_1 \mid \gamma)>0$ and $\eps(\omega_2 \mid \gamma)>0$, namely $\gamma \in \N(f)^\eps$. 
					Let $\delta$ be a positive real number satisfying
					$$\delta \geq \max\left(\frac{((1,0)\mid \omega_1)}{\eps (\gamma\mid \omega_1)} , \frac{((1,0)\mid \omega_2)}{\eps (\gamma\mid \omega_2)} \right).	$$
					In that case the cone $C_{\eps,\gamma}^{\delta,=}$ is equal to $C_\gamma$.
					Indeed, let $(\alpha,\beta)$ be in $C_\gamma$. 
					There are real numbers $\lambda$ and $\mu$ {in $\mathbb R_{>0}$} such that 
					$(\alpha, \beta) = \lambda \omega_1 + \mu \omega_2$. 
					Then, the inequalities defining the cone $C_{\eps,\gamma}^{\delta,=}$ are satisfied:
					$$\text{$\eps\m(\alpha,\beta) = \eps((\alpha,\beta)\mid \gamma) > 0$ and 
					$\delta\eps\m(\alpha,\beta) = \eps\delta(\lambda (\gamma\mid \omega_1) + \mu(\gamma\mid \omega_2)) \geq ((1,0)\mid (\alpha,\beta)) = \alpha$.}
					$$
					Then, applying Lemma \ref{lemmerationalitedescones} we obtain 
					\begin{equation} \label{Rgammacas2} 
						R_{\gamma,\eps,\omega}^{\delta}(T)=\sum_{(\alpha_0,\beta_0)\in \mathcal P_\gamma}
						\frac{\mathbb L^{-((\nu,1)\mid (\alpha_0 , \beta_0))}T^{\eps((a,b)\mid(\alpha_0,\beta_0))}}
						{(1-\mathbb L^{-((\nu,1)\mid \omega_1)}T^{\eps((a,b)\mid \omega_1)})
						(1-\mathbb L^{-((\nu,1)\mid \omega_2)}T^{\eps((a,b)\mid \omega_2)})}
					\end{equation}
					        with $\mathcal P_\gamma = (]0,1]\omega_1 + ]0,1]\omega_2)\cap \mathbb N^2$,
						and $\lim_{T\ra \infty}R_{\gamma,\eps,\omega}^{\delta}(T)=1$.

				\item Assume $\eps(\gamma\mid \omega_2)>0$ and $(\gamma\mid \omega_1)=0$. 
					The case $\eps(\gamma\mid \omega_1)>0$ and $(\gamma\mid \omega_2)=0$ is symmetric{al}.
					By definition we have $C_{\eps,\gamma}^{\delta,=} = C_\gamma \cap L_{\delta}^{-1}(\mathbb R_{\geq 0})$
					with $L_{\delta}:(\alpha,\beta)\mapsto \eps\delta(\gamma\mid (\alpha,\beta))-\alpha.$
					For any $\delta > \frac{(\omega_2 \mid (1,0))}{\eps(\gamma\mid \omega_2)}$, we have
					{$$L_{\delta}(\omega_1)= -(\omega_1 \mid (1,0)) <0\: \text{and} \: L_\delta(\omega_2)>0.$$}
					Let $\omega^{\delta}=\frac{L_\delta(\omega_2)}{(\omega_1\mid(1,0))}\omega_1 + \omega_2$ element of $C_\gamma$.
					As $L_\delta(\omega^\delta)=0$ and $L_{\delta}^{-1}(\mathbb R_{\geq 0})$ is a half-plane, 
					we can conclude $C_{\eps,\gamma}^{\delta,=}=\mathbb R_{\geq 0}\omega_2 + \mathbb R_{>0}\omega^\delta.$
					Then, applying Lemma \ref{lemmerationalitedescones} we obtain
					\begin{equation} \label{Rgammacas3}
						R_{\gamma,\eps,\omega}^{\delta}(T)=\sum_{(\alpha_0,\beta_0)\in \mathcal P_\gamma^{\delta}}
						\frac{\mathbb L^{-\nu \alpha_0 - \beta_0}T^{\eps((a,b)\mid(\alpha_0,\beta_0))}}
						{(1-\mathbb L^{-((\nu,1)\mid \omega^{\delta})}T^{\eps((a,b)\mid \omega^{\delta})})
						(1-\mathbb L^{-((\nu,1)\mid \omega_2)}T^{\eps((a,b)\mid \omega_2)})} 
						+ \frac{\mathbb L^{-((\nu,1)\mid {\omega^\delta})}T^{\eps((a,b)\mid {\omega^\delta})}}
						{1-\mathbb L^{-((\nu,1)\mid {\omega^\delta})}T^{\eps((a,b)\mid {\omega^\delta})}}
					\end{equation}
					with $\mathcal P_\gamma^{\delta}= (]0,1]\omega^\delta + ]0,1]\omega_2)\cap \mathbb N^2$ 
					and $\lim R_{\gamma,\eps}(T) = 0$.

				\item Assume $\eps(\gamma\mid \omega_2)>0$ and $\eps(\gamma\mid \omega_1)<0$. 
					The case $\eps(\gamma\mid \omega_1)>0$ and $\eps(\gamma\mid \omega_2)<0$ is symmetric{al}.
					By convexity, annulling the linear form $(\gamma\mid .)$, the vector $\omega=(b,-a)$, 
					is an element of the cone $C_\gamma$. We observe that any element $u$ in 
					$\mathbb R_{>0}\omega + \mathbb R_{\geq 0}\omega_1 \subset C_\gamma$ does not belong to 
					$C_{\eps,\gamma}^{\delta,=}$, because $\eps \m(u)<0$. Then, we conclude that 
					$$C_{\eps,\gamma}^{\delta,=} = 
					\{(\alpha,\beta) \in \mathbb R_{> 0}\omega_2 + \mathbb R_{>0}\omega \mid 0< \alpha < \eps \m(\alpha,\beta)\delta,\:0<\beta\}.$$ 
					As $\eps(\omega_2 \mid \gamma)>0$ and $\eps(\omega\mid \gamma)=0$ we can apply the previous point.

				\item Assume $\eps(\gamma\mid \omega_2)\leq 0$ and $\eps(\gamma\mid \omega_1)\leq0$.
				{Then $C_{\eps,\gamma}^{\delta,=}$ is empty, because for any $(\alpha,\beta)$ in $C_\gamma$, $\eps\m(\alpha,\beta) \leq 0$.}
				
			\end{itemize}
		\item {Assume $\gamma$ is a one dimensional compact face} supported by a line with
			equation $mp+nq=N$ with $(p,q)$ non negative integers and $gcd(p,q)=1$.
			In particular in that case we have $C_\gamma = \mathbb R_{>0}(p,q)$. 
			Let $(a,b)$ be a point in $\gamma$ with integral coordinates. 
			Then, for any $(kp,kq)$ in $C_\gamma$ we have $\m(kp,kq) = ((a,b)\mid (kp,kq)) = kN$ 
			and 
			$$C_{\eps,\gamma}^{\delta,=} = \{(kp,kq)\in C_\gamma\mid 1\leq kp \leq \eps kN \delta \}.$$
			In particular, under the condition
			$\eps N\leq 0$ the set $C_{\eps,\gamma}^{\delta,=}$ is empty, otherwise for any $\delta \geq \frac{p}{\eps N}$, the set $C_{\eps,\gamma}^{\delta,=}$ is equal to $C_\gamma$. Applying Lemma \ref{lemmerationalitedescones} we obtain the following expression of
			$R_{\gamma,\eps,\omega}^{\delta}(T)$ implying its convergence to $-1$
			\begin{equation} \label{Rgammacas1} 
				R_{\gamma,\eps,\omega}^{\delta}(T) = \frac{\mathbb L^{-(\nu p +q)}T^{\eps N}}{1-\mathbb L^{-(\nu p +q)}T^{\eps N}}.
			\end{equation}
				\end{itemize}
\end{proof}

\paragraph{Rationality and limit of $Z_{\eps,\gamma,\omega}^{\delta,<}(T)$ for a one-dimensional face $\gamma$} \label{sec:rationalite-limite-Zgamma<}

\begin{rem} \label{nulliteZf}  Let $\gamma$ be a one dimensional face in $\N(f)$ and assume ``$\eps=-1$''. The face $\gamma$ is supported by a line of equation $mp+nq=N$
	with $(p,q)$ in $(\mathbb N^*)^2$ and $gcd(p,q)=1$. If $N\geq 0$, namely $\gamma$ does not belong to $\N(f)^{-}$ (Remark \ref{Nfepsfacedim1}), then the cone
        $C_{-,\gamma}^{\delta,<} = 
	\{(n,kp,kq)\in \mathbb R_{>0}\times C_\gamma \mid kN< -n,\: 0<kp\leq n\delta\}$ 
	is empty and $Z^{\delta,<}_{-,\gamma,\omega}(T)=0$. This is the reason why the second summation in {equations (\ref{resS1/fb0}) and (\ref{res1/f}) is only on} $\N(f)^{-}$.
\end{rem}

\begin{rem} By additivity of the measure, by equation (\ref{sommme<}) and Definition \ref{def:zero-partie-initiale} and Remark \ref{rem:sommefinie}, for any one dimensional face $\gamma$ in $\N(f)$,
the motivic zeta function $Z_{\eps,\gamma}^{\delta,<}(T)$ has the following
decomposition 
$$Z_{\eps,\gamma,\omega}^{\delta,<}(T) = \sum_{\mu \in R_{\gamma}}
\sum_{(n,\alpha,\beta)\in C_{\eps,\gamma}^{\delta,<}\cap \mathbb N^3} \mathbb L^{-(\nu-1)\alpha}\mes(X_{(n,\alpha,\beta),\mu}) T^n
$$ 
where
$X_{(n,\alpha,\beta),\mu} = 
\left\{ \begin{array}{c|l}
	\left(x(t),y(t)\right) \in \mathcal L(\mathbb A^2_{\k})   
	& \begin{array}{l} 
		\ac(y(t))^p = \mu \ac(x(t))^q, \ord x(t)=\alpha,\: \ord y(t)=\beta,\: \ord f^{\eps}(x(t),y(t))=n
	\end{array} 
\end{array} 
	  \right\}.
$
\end{rem}
\begin{prop} \label{prop:mesXnmu} 
	Let $\gamma$ be a one dimensional face of $\N(f)$, let $\mu$ be a root of $f_\gamma$ and ${\sigma_{(p,q,\mu)}}$ the induced Newton transform 
	(Definition \ref{defn:Newton-map-local}). 
	For any $k>0$, denoting $\alpha = pk$ and $\beta = qk$, we have
	$$\mes(X_{(n,\alpha,\beta),\mu})=\mathbb L^{-(p+q-1)k}
	\mes(Y_{(n,k)}^{{\sigma_{(p,q,\mu)}}}) $$
	with 
	$ Y_{(n,k)}^{{\sigma_{(p,q,\mu)}}} = 
	\left\{ 
		\begin{array}{c|c}
			\left(v(t),w(t)\right) \in \mathcal L(\mathbb A^2_k) &
			\begin{array}{l} 
				\ord v(t)=k,\:\ord w(t)>0,\: \ord f^{\eps}_{{\sigma_{(p,q,\mu)}}}(v(t),w(t)) = n 
			\end{array}
		\end{array} 
	\right\}.
	$ 
\end{prop} 
\begin{proof}
The proof is similar to that of \cite[Lemma 3.3]{CassouVeys13} (see also \cite[Proposition 6]{carai-antonio} or the
proof of Proposition \ref{prop:mesXnmu-infini-infini} below).
\end{proof}
\begin{rem} 
	Let $\gamma$ be a one dimensional face of $\N(f)$ supported by a line of equation 
	$ap+bq=N$. Let $\mu$ be a root of $f_\gamma$ and ${\sigma_{(p,q,\mu)}}$ the induced Newton
	transform. By Lemma \ref{lem:Newton-alg}, there is a polynomial $\tilde{f}_{{\sigma_{(p,q,\mu)}}}$ in
	$\k[v,w]$ such that
	$$f_{{\sigma_{(p,q,\mu)}}}(v,w)=v^{N}\tilde{f}_{{\sigma_{(p,q,\mu)}}}(v,w).$$ 
	In particular the Newton transform $f_{{\sigma_{(p,q,\mu)}}}$ satisfies the same type of conditions on $f$, {defined in subsection} \ref{setting}, and the motive 
	 $(S_{f^{\eps}_{{\sigma_{(p,q,\mu)}}}, \omega_{p,q,\mu}, v\neq 0})_{((0,0),0)}$ is well defined.
\end{rem}

\begin{prop}\label{prop:rationalityZ<} 
	Let $\gamma$ be a one dimensional face of $\N(f)$ supported by a line of equation $ap+bq=N$ with $p$ and $q$ in $(\mathbb N^*)^2$ and coprime.
	{With Notations \ref{notationsactionsformdiff}}, for $\delta$ large enough, the motivic zeta function
	$Z_{\eps,\gamma,\omega}^{\delta,<}$ can be decomposed as 
	\begin{equation} \label{eqZdelta<} 
		Z^{\delta,<}_{\eps,\gamma,\omega}(T) = 
		\sum_{\mu \in R_\gamma} (Z_{f_{{\sigma_{(p,q,\mu)}}}^{\eps}, \omega_{p,q,\mu}, v\neq 0}^{\delta/p})_{((0,0),0)},
	\end{equation}
	In particular we have,
	$$-\lim_{T \ra \infty}	
	(
	Z_{f^{\eps}_{{\sigma_{(p,q,\mu)}}},\omega_{p,q,\mu}, v\neq 0}^{\delta/p}(T)
	)_{((0,0),0)} = 
	(S_{f^{\eps}_{{\sigma_{(p,q,\mu)}}},v\neq 0})_{((0,0),0)} \in \mgg.
	$$
	Furthermore, by Remark \ref{nulliteZf}, if ``$\eps=-$'' then $Z_{-,\gamma,\omega}^{\delta,<}=0$ for all face $\gamma$ not in $\N(f)^{-}$.
\end{prop}

\begin{proof} Let $\gamma$ be a one dimensional face of $\N(f)$.
	For any element $(\alpha,\beta)$ in $C_\gamma$ there is $k>0$ such that
	$\alpha =pk$ and $\beta = qk$. 
	Note that $\m(\alpha,\beta) = \m(pk,qk)=k{N}$.
	By Notations \ref{notationsconezeta}, we remark that the cone
	$$ C_{\eps,\gamma}^{\delta,<} \cap \mathbb N^{3} =
	\left\{ 
		\begin{array}{c|l} 
		   (n,\alpha,\beta) \in \mathbb R_{>0} \times C_\gamma & 
		   \begin{array}{l} 
			\m(\alpha,\beta) < \eps n, \:	1 \leq \alpha \leq n\delta 
		  \end{array}
	        \end{array} 
        \right\} \cap \mathbb N^3,
        $$ 
is in bijection with the set $\overline{C}_{\eps,\gamma}^{\delta,<} \cap \mathbb N^2$ with
$\overline{C}_{\eps,\gamma}^{\delta,<} = 
\left\{ 
	\begin{array}{c|l} 
		(n,k) \in \mathbb R_{>0}^2 & 
		\begin{array}{l} 
			\m(pk,qk) < \eps n, \: 	1 \leq pk \leq n\delta 
		\end{array} 
	\end{array}
\right\}.
$
Using this notation we prove equality {(\ref{eqZdelta<})}. 
For any integer $n$, we consider 
$$
\text{$C_{\eps,\gamma,n}^{\delta,<}=\left\{(\alpha,\beta)\in (\mathbb R_{>0})^2 \mid (n,\alpha,\beta) \in C_{\eps,\gamma}^{\delta,<}\right\}$ and 
$\overline{C}_{\eps,\gamma,n}^{\delta,<}=\left\{k\in \mathbb R_{>0} \mid (n,k) \in \overline{C}_{\eps,\gamma}^{\delta,<}\right\}.$}
$$
\noindent Using Proposition \ref{prop:mesXnmu} we have 
$$ \begin{array}{ccl} 
	Z^{\delta,<}_{\eps,\gamma,\omega}(T) & = & 
	\sum_{\mu \in R_\gamma}  
	\sum_{n\geq 1} 
	\sum_{(\alpha,\beta) \in C_{\eps,\gamma,n}^{\delta,<}\cap \mathbb N^2}
	\mathbb L^{-(\nu-1)pk}\mes(X_{(n,\alpha,\beta),\mu})T^n \\ 
		& = & \sum_{\mu \in R_\gamma} 
		      \sum_{n\geq 1} 
		      \sum_{k \in {\overline C}_{\eps,\gamma,n}^{\delta,<} {\cap \mathbb N}}
		      \mathbb L^{-(\nu-1)pk} \mathbb L^{-(p+q-1)k} \mes (Y_{(n,k)}^{{\sigma_{(p,q,\mu)}}})T^{n},
    \end{array} 
$$ 
but using the definition of the zeta function {in formula (\ref{zeta})} we have  
	$$ \begin{array}{ccl}
		(
		Z_{f^{\eps}_{{\sigma_{(p,q,\mu)}}},\omega_{p,q,\mu},v\neq 0}^{\delta/p}(T)
		)_{((0,0),0)}& = & 
		\sum_{n \geq 1} (\sum_{n\delta/p \geq k \geq 1}  \mathbb L^{-(\nu p+q-1)k}
		\mes(Y_{n,k}^{\delta/p,{\sigma_{(p,q,\mu)}}}))T^{n} \\  
		& = & \sum_{n\geq 1} 
		      \sum_{k \in {\overline C}_{\eps,\gamma,n}^{\delta,<}}
		\mathbb L^{-(\nu p+q-1)k}
		\mes({Y_{(n,k)}^{{\sigma_{(p,q,\mu)}}}})T^{n}
	\end{array}
	$$ 
	with  
	$Y_{n,k}^{\delta/p, {\sigma_{(p,q,\mu)}}} = 
	\left\{
		\begin{array}{c|c} 
			\left(v(t),w(t)\right) \in \mathcal L(\mathbb A^2_k) &
			\begin{array}{l} 
				(v(0),w(0))=0, \:
				\ord v(t) = k \leq n\delta /p \\
				\ord f^{\eps}_{{\sigma_{(p,q,\mu)}}} (v(t),w(t)) = n, \:
				\ord \omega_{p,q,\mu}(v(t),w(t)) = (\nu p+q-1)k 
			\end{array}
		\end{array} 
	\right\}.
	$\\
	In particular we can conclude thanks to  section \ref{section:mot-zeta-omega},
	for $\delta$ large enough, that the motivic zeta function is rational and
	has a limit independent on $\delta$ when $T$ goes to infinity
	$$-\lim_{T \ra \infty}	
	(
	Z_{f^{\eps}_{{\sigma_{(p,q,\mu)}}},\omega_{p,q,\mu}, v\neq 0}^{\delta/p}(T)
	)_{((0,0),0)} = 
	(S_{f^{\eps}_{{\sigma_{(p,q,\mu)}}}, v\neq 0})_{((0,0),0)} \in \mgg.
	$$
\end{proof}

{\paragraph{Base case $f(x,y)=U(x,y)x^{-M}y^m$}
\begin{example} \label{ex:monomial}
	Let $f(x,y)= U(x,y)x^{-M}y^{m}$ with $M$ {in} $\mathbb Z$, $m$ {in} $\mathbb N$ and $U$ {in} $\k[[x,y]]$ with $U(0,0)\neq 0$.
	The Newton polygon $\N(f)$ has only one face $(-M,m)$ denoted by $\gamma_h$.
              \begin{itemize}
		\item If $m=0$, then 
			\begin{itemize}
				\item if $-\eps M \leq 0$, then for any $\delta \geq 1$, $Z^\delta_{\eps,\gamma_h, \omega}(T)=0$ 
					{and $\left(S_{f,x\neq 0} \right)_{((0,0),0)}=0$}
				\item if $-\eps M >0$, then there is $\delta_0>0$ such that for any $\delta\geq \delta_0$,
				       $Z^\delta_{\eps,\gamma_h, \omega}(T)$ {is computed in formula (\ref{gamma=(a,0)})} and 
					\begin{equation} \label{cam0}
						{\left(S_{f^{\eps},x\neq 0} \right)_{((0,0),0)}} = [x^{-\eps M}:\Gm \ra \Gm, \sigma_{\Gm}]
				        \end{equation}
					where $\sigma_{\Gm}$ is the action by translation of $\Gm$ on $\Gm$.
			\end{itemize}
		\item If $m\neq 0$ then 
			$$Z^\delta_{\eps,\gamma_h, \omega}(T) = [(x^{-M}y^m)^{\eps}:\Gm^2\to \Gm, {\sigma_{\Gm^2}}]  R_{\gamma_h,\eps,\omega}^{\delta}(T),$$
			\begin{itemize}
				\item if ``$\eps=+$'' and $-M>0$ then $R_{\gamma_h,\eps,\omega}^{\delta}(T)$ is computed in formula (\ref{Rgamma1}) and 
					$${\left(S_{f,x\neq 0} \right)_{((0,0),0)} = -[(x^{-M}y^m)^{\eps}:\Gm^2\to \Gm, {\sigma_{\Gm^2}}]},$$
			        \item if ``$\eps=+$'' and $-M\leq 0$ then $R_{\gamma_h,\eps,\omega}^{\delta}(T)$ is computed in formula (\ref{Rgamma2}) and 
					{$\left(S_{f,x\neq 0} \right)_{((0,0),0)}=0$},
				\item if ``$\eps=-$'' and $-M\geq 0$ then $R_{\gamma_h,\eps,\omega}^{\delta}(T)=0$ and {$\left(S_{1/f,x\neq 0} \right)_{((0,0),0)}=0$,}
				\item if ``$\eps=-$'' and $-M<0$ then $R_{\gamma_h,\eps,\omega}^{\delta}(T)$ is computed in formula (\ref{Rgamma3}) and 
					{$\left(S_{1/f,x\neq 0} \right)_{((0,0),0)}=0$.}
			\end{itemize}
	\end{itemize}
\end{example}
\begin{proof}  As $U$ is a unit, as all the arcs $(x(t),y(t))$ used in the computation of the motivic Milnor fiber at the origin satisfy $(x(0),y(0))=(0,0)$, we can assume  $U(x,y)=1$. As $f$ is a monomial its Newton polygon $\N(f)$ has only one face, the horizontal face $\gamma_h=(-M,{m})$ and the proof follows immediately from Proposition \ref{ex:cas-x^N}.
\end{proof}
}
\paragraph{Base case $f(x,y)=U(x,y)x^{-M}(y-\mu x^q+g(x,y))^m$} \label{subsection:basecase}

\begin{example} \label{ex:casdebase-f}
	Let $f(x,y)=U(x,y)x^{-M}(y-\mu x^q+g(x,y))^m$ with {$\mu$ in $\Gm$}, $M$ {in} $\mathbb Z$, $q$ {in} $\mathbb N$, $m$ {in} $\mathbb N^{*}$, 
        $U$ {in} $\k[[x,y]]$ with $U(0,0)\neq 0$ and $g(x,y)=\sum_{a+bq>q} c_{a,b}x^ay^b$ in $\k[x,y]$. We denote by $\gamma$ the one dimensional compact face.
	 Let $\nu \in \mathbb N_{\geq 1}$ and $\omega$ the associated differential form in Notations \ref{notations:Xeps-omega}.
	Then, there is $\delta_0>0$, such that for any $\delta \geq \delta_0$, we have 
         $$ \begin{array}{ccl}
					\left(Z_{f^\eps,\omega,x\neq 0}^{\delta}(T)\right)_{((0,0),0)} & = & 
					          \left[x^{\eps(-M+mq)}:\Gm \to \Gm, \sigma_{\Gm}\right] R_{(-M+mq,0),\eps,\omega}^{\delta}(T) \\
						  & + & \left[x^{-\eps M}(y-\mu x^q)^{\eps m}:\Gm^2 \setminus (y=\mu x^q) \to \Gm, \sigma_{ {\gamma} }\right] R_{\gamma,\eps,\omega}^{\delta}(T)\\
					      & + & 
					      \left[x^{-\eps M}y^{\eps m}:\Gm^2 \to \Gm,\sigma_{\Gm^2}\right] R_{(-M,m),\eps,\omega}^{\delta}(T)\\
					      & + & \left[x^{-\eps M}\xi^{\eps m} : ((y=\mu x^q)\cap \Gm^2) \times \Gm \to \Gm,\sigma_{\Gm^3}\right] S^{\delta}_{\omega}(T)\\
					    \end{array}
	 $$
	 where 
	 $\sigma_{\Gm^3}$ is the action $\sigma_{\Gm^3}(\lambda,(x,y,\xi)) = (\lambda x, \lambda y, \lambda \xi)$ and
	 $R_{(-M+mq,0),\eps,\omega}^{\delta}$, $R_{\gamma,\eps,\omega}^{\delta}$, $ R_{(-M,m),\eps,\omega}^{\delta}$, 
	 and $S^{\delta}_{\omega}(T)$  are rational functions defined in Propositions \ref{ex:cas-x^N} and \ref{lem:zeta=} {and depending on the context, formulas (\ref{eqSdelta1}), (\ref{eqSdelta2}), (\ref{eqSdelta3}) and (\ref{eqSdelta4})}.
         
		Furthermore,
		\begin{itemize}
			\item  If {$-M> 0$} then 
				\begin{equation}\label{eq1}
					{\left(S_{f,x\neq 0} \right)_{((0,0),0)} = - [x^{-M}y^m:\Gm^2 \to \Gm,\sigma_{\Gm^2}] \: \text{and} \: 
					\left(S_{1/f,x\neq 0} \right)_{((0,0),0)} = 0.}
				\end{equation}
			\item If {$-M\leq 0$} then 
				\begin{equation}\label{eq2}
					{\left(S_{f,x\neq 0} \right)_{((0,0),0)} = 0 \: \text{and} \: 
					\left(S_{1/f,x\neq 0} \right)_{((0,0),0)} = 0.}
				\end{equation}
		\end{itemize}
\end{example}

\begin{proof}
	The proof is similar to \cite[Example 3]{carai-antonio}. We start the proof by some preliminary remarks. 
	\begin{itemize}
		\item The dual cone to the face $x^{-M}y^m$ is $\mathbb R_{>0}(1,0) + \mathbb R_{>0}(1,q)$. Thus, the face $x^{-M}y^m$ belongs to 
			$\N(f)^+$ if and only if {$-M>0$}. Furthermore, if $M>0$ then the face $x^{-M}y^m$ belongs to $\N(f)^-$ if and only if $-M+mq{<} 0$. 
	        \item  As $U$ is a unit, as all the arcs $(x(t),y(t))$ used in the computation of the motivic Milnor fiber at the origin satisfy $(x(0),y(0))=(0,0)$, 
			we can assume in the following $U(x,y)=1$. We denote by $h(x,y)$ the polynomial $y-\mu x^q + g(x,y)$. We denote by $\gamma$ the compact one-dimensional face of the Newton polygon of $f$ with face polynomial $x^{-M}(y-\mu x^q)^m$. 

		\item The Newton polygon of $f$ has three face polynomials $x^{-M+qm}$, $x^{-M}y^{m}$ and $f_{\gamma}(x,y)=x^{-M}(y-\mu x^q)^m$. 
			Applying the decomposition formula (\ref{formula:decomposition}), we get for any $\delta \geq 1$: 
			$$Z_{f^\eps, \omega,x\neq 0}^{\delta}(T)= Z_{\eps,x^{qm-M},\omega}^{\delta}(T) + Z_{\eps,x^{-M}y^m,\omega}^{\delta}(T) + Z_{\eps,\gamma, \omega}^{\delta,=}(T)+ Z_{\eps,\gamma,\omega}^{\delta,<}(T).$$
			The rationality and the limit of $Z_{\eps,x^{qm-M},\omega}^{\delta}(T)$, $Z_{\eps,x^{-M}y^m,\omega}^{\delta}(T)$ and 
			$Z_{\eps,\gamma,\omega}^{\delta,=}(T)$ are given in Propositions \ref{ex:cas-x^N} and \ref{lem:zeta=}.
	\end{itemize}
	As $C_\gamma = \mathbb R_{>0}(1,q)$, the set $C_{\eps,\gamma}^{\delta,<}\cap \mathbb N^3$ (Notations \ref{notationsconezeta})  
	is in bijection with $C^{\delta} = \left\{ (n,k) \in (\mathbb N^{*})^2 \mid -Mk+qkm <\eps n,\: 0 < k \leq n\delta \right\}.$
	Then, by its definition in formula (\ref{sommme<}), we have
	$$Z_{\eps,\gamma,\omega}^{\delta,<}(T) = \sum_{(n,k)\in C^{\delta}} \mathbb L^{-(\nu-1)k}\mes(X_{n,(k,kq)})T^n$$
	with for any $(n,k)$ in $\mathbb N^2$
	$$X_{n,(k,kq)} = 
	\left\{ 
		\begin{array}{c|c}
			\varphi = (x(t),y(t)) \in \mathcal L(\mathbb A^2_\k) 
			& \begin{array}{l}
				\ord x(t)=k,\: \ord y(t)=qk,\: \ac y(t)= \mu \ac x(t)^q,\: \ord h(\varphi(t))=\frac{n+Mk\eps}{m\eps}
			\end{array}
		\end{array}
	\right\}.
	$$
	We introduce, for any $(k,l)$ in $(\mathbb N^{*})^2$ 
	$$X_{l,k}^{<}(h)=
	 \left\{ 
		\begin{array}{c|c}
			\varphi = (x(t),y(t)) \in \mathcal L(\mathbb A^2_\k) 
			& \begin{array}{l}
				\ord x(t)=k,\: \ord y(t)=qk,\:
				\ac y(t)= \mu \ac x(t)^q,\:
				\ord h(\varphi(t))=l
			\end{array}
		\end{array}
	 \right\}
	$$
	endowed with the map to $\Gm$ : $(x(t),y(t)) \mapsto \left((\ac x)^{-M}(\ac h(\varphi))^m\right)^\eps$.
        Remark that, if $X_{l,k}^{<}$ is not empty, then $l > kq$.\\
        We introduce
	$\tilde{C}^{\delta} =  
	\{(k,l) \in (\mathbb R_{>0})^2 \mid l>kq,\: \eps(-Mk + ml)>0,\: k\leq \eps(-Mk + ml)\delta\}.
	$
	{Writing $n=\eps(-Mk+ml)$,} we have 
	$$Z_{\eps,\gamma,\omega}^{\delta,<}(T)=\sum_{(k,l)\in \tilde{C}^{\delta} \cap (\mathbb N^*)^2} \mathbb L^{-(\nu-1)k}\mes(X_{l,k}^{<}(h)) T^{\eps(-Mk+ml)}.$$
	As $h$ is a polynomial Newton non-degenerate, we have (see for instance \cite[Lemme 2.1.1]{Gui02a}, \cite{GuiLoeMer06a} or \cite[Lemme 3.17]{Rai11})
	$$\mes(X_{l,k}^{<}(h)) = [(x^{-M}\xi^m)^{\eps}:(y=\mu x^q)\cap \Gm^2 \times \Gm \to \Gm, \sigma_{k,l}]\mathbb L^{-k -l}$$
	with $\sigma_{k,l}(\lambda,(x,y,\xi))=(\lambda^kx, \lambda^{kq}y,\lambda^{l}\xi)$. Using the construction of the Grothendieck ring $\mgg$, we obtain the equality 
	$$[(x^{-M}\xi^m)^{\eps}:(y=\mu x^q)\cap \Gm^2 \times \Gm \to \Gm, \sigma_{k,l}] = [(x^{-M}\xi^m)^{\eps}:(y=\mu x^q)\cap \Gm^2 \times \Gm \to \Gm, \sigma_{1,1}]$$
	(see \cite[Example 3]{carai-antonio} for details).
	Then we have, 
	$Z_{\eps,\gamma,\omega}^{\delta,<}(T) =  [(x^{-M}\xi^m)^{\eps}:(y=\mu x^q)\cap \Gm^2 \times \Gm \to \Gm, \sigma_{\Gm^{2}}] S^{\delta}_{\omega}(T) $
	with 
	$$S^{\delta}_{\omega}(T)= \sum_{n\geq 1}\sum_{\text{\tiny{$\begin{array}{c}(k,l)\in \tilde{C}^{\delta} \\  n = \eps(-Mk+ml) \end{array}$}}} 
	\mathbb L^{-\nu k -l}T^{n}.$$
  
	The rationality result is a consequence of Lemma \ref{lemmerationalitedescones}.
	
	\begin{itemize}
		\item If $M\leq 0$ {and ``$\eps = +$'' then} the assumption $-Mk+ml>0$ is always satisfied, 
			the condition $k\leq (-Mk+ml)\delta$ is also satisfied for any $\delta \geq 1$. Then we have 
			$\tilde{C}^{\delta} = \{(k,l)\in (\mathbb R_{>0})^2\mid kq<l\} = \mathbb R_{>0}(0,1) + \mathbb R_{>0}(1,q)$
			and by Lemma \ref{lemmerationalitedescones} denoting
		$\mathcal P =\left( ]0,1](0,1) + ]0,1](1,q) \right)\cap \mathbb N^2$ we have 
		\begin{equation} \label{eqSdelta1}
			S^{\delta}_{\omega}(T)=
			\sum_{(k_0,l_0) \in \mathcal P}
			\frac{\mathbb L^{-\nu k_0-l_0}T^{\eps(-Mk_0+ml_0)}}
			{(1-\mathbb L^{-1}T^{\eps m})(1-\mathbb L^{-\nu -q}T^{\eps(-M+mq)})} \underset{T \to \infty}{\to} 1.
		\end{equation}

	\item If $M>0$ and {``$\eps=+$''} remark that for any $\delta >0$, $\frac{m\delta}{1+M\delta} < \frac{m}{M}$ and
		$\tilde{C}^{\delta} = 
		\left\{\begin{array}{c|c} (k,l)\in (\mathbb R_{>0})^2 & 
		kq<l ,\:  k \leq \frac{lm\delta}{1+M\delta}\end{array}\right\}$
		Furthermore, $\frac{m\delta}{1+M\delta} \to \frac{m}{M}$ when $\delta \to +\infty$. Thus, 
		\begin{itemize}
		\item if $mq-M\leq 0$, namely $\gamma \notin \N(f)^+$ {then we have the inequalities}
				$m\delta/(1+M\delta) < m/M \leq 1/q$
				{implying}
				$$\tilde{C}^{\delta} = \left\{\begin{array}{c|c} (k,l) \in (\mathbb R_{>0})^2 & k \leq \frac{m\delta}{1+M\delta} l \end{array} \right\}
				= \mathbb R_{\geq 0}(0,1)+ \mathbb R_{>0}\omega^\delta$$
				with $\omega^\delta  = \left(1,(1+M\delta)/(m\delta)\right)$. By Lemma \ref{lemmerationalitedescones} denoting
			$\mathcal P =( ]0,1](0,1) + ]0,1]\omega^{\delta} )\cap \mathbb N^2$
			we have 	
			\begin{equation} \label{eqSdelta2}
				S^{\delta}_{\omega}(T)=
				\sum_{(k_0,l_0) \in \mathcal P}
				\frac{\mathbb L^{-\nu k_0 -l_0}T^{\eps(-M k_0 + ml_0)}}
				{(1-\mathbb L^{-1}T^{\eps m})(1-\mathbb L^{-((\nu,1)\mid \omega^\delta))}T^{(\eps(-M,m)\mid \omega^\delta)})} 
				+ \frac{\mathbb L^{-((\nu,1)\mid \omega^\delta))}T^{(\eps(-M,m)\mid \omega^\delta)}}
				{1-\mathbb L^{-((\nu,1)\mid \omega^\delta))}T^{(\eps(-M,m)\mid \omega^\delta)}} \underset{T \to \infty}{\to} 0.
			\end{equation}

		\item if $mq-M> 0$, namely $\gamma \in \N(f)^+$, then for $\delta$ large enough we have, the inequalities
			$1/q < m\delta/(1+M\delta) < m/M$
			inducing 
			$$\tilde{C}^{\delta} = \left\{\begin{array}{c|c} (k,l)\in (\mathbb R_{>0})^2 & k < l/q  \end{array} \right\}
			= \mathbb R_{>0}(0,1) + \mathbb R_{>0}(1,q).$$
			By Lemma \ref{lemmerationalitedescones} denoting
		$\mathcal P =\left( ]0,1](0,1) + ]0,1](1,q) \right)\cap \mathbb N^2$, we have  
		\begin{equation} \label{eqSdelta3}
			S^{\delta}_{\omega}(T)=
			\sum_{(k_0,l_0) \in \mathcal P}
			\frac{\mathbb L^{-\nu k_0-l_0}T^{\eps(-Mk_0+ml_0)}}
			{(1-\mathbb L^{-1}T^{-\eps M})(1-\mathbb L^{-\nu -q}T^{\eps(-M+mq)})} \underset{T \to \infty}{\to} 1.
		\end{equation}
\end{itemize}
		\item {If $M\leq 0$ and ``$\eps=-$", {namely $\gamma \notin \N(f)^-$}, then the cone $\tilde{C}^{\delta}$ is empty and $S_{\omega}^{\delta}(T)=0$.}
		\item  If $M>0$ and {``$\eps=-$"}, {namely $\gamma \in \N(f)^-$}, {then} we have 
			$\tilde{C}^{\delta} =  
			\{(k,l) \in \mathbb R_{>0}^2 \mid l/q>k> ml/M,\: k\leq (M k - ml)\delta\}.
			$
			\begin{itemize}
				\item If $M-mq\leq 0$, then the cone $\tilde{C}^{\delta}$ is empty and $S^{\delta}_{\omega}(T)=0$.
				\item If $M-mq>0$, then there is $\delta_0>0$ such that for any $\delta\geq \delta_0$,
					$\frac{m}{M}<\frac{m\delta}{M\delta-1}<\frac{1}{q}$ and we conclude that
					$$\tilde{C}^{\delta} =  
					\left\{
						\begin{array}{c|l} (k,l) \in \mathbb R_{>0}^2 & \frac{m l \delta}{M\delta-1} \leq k < \frac{l}{q} \end{array} 
					\right\}
					=\mathbb R_{\geq 0}(1,q) + \mathbb R_{>0}\omega^{\delta}
					$$
					with here $\omega^\delta=(1,(M\delta-1)/(m\delta))$ and by Lemma \ref{lemmerationalitedescones} denoting
				$\mathcal P =( ]0,1](1,q) + ]0,1]\omega^{\delta} )\cap \mathbb N^2$ we have 
				$\lim_{T \ra \infty} S^{\delta}_{\omega}(T)=0$ with
				\begin{equation} \label{eqSdelta4}
					S^{\delta}_{\omega}(T)=
					\sum_{(k_0,l_0) \in \mathcal P}
					\frac{\mathbb L^{-\nu k_0 -l_0}T^{\eps(-M k_0 + ml_0)}}
					{(1-\mathbb L^{-\nu-q}T^{\eps( -M+mq)})(1-\mathbb L^{-((\nu,1)\mid \omega^\delta))}T^{(\eps(-M,m)\mid \omega^\delta)})} 
					+ \frac{\mathbb L^{-((\nu,1)\mid \omega^\delta))}T^{(\eps(-M,m)\mid \omega^\delta)}}
					{1-\mathbb L^{-((\nu,1)\mid \omega^\delta))}T^{(\eps(-M,m)\mid \omega^\delta)}} \underset{T\to \infty}{\to} 0.
				\end{equation}
		\end{itemize} 

\end{itemize}
Finally applying Proposition \ref{ex:cas-x^N} and Proposition \ref{lem:zeta=} we obtain 
	\begin{itemize}
				\item if $M< 0$ then 
					$$ \begin{array}{ccl}
						\left(S_{f,x\neq 0} \right)_{((0,0),0)} & = & \left[x^{-M+mq}:\Gm \to \Gm, \sigma_{\Gm}\right] 
						+ [x^{-M}(y-\mu x^q)^m:\Gm^2 \setminus (y=\mu x^q) \to \Gm, \sigma] \\
						&  & 
						-\: [x^{-M}y^m:\Gm^2 \to \Gm,\sigma_{\Gm^2}] 
						- [x^{-M}\xi^m : (y=\mu x^q)\cap \Gm^2 \times \Gm \to \Gm,\sigma_{\Gm^3}] 
					    \end{array}
					$$
				\item if $M=0$ then 
                                        $$ \begin{array}{ccl}
						\left(S_{f,x\neq 0} \right)_{((0,0),0)} & = & \left[x^{mq}:\Gm \to \Gm, \sigma_{\Gm}\right] 
						+ [(y-\mu x^q)^m:\Gm^2 \setminus (y=\mu x^q) \to \Gm, \sigma] 
						\\
						&  & -\: [\xi^m : (y=\mu x^q)\cap \Gm^2 \times \Gm \to \Gm,\sigma_{\Gm^3}] 
					    \end{array}
					$$
				\item if $M>0$ then 
					$$ \begin{array}{ccl}
						\left(S_{f,x\neq 0} \right)_{((0,0),0)} & = & s^{(+)} \left[x^{-M+mq}:\Gm \to \Gm, \sigma_{\Gm}\right] 
						 +  s^{(+)}[x^{-M}(y-\mu x^q)^m:\Gm^2 \setminus (y=\mu x^q) \to \Gm, \sigma] \\
						&  & - \: s^{(+)} [x^{-M}\xi^m : (y=\mu x^q)\cap \Gm^2 \times \Gm \to \Gm,\sigma_{\Gm^3}] 
					    \end{array}
					$$
		\item the motivic Milnor fiber $\left(S_{1/f,x\neq 0} \right)_{((0,0),0)}$ is 0 if $M\leq 0$ otherwise, if $M>0$ 				
			                $$ \begin{array}{ccl}
		 				\left(S_{1/f,x\neq 0} \right)_{((0,0),0)} & = & s^{(-)} \left[x^{M-mq}:\Gm \to \Gm, \sigma_{\Gm}\right] 
						+ s^{(-)}[x^{M}(y-\mu x^q)^{-m}:\Gm^2 \setminus (y=\mu x^q) \to \Gm, \sigma] \\
						&  & -\:s^{(-)} [x^{M}y^{-m}:\Gm^2 \to \Gm,\sigma_{\Gm^2}]
					   \end{array}
					$$
			\end{itemize}
        with for any $\eps$ in $\{\pm\}$, $s^{(\eps)} = 1$ if $-M+mq>0$ and otherwise $s^{(\eps)} = 0$.
	Then for any $M$ in $\mathbb Z$ and $m\geq 1$, equalities (\ref{eq1}) and (\ref{eq2}) are induced by 
	the following equalities
				\begin{equation} \label{seq1}
					[x^{-M+mq}:\Gm\to \Gm,\sigma_{\Gm}] + [x^{-M}(y-\mu x^q)^m:\Gm^2 \setminus (y=\mu x^q) \to \Gm, \sigma] = 
				        [x^{-M}(y-\mu x^q)^m:\Gm \times \mathbb A^{1}_{\k} \setminus (y=\mu x^q) \to \Gm, \sigma]
			        \end{equation}
				\begin{equation} \label{seq2}
					[x^{-M}(y-\mu x^q)^m:\Gm \times \mathbb A^{1}_{\k} \setminus (y=\mu x^q) \to \Gm, \sigma] = 
					[x^{-M}z^m : \Gm^2 \to \Gm, \sigma_{1,1}]
				\end{equation}
				\begin{equation} \label{seq3}
					[x^{-M}\xi^m:(y=\mu x^q)\cap \Gm^2 \times \Gm \to \Gm, \sigma_{\Gm^3}] 
					= [x^{-M}\xi^m:\Gm^2 \to \Gm, \sigma_{1,1}]
				\end{equation}
				These equalities follow from the construction of the Grothendieck ring $\mgg$ and the isomorphisms in the category $Var_{\Gm}^{\Gm}$ 
				$$
				\begin{array}{ccc}
					\begin{array}{ccc}
					      \Gm \times \mathbb A^{1}_{\k} \setminus (y=\mu x^q) & \to & \Gm^2 \\
					      (x,y) & \mapsto & (x,z=y-\mu x^q)
				        \end{array}
					& \text{and} &
					\begin{array}{ccc}
						(y=\mu x^q)\times \Gm^2 \times \Gm & \to & \Gm^2 \\
						(x,y,\xi) & \mapsto & (x,\xi)
					\end{array}
				\end{array}.
				$$
\end{proof}

\section{Motivic invariants at infinity and Newton transformations} \label{section3}

\begin{defn} \label{def:compactification} A \emph{compactification} of {a polynomial $f$ in $\k[x,y]$} is a data $(X,i,\hat{f})$ with $X$ an algebraic $\k$-variety, $\hat{f}$ a proper map and $i$ an open dominant immersion, such that the following diagram is commutative 
$$ \xymatrix{ 
	\mathbb A_{\k}^{2} \ar[r]^{i} \ar[d]_{f} & X \ar[d]^{\hat{f}} \\ 
        \mathbb A^{1}_{\k} \ar[r]_{j} & \mathbb P^{1}_{\k} 
} 
$$ 
where $j$ is the open dominant immersion from $\mathbb A^{1}_{\k}$ to $\mathbb P^{1}_{\k}$ which maps a point $a$ to $[1:a]$. 
With these notations, we denote by $X_\infty$ the closed subset $X\setminus i(\mathbb A^2_\k)$ and by $\infty$ the point $[0:1]$.
We identify $a$ with the point $[1:a]$ and we denote $1/\hat{f}$ the extension of $1/f$ on 
$X\setminus \hat{f}^{-1}(0)$ and for any value $a$, we consider $\hat{f}-a$ the extension of $f-a$ on $X\setminus \hat{f}^{-1}(\infty)$. {In the following, we will compute rational forms of motivic zeta function using a specific compactification defined in subsection \ref{compactification}.}
\end{defn}

\subsection{Motivic Milnor fiber at infinity} \label{subsection:decompositionSfinfini} 
In this subsection, we recall the notion of {\emph{Milnor fiber at infinity} and}  \emph{motivic Milnor fibers at infinity} {of a polynomial $f$ in $\mathbb C[x,y]$}, which is a consequence of {the studies of} Bittner \cite{Bit05a}, Guibert-Loeser-Merle \cite{GuiLoeMer09a} on motivic Milnor fibers and developed for instance in \cite{Matsui-Takeuchi-13} and \cite{Rai11}.
{\subsubsection{Milnor fibration at infinity} 
The following result is a consequence of ideas of Thom, see for instance \cite{Pha83a}.

\begin{thm} Let $f$ be a polynomial in $\mathbb C[x,y]$. There is $R>0$ such that the restriction 
	$$f : \mathbb C^2 \setminus f^{-1}(D(0,R)) \ra \mathbb C \setminus D(0,R)$$ is a $C^\infty$ -- locally trivial fibration called \emph{Milnor fibration at infinity} of $f$. The {\emph Milnor fiber at infinity} is up to an homeomorphism the fiber $f^{-1}(a)$ for $a>R$. The \emph{monodromy at infinity} is induced by the action of $\pi_1(\mathbb C \setminus D(0,R))$ on $f^{-1}(a)$.
\end{thm}
}
\subsubsection{Motivic Milnor fiber at infinity} \label{section:Sfinfini}
\begin{defn} \label{defSfinfini}
	Let $(X,i,\hat{f})$ be a compactification of {a polynomial $f$ in $\k[x,y]$}.
	For any $\delta>0$ and $n$ in $\mathbb N^{*}$, we consider
	$$ X_{n}^{\delta}(1/\hat{f}) = \{ 
				  \varphi(t) \in \mathcal L(X) \mid 
					  \ord \varphi^{*}\left(\mathcal I_{X_\infty}\right) \leq n \delta, \:
					  \ord 1/\hat{f}(\varphi(t)) = n 
	            \}. $$ 
	with its structural map to $\hat{f}^{-1}(\infty) \times \Gm$, $\varphi \mapsto \left(\varphi(0),\ac 1/ \hat{f}(\varphi(t))\right)$.
	\end{defn}
\begin{rem} The motivic measure of $X_{n}^{\delta}(1/\hat{f})$ belongs to 
	$\mathcal M_{\hat{f}^{-1}(\infty)\times \Gm}^{\Gm}$. Indeed, even if $X$ is singular, the singular locus is contained in $X_\infty$, and it follows from \cite[Lemma 4.1]{DenLoe99a} that 
	the condition $\ord \varphi^{*}\left(\mathcal I_{X_\infty}\right) \leq n \delta$
	implies that it is not necessary to complete the Grothendieck ring to compute the measure of $X_{n}^{\delta}(1/\hat{f})$ for any $n$ and $\delta$.
\end{rem}

Applying \cite[\S 3.9]{GuiLoeMer09a}, {with notations of subsection \ref{deffctzetamodifiee}, Remark \ref{continuity-pushforward} and Theorem \ref{extensionaugroup}}, we get \cite[Theorem 3.4]{Rai10a}:  

\begin{thm}[Motivic Milnor fiber at infinity]
	Let $(X,i,\hat{f})$ be a compactification of {a polynomial $f$ in $\k[x,y]$} and $\delta>0$. The modified zeta function 
$$Z_{1/\hat{f},i(\mathbb A^2_\k)}^{\delta}(T)=\sum_{n\geq 1} \mes\left(X_{n}^{\delta}(1/\hat{f})\right)T^n$$
is rational
for $\delta$ large enough and has a limit when $T$ goes to infinity independent from the parameter $\delta$.
We denote 
$$S_{1/\hat{f}}([i:\mathbb A^2_\k\to X]) = -\lim_{T\ra \infty} Z_{1/\hat{f},i(\mathbb A^2_\k)}^{\delta}(T) 
\in \mathcal M_{\hat{f}^{-1}(\infty)\times \Gm}^{\Gm} \:\:\text{and}\:\:
S_{f,\infty} = \hat{f}_{!}S_{1/\hat{f}}\left([i:\mathbb A^2_\k \to X]\right)\in \mathcal M_{\{\infty\}\times\Gm}^{\Gm}
.$$ 
The motive $S_{f,\infty}$ does not depend on the chosen compactification and is called \emph{motivic Milnor fiber at infinity} of $f$.
\end{thm}
\begin{rem}\label{rem:casgeneric} \label{assumption-generic} Some remarks:
	\begin{itemize}
	  \item In the following, {similarly to Remark \ref{continuity-pushforward}}, we will identify $\mathcal M_{\{\infty\}\times\Gm}^{\Gm}$ with $\mgg$.
	  \item  For any constant $c$, for any arc $\varphi$ in $\mathcal L(X)$, for any positive integer $n$, $\ord (\hat{f}-c)(\varphi) = -n$ if and only if $\ord \hat{f}(\varphi)=-n$ which implies the equality $S_{f,\infty} = S_{f-c,\infty}$. So to compute $S_{f,\infty}$, we will always assume that $(0,0)$ is a point of the support, namely $f(0,0)\neq 0$.
	  In that case, by Remark \ref{rem:egalitepolygones}, $\overline{\N}(f)$ is equal to $\N_{\infty}(f)$.
	\end{itemize}
\end{rem}

\begin{notation} \label{notation-c-omega}
For a one-dimensional face $\gamma$ in $\mathcal N_{\infty}(f)$ (or $\overline{\N}(f)$) with primitive exterior normal vector $(p,q)$, we define
	\begin{equation}
		c(p,q) = \left\{ 
			\begin{array}{l}
				\text{$p+q$, if $p> 0$ and $q> 0$},\\
				\text{$p$, if $p>0$ and $q< 0$},\\
				\text{$q$, if $p< 0$ and $q>0$},\\
				\text{1, if $(p,q)=(1,0)$ or $(p,q)=(0,1)$} 
			\end{array}
			\right.
			\:\text{and}\:\:\omega_{p,q}(v,w) = v^{( \abs{p} + \abs{q} -1)}dv\wedge dw.
	\end{equation}
\end{notation}

\begin{thm}[Computation of $S_{f,\infty}$] \label{thm:thmSfinfini} \label{casligne} 
	Let $f$ in $\k[x,y]$ {and not in $\k[x]$ or $\k[y]$}.
	Using the compactification $(X,i,\hat{f})$ of subsection \ref{compactification}, 	\begin{itemize}
               \item {if} $f(x,y)=P(x^ay^b)$ with $P$ in $\k[s]$ {of degree $d$} with $(a,b)$ in $(\mathbb N^*)^2$, 
		       then if $\delta > \max(\frac{1}{da},\frac{1}{bd})$  
			we have
			\begin{equation} \label{resultat:cashomogene}
			Z^{\delta}_{1/\hat{f},i(\mathbb A^2)}(T) = [{1/(x^{a}y^{b})^d}:\Gm^2 \to \Gm, \sigma_{{\Gm^2}}] R_{\gamma}^\delta(T)
			\:\:\text{and}\:\:S_{f,\infty} = [1/(x^{a}y^{b})^d:\Gm^2 \to \Gm, \sigma_{{\Gm^2}}]\end{equation}
			
		\item {otherwise, there is $\delta'>0$ such that for any $\delta\geq \delta'$,}
				\begin{equation}\label{ratinfini}
					\begin{array}{c}
						Z_{1/\hat{f},i(\mathbb A^2_k)}^{\delta}(T)  =  \eps_{(a_0,0)}[1/x^{a_0}:\Gm \to \Gm,\sigma_{\Gm}] R_{(a_0,0)}^{\delta}(T) + 
						\eps_{(0,b_0)}[1/y^{b_0}:\Gm \to \Gm,\sigma_{\Gm}] R_{(0,b_0)}^{\delta}(T) \\
						 +  \sum_{\gamma \in \N_\infty(f)^o}[1/f_\gamma : \Gm^{2}\setminus f_\gamma^{-1}(0)\ra \Gm,\sigma_{\gamma}]R_{\gamma}^{\delta,=}(T) 
						 +  \sum_{\gamma \in \N_{\infty}(f)^o}\sum_{\mu \in R_\gamma} 
						(Z^{\delta/c(p,q)}_{1/{f_{{\sigma_{(p,q,\mu)}}}}, \omega_{p,q}, v\neq 0}(T))_{((0,0),0)}
						\end{array}
				\end{equation}

			{and the motivic Milnor fiber at infinity is}  
			\begin{equation}\label{formuleSfinfini}
				\begin{array}{lll}
					S_{f,\infty} & = & \eps_{(a_0,0)}[1/x^{a_0}:\Gm \to \Gm,\sigma_{\Gm}] + 
					\eps_{(0,b_0)}[1/y^{b_0}:\Gm \to \Gm,\sigma_{\Gm}] \\ 
					& + & \sum_{\gamma \in \N_\infty(f)^o}\eps_{\gamma}
					[1/f_\gamma : \Gm^{2}\setminus f_\gamma^{-1}(0)\ra \Gm,\sigma_{\gamma}]
					 + { \sum_{\gamma \in \N_{\infty}(f)^o,\:\dim \gamma = 1 }\sum_{\mu \in R_\gamma} 
					(S_{1/{ f_{{\sigma_{(p,q,\mu)}}}}, v\neq 0})_{((0,0),0)}}.
				\end{array}
			\end{equation}
			\end{itemize}
			In particular, we have 
			\begin{equation} \label{formuleCorollaireAntonio}
				\begin{array}{lll}
					S_{f,\infty} & =  & \sum_{\gamma \in \N_\infty(f)^o} \eps_{\gamma}S_{f_\gamma,\infty}
					+  \sum_{\gamma \in \N_{\infty}(f)^o,\dim \gamma = 1 }\sum_{\mu \in R_\gamma} 
					(S_{1/f_{{\sigma_{(p,q,\mu)}}}, v\neq 0})_{((0,0),0)} 
					- 
					(S_{1/(f_\gamma)_{{\sigma_{(p,q,\mu)}}}, v\neq 0})_{((0,0),0)}.
				\end{array}
			\end{equation}
			{All these formulas use the notation $\N_{\infty}(f)^o$ of Definition \ref{def:Ninfinif}, Notations \ref{rem:factorisationinfini} and \ref{notation-c-omega}, and the following:}
\begin{itemize}
	\item  $\eps_{(a_0,0)}$ and $\eps_{(0,b_0)}$ {are respectively} equal to $1$ and otherwise 0, if and only if $(a_0,0)$ and {$(0,b_0)$ are respectively} faces of $\N_\infty(f)$, $ R_{(a_0,0)}^{\delta}(T)$ and $ R_{(0,b_0)}^{\delta}(T)$ {are respectively} defined in equations (\ref{casb01}) and (\ref{casb02})
	with limit equal to $-1$.

\item the expression of $R_{\gamma}^{\delta,=}(T)$ is given for any zero dimensional face $\gamma$ in $\N_\infty(f)^o$ by equations (\ref{eqR}), (\ref{eqcasa1}), (\ref{eqcasa2}), (\ref{eqcasb1}), (\ref{eqcasb2}), (\ref{eqcasb3}), (\ref{eqcasb4}), (\ref{eqcasb5}) {(and (\ref{casmonom}) for the case $f=P(x^ay^b)$)}, 
	and equation (\ref{eqcasc1}) for one-dimensional faces of $\N_\infty(f)^o$. In particular, in the general case (formula (\ref{ratinfini})), we have
	$ -\lim_{T \to \infty} R_{\gamma}^{\delta,=}(T) = \eps_\gamma$
	where $\eps_\gamma$ is {$(-1)^{\dim \gamma +1}$} {(and otherwise 0)} if $\gamma$ is not contained in a face which contains the origin.
\end{itemize}
\end{thm}

\begin{proof}
	We give the ideas of the general proof using above notations and refer to subsection \ref{section:preuveformules} for details. It is similar to the proof of Theorem \ref{thmSfeps}. 
{We consider first a polynomial $f$ in $\k[x,y]$ which is not of the form $f(x,y)=P(x^ay^b)$ with $P$ in $\k[s]$ with $(a,b)$ in 
		$(\mathbb N^*)^2$. We work with the Newton polygon at infinity $\N_{\infty}(f)$. Note that by Remark \ref{rem:casgeneric}, we can assume $f(0,0)\neq 0$ and in that case we have $\GN(f)=\N_{\infty}(f)$.} 
	In subsection \ref{compactification} we consider the compactification $(X,i,\hat{f})$ of the graph of $f$ in $(\mathbb P^1_{\k})^3$.
	In Proposition \ref{rem:decomposition}, we decompose the motivic zeta function of $1/\hat{f}$ along $\overline{\N}(f)$
	\begin{equation} 
		Z^{\delta}_{1/\hat{f},i(\mathbb A^2_\k)}(T) = \sum_{\gamma \in \overline{\N}(f)} Z^{\delta}_{\gamma,-}(T).
	\end{equation}
	By Remark \ref{rem:NfinfinipourSfinfini}, it is enough to consider faces $\gamma$ in $\N_\infty(f)^{o}$ namely faces of $\N_\infty(f)$ which do not contain the origin.  
	In Proposition \ref{lem:fctzetainfiniaxe}, we show the rationality and compute the limit of $Z^{\delta}_{\gamma,-}(T)$ in the case of a zero dimensional face $\gamma$ contained in a coordinate axis.
	In Proposition \ref{lem:fctzetaegalinfini} and section \ref{sec:casdim1<}, we consider the case of a face $\gamma$ not contained in a coordinate axis. Depending on the fact that the face polynomial $f_\gamma$ vanishes or not on the angular components of the coordinates of an arc 
	(Remark \ref{rem:m}), we decompose in formula (\ref{decompositionzeta=<}) the zeta function $Z^{\delta}_{\gamma,-}(T)$ as a sum of $Z^{\delta,=}_{\gamma,-}(T)$ and $Z^{\delta,<}_{\gamma,-}(T)$. 
	In particular if the face $\gamma$ is zero dimensional then $Z^{\delta,<}_{\gamma,-}(T)$ is zero.  
	In Proposition \ref{lem:fctzetaegalinfini} we show the rationality and compute the limit of the zeta function $Z^{\delta,=}_{\gamma,-}(T)$. 
	In Proposition \ref{prop:rationalityZ<infini-infini} and Proposition \ref{prop:rationalityZ<infini-zero},
	we prove the decomposition 
	$$Z^{\delta,<}_{\gamma,-}(T) = 
	\sum_{\mu \in R_\gamma} 
        (Z^{\delta/c(p,q)}_{1/{ f_{{\sigma_{(p,q,\mu)}}}}, \omega_{p,q}, v\neq 0})_{((0,0),0)}.$$ 
	We use section \ref{thmSfeps} to obtain the rationality and an expression of the limit of 
	$(Z^{\delta/c(p,q)}_{1/{ f_{{\sigma_{(p,q,\mu)}}}}, \omega_{p,q}, v\neq 0})_{((0,0),0)}$.
	Similarly in subsection \ref{sec:cashorizontalinfini} and \ref{sec:casverticalinfini} we consider the case of the horizontal and vertical faces.
	
	{We assume now that $f(x,y)=P(x^ay^b)$ with $f(0,0)\neq 0$, $(a,b)$ in $(\mathbb N^*)^2$ and $P$ in $\k[s]$ of degree $d$.}
	We denote by $\gamma$ the face $(ad,bd)$.
	Remark that by assumption $\N_{\infty}(f)$ is the segment {$[(0,0),(da,db)]$} not contained in a coordinate axis. 
	Then, by Proposition \ref{rem:decomposition}, Remark \ref{rem:NfinfinipourSfinfini} and Remark \ref{annulation-f-gamma}, we have the equality 
	$$Z^{\delta}_{1/\hat{f}, i(\mathbb A_\k^2)}(T) = Z^{\delta}_{\gamma,-}(T) = Z^{\delta,=}_{\gamma,-}(T).$$
	Then formulas (\ref{resultat:cashomogene}) follow from formula (\ref{casmonom}) of Proposition \ref{lem:fctzetaegalinfini}.
	{The proof of equation (\ref{formuleCorollaireAntonio}) follows from Theorem \ref{thm:thmSfinfini} applied to $f$ and each {face} polynomials $f_\gamma$ for $\gamma$ in {$\N_{\infty}(f)^{o}$} {applying Remark \ref{rem:casgeneric}}}. 
\end{proof}

{\begin{rem} This theorem extends in the case of curves (without non degeneracy or convenient conditions), the computation in the non degenerate case of $S_{f,\infty}$ done in \cite{Matsui-Takeuchi-14} and \cite{Rai10a}.
\end{rem}
}

\newpage
\subsubsection{Realization results}
{In this section we assume $\k = \mathbb C$.}
\paragraph{Generalized Kouchnirenko formula for the generic fiber}
Using Denef-Loeser results (see for instance \cite[3.17]{GuiLoeMer06a}) we have (\cite[\S 2.4]{Rai10a}, \cite{Matsui-Takeuchi-13, Matsui-Takeuchi-14}).

\begin{thm} Let $f$ be a polynomial in $\mathbb C[x,y]$. We have the equality
	$\tilde{\chi}_{c}(S_{f,\infty}^{(1)}) = \chi_{c}\left(F_\infty\right)$
	where $F_\infty$ is the Milnor fiber at infinity of $f$ and $\tilde{\chi}_{c}:\mathcal M_{\k}^{\hat{\mu}} \to \mathbb Z$ is the Euler characteristic realization.
\end{thm}

{\begin{rem} \phantomsection \label{thm:lambda}
	We recall that the Milnor fiber at infinity $F_\infty$ of $f$ is the fiber $f^{-1}(R)$ for $R$ large enough, then it is homeomorphic to the generic fiber 
	$f^{-1}(a_{gen})$ of $f$, then we have 
	$\chi_c(F_\infty) = \chi_c(f^{-1}(a_{gen})).$
\end{rem}
}

Using that result, Proposition \ref{prop:caracteristiceuleraire} and Theorem \ref{thm:thmSfinfini} we have 
{
\begin{cor}[{Generalized Kouchnirenko formula for the generic fiber}] \label{Kouchnirenkoformulagenfiber}
	Let $f$ be a polynomial in $\mathbb C[x,y]$, {not in $\mathbb C[x]$ or $\mathbb C[y]$.} 
        With notations of Theorem \ref{thm:thmSfinfini} and subsection \ref{subsection:area}, we have
	\begin{itemize}
		\item if $f(x,y)=P(x^ay^b)$ with $P$ in $\k[s]$ with $(a,b)$ in $(\mathbb N^*)^2$ then we have 
			$\chi_{c}\left(F_\infty\right) = \chi_{c}\left(f^{-1}(a_{gen})\right) = 0,$
		\item otherwise in the general case we have
			\begin{equation} \label{formuleKouchnirenkogenfiber}
				\begin{array}{lll}
					\chi_{c}\left(F_\infty\right) = \chi_{c}\left(f^{-1}(a_{gen})\right)& = & \eps_{(a_0,0)}a_0 + \eps_{(0,b_0)}b_0 
					-2\sum_{\gamma \in \N_\infty(f)^o,\dim \gamma =1} \mathcal S_{\N(f_\gamma),f_{\gamma}}\\
					& + &  \sum_{\gamma \in \N_{\infty}(f)^o,\:\dim \gamma = 1 }\sum_{\mu \in R_\gamma} 
					\tilde{\chi}_{c}\left((S_{{(1/f)_{{\sigma_{(p,q,\mu)}}}}, v\neq 0})^{(1)}_{((0,0),0)}\right),
				\end{array}
			\end{equation}
			with $\eps_{(a_0,0)}$ (resp. $\eps_{(0,b_0)}$) is equal to $1$ and otherwise 0, if and only if $(a_0,0)$ (resp. $(0,b_0)$) is a face of $\N_\infty(f)$.
	\end{itemize}
\end{cor}
}
\begin{rem}Using as usual other realizations, we can obtain formula of the monodromy zeta function at infinity of $f$ or the spectrum at infinity of $f$ in terms of the iterated Newton polygons of $f$ in the Newton algorithm at infinity.
\end{rem}
{\begin{example} \label{example:example1Sfinfini}
We extend Example \ref{example:example1}. We observe that $\GN(f)=\N_{\infty}(f)$ and by Theorem \ref{thm:thmSfinfini} we have
\begin{equation} \label{example:formuleexample1Sfinfini}
\begin{array}{ccl}
	S_{f,\infty}  & = & [1/y:\Gm\ra\Gm,\sigma_{\Gm}]+[1/x^2:\Gm\ra\Gm,\sigma_{\Gm}] \\
	& + & 
	[(y(x^2y+1)^3)^{(-1)} : \Gm^{2}\setminus (x^2y+1=0)\ra \Gm,\sigma_{\gamma_{1}^{0}}] +
	[(x^2(xy+1)^4)^{-1}: \Gm^{2}\setminus (xy+1=0)\ra \Gm,\sigma_{\gamma_{2}^{0}}] 
	\\
	& - & [1/x^6y^4: \Gm^{2}\ra \Gm,\sigma_{x^6y^4}]
	+ \left(S_{1/f_1, v\neq 0}\right)_{((0,0),0)} 
	+\left(S_{1/f_2, v\neq 0}\right)_{((0,0),0)}
\end{array}
\end{equation}
with $f_1$ and $f_2$ defined in formulas (\ref{exemplef1}) and (\ref{exemplef2}).
\begin{itemize}
	\item[$\bullet$] Applying Theorem \ref{thmSfeps}, we have 
$$
\begin{array}{ccl}
	\left(S_{1/f_1, v\neq 0}\right)_{((0,0),0)} & = & [v:\Gm \to \Gm,\sigma_{\Gm}] +
	[(8v^{-2}w^3+5v^{-1})^{-1}:\Gm^2 \setminus (8v^{-2}w^3+5v^{-1}=0) \to \Gm, \sigma_{\gamma_1^{(1)}}]\\ 
	& & - [v^2w^{-3}:\Gm^2 \to \Gm] + (S_{1/(f_{1})_{\sigma_{\gamma_1}^{(1)}},v_1\neq 0})_{(0,0),0}.
\end{array}
$$
As $(f_{1})_{\sigma_{\gamma_1}^{(1)}}$ is a base case of Theorem \ref{thm:algo-Newton}, with $M=3$ and $m=1$, by Example \ref{ex:casdebase-f} we have 
$(S_{1/(f_{1})_{\sigma_{\gamma_1}^{(1)}},v_1\neq 0})_{(0,0),0} = 0.$

\item[$\bullet$] As the set $\N(f_2)^{-}$ is empty, applying Theorem \ref{thmSfeps}, we have
$(S_{1/f_2, v\neq 0})_{((0,0),0)} = 0.$
\end{itemize}

Assume now $\k=\mathbb C$. By Corollary \ref{Kouchnirenkoformulagenfiber}, we compute the Euler characteristic of the generic fiber of $f$. We have 
$$ 
\chi_{c}((y(x^2y+1)^3=1)\cap \Gm^{2}) = -2,\: \chi_{c}((x^2(xy+1)^4=1)\cap \Gm^{2}) = -2,\: \chi_{c}((8v^{-2}w^3+5v^{-1}=1) \cap \Gm^2) = -3 
$$
by Corollary \ref{KFSfeps} and Proposition \ref{prop:caracteristiceuleraire}. We conclude by formulas (\ref{formuleKouchnirenkogenfiber}) and (\ref{KFSf+}) that 
$$\chi_{c}(F_\infty)= 1 + 2 - 2 -2 - 0 + (1 -3 - 0) + 0 = -3.$$
\end{example}
}
\subsection{Topological bifurcation set, motivic bifurcation set, motivic nearby cycles at infinity}

\subsubsection{Topological bifurcation set} 
The following result is a consequence of ideas of Thom, see for instance \cite{Pha83a}.
\begin{thm} Let $f$ be a polynomial in {$\mathbb C[x,y]$}. There is a finite set $B$ such that the restriction 
	$$f : \mathbb C^2 \setminus f^{-1}(B) \ra \mathbb C \setminus B$$ is a $C^\infty$ -- locally trivial fibration.
The smallest convenient set $B$, denoted by $B_f^{\text{top}}$, is called \emph{topological bifurcation set} of $f$.  
\end{thm}
H\`a and L\^e gave the following description of the topological bifurcation set		
\begin{thm}[H\`a-L\^e \cite{VuiTra84a}] \label{Le-Ha} 
	Let $f$ be a polynomial in $\mathbb C[x,y]$. The topological bifurcation set is 
	$$B_{f}^{\text{top}} = \{ a \in \mathbb C  \mid \chi_c(f^{-1}(a)) \neq \chi_c(f^{-1}(a_{gen}))\}.$$ 
\end{thm} 

\subsubsection{$\lambda$-invariant}
{
	Let $f$ be a non constant polynomial in $\mathbb C[x,y]$. We denote by $d$ its degree and by $f_0$, \dots, $f_d$, its homogeneous components. 
	In this paragraph, we recall for any value $a$ the definition of the invariant $\lambda_{a}(f)$, which roughly speaking measures the non equisingularity at infinity of the fibers of $f$ in $\mathbb P^2_{\mathbb C}$, see for instance formula (\ref{def:lambda}). We consider first, the algebraic variety
	$\mathcal V = \{([x:y:z],a)\in \mathbb P^2_{\mathbb C} \times \mathbb A^1_{\mathbb C} \mid G(x,y,z,a) = 0\}$
	with 
	$G(x,y,z,a) = \tilde{f}(x,y,z)-az^{d}$
	where $\tilde{f}$ is the homogeneous polynomial associated to $f$. We denote by $i_{\mathcal V}:\mathbb A^2_{\mathbb C} \to \mathcal V$ the open dominant immersion which maps $(x,y)$ on $([x:y:1],f(x,y))$. We denote by $\hat{f}_{\mathcal V}:\mathcal V \to \mathbb A^1_{\mathbb C}$ the application 
	which maps $([x:y:z],a)$ to $a$. The triple $(\mathcal V,i_{\mathcal V}, \hat{f}_{\mathcal V})$ is a compactification of $f$. 
	The singular locus of $\mathcal V$ is the closed subset, denoted by $\mathcal V_{sing}$, and given by the equations
	\begin{equation} \label{eqsing} 
		\frac{\partial f_d}{\partial x} = \frac{\partial f_d}{\partial y} = f_{d-1} = z = 0.
	\end{equation}
	We observe that $\mathcal V_{sing}$ is equal to the product $\mathcal P \times \mathbb A^1_{\mathbb C}$ where $\mathcal P$ is the closed subset of 
	$\mathbb P^2_{\mathbb C}$ defined by (\ref{eqsing}). As $\frac{\partial f_d}{\partial x}$, $\frac{\partial f_d}{\partial y}$ and $f_{d-1}$ are homogeneous polynomials in two variables, their zero locus in $\mathbb P^1_{\mathbb C}$ is a finite set, thus $\mathcal P$ is a finite set. 

	We assume $f$ has isolated singularities. Let $p_0=[x_0:y_0:0]$ be an element of $\mathcal P$.
	We can assume for instance $x_0 \neq 0$. Then we work in the chart $x \neq 0$ with the coordinates $u=y/x$ and $v=z/x$. We define $u_0 = y_0/x_0$, $v_0=0$ and for any $a$ in $\mathbb C$, we consider the polynomial $H_{a}(u,v) = G(1,u,v,a).$
	The point $(u_0,v_0)$ is an isolated critical point of $H_a$, and we denote by $\mu_{p_0}(a)$ the Milnor number $\mu_{(u_0,v_0)}(H_a)$, it does not depend on the choice of the chart. It follows from Thom-Mather theorem that the set of values $\mu_{p_0}(a)$ parametrised by $a$ is finite. 
	We finally define for any value $a$, the classical invariants
	\begin{equation}\label{def:lambda} 
		\lambda_{p_0,a}(f) = \mu_{p_0}(a) - \mu_{p_0}(a_{\text{gen}}) \:\: \text{and} \:\: \lambda_{a}(f)= \sum_{p_0 \in \mathcal P} \lambda_{p_0,a}(f).
	\end{equation}
	By upper semi-continuity of the function $a \mapsto \mu_{p_0}(a)$ (see for instance \cite[Prop 2.3]{Bro88a}), we observe that $\lambda_{a}(f)\geq 0$ and equal to zero for almost every value $a$.
	This invariant was studied by Suzuki in \cite{Suz74}, H\`a and L\^e in \cite{VuiTra84a} or the first author in \cite{Cas96}.
	We refer to \cite{ArtLueMel00b} or \cite{Tibar} for generalization in dimension $\geq 3$. It follows from these references and Theorem \ref{Le-Ha} that
}
	\begin{thm} \label{chigenchia} Let $f$ be a polynomial in $\mathbb C[x,y]$ with isolated singularities. 
	We denote by $\chi_{c}(f^{-1}(a_{\text{gen}}))$ the Euler characteristic of the generic fiber of $f$. 
	For any value $a$ in $\mathbb C$, we have 
	\begin{equation} \label{formulechigenerique}
	\chi_{c}(f^{-1}(a))=\chi_{c}(f^{-1}(a_{\text{gen}}))+\mu_a(f)+\lambda_a(f)
	\:\:\text{and}\:\:
	\chi_{c}(f^{-1}(a_{\text{gen}}))=1-(\mu(f)+\lambda(f)) 
        \end{equation}
	with $\mu_a(f)$ equal to the sum of Milnor numbers of critical points of $f^{-1}(a)$, 
	$\mu(f)$ is the sum of Milnor numbers and $\lambda(f)$ is the sum of all $\lambda_a(f)$.
	In particular, we have the equality
	$B_f^{\text{top}} = \text{disc}(f) \cup \{a \in \mathbb C \mid \lambda_a(f)\neq 0\}.$
        \end{thm}

	\subsubsection{Motivic nearby cycles at infinity} \label{section:Sfainfini}

\begin{defn}[The morphism $S_{\hat{f}-a}^{\infty}$] \label{Sg^infini} 
	Let $f$ be a polynomial in $\k[x,y]$ and $(X,i,\hat{f})$ be a compactification of $f$, with $X_\infty = X\setminus i(\mathbb A^2_{\k})$.
	Let $a$ be a value in $\mathbb A^{1}_{\k}$. We denote by 
	$S_{\hat f -a}^{\infty} : \mathcal M_X \ra \mathcal M_{(X_\infty \cap \hat{f}^{-1}(a))\times \Gm}^{\Gm}$
	the composition of the morphism $S_{\hat{f}-a}:\mathcal M_X \to \mathcal M_{\hat{f}^{-1}(a)\times \Gm}^{\Gm}$ 
	(Theorem \ref{extensionaugroup}) with the morphism 
	$i_{(X_\infty \cap \hat{f}^{-1}(a))\times \Gm}^*:\mathcal M_{\hat{f}^{-1}(a)\times \Gm}^{\Gm} \to \mathcal M_{(X_\infty \cap \hat{f}^{-1}(a))\times \Gm}^{\Gm}$  (subsection \ref{inverse-direct}) and induced by the canonical injection $i:(X_\infty \cap \hat{f}^{-1}(a))\times \Gm \to \hat{f}^{-1}(a)\times \Gm$. 
\end{defn}
We recall some notions of 
\cite[\S 4]{Rai11} (see also constructions of Matsui, Takeuchi and Tib\u{a}r in \cite{Matsui-Takeuchi-13}, \cite{Matsui-Takeuchi-14} and \cite{Takeuchi-Tibar-16}).

\begin{thm}[Motivic nearby cycles at infinity] For any value $a$ in $\mathbb A_{\k}^{1}$, the motive $S_{f,a}^{\infty}$ defined as
	$$S_{f,a}^{\infty} = \hat{f_!}S_{\hat{f}-a}^{\infty}({[i:\mathbb A_{\k}^2\to X]}) \in {\mathcal M_{ \{a\}\times \Gm}^{\Gm}}$$
	does not depend on the chosen compactification $(X,i,\hat{f})$ and is called \emph{motivic nearby cycles at infinity} of $f$ for the value $a$.
\end{thm}
\begin{rem}
Some remarks:
\begin{itemize}
	\item The motive $S_{\hat{f}-a}^{\infty}({[i:\mathbb A_{\k}^2\to X]})$ is the limit of the motivic zeta function
\begin{equation} \label{Zetafunctiona}
	Z_{\hat{f}-a,i(\mathbb A^2_{\k})}^{\delta, \infty}(T) = \sum_{n\geq 1} \mes(X_{n}^{\delta,\infty}(\hat{f}-a))T^n
\in \mathcal M_{(X_\infty \cap \hat{f}^{-1}(a))\times\Gm}^{\Gm}[[T]],
\end{equation}
		with $\delta$ an integer large enough (Proposition \ref{thmrationaliteopen}) and for any $n\geq 1$
		$$X_{n}^{\delta,\infty}(\hat{f}-a) = 
		\left\{ \begin{array}{c|c} 
			        \varphi(t) \in \mathcal L(X) & 
				\begin{array}{l} \varphi(0) \in X_{\infty},\: 
				  \ord \varphi^{*}\mathcal I_{X_{\infty}} \leq n \delta,\:
				  \ord (\hat{f}-a)(\varphi(t)) = n 
			        \end{array} 
		         \end{array} 
		\right\}.
		$$
		endowed with the structural application to $(\hat{f}^{-1}(a) \cap X_\infty) \times \Gm$, 
		$\varphi \mapsto (\varphi(0),\ac((\hat{f}-a)(\varphi)))$.
	\item Similarly to Remark \ref{continuity-pushforward}, we will identify $\mathcal M_{\{a\}\times\Gm}^{\Gm}$ with $\mgg$. 
\end{itemize}
\end{rem}

In \cite{FR}, Fantini and the second author proved the following result stated in dimension 2 here:

\begin{thm} \label{thm:inclusion-Btop-Bmot} \label{chilambda} 
	Let $f$ be a polynomial in $\mathbb C[x,y]$ with isolated singularities. Then, for any value $a$, we have the equality
	$$\tilde{\chi}_{c}\left(S_{f,a}^{\infty,(1)}\right) = -\lambda_a(f),$$
	with $S_{f,a}^{\infty,(1)}$ in $\mathcal M^{\hat{\mu}}$ the fiber in 1 of $S_{f,a}^{\infty}$ and 
	$\tilde{\chi}_c : \mathcal M_\k^{\hat{\mu}}\to \mathbb Z$ the Euler characteristic realization.
\end{thm}

{\begin{rem} Let $f$ be a polynomial in $\k[x,y]$. In the following we compute $S_{f,0}^{\infty}$ in terms of the combinatorics of iterated Newton polygons. The general case $a\neq 0$ can be deduced by translation. 
	We use Notations \ref{def:polygon-global} of Newton polygons.
\end{rem}
}
\begin{thm}[Computation of $S_{f,0}^{\infty}$] \label{thmcyclesprochesinfini} \label{cassegment}

	Let $f$ be a polynomial in $\k[x,y]$, {not in $\k[x]$ or $\k[y]$}. {Using the compactification $(X,i,\hat{f})$ of subsection \ref{compactification},} there is $\delta'>0$ such that for any $\delta\geq \delta'$:
	\begin{itemize}
		\item  If $\GN(f)$ is not a segment,  then we have
{\begin{equation}\label{ratzeroinfini}
		Z_{\hat{f},i(\mathbb A^2_k)}^{\delta,\infty}(T)  =  
		\sum_{\gamma \in {\widetilde{\N}(f)}}[f_\gamma : \Gm^{2}\setminus f_\gamma^{-1}(0)\ra \Gm,\sigma_{\gamma}]R_{\gamma}^{\delta,=}(T) 
		 +  \sum_{\gamma \in {\widetilde{\N}(f)},\: \dim \gamma = 1} \sum_{\mu \in R_\gamma} 
		(Z^{\delta/c(p,q)}_{{f_{{\sigma_{(p,q,\mu)}}}}, \omega_{p,q}, v\neq 0}(T))_{((0,0),0)}
\end{equation}
}
\begin{equation}\label{formulecyclesprochesinfini}
		S_{f,0}^{\infty}  =  \sum_{\gamma \in {\widetilde{\N}(f)}}
		\eps_{\gamma}[f_\gamma : \Gm^{2}\setminus f_\gamma^{-1}(0)\ra \Gm,\sigma_{\gamma}] 
		 +  { \sum_{\gamma \in {\widetilde{\N}(f)},\: \dim \gamma = 1}\sum_{\mu \in R_\gamma} 
		\left(S_{{f_{{\sigma_{(p,q,\mu)}}}}, v\neq 0}\right)_{((0,0),0)} } 
\end{equation}

\item If $f$ is the monomial $x^ay^b$, then we have
	\begin{equation}
					Z_{\hat{f},i(\mathbb A^2_k)}^{\delta,\infty}(T)  = [x^{a}y^{b} : \Gm^{2}\ra \Gm,\sigma_{\Gm^2}]R_{(a,b)}^{\delta,=}(T)
					\:\:\text{and}\:\:
					S_{f,0}^{\infty}  = 0.
			\end{equation}

		\item If $\GN(f)$ is a segment, we denote by $\gamma$ the one dimensional face of $\overline{\N}(f)$. 
	It is a segment with vertices $(a_0,b_0)$ and $(a_1,b_1)$ {such that $a_0\leq a_1$ and if $a_0=a_1$ then $b_0<b_1$}. We define $\delta_{(a_0,b_0)}$ as $1$ if $(a_0,b_0)\neq (0,0)$ otherwise $0$. 
	{We choose an equation of the underlying line of $\gamma$ as $ap+bq=N$ with $(p,q)$ in $\mathbb Z^2 \setminus (\mathbb Z_{\leq 0})^2$ and coprime}
	\begin{itemize}
		\item If {$pq<0$}, we have 
			\begin{equation} \label{ratzetaqh1}
				\begin{array}{lll}
					Z_{\hat{f},i(\mathbb A^2_k)}^{\delta,\infty}(T) & = & 
					\delta_{(a_0,b_0)}[x^{a_0}y^{b_0} : \Gm^{2}\ra \Gm,\sigma_{\Gm^2}]R_{(a_0,b_0)}^{\delta,=}(T)
					+ [x^{a_1}y^{b_1} : \Gm^{2}\ra \Gm,\sigma_{\Gm^2}]R_{(a_1,b_1)}^{\delta,=}(T) \\ 
					& + &[f:\Gm^{2}\setminus f^{-1}(0) \ra \Gm,\sigma_{\gamma}]R_{\gamma}^{\delta,=}(T)\\
					& + &\sum_{\mu \in R_\gamma} 
					(Z^{\delta/c(p,q)}_{{f_{{\sigma_{(p,q,\mu)}}}}, \omega_{p,q}, v\neq 0}(T))_{((0,0),0)}
					+ (Z^{\delta/c(p,q)}_{ {f_{\sigma(-p,-q,\mu)}}, \omega_{p,q}, v\neq 0}(T))_{((0,0),0)},
				\end{array}
			\end{equation}
			\begin{equation} \label{formulescyclesprochesinfinisegment1}
				\begin{array}{lll}
					S_{f,0}^{\infty} & = & 
					\delta_{(a_0,b_0)}\eps_{(a_0,b_0)}[x^{a_0}y^{b_0} : \Gm^{2}\ra \Gm,\sigma_{\Gm^2}]
					+ \eps_{(a_1,b_1)} [x^{a_1}y^{b_1} : \Gm^{2}\ra \Gm,\sigma_{\Gm^2}]\\
					& + & \eps_{\gamma}[f:\Gm^{2}\setminus f^{-1}(0) \ra \Gm,\sigma_{\gamma}]
					+  \delta_N \sum_{\mu \in  R\gamma}[x^{{\abs{N}}}y^{\nu(\mu)}:\Gm^2 \to \Gm, \sigma_{\Gm^2}], 
				\end{array}
			\end{equation}
			with $\delta_N=-1$ if $N\neq 0$ otherwise 0, and for any root $\mu$, $\nu(\mu)$ is the multiplicity of $\mu$.

	\item  If $pq\geq 0$ we have 
			\begin{equation} \label{ratzetaqh2}
				\begin{array}{lll}
					Z_{\hat{f},i(\mathbb A^2_k)}^{\delta,\infty}(T) & = & 
					[x^{a_0}y^{b_0} : \Gm^{2}\ra \Gm,\sigma_{\Gm^2}]R_{(a_0,b_0)}^{\delta,=}(T)
					+ [x^{a_1}y^{b_1} : \Gm^{2}\ra \Gm,\sigma_{\Gm^2}]R_{(a_1,b_1)}^{\delta,=}(T)\\
					& + & [f:\Gm^{2}\setminus f^{-1}(0) \ra \Gm,\sigma_{\gamma}] R_{\gamma}^{\delta,=}(T) 
					 +   \sum_{\mu \in R_\gamma} 
					 (Z^{\delta/c(p,q)}_{ {f_{{\sigma_{(p,q,\mu)}}}}, \omega_{p,q}, v\neq 0}(T))_{((0,0),0)}
				\end{array}
			\end{equation}
			\begin{equation}\label{formulescyclesprochesinfinisegment2}
				S_{f,0}^{\infty} = 0.
			      \end{equation}
	\end{itemize}

\end{itemize}
{All these formulas use Notations \ref{notation-c-omega} and for any face $\gamma$, the expression of $R_{\gamma}^{\delta,=}(T)$ is given in  formula (\ref{eqRgammaa}) of Proposition \ref{prop:casfacecontenuedansaxe}, and formulas (\ref{eqdecompozeta+}), (\ref{eqcas+a1}), (\ref{eqcas+a2}) and (\ref{eqcas+b1})
of Proposition \ref{lem:epsgammafacedim0-infini} 
and $\eps_\gamma$  belongs to $\{-2, -1, 0\}$ if $\gamma$ is zero-dimensional and to $\{0,1\}$ if $\gamma$ is one-dimensional.}
\end{thm}

\begin{proof}
	We give the ideas of the general proof using above notations and refer to subsection \ref{proofthmSfeps} for details. 
	It is similar to the proof of Theorem \ref{thmSfeps} {or Theorem \ref{thm:thmSfinfini}}.
	We consider a polynomial $f$ in $\k[x,y]$ not in $\k[x]$ or $\k[y]$. 
	In subsection \ref{compactification} we consider the compactification $(X,i,\hat{f})$ of the graph of $f$ in $(\mathbb P^1_{\k})^3$.
	In Proposition \ref{rem:decomposition}, we consider the decomposition 
	\begin{equation} 
		Z^{\delta}_{\hat{f},i(\mathbb A^2_\k)}(T) = \sum_{\gamma \in \overline{\N}(f)} Z^{\delta}_{\gamma,+}(T) 
	\end{equation}
	Assume $\GN(f)$ is not a segment. We only consider faces $\gamma$, such that the dual cone $C_\gamma$ is not contained in $\mathbb Z_{\leq 0}\times \mathbb Z_{\leq 0}$, then we only consider faces in $\widetilde{N}(f)$.
	Depending on the fact that the face polynomial $f_\gamma$ vanishes or not on the angular components of the coordinates of an arc 
	(Remark \ref{rem:m}), we decompose in formula (\ref{decompositionzeta=<}), the zeta function $Z^{\delta}_{\gamma,+}(T)$ as a sum of $Z^{\delta,=}_{\gamma,+}(T)$ and $Z^{\delta,<}_{\gamma,+}(T)$. 
	In particular if the face $\gamma$ is zero dimensional then $Z^{\delta,<}_{\gamma,+}(T)$ is zero {and the case of faces contained in coordinate axes is done in Proposition \ref{prop:casfacecontenuedansaxe}}.  
	In Proposition \ref{lem:epsgammafacedim0-infini} we show the rationality and compute the limit of the zeta function $Z^{\delta,=}_{\gamma,+}(T)$. 
	In Proposition \ref{prop:rationalityZ<infini-infini} and Proposition \ref{prop:rationalityZ<infini-zero},
	we prove the decomposition 
	$$Z^{\delta,<}_{\gamma,+}(T) = 
	\sum_{\mu \in R_\gamma} 
	\left(Z^{\delta/c(p,q)}_{{ f_{{\sigma_{(p,q,\mu)}}}}, \omega_{p,q}, v\neq 0}\right)_{((0,0),0)}.$$ 
	We use section \ref{thmSfeps} to obtain the rationality and an expression of the limit of 
	$\left(Z^{\delta/c(p,q)}_{{ f_{{\sigma_{(p,q,\mu)}}}}, \omega_{p,q}, v\neq 0}\right)_{((0,0),0)}$.
	Similarly in subsection \ref{sec:cashorizontalinfini} and \ref{sec:casverticalinfini} we consider the case of the horizontal and vertical faces.
	
	Assume $\GN(f)$ is a segment. We denote by $\gamma$ the one dimensional face of $\overline{\N}(f)$. 
	It is a segment with vertices $(a_0,b_0)$ and $(a_1,b_1)$ {such that $a_0\leq a_1$ and if $a_0=a_1$ then $b_0<b_1$}.
	We choose an equation of the underlying line of $\gamma$ as $ap+bq=N$ with $(p,q)$ in $\mathbb Z^2 \setminus (\mathbb Z_{\leq 0})^2$ and coprime.
	\begin{itemize}
		\item Assume $pq<0$. In that case we have 
			$C_\gamma {\cap \Omega} = \mathbb R_{>0}(p,q) + \mathbb R_{>0}(-p,-q),$
			then using similar ideas as in the previous case, we obtain formula (\ref{ratzetaqh1}).
			We remark also that for any Newton transformation at infinity $\sigma$ associated to a root $\mu$ of $f$ with multiplicity $\nu$, the Newton transforms $f_{{\sigma_{(p,q,\mu)}}}$ and $f_{\sigma(-p,-q,\mu)}$ have the form $u(v,w)v^{-N}w^{\nu}$ and $u(v,w)v^{N}w^{\nu}$ with $u$ a unit.
			Applying Example \ref{ex:monomial}, formula (\ref{formulescyclesprochesinfinisegment1}) is satisfied.
			\item   Assume $pq\geq 0$. In that case we have 
			$C_\gamma {\cap \Omega}= \mathbb R_{>0}(p,q)$
			and then, using similar ideas as in two previous cases, we obtain formula (\ref{ratzetaqh2}). We prove now formula (\ref{formulescyclesprochesinfinisegment2}). 
			First of all, as $pq\geq 0$ we necessarily have $N>0$, then $\eps_\gamma = 0$ and as above 
			$(S_{f_{{\sigma_{(p,q,\mu)}}}, v\neq 0})_{((0,0),0)}$ is zero for any roots $\mu$ of $f$. 
			Furthermore, 
			\begin{itemize}
				\item  if $a_0 = 0$ (resp $b_1 = 0$) then the intersection $C_{(a_0,b_0)}\cap H_{(a_0,b_0)} \cap \Omega$ is empty and 
					$\eps_{(0,b_0)}=0$ (similarly $\eps_{(a_1,0)} = 0$). 
				\item  If $a_0 \neq 0$ (resp $b_1 \neq 0$) then we have
					$\Omega \cap C_{(a_0,b_0)} \cap H_{(a_0,b_0)} = \mathbb R_{>0} (0,-1) + \mathbb R_{>0} (b_0,-a_0) \subset \mathbb R_{>0} \times \mathbb R_{<0}$
					and we obtain $\eps_{(a_0,b_0)} = 0$ (similarly $\eps_{(a_1,b_1)} = 0$) by Proposition \ref{lem:epsgammafacedim0-infini}.
			\end{itemize}
	\end{itemize}
	The monomial case follows from Proposition \ref{lem:epsgammafacedim0-infini}.
\end{proof}

{\begin{rem}  \label{annulationesp} 
	Let  $f$ be a polynomial in $\k[x,y]$ with $\GN(f)$ not a segment. 
	We use notations of Theorem \ref{thmcyclesprochesinfini}.
	\begin{itemize}
		\item We refer to Proposition \ref{prop:casfacecontenuedansaxe} for the value of $\eps_\gamma$ for $\gamma$ a face equal to $(a,0)$ with $a>0$ or $(0,b)$ with $b>0$. We precise two particular cases
			\begin{itemize}
				\item  if the dual cone of the face $(a,0)$ is 
					$C_{(a,0)}=\mathbb R_{>0}(0,-1)+\mathbb R_{>0}\eta$ with $(\eta \mid (1,0))>0$. 
					Then we have $\eps_{(a,0)} = 0$. Indeed, the half-space $H_{(a,0)}=\{(\alpha,\beta)\mid a\alpha<0\}$ does not intersect 
					$C_{(a,0)}$, and we conclude by Proposition \ref{prop:casfacecontenuedansaxe}. 	

				\item[$\bullet$] if the dual cone of the face $(0,b)$ is 
					$C_{(0,b)}=\mathbb R_{>0}(-1,0)+\mathbb R_{>0}\eta$ with $(\eta \mid (0,1))>0$. 
					Then we have $\eps_{(0,b)} = 0$. Indeed, the half-space $H_{(0,b)}=\{(\alpha,\beta)\mid b\beta<0\}$ does not intersect 
					$C_{(0,b)}$, and we conclude by Proposition \ref{prop:casfacecontenuedansaxe}. 	
			\end{itemize}
		\item By Proposition \ref{lem:epsgammafacedim0-infini} and its notations, we have $\eps_\gamma=0$ for any zero-dimensional face $\gamma$ intersection of two one-dimensional faces $\gamma_1$ and $\gamma_2$, with exterior normal vectors $\omega_1$ and $\omega_2$ such that 
			$(\gamma \mid \omega_1) \geq 0$ and $(\gamma \mid \omega_2) \geq 0$. Indeed, under this assumption, the intersection $C_\gamma \cap H_\gamma$ is empty.
		\item If a face $\gamma$ belongs to $\N(f)$ but is not the horizontal or vertical face $\gamma_h$ or $\gamma_v$ of $f$ in Definition 
			\ref{def:Newton-polygon-height}, then its dual cone $C_\gamma$ does not intersect $\Omega$, and $\eps_\gamma = 0$.
	\end{itemize}
\end{rem}
}

From Theorem \ref{thm:inclusion-Btop-Bmot} and Theorem \ref{thmcyclesprochesinfini}, we deduce the following Kouchnirenko {type} formula for the invariant $\lambda$.

\begin{cor}[Kouchnirenko {type} formula for the invariant $\lambda$] \label{calcullambdaaire} 
	{Assume $\k= \mathbb C$}. Let $f$ be a polynomial in $\k[x,y]$ {not in $\k[x]$ or $\k[y]$}, with isolated singularities.

	\begin{itemize}
		\item  Assume $\GN(f)$ is not a segment. Then we have,
			\begin{equation} \label{formule:lambda}	
				\lambda_{0}(f)  =  
				2\sum_{\gamma \in {\widetilde{\N}(f)},\dim \gamma=1}\eps_{\gamma}\mathcal S_{\N(f_\gamma),f_\gamma} 
				- \sum_{\gamma \in {\widetilde{\N}(f)},\: \dim \gamma = 1} \sum_{\mu \in R_\gamma} 
				\tilde{\chi}_{c}\left(\left(S_{f_{{\sigma_{(p,q,\mu)}}}, v\neq 0}\right)^{(1)}_{((0,0),0)}\right),
			\end{equation}
		
		where $\eps_\gamma$ belongs to $\{0,1\}$ for faces of dimension 1.
	\item  Assume $f$ to be quasi homogeneous polynomial with simple roots and which is not monomial. Let $\gamma$ be its one-dimensional face with primitive exterior normal vector $(p,q)$ and underlying line of equation $ap+bq=N$. Then, we have 
		\begin{equation} \label{formule:lambda-segment} 
			\lambda_{0}(f)  =  2\eps_{\gamma}\mathcal S_{\N(f_\gamma),f_\gamma}.
		\end{equation}
		where $\eps_\gamma$ belongs to $\{0,1\}$.
		More precisely we have, if $pq\geq 0$ then $\lambda_{0}(f)=0$, if $pq<0$ and $N=0$ then $\lambda_{0}(f)=0$ and
		if $pq<0$ and $N\neq 0$ then $\lambda_{0}(f) = 2\mathcal S_{\N(f_\gamma),f_\gamma}=2\mathcal S_\gamma$ where $\mathcal S_\gamma$ is the area associated to the face $\gamma$ (Proposition \ref{prop:caracteristiceuleraire}).
	\end{itemize}
\end{cor}
\begin{proof}
	This corollary follows from Theorem \ref{chilambda}, Theorem \ref{thmcyclesprochesinfini}, Remark \ref{annulationesp} 
	and Proposition \ref{prop:caracteristiceuleraire}.
\end{proof}

\begin{example}[Broughton's example] \label{ex:Broughton} 
	We study here Broughton example $f(x,y)=x(xy-1)$.
	\begin{itemize}
		\item The global Newton polygon of $f$ is a segment, then applying {Theorem} \ref{cassegment} we have 
			\begin{equation}\label{eqSfBroughton}
			\begin{array}{ccl}
				S_{f,0}^{\infty} & = & \eps_{x}[x:\Gm^2 \to \Gm, \sigma_{\Gm^2}] + \eps_{x^2y} [x^2y : \Gm^2 \to \Gm, \sigma_{ {\Gm^2} }] 
			+ \eps_{x^2y-x}[x^2y-x : \Gm^2 \setminus (xy=1) \to \Gm,\sigma_{\gamma}] \\
			& + & (S_{ {f_{\sigma(-1,1,1)}},x_1 \neq 0})_{(0,0),0} 
			+ (S_{ { f_{\sigma(1,-1,1)}},x_1 \neq 0})_{(0,0),0}.
		        \end{array}
		\end{equation}
			{We compute now the coefficients $\eps_{x}$, $\eps_{x^2y}$, $\eps_{x^2y-x}$ and the motives 
			$(S_{ {f_{\sigma(-1,1,1)}},x_1 \neq 0})_{(0,0),0}$, 
			$(S_{ { f_{\sigma(1,-1,1)}},x_1 \neq 0})_{(0,0),0}$.}
			\begin{itemize}
				\item  We have $C_x=\{(\alpha,\beta) \mid \alpha+\beta<0\}$ and $H_x=\{(\alpha,\beta)\mid \alpha<0\}$.
					In particular we have
					$$C_x \cap H_x \cap (\mathbb R_{>0} \times \mathbb R_{<0}) = \emptyset \: \text{and} 
					\: 
					C_x \cap H_x \cap (\mathbb R_{<0} \times \mathbb R_{>0}) = \mathbb R_{>0} (-1,0) + \mathbb R_{>0}(-1,1).$$
					Then, by Proposition \ref{prop:casfacecontenuedansaxe} we conclude that $\eps_x=-1$.
				\item We have $C_{x^2y}=\{(\alpha,\beta) \mid \alpha+\beta>0\}$ and 
					$H_{x^2y}=\{(\alpha,\beta)\mid 2\alpha+ \beta<0\}$. 
                                        In particular we have
					$$C_{x^2y} \cap H_{x^2y} \cap (\mathbb R_{>0} \times \mathbb R_{<0}) = \emptyset \: \text{and} 
					\: 
					C_{x^2y} \cap H_{x^2y} \cap (\mathbb R_{<0} \times \mathbb R_{>0}) = 
					\mathbb R_{>0} (-1,1) + \mathbb R_{>0}(-1,2).$$
					Then, by point \ref{cas2biieps} of Proposition \ref{lem:epsgammafacedim0-infini} we conclude that $\eps_{x^2y}=0$.
				\item   By point \ref{cas4beps} of Proposition \ref{lem:epsgammafacedim0-infini}, we have $\eps_{x^2y-x}=1$.
				\item We have $f_{\sigma(-1,1,1)}(x_1,y_1)=x_1y_1(y_1+1)$ and 
					by Example \ref{ex:monomial} 
					$\left(S_{ {f_{\sigma(-1,1,1)}},x_1 \neq 0}\right)_{(0,0),0}
					= -[x_1y_1:\Gm^2 \to \Gm, \sigma_{ {\Gm^2}}].$
				\item We have $f_{\sigma(1,-1,1)}(x_1,y_1)=x_1^{-1}y_1$ and 
					by Example \ref{ex:monomial}  
					$\left(S_{ {f_{\sigma(1,-1,1)}},x_1 \neq 0}\right)_{(0,0),0} = 0.$
			\end{itemize}
			From equality \ref{eqSfBroughton}, we obtain 
			$$S_{f,0}^{\infty} = -[x:\Gm^2 \to \Gm, \sigma_{ {\Gm^2}}]   
			+ [x^2y-x : \Gm^2 \setminus (xy=1) \to \Gm,\sigma_{\gamma}] - [xy:\Gm^2\to \Gm, \sigma_{{\Gm^2}}]$$
			and we have $$S_{f,0}^{\infty,(1)} = -\mathbb L\:\:\text{and}\:\: {\lambda_{0}(f)=1}$$
			using the equalities 
			$[(xy=1) \cap \Gm^2, \sigma_{\hat{\mu}}]=[(z=1) \cap \Gm^2, \sigma_{\hat{\mu}}]=\mathbb L-1,$ by the isomorphism $(x,y)\to (z=xy,y)$ of 
			$\Gm^2$ and
				$[((xy-1)x=1)\cap \Gm^2, \sigma_{\Gm^2}]=[\Gm\setminus \{-1\},\sigma_{\hat{\mu}}]=\mathbb L - 2$
				writing $y=(1+x)x^{-2}$ with the condition $y\neq 0$.
			\item Let $c\neq 0$ be in $\mathbb C$. Applying Theorem \ref{thmcyclesprochesinfini} we have 
				$$\begin{array}{ccl}
					S_{f,c}^{\infty} & = & \eps_x[x:\Gm^2 \to \Gm, \sigma_{\Gm^2}] + \eps_{x^2}[x^2y:\Gm^2 \to \Gm, \sigma_{\Gm^2}] \\
					& + & \eps_{x-c}[x-c:\Gm^2\setminus(x=c) \to \Gm, \sigma_{\Gm^2}] 
					+ \eps_{x^2y-c}[x^2y-c:\Gm^2\setminus(x^2y=c)\to \Gm, \sigma_{\Gm^2}] \\
					& + & \eps_{x^2y-x}[x^2y-x:\Gm^2\setminus (xy-1) \to \Gm, \sigma_{\Gm^2}] 
					+ (S_{ {f_{\sigma(-1,2,c)}},x_1 \neq 0})_{(0,0),0} 
					+ (S_{ {f_{\sigma(1,-1,1)}} ,x_1 \neq 0})_{(0,0),0}.
				   \end{array} $$
				   {We compute now the coefficients $\eps_x$, $\eps_{x^2}$, $\eps_{x-c}$, $\eps_{x^2y-c}$, $\eps_{x^2y-x}$ and the motives 
				   $(S_{ {f_{\sigma(-1,2,c)}},x_1 \neq 0})_{(0,0),0}$, 
				   $(S_{ {f_{\sigma(1,-1,1)}} ,x_1 \neq 0})_{(0,0),0}$.}
				   
				   \begin{itemize}
					   \item We have $C_{x}=\mathbb R_{>0}(0,-1) + \mathbb R_{>0}(1,-1)$ and 
						   $H_{x}=\{(\alpha,\beta)\mid \alpha<0\}$. In particular, $C_x \cap H_x$ is empty, so by Proposition 
						   \ref{prop:casfacecontenuedansaxe}, $\eps_x=0$.
					   \item We have $C_{x^2y} = \mathbb R_{>0}(-1,2) + \mathbb R_{>0}(1,-1)$ and 
						   $H_{x^2y}=\{(\alpha,\beta)\mid 2 \alpha + \beta <0\}$.
						   In particular, $C_{x^2y} \cap H_{x^2y}$ is empty, 
						   so by point \ref{cas2aeps} 
						   of Proposition \ref{lem:epsgammafacedim0-infini} 
						   we have $\eps_{x^2y} = 0$.
					   \item We have $C_{x-c}=\mathbb R_{>0}(0,-1)$, so $C_{x-c}\cap \Omega$ is empty, so by point \ref{cas3aeps} 
					   of Proposition \ref{lem:epsgammafacedim0-infini}, $\eps_{x-c}=0$.
					   \item By point \ref{cas3aeps} of Proposition 
					   \ref{lem:epsgammafacedim0-infini}, we have  $\eps_{x^2y-c}=0$ and $\eps_{x^2y-x} = 0$.
					   \item We have $f_{\sigma(-1,2,c)}(x_1,y_1)=2cy_1-x_1+ cy_1^2-x_1y_1$.
						   Then, by Example \ref{ex:casdebase-f} we have 
						   $(S_{ { f_{\sigma(-1,2,c)}},x_1 \neq 0})_{(0,0),0} = 0.$
					   \item We have $f_{\sigma(1,-1,1)}(x_1,y_1)=x_1^{-1}(y_1-cx_1)$.
						   Then, by Example \ref{ex:casdebase-f} we have 
						   $(S_{ { f_{\sigma(1,-1,1)}},x_1 \neq 0})_{(0,0),0} = 0.$
				   \end{itemize}
				   Then, we conclude that $S_{f,c}^{\infty}=0$ and {$\lambda_{c}(f)=0$}.
	\end{itemize}
\end{example}

{
\begin{example}[Non degenerate and convenient case] 
	In \cite[Th\'eor\`eme 4.8]{Rai10a}, the second author proved that for any convenient and non-degenerate polynomial $f$ at infinity, the motivic nearby cycle 
	$S_{f,c}^{\infty}$ is zero for any value $c$ in $\k$. We can also deduce this result in dimension 2 from Theorem \ref{thmcyclesprochesinfini}.
	Let $c$ be in $\k$ and consider formula (\ref{formulecyclesprochesinfini}) of Theorem \ref{thmcyclesprochesinfini}.
\begin{itemize}
	\item As $f$ is convenient, the Newton polygon at infinity $\N_{\infty}(f)$ has two zero dimensional faces $(a,0)$ and $(0,b)$. Then, we observe that for any face $\gamma$ in $\overline{\N}(f) \setminus \N_{\infty}(f)^o$, the intersection $C_\gamma \cap \Omega$ is empty, so $\eps_\gamma = 0$. Furthermore, any one dimensional face $\gamma$ in $\N_{\infty}(f)^o$, is supported by a line of equation $ap+bq=N$ with $N>0$, and $(p,q)$ the primitive exterior normal vector to $\N_{\infty}(f)$. Then, by point \ref{cas3aeps} of Proposition \ref{lem:epsgammafacedim0-infini} we also have $\eps_\gamma=0$.

	\item  As the Newton polygon is convenient and convex, it follows from Proposition \ref{prop:casfacecontenuedansaxe} and 
		Proposition \ref{lem:epsgammafacedim0-infini} (see also Remark \ref{annulationesp}) that $\eps_\gamma=0$ for $\gamma$ a face of dimension zero. 
	\item Let $\gamma$ be a one dimensional face in $\N_{(\infty,\infty)}$ (the cases $\N_{(0,\infty)}$ and $\N_{(\infty,0)}$ are similar).
		This face $\gamma$ is supported by a line of equation $ap+bq=N$, with $(p,q)$ the primitive exterior normal vector to $\overline{\N}(f)$ and $N>0$. 
		As $f$ is non degenerate for its Newton polygon $\N_{\infty}(f)$, for any roots $\mu$ of $f_\gamma$, applying the Newton transformation ${\sigma_{(p,q,\mu)}}$, 
		we obtain a Newton transform of the form 
		$f_{{\sigma_{(p,q,\mu)}}}(x_1,y_1) = x_1^{-N}y_1 u(x_1,y_1)$ 
	        with $u$ a unit or 
		$f_{{\sigma_{(p,q,\mu)}}}(x_1,y_1) = x_1^{-N}(c_y y_1 + c_x x_1^m + g(x_1,y_1)),$
		with $c_y$ and $c_x$ in $\k$, $m>0$ and $g(x_1,y_1)=\sum_{a+bm>0}c_{a,b}x_1^ay_1^b$. 
		Then, applying Examples \ref{ex:monomial} and \ref{ex:casdebase-f},
		we have $(S_{ f_{{\sigma_{(p,q,\mu)}}},x_1\neq 0})_{(0,0),0} = 0$. 

	\item For the vertical face (not contained in the coordinate axis), the face polynomial is $f_\gamma(x,y)=x^MT(y)$, with $M>0$.
		We remark that for any root $\mu$ of $T$, the polynomial $f(1/x,y+\mu)$ as the form $x^{-M}(c_y y + c_x x^m + g(x,y))$
		with $c_y$ and $c_x$ in $\k$, $m>0$ and $g(x,y)=\sum_{a+bm>0}c_{a,b}x^ay^b$. Then, applying Example \ref{ex:casdebase-f},
		we have $\left(S_{f_{\infty,0,\mu},x\neq 0}\right)_{(0,0),0} = 0$.
		The result and the proof are the same for the horizontal case.
\end{itemize}
Then, by formula (\ref{formulecyclesprochesinfini}), $S_{f,c}^{\infty}=0$. 
\end{example}
}

\begin{example} \label{example:example1Sfcinfini} We consider the polynomial $f$ defined in Example \ref{example:example1}. We remark that $f(0,0)=1$. Let $c$ be in $\k$.  We use all the notations of Example \ref{example:example1}. By formula (\ref{formulecyclesprochesinfini}) of Theorem \ref{thmcyclesprochesinfini}, we have 
	$$
	\begin{array}{ccl}
		S_{f,c}^{\infty} & = & \eps_{y}[y:\Gm^2 \to \Gm] + \eps_{x^2}[x^2:\Gm^2 \to \Gm] + 
		\eps_{\gamma_{1}^{(0)}}[y(x^2y+1)^3 : \Gm^2 \setminus (x^2y+1=0) \to \Gm, \sigma_{\gamma_{1}^{(0)}} ] \\
		& + & \eps_{\gamma_{2}^{(0)}}[x^2(xy+1)^4:\Gm^2 \setminus (xy+1=0) \to \Gm, \sigma_{\gamma_{2}^{(0)}} ] + 
		\eps_{x^4y^6}[x^4y^6:\Gm^2 \to \Gm, \sigma_{\Gm^2}] \\
		& + & \left( S_{f_1-c,v\neq 0} \right)_{(0,0),0} + \left( S_{f_2-c,v\neq 0} \right)_{(0,0),0}.
	\end{array}
	$$
	By Remark \ref{annulationesp} and Propositions \ref{prop:casfacecontenuedansaxe} and \ref{lem:epsgammafacedim0-infini} , we have 
	$\eps_y = 0$, $\eps_{x^2} = 0$, $\eps_{\gamma_{1}^{(0)}}=0$, $\eps_{\gamma_{2}^{(0)}}=0$ and $\eps_{x^4y^6} = 0$.
	Furthermore,  $\N(f_1-c)^{+}$ is empty and by Theorem \ref{thmSfeps}, we have
	$(S_{f_1-c,v\neq 0})_{(0,0),0} = (S_{(f_1)_{\sigma_{\gamma_1}^{(1)}}-c,v\neq 0})_{(0,0),0} = 0$ 
	because $(f_1)_{\sigma_{\gamma_1}^{(1)}}-c$ is an Example \ref{ex:casdebase-f} with $M=3$ and $m=1$ by formula (\ref{exemplef1}), then we obtain the equality
	$$\begin{array}{ccl}
		S_{f,c}^{\infty} = (S_{f_2-c,v\neq 0})_{(0,0),0}.
	\end{array}
	$$

	\begin{itemize}
		\item We assume $c \notin \{1,2\}$. The set $\N(f_2-c)^{+}$ is empty and by Theorem \ref{thmSfeps}, we have
			$$(S_{f_2-c,v\neq 0})_{(0,0),0} = 
			(S_{(f_2)_{\sigma_{\gamma_1,\mu_1}^{(2)}}-c,v\neq 0})_{(0,0),0} + 
			(S_{(f_2)_{\sigma_{\gamma_1,\mu_2}^{(2)}}-c,v\neq 0})_{(0,0),0} = 0
			$$
			because $(f_2)_{\sigma_{\gamma_1,\mu_i}^{(2)}}-c$ for $i$ in $\{1,2\}$  
			is an Example \ref{ex:monomial} or Example \ref{ex:casdebase-f} with $M=0$ and $m=1$ by formula (\ref{exemplef2}).
			Assume $\k=\mathbb C$, as $S_{f,c}^{\infty}=0$, it follows from Corollary \ref{calcullambdaaire} or Theorem \ref{chilambda} that $\lambda_{c}(f)=0$.

		\item We assume $c=2$. We observe that $\N(f_2)^{+} = \{\gamma_{v}, \gamma_{2}^{(2)}\}$. Then, applying Theorem \ref{thmSfeps}, we have
			$$
			\begin{array}{ccl}
				\left( S_{f_2-2,v\neq 0} \right)_{(0,0),0} & = & 
				[v:\Gm \to \Gm] + [2v^{-1}w^2-4w-v : \Gm^2 \setminus (2v^{-1}w^2-4w-v = 0) \to \Gm, \sigma_{\gamma_{2}^{(2)}}] \\
				& + & 
				(S_{(f_2)_{\sigma_{\gamma_1}^{(2)}}-2,v_1 \neq 0})_{(0,0),0} 
				+ (S_{(f_2)_{\sigma_{\gamma_2,r_1}^{(2)}}-2,v_1 \neq 0})_{(0,0),0} 
				+ (S_{(f_2)_{\sigma_{\gamma_2,r_2}^{(2)}}-2,v_1 \neq 0 })_{(0,0),0}
			\end{array}.
			$$
			As the Newton transform $(f_2)_{\sigma_{\gamma_1}^{(2)}}-2$ is an Example \ref{ex:casdebase-f} with $M=0$ and $m=1$, we have 
			$(S_{(f_2)_{\sigma_{\gamma_1}^{(2)}}-2,v_1 \neq 0})_{(0,0),0} = 0.$
			As well, the Newton transforms $(f_2)_{\sigma_{\gamma_2,r_1}^{(2)}}-2$ and $(f_2)_{\sigma_{\gamma_2,r_2}^{(2)}}-2$ are Examples \ref{ex:casdebase-f} with $M=-1$ and $m=1$, then we have 
			$$(S_{(f_2)_{\sigma_{\gamma_2,r_1}^{(2)}}-2,v_1 \neq 0})_{(0,0),0} = 
			(S_{(f_2)_{\sigma_{\gamma_2,r_2}^{(2)}}-2,v_1 \neq 0})_{(0,0),0} = 
			-[xy:\Gm^2 \to \Gm, \sigma_{\Gm^2}].
			$$
			We conclude that 
			$$S_{f,2}^{\infty} = [v:\Gm \to \Gm] + [2v^{-1}w^2-4w-v : \Gm^2 \setminus (2v^{-1}w^2-4w-v = 0) \to \Gm, \sigma_{\gamma_{2}^{(2)}}] 
			-2[xy:\Gm^2 \to \Gm, \sigma_{\Gm^2}].$$
			Assume $\k=\mathbb C$. By Proposition \ref{prop:caracteristiceuleraire} we have
			$\chi_{c}((2v^{-1}w^2-4w-v=1)\cap \Gm^2) = -2$.
			It follows from Corollary \ref{calcullambdaaire} or Theorem \ref{chilambda} that
			$\lambda_{2}(f)=-\chi_{c}(S_{f,2}^{\infty, (1)})=-(1-2+0) = 1.$
		\item We assume $c=1$. Applying Theorem \ref{thmSfeps}, we have
			$$(S_{f_2-1,v\neq 0})_{(0,0),0} = (S_{f_3,v_1 \neq 0})_{(0,0),0},$$
			with
			$(S_{f_3,v_1\neq 0})_{(0,0),0} = 
			[v_1:\Gm \to \Gm,\sigma_{\Gm}]+[4w_1^2-7v_1:\Gm^2 \setminus (4w_1^2=7v_1) \to \Gm,\sigma_{\gamma_1^{(2)}}]
			+ (S_{(f_3)_{\sigma_{\gamma}^{(3)}},v_2 \neq 0})_{(0,0),0},$
			with 
			$(S_{(f_3)_{\sigma_{\gamma}^{(3)}},v_2 \neq 0})_{(0,0),0} = 
			-[v_2^{2}w_2:\Gm^2\to \Gm,\sigma_{\Gm^2}],$
			because $f_3$ is an Example \ref{ex:casdebase-f} with $M=-2$ and $m=1$. 
			Assume $\k=\mathbb C$. By Proposition \ref{prop:caracteristiceuleraire} we have
			$\chi_{c}( (4w_1^{2}-7v_1=1) \cap \Gm^2) = -2.$
			It follows from Corollary \ref{calcullambdaaire} or Theorem \ref{chilambda} that   
			$\lambda_{1}(f)=-\chi_{c}(S_{f,1}^{\infty, (1)})=-(1-2+0) = 1.$	
	\end{itemize}
	By Example \ref{example:example1} and Example \ref{example:example1Sfinfini}, we have
	$\chi_{c}(f^{-1}(a_{\text{gen}}))=-3$, $\mu(f)=2$ and $\lambda(f)=2$. The second formula \ref{formulechigenerique} is satisfied.
\end{example}

\subsubsection{Motivic bifurcation set} \label{sec:motivic-bifurcation-set}
{We recall in dimension 2, the notion of motivic bifurcation set defined in \cite{Rai11} and \cite{FR}.}
\begin{defn}[Motivic bifurcation set]
	 For any polynomial $f$ in $\k[x,y]$,
				the \emph{motivic bifurcation set} defined as
			$$\mathcal B_{f}^{\text{mot}} = \{a\in \mathbb A^1_{\k} \mid S_{f,a}^{\infty} \neq 0\} \cup \text{disc}(f)$$ is a finite set.
\end{defn}

\subsubsection{Equalities of bifurcation sets}

In this section we prove the following theorem. 
\begin{thm} \label{Bftop=Bfmot} Let $f$ be a polynomial in ${\mathbb C}[x,y]$ not in $\mathbb C[x]$ or $\mathbb C[y]$. If $f$ has isolated singularities then 
$$\mathcal B_{f}^{\text{top}} = \mathcal B_{f}^{\text{Newton}} = \mathcal B_{f}^{\text{mot}}.$$
	\end{thm}
\begin{proof}
	Indeed, we deduce from Theorem \ref{chigenchia} and Theorem \ref{chilambda} that $\mathcal B_{f}^{\text{top}}$ is included in $\mathcal B_{f}^{\text{mot}}$ (see \cite{FR}).
	We deduce from Proposition \ref{lambdanulgenerique} that $B_{f}^{\text{Newton}}$ is included in $\mathcal B_{f}^{\text{top}}$. We deduce from Proposition \ref{generic-Sfnul} the inclusion of $\mathcal B_{f}^{\text{mot}}$ in $\mathcal B_{f}^{\text{Newton}}$.
\end{proof}

\begin{rem} This result generalize the result in the non-degenerate case of Némethi and Zaharia in \cite{NemZah90a}. We recover also from Proposition \ref{BfNewtonfini} that $B_{f}^{\text{mot}}$ is finite. From Theorem \ref{Bftop=Bfmot} and \cite{FR}, we deduce the equality  
	$$\mathcal B_{f}^{\text{top}} = \mathcal B_{f}^{\text{Newton}} = \mathcal B_{f}^{\text{Serre}} = \mathcal B_{f}^{\text{mot}},$$
	where $\mathcal B_{f}^{\text{Serre}}$ is the Serre bifurcation set defined in \cite{FR} using analytic geometry in the Berkovich sense.
\end{rem}

To prove Proposition \ref{lambdanulgenerique}, we need the following Lemma \ref{lem:positivite-nullite} and Lemma \ref{lambdanulgenerique-h}. 

\begin{lem} \label{lem:positivite-nullite} \label{lem:positivite} Let  $h$ be polynomial in $\k[x^{-1},x,y]$ which is not of the form $x^{-M}u(x,y)$ with $u$ a unit. Then, we have  
	$$\tilde{\chi}_c\left((S_{h,x\neq 0})^{(1)}_{(0,0),0}\right) \leq 0.$$
\end{lem}

\begin{proof}
	To compute  $\tilde{\chi}_c\left((S_{h,x\neq 0})_{(0,0),0}^{(1)}\right)$, we use formula (\ref{KFSf+}) and (\ref{KFSf+b}) of Corollary \ref{KFSfeps}. 
	We prove the result by induction on the Newton process.
	\begin{itemize}
		\item For the base case {Example} \ref{ex:monomial} with $m \neq 0$ and base case {Example} \ref{ex:casdebase-f}, we have $\tilde{\chi}_c\left(\left(S_{h,x\neq 0}\right)_{(0,0),0}^{(1)}\right) = 0$. 
		\item By Newton algorithm, at each step the considered polynomial is not of the type of {Example} \ref{ex:monomial} with 
		$m = 0$. Indeed, {by assumption $h$ is not of the form} $x^{-M}u(x,y)$ with $u$ a unit. {If $h$ is {an Example} \ref{ex:monomial} necessarily $m \neq 0$.} Otherwise, $h$ has at least one compact one dimensional face, and for each of them, for each root $\mu$ of the face polynomial the corresponding multiplicity $\nu$ is larger than $1$ and the Newton transform is 	
		\begin{center} $h_{{\sigma_{(p,q,\mu)}}}(x,y) = x^{N}\left(y^\nu u(y) + \sum_{(k,l)\notin \gamma} x^{((k,l)\mid (p,q))-N} u_{k,l}(y)\right),$
		\end{center}
		where $u$ and $u_{k,l}$ are units in $y$. Then, $h_{{\sigma_{(p,q,\mu)}}}$ is not of the type of Example \ref{ex:monomial} with $m=0$. 
	        \end{itemize}
	Now assume the induction hypothesis: $\tilde{\chi}_c\left(\left(S_{h_{{\sigma_{(p,q,\mu)}}},x\neq 0}\right)_{(0,0),0}^{(1)}\right) \leq 0$, for each face of the Newton polygon of $h$ and for each root of the corresponding face polynomial. Then, by formulas (\ref{KFSf+}) and (\ref{KFSf+b}) of Corollary \ref{KFSfeps} we get the result. Indeed, 
	\begin{itemize}
		\item  for each one dimensional face $\gamma$ in $\N(h)^{+}$ we have 
			$-2S_{\N(h_\gamma),f_\gamma}\leq 0$. 
		\item  if we assume that the Newton polygon $\N(h)$ has a face $\gamma$ with vertices $(-M,d)$ and $(a,0)$, with $-M$ and $a$ in 
	$\mathbb Z$, $d$ in $\mathbb N^*$ and $a>-M$. We denote by $h_\gamma$ the face polynomial. It follows from Definition \ref{area-f} and its notations that 
	\begin{equation} \label{lemcleairebis} a - 2\mathcal S_{\N({h}_\gamma),h_\gamma} = a\left(1 - \frac{rd}{s+1}\right) {\leq 0} \end{equation}
	because $d\geq s+1$. In particular this is equal to zero if and only if $d=s+1$ and $r=1$. 
	\end{itemize}
	Then, we conclude by induction using the Newton algorithm and get the result.
\end{proof}
\begin{lem} \label{lambdanulgenerique-h} Let $h$ be in $\mathbb C[x^{-1},x,y]$ with isolated singularities and $c_0$ be in {$\mathbb C$}. 
	If ${\tilde{\chi_{c}}}\left((S_{h-c_0,{x}\neq 0})_{(0,0),0}^{(1)}\right) = 0$
	then $c_0$ does not belong to $B_{h}^{\text{Newton}}$.
\end{lem}

	\begin{proof}
	{As $h$ has isolated singularities in $\mathbb C^{*} \times \mathbb C$, then by Lemma \ref{lemme:singularitesisolees}, for any composition $\Sigma$ of Newton transformations, the polynomial $h_\Sigma$ has isolated singularities in $\mathbb C^{*} \times \mathbb C$.}
	Let $c_0$ be in {$\mathbb C$}.
	We assume ${\tilde{\chi}_{c}}\left(\left(S_{h-c_0,{x}\neq 0}\right)_{(0,0),0}^{(1)}\right) = 0$. 
	We remark that using {Corollary \ref{KFSfeps} and Lemma \ref{lem:positivite} (and formula (\ref{lemcleairebis}) in its proof)} we have
	\begin{equation} \label{nulliteh}
		\tilde{\chi}_{c}\left(\left(S_{ (h-c_0)_{ {\Sigma}}, v\neq 0}\right)^{ {(1)} }_{((0,0),0)}\right) =  0
	\end{equation}
	for any {composition $\Sigma$ of Newton transformations}. {We denote by $c_{(0,0)}(h)$ the constant term of $h$.} 
	\begin{itemize}
		\item  Assume $c_0$ is a Newton non generic value of $h$. Then, the polynomial $h$ has a dicritical face $\gamma_0$.
	\begin{enumerate}
		\item Assume $c_0\neq c_{(0,0)}(h)$. {Then the face $\gamma_0$ belongs to $\N(h-c_0)$}. {The associated polynomial of the face $\gamma_0$ (Definition \ref{def:face-dicritique-discriminant})} is equal to $P_{\gamma_0}(x^{-a}y^b)$ with 
			$P_{\gamma_0}$ a polynomial in ${\mathbb C}[s]$. {As $c_0$ is a Newton non generic value and $c_0\neq c_{(0,0)}(h)$, by Definition \ref{Newtongenericvalues-local}} the polynomial $P_{\gamma_0}(s)-c_0$ has a non simple root $\mu \neq 0$ {with multiplicity $\nu \geq 2$. Applying the corresponding Newton transformation ${\sigma_{(p,q,\mu)}}$ defined in Definition \ref{defn:Newton-map-local}, we obtain the expression }
			\begin{equation}
				(h-c_0)_{{\sigma_{(p,q,\mu)}}} = u(w)w^{\nu} + \sum_{(k,l)\in \supp{(h-c_0)}\setminus \{\gamma_0\}} v^{\abs{((k,l)\mid((p,q))}}u_{(k,l)}(w),
			\end{equation}
			where $u$ and $u_{k,l}$ are units.
			Then, {$(h-c_0)_{{\sigma_{(p,q,\mu)}}}$ has at least a one dimensional face with a non zero area and by 
			Lemma \ref{lem:positivite-nullite} and Corollary \ref{KFSfeps}} we get 
			$\tilde{\chi}_{c}\left((S_{( h-c_0)_{{\sigma_{(p,q,\mu)}}}, v\neq 0})^{(1)}_{((0,0),0)}\right) \neq 0$ which contradicts 
			{formula (\ref{nulliteh})}. 
		\item Assume $c_0 = c_{(0,0)}(h)$. Let $\gamma$ be the face of maximal dimension of $\N(h-c_0)$ such that $\gamma \cap \gamma_0$ is not empty.

			{It has only one vertex $(a_0,b_0)$ if its dimension is zero, otherwise it has two vertices $(a_0,b_0)$ and $(a_1,b_1)$ with $a_0>a_1$. }
			\begin{itemize}
			\item Assume the polynomial $P_{\gamma_0}(s)-c_0$ has {a root $\mu$ with multiplicity $\nu \geq 2$.}
					\begin{enumerate}
						\item If $\mu\neq 0$ then as above
							$\tilde{\chi}_{c}\left((S_{ (h-c_0)_{{\sigma_{(p,q,\mu)}}}, v\neq 0})^{(1)}_{((0,0),0)}\right) \neq 0$ which contradicts formula (\ref{nulliteh}). 
						\item \label{casb0>1} Assume $\mu = 0$. Necessarily $b_0 \geq 2$ otherwise {$\mu$} is not a root of 
							{$P_{\gamma_0}(s)-c_0$} with multiplicity {$\nu \geq 2$.}. If the point $(a_0,b_0)$ is the horizontal face of $h-c_0$ then the singularities of 
							$h-{c_0}$ are not isolated (the line $y=0$ is critical). Then $(a_0,b_0)$ is not the horizontal face, then there is a {one dimensional} face $\gamma'\neq \gamma$ with {vertex} $(a_0,b_0)$.
							Necessarily $\N^{+}(h-c_0)$ contains $\gamma'$ {with $S_{\N(h_{\gamma'}),h_{\gamma'}}$ a non zero area, then by Lemma \ref{lem:positivite-nullite} and Corollary \ref{KFSfeps}}, 
							$\tilde{\chi}_{c}\left(\left(S_{h-c_0, x\neq 0}\right)^{(1)}_{((0,0),0)}\right) \neq 0$
							which contradicts formula (\ref{nulliteh}). 
					\end{enumerate}
				\item If all the roots of $P_\gamma(s)-{c_0}$ are simple, {then} necessarily the face $\gamma_0$ is not smooth and $b_0\geq 2$. Then, arguing as in (\ref{casb0>1}) we obtain a contradiction {with formula (\ref{nulliteh})}. 

			\end{itemize}

	\end{enumerate}
\item  If $c_0$ is a Newton generic value for $h$, {and} $c_0$ belongs to $B_{h}^{\text{Newton}}$, {then} there is a polynomial $\tilde{h}$ in the Newton algorithm of 
	$h-{c_0}$, which by assumption has isolated singularities and such that 
	$c_0$ belongs to $B_{\tilde{h}}^{\text{Newton}}$, then as above we have 
	$\tilde{\chi}_{c}\left((S_{ {(\tilde{h}-c_0)_{{\sigma_{(p,q,\mu)}}}}, v\neq 0})^{(1)}_{((0,0),0)}\right) \neq  0$
	which is a contradiction {with formula (\ref{nulliteh}).} 
	\end{itemize}
\end{proof}

\begin{prop} \label{lambdanulgenerique} Let $f$ be a polynomial in ${\mathbb C}[x,y]$ with isolated singularities. Let {$c_0$} be in ${\mathbb C}\setminus disc(f)$. If $\lambda_{ {c_0} }(f)=0$ then {$c_0$} does not belong to $B_{f}^{\text{Newton}}$. {Then, we have the inclusion $B_f^{\text{Newton}} \subset B_f^{\text{top}}$.} 
\end{prop}

\begin{proof}
	{As $f$ has isolated singularities in $\mathbb C^2$, it follows from Lemma \ref{lemme:singularitesisolees} that for any composition $\Sigma$ of Newton transformations, 
	$f_\Sigma$ has isolated singularities in $\mathbb C^* \times \mathbb C$.}
	Let $c_0$ be in $\mathbb C\setminus disc(f)$ {with $\lambda_{c_0}(f)=0$}. We consider the computations of $\lambda_{c_0}(f)$ in Corollary \ref{calcullambdaaire} for $f-c_0$.\\
	{$\bullet$ Assume $f-c_0$ is not a quasi homogeneous polynomial.}
	As $\lambda_{c_0}(f)=0$, {using Lemma \ref{lem:positivite} we observe that} in formula (\ref{formule:lambda}) we have $\eps_{\gamma} = 0$ for any one dimensional face $\gamma$ in 
	$\overline{\N}(f-c_0)$ and for any Newton transformation in the first step of {the} Newton algorithm at infinity of $f-c_0$, we have
	\begin{equation} \label{nullite}
				\tilde{\chi}_{c}\left((S_{ {(f-c_0)_{{\sigma_{(p,q,\mu)}}}}, v\neq 0})^{(1)}_{((0,0),0)}\right) =  0.
	\end{equation}
	As $f_{{\sigma_{(p,q,\mu)}}}\in {\mathbb C}[x^{-1},x,y]$ {and $f_{{\sigma_{(p,q,\mu)}}}-c_0 = (f-c_0)_{{\sigma_{(p,q,\mu)}}}$}, we deduce by Lemma \ref{lambdanulgenerique-h} that $c_0$ does not belong to 
	$B_{f_{{\sigma_{(p,q,\mu)}}}}^{\text{Newton}}$ and then neither, to any $B_{f_\Sigma}^{\text{Newton}}$ for any composition $\Sigma$ of Newton transformations.
	{Thus,} to prove the result it is enough to show that $c_0$ is not a non Newton generic value at infinity. 
	{We proceed by contradiction. Assume $c_0$ is a non Newton generic value at infinity. So, there is a dicritical face at infinity $\gamma_0$ of $f$ 
		with face polynomial $P_{\gamma_0}(x^{m}y^{n})$ with $P_{\gamma_0}$ in $\mathbb C[s]$.}
	\begin{enumerate}
		\item \label{(case1)} Assume $c_0 \neq f(0,0)$.
			The polynomial $P_{\gamma_0}(s)-c_0$ has a root $\mu$ of multiplicity $\nu\geq 2$. {We denote by $(p,q)$ the primitive exterior normal vector to the face $\gamma_0$.} Applying the corresponding Newton transformation ${\sigma_{(p,q,\mu)}}$, we {get} 
		\begin{equation} \label{NTfc0}
			(f-c_0)_{{\sigma_{(p,q,\mu)}}} = u(w)w^{\nu} + \sum_{(k,l)\in \supp{(f-c_0)}\setminus \{\gamma_0\}} v^{\abs{((k,l)\mid((p,q))}}u_{(k,l)}(w),
		\end{equation}
		where $u$ and $u_{k,l}$ are units. We remark that $\N^{+}(f-c_0)$ is not empty and contains a face of dimension 1, then it follows from Corollary \ref{KFSfeps} that 
		$\tilde{\chi}_{c}\left((S_{ (f-c_0)_{{\sigma_{(p,q,\mu)}}}, v\neq 0})^{(1)}_{((0,0),0)}\right) \neq 0$ which is a contradiction {with formula (\ref{nullite})}. 		
	\item Assume $c_0 = f(0,0)$. 
		Let $\gamma $ be  the face of maximal dimension of $\overline{\N}(f-c_0)$ such that $\gamma \cap \gamma_0 \neq \emptyset$. It has only one vertex $(a_0,b_0)$ if its dimension is zero, otherwise it has two vertices $(a_0,b_0)$ and $(a_1,b_1)$ with $a_0<a_1$.
			\begin{itemize}
				\item Assume that the polynomial $P_{\gamma_0}(s)-c_0$ has a root $\mu$ with a multiplicity $\nu\geq 2$. 
					\begin{enumerate}
				\item \label{cas:dicriticracinenonzero} Assume the root $\mu\neq 0$. The situation is similar to above case \ref{(case1)} and gives a contradiction.
				\item \label{cas:dicriticracinezero} Assume the root $\mu=0$. As $\nu \geq 2$ we have $a_0\geq 2$.
					{If $x^{a_0}$ divides $f-c_0$ then $f-c_0$ does not have isolated singularities in $\mathbb C^2$. Then, $x^{a_0}$ does not divide $f-c_0$, so there is a one dimensional face $\gamma'$ in $\GN(f-c_0)\setminus \N(f-c_0)$ with vertex $(a_0,b_0)$. As $\gamma_0$ is a dicritical face, the face $\gamma'$ is supported by a line of equation $ap+bq=N$ with 
					$(p,q)$ the normal vector exterior to $\overline{\N}(f-c_0)$ and $N<0$. This one dimensional face $\gamma'$ induces a non zero area with $\eps_\gamma'=1$ (by Proposition \ref{lem:epsgammafacedim0-infini}), then by Lemma \ref{lem:positivite} and formula (\ref{formule:lambda}), we have $\lambda_{c_0}(f)\neq 0$. Contradiction.}
			\end{enumerate}
		
		\item Assume that the dicritical face at infinity $\gamma_0$ is non smooth. 
			As $c_0$ is a non critical value {of $f$ and $c_0=f(0,0)$}, the point $(0,0)$ is not critical {and} $x$ or $y$ is in the support of $f-c_0$. Assume for instance $x$ is in the support (the case of $y$ in the support is similar). The dicritical face $\gamma_0$ is supported by an equation $p\alpha+q\beta=0$ with $p<0$ and $q>0$. {By assumption $\gamma_0$} is non smooth, {then}  $q$ is different from $1$. In particular, the face $\gamma$ does not have a point of coordinate $(1,\abs{p})$ {and} we necessarily have $a_0>1$. Then, we are in the similar case as \ref{cas:dicriticracinezero} and obtain a contradiction. 
	\end{itemize}
		\end{enumerate}
		$\bullet$ Assume $f-c_0$ is a quasi homogeneous polynomial. As $\lambda_{c_0}(f)=0$, by Corollary \ref{calcullambdaaire}, we consider the cases $pq<0$ with $N=0$ and $pq\geq 0$. 
		\begin{itemize}
			\item  Assume $pq<0$ and $N=0$. If $f-c_0$ has a root $\mu$ with multiplicity $\nu \geq 2$, then $f$ does not have isolated singularities. If $f(0,0)=c_0$ then as $c_0$ is not a critical value, $(0,0)$ is not a critical point then $x$ or $y$ belongs to the support of $f-c_0$ which is a contradiction with the assumption $N=0$. Then $c_0$ does not belong to $B_f^{\text{Newton}}$.
			\item Assume $pq\geq 0$. As $f-c_0$ is quasi homogeneous, we have $f(0,0)=c_0$. As $c_0$ is not critical then as above $(0,0)$ is not a critical point and we can assume that $x$ is in the support, then up to a constant, there is $p$ in $\mathbb N^*$ such that $f(x,y)-c_0=y^p - \mu x$.
			Then $c_0$ does not belong to $B_f^{\text{Newton}}$.
	\end{itemize}
\end{proof}

\begin{prop} \label{Shcnulle}
	Let $h$ be in $\k[x^{-1},x,y]$ and not in $\k[x,y]$. Let $c$ {be in $\k \setminus B_{h}^{\text{Newton}}$.}
	Then we have $\left(S_{h-c,{x}\neq 0}\right)_{(0,0),0} = 0.$
\end{prop}
\begin{proof}
	{(The base cases). We prove here Proposition \ref{Shcnulle} for the base cases Examples \ref{ex:monomial} or \ref{ex:casdebase-f} with in each case $M>0$. The general case is proved below by induction using Newton algorithm.}
	\begin{itemize}
		\item Consider $h(x,y)=U(x,y)x^{-M_0}y^{m_0}$ with $U$ a unit {and $M_0>0$}. We assume $U(0,0)=1$. 
		Let $c$ be in $\k \setminus B_{h}^{\text{Newton}}$.
		
			\begin{itemize}
				\item  Assume $c\neq 0$. The polynomial $h$ has a dicritical face $\gamma$ which is also a face of $h-c$. 
					{The face polynomial of $h-c$ is $x^{-M_0}y^{m_0}-c = \prod_{i=1}^{d}(x^{-q}y^{p}-\mu_i),$
					where $d=\gcd(M_0,m_0)$, $(p,q)=(m_0,M_0)/d$ and $(\mu_i)$ are the $d$-roots of $c$.
				}
					In particular the set $\N(h-c)^{+}$ is empty.
					{Then, applying Theorem \ref{thmSfeps} we have 
					$$\left(S_{h-c,{x}\neq 0}\right)_{(0,0),0} = 
					\sum_{i=1}^{d} \left(S_{(h-c)_{\sigma_{(p,q,\mu_i)}},v\neq 0}\right)_{(0,0),0} = 0$$
					because for each root $\mu_i$ the Newton transform $(h-c)_{\sigma_{(p,q,\mu_i)}}$ 
					is an Example \ref{ex:monomial} or \ref{ex:casdebase-f} with $M=0$ and $m=1$.} 
					
				\item  Assume $c=0$, by Example \ref{ex:monomial}, we have $\left(S_{h,v\neq 0}\right)_{(0,0),0} = 0$ because $M_0>0$. 			\end{itemize}
		\item We consider $h(x,y)=U(x,y)x^{-M_0}(y-\mu_0 x^q+g(x,y))^{m_0}$ with $U$ a unit, {$M_0>0$} and $g(x,y)=\sum_{a+bq>q}c_{a,b}x^ay^b$.
			{We assume $U(0,0)=1$.}
			The polynomial $h$ has {a one dimensional face} with  vertices $(-M_0,m_0)$ and $(-M_0+m_0q,0)$. We denote the constant term of $h$ by $c_{(0,0)}(h)$.
			\begin{itemize}
				\item Assume $-M_0+m_0q<0$. Then, the set $\N(h-c)^{+}$ is empty for any value $c$. The polynomial $h-c$ has only one one dimensional face {denoted by $\gamma$} and its face polynomial $(h-c)_\gamma$ is $x^{-M_0}(y-\mu_0 x^q)^{m_0}$. The Newton transform $(h-c)_{\sigma_{(1,q,\mu_0)}}$ is an Example \ref{ex:monomial} or \ref{ex:casdebase-f} with $M=-M_0+m_0q<0$ and $m=m_0$, 
					so we have $\left(S_{(h-c)_{\sigma_{(1,q,\mu_0)}},v\neq 0}\right)_{(0,0),0}=0.$ 
			        Applying Theorem \ref{thmSfeps} we have $\left(S_{h-c,{x}\neq 0}\right)_{(0,0),0} = 0$.
				\item Assume $-M_0+m_0q= 0$. The polynomial $h$ has a dicritical face $\gamma$. The polynomial face $h_\gamma$ is equal to 
					$P(x^{-q}y)$ with $P(s)=(s-\mu_0)^{m_0}$ in $\k[s]$. Let $c \notin B_{f}^{\text{Newton}}$. 
				      	\begin{itemize}
						\item If $c\neq c_{(0,0}(h)$ then $\gamma$ is a face of $h-c$ and the set $\N(h-c)^{+}$ is empty. The polynomial $h-c$ has only one  one dimensional face, which is the dicritical face $\gamma$. As $c$ does not belong to $B_{f}^{\text{Newton}}$, 
							the face polynomial $(h-c)_\gamma$ has simple roots. Then, for any root $\mu$ of $(h-c)_\gamma$, the Newton transform $(h-c)_{\sigma_{(1,q,\mu)}}$ is an Example \ref{ex:monomial} or \ref{ex:casdebase-f} with $M=0$ and $m=1$. 
							Then, we have $(S_{(h-c)_{\sigma_{(1,q,\mu)}},v\neq 0})_{(0,0),0} = 0$ and as above applying Theorem \ref{thmSfeps} we have $(S_{h-c,{x}\neq 0})_{(0,0),0} = 0$.
						\item If $c=c_{(0,0)}(h)$, then as $c \notin B_{f}^{\text{Newton}}$, the face $\gamma$ is smooth and the polynomial $P(s)-c$ has simple roots. In particular, the polynomial $h$ has a monomial $x^{-M'}y$ with $M'>0$. 
							\begin{itemize}
								\item Assume the horizontal face of $h-c$ is $(a,0)$. Necessarily we have $a>0$. 
									The polynomial $h-c$ has at most two one dimensional faces: the associated face to the dicritical face of $h$ still denoted by $\gamma$, and a one dimensional with vertices $(a,0)$ and $(-M',1)$ denoted by $\gamma'$. We can assume the face polynomial $(h-c)_{\gamma'}$ to be equal to 
									$x^{-M'}(y-\lambda x^{a+M'})$. We have $\N(h-c)^{+}=\{(a,0),\gamma'\}$.
									The polynomial $(h-c)_{\sigma_{(1,a+M',\lambda)}}$ is an Example \ref{ex:monomial} or
									\ref{ex:casdebase-f} 
									with $M=a>0$ and $m=1$ then we have 
									$$(S_{(h-c)_{\sigma_{(1,a+M',\lambda)}},v\neq 0})_{(0,0),0} = 
									-[x^{a}y:\Gm^2 \to \Gm, \sigma_{\Gm}].$$

									By Theorem \ref{thmSfeps} we have 
									$$ 
									\begin{array}{ccl}
										\left(S_{h-c,x\neq 0}\right)_{(0,0),0} & = & [x^{a}:\Gm \to \Gm,\sigma_{\Gm}] + 
										[x^{-M'}(y-\lambda x^{a+M'})  :
										\Gm^2 \setminus (y=\lambda x^{a+M'}) \to \Gm, \sigma_{\Gm}] \\ \\
										& - & [x^{a}y:\Gm^2 \to \Gm] 
										+  \sum_{\mu' \in R_{\gamma}} 
										(S_{(h-c)_{\sigma(1,M',\mu')},v\neq 0})_{(0,0),0}
									\end{array}.
									$$
									As in the proof of Example \ref{ex:casdebase-f} we have 
									$$
									[x^{a}:\Gm \to \Gm,\sigma_{\Gm}] + 
									[x^{-M'}(y-\lambda x^{a+M'})  : \Gm^2 \setminus (y=\lambda x^{a+M'}) \to \Gm, \sigma_{\Gm}] 
									 =  
									[x^{-M'}y:\Gm^2 \to \Gm, \sigma_{\Gm}] = [x^{a}y:\Gm^2 \to \Gm].$$
									Furthermore, as $c$ does not belong to $B_h^{\text{Newton}}$, any roots $\mu$ in $R_\gamma$ has multiplicity one, then the Newton transform $(h-c)_{\sigma_{(1,M',\mu)}}$ is an Example 
									\ref{ex:monomial} or \ref{ex:casdebase-f}  with $M=0$ and $m=1$, then we have 
									$(S_{(h-c)_{\sigma(1,M',\mu)},v\neq 0})_{(0,0),0} = 0,$
									and, we conclude that 
									$(S_{h-c,{x}\neq 0})_{(0,0),0}  = 0.$

								\item If the horizontal face is not of the form $(a,0)$ then, it is the face $(-M',1)$. Then, the set $\N(h-c)^{+}$ is empty, and as above as $c$ does not belong to $B_{h}^{\text{Newton}}$ we have 
									$\left(S_{h-c,{x}\neq 0}\right)_{(0,0),0} = 0$.
							\end{itemize}
					\end{itemize}
				\item Assume $-M_0+m_0q>0$. 
					\begin{itemize}
						\item  If $c$ is different from $0$, then $\N(h-c)^{+}$ is empty. The polynomial $h-c$ has only one face of dimension one with face polynomial $x^{-M_0}y^{m_0}-c$, the Newton transform $(h-c)_{\sigma_{(m_0,M_0,c)}}$ is an Example \ref{ex:monomial} or \ref{ex:casdebase-f} with $M=0$ and $m=1$, then we conclude as above that $\left(S_{h-c,{x}\neq 0}\right)_{(0,0),0} = 0$.
						\item  If $c$ is equal to zero, and does not belong to $B_{h}^{\text{Newton}}$, then with the same proof as in the case 
							$-M_0+m_0q=0$ and $c=c_{(0,0)}(h)$, we obtain $\left(S_{h-c,{x}\neq 0}\right)_{(0,0),0} = 0$. 
					\end{itemize}

			\end{itemize}

	\end{itemize}
\end{proof}
\begin{proof}(The general case).
We consider the general case of $h$ in $\k[x^{-1},x,y] \setminus \k[x,y]$ and argue by induction using the Newton algorithm. {The base cases in the induction are Examples \ref{ex:monomial} or \ref{ex:casdebase-f} with in each case $M>0$; they are treated above.} Let $c$ be a value not in $B_h^{\text{Newton}}$.
	\begin{itemize}
		\item Assume $c$ to be different from the constant term $c_{(0,0)}(h)$ of $h$.
			Thus, $(0,0)\in \text{Supp}(h-c)$ and $\N(h-c)^{+}$ is empty.
			\begin{itemize}
				\item Assume $h$ has a dicritical face $\gamma$. Then, $\gamma$ is  a face of $\N(h-c)$. Furthermore, as $c\notin B_{h}^{\text{Newton}}$,
					all the roots $\mu$ of the face polynomial $(h-c)_\gamma$ have multiplicity one, and for each associated Newton transformation $\sigma_{(p,q,\mu)}$, the polynomial $(h-c)_\sigma$ is a base case {Example} \ref{ex:monomial} or \ref{ex:casdebase-f} with $M=0,m=1$, then 
					$(S_{(h-c)_{\sigma_{(p,q,\mu)}},v\neq 0})_{(0,0),0}=0.$
					Thus, applying Theorem \ref{thmSfeps} we have 
					$$(S_{h-c,{x}\neq 0})_{(0,0),0} = 
					\sum_{\gamma' \in \N(h-c)\setminus \{\gamma\}} \sum_{\mu \in R_{\gamma'}} 
					(S_{(h-c)_{\sigma_{(p',q',\mu)}},v\neq 0})_{(0,0),0}.$$
					Furthermore, for any face $\gamma' \in \N(h-c)\setminus \{\gamma\}$, $\gamma'$ is supported by a line of equation $\alpha p' + \beta q' = N_{\gamma'}$ with $(p',q')$ the normal of the face $\gamma'$ in $\N(h-c)$ and $N_\gamma'<0$. Then, for each Newton transformation $\sigma_{(p',q',\mu)}$ associated to a root $\mu$ of the face polynomial $(h-c)_{\gamma'}$, the Newton transformation 
					$(h-c)_{\sigma_{(p',q',\mu)}}{=h_{\sigma_{(p',q',\mu)}}-c}$ belongs to $\k[x^{-1},x,y]\setminus \k[x,y]$. Finally, as $c$ does not belong to $B_{h}^{\text{Newton}}$, by definition
					$c$ does not belong to ${B_{h_{\sigma_{(p',q',\mu)}}}^{\text{Newton}}}$. Then, applying the induction we obtain 
					$(S_{(h-c)_{\sigma_{(p',q',\mu)}},v\neq 0})_{(0,0),0}=0$ and we conclude that 
                                        $(S_{h-c,{x}\neq 0})_{(0,0),0} = 0$.

				\item Assume $h$ does not have a dicritical face. Then as $h$ belongs to $\k[x^{-1},x,y] \setminus \k[x,y]$, necessarily the horizontal face of $h$ is $(a,0)$ with $a<0$. Then, the set $\N(h-c)^{+}$ is empty and as above applying Theorem \ref{thmSfeps} we have 
					$$\left(S_{h-c,{x}\neq 0}\right)_{(0,0),0} = 
					\sum_{\gamma' \in \N(h-c)} \sum_{\mu \in R_{\gamma'}} 
					(S_{(h-c)_{\sigma_{(p',q',\mu)}},v\neq 0})_{(0,0),0},$$
					where for each face $\gamma'$, for each associated Newton transformation $\sigma_{(p',q',\mu)}$, the polynomial 
					$(h-c)_{\sigma_{(p',q',\mu)}}$ belongs to $\k[x^{-1},x,y] \setminus \k[x,y]$ and $c$ does not belong to ${B_{h_{\sigma_{(p',q',\mu)}}}^{\text{Newton}}}$.
					Then, applying the induction we obtain 
					$(S_{(h-c)_{\sigma_{(p',q',\mu)}},v\neq 0})_{(0,0),0}=0$ and we conclude that 
                                        $(S_{h-c, x\neq 0})_{(0,0),0} = 0$.
			\end{itemize}
		\item Assume $c$ is equal to the constant term $c_{(0,0)}(h)$ of $h$. In particular $(0,0)$ does not belong to the support of $h-c$.
			\begin{itemize}
				\item If $h$ does not have a dicritical face, then we argue as above and get the result.
				\item Assume that $h$ has a dicritical face $\gamma$. As $c$ does not belong to $B_{h}^{\text{Newton}}$, the dicritical face $\gamma$ is smooth and supported by a line of equation $\alpha + \beta q = 1$. Furthermore the face polynomial of $h_\gamma$ is $P_\gamma(x^{-q}y)$ and the polynomial $P_\gamma(s)-c$ has simple roots. Then, we deduce that $h$ has a monomial $x^{-q}y$ and for each non zero root $\mu$ of $P_\gamma(s)-c$ we have as above
				$(S_{(h-c)_{\sigma_{(p,q,\mu)}},v\neq 0})_{(0,0),0}=0.$
				\begin{itemize}
					\item Assume the horizontal face of $h-c$ is $(-q,1)$ then $\N(h-c)^{+}$ is empty and 
						again by Theorem \ref{thmSfeps} we have 
						$$(S_{h-c,{x}\neq 0})_{(0,0),0} = 
						\sum_{\gamma' \in \N(h-c)\setminus \{\gamma\}} \sum_{\mu \in R_{\gamma'}} 
						(S_{(h-c)_{\sigma_{(p',q',\mu)}},v\neq 0})_{(0,0),0},$$
						where for each face $\gamma'$, for each associated Newton transformation $\sigma_{(p',q',\mu)}$, the polynomial 
						$(h-c)_{\sigma_{(p',q',\mu)}}$ belongs to $\k[x^{-1},x,y] \setminus \k[x,y]$ and $c$ does not belong to 
						{$B^{\text{Newton}}_{h_{\sigma_{(p',q',\mu)}}}$}.
						Then, applying the induction we obtain 
						$(S_{(h-c)_{\sigma_{(p',q',\mu)}},v\neq 0})_{(0,0),0}=0$ and we conclude that 
						$(S_{h-c, x\neq 0})_{(0,0),0} = 0$.

					\item Assume the horizontal face $h-c$ is $(a,0)$. Necessarily we have $a>0$. We denote by $\gamma'$ the face with vertices $(a,0)$ and $(-q,1)$. In that case we have $\N(h-c)^{+} = \{(a,0), \gamma'\}$. The face polynomial $(h-c)_{\gamma'}$ is equal to 
						$\alpha x^{-q}(y-\mu x^{a+q})$. The Newton transformation $(h-c)_{\sigma_{(1,a+q,\mu)}}$ is a base case   \ref{ex:monomial} or \ref{ex:casdebase-f} with $M=-a<0$ and $m=1$. Then, we have 
						$(S_{(h-c)_{\sigma_{(p',q',\mu)}},v\neq 0})_{(0,0),0}= -[x^{a}y:\Gm^2 \to \Gm, \sigma_{\Gm}].$
						By Theorem \ref{Sfeps} we have
						$$
						\begin{array}{ccl}
							\left(S_{h-c, {x} \neq 0}\right)_{(0,0),0} & = & [x^a:\Gm \to \Gm, \sigma_{\Gm}] + [ x^{-q}(y-\mu x^{a+q}) : \Gm^2 \setminus (y=\mu x^{a+q}) \to \Gm, \sigma_{\Gm}] \\
							& - & [x^{a}y:\Gm^2 \to \Gm, \sigma_{\Gm}]
							 +  \sum_{\gamma''\in \N(h-c)\setminus\{\gamma,\gamma'\}} 
							\sum_{\mu \in R_{\gamma'}} (S_{(h-c)_{\sigma_{(p'',q'',\mu)}},v\neq 0})_{(0,0),0}
						\end{array}
						$$
						Using similar ideas of the end of the proof of Example \ref{ex:casdebase-f}, we have 
						$$[x^a:\Gm \to \Gm, \sigma_{\Gm}] + [ x^{-q}(y-\mu x^{a+q}) : \Gm^2 \setminus (y=\mu x^{a+q}) \to \Gm, \sigma_{\Gm}] \\
						 - [x^{a}y:\Gm^2 \to \Gm, \sigma_{\Gm}] = 0$$
						Furthermore, for any face $\gamma''$ in $\N(h-c)\setminus\{\gamma,\gamma'\}$, for each associated Newton transformation 
						$\sigma_{(p'',q'',\mu)}$, the polynomial 
						$(h-c)_{\sigma_{(p'',q'',\mu)}}$ belongs to $\k[x^{-1},x,y] \setminus \k[x,y]$ and $c$ does not belong to 
						$B_{h_{\sigma_{(p'',q'',\mu)}}}^{\text{Newton}}$.
						Then, applying the induction we obtain 
						$(S_{(h-c)_{\sigma_{(p'',q'',\mu)}},v\neq 0})_{(0,0),0}=0$ and we conclude that 
						$(S_{h-c,x\neq 0})_{(0,0),0} = 0$.
				\end{itemize}
			\end{itemize}

	\end{itemize}
\end{proof}

\begin{prop} \label{generic-Sfnul}
	For any polynomial $f\in \k[x,y]$, we have the inclusion $B_f^{\text{mot}} \subset B_f^{\text{Newton}}$ and $B_f^{\text{mot}}$ is finite.
\end{prop}

\begin{proof} 
We assume $c\notin B^{\text{Newton}}$ and prove that $S_{f,c}^{\infty}=0$ using formulas (\ref{formulecyclesprochesinfini}), (\ref{formulescyclesprochesinfinisegment1}) and (\ref{formulescyclesprochesinfinisegment2}). 
\begin{itemize}
	\item  Assume $c\neq f(0,0)$.
		We remark first that the point $(0,0)$ belongs to the support of $f-c$, then we have $\eps_\gamma = 0$ for any face $\gamma$ in 
		$\overline{\N}(f-c)=\N_{\infty}(f)$.
		Indeed:
		\begin{itemize}
			\item  each one-dimensional face $\gamma$ not contained in a coordinate axis is supported by a line of equation $ap+bq=N$ with $N>0$, and $(p,q)$ the primitive exterior normal vector to $\N_{\infty}(f)$. Then applying Proposition \ref{lem:epsgammafacedim0-infini}, we obtain that $\eps_\gamma=0$.
			\item  each one dimensional face $\gamma$ contained in a coordinate axis, has a dual cone $C_\gamma$ which does not intersect $\Omega$, then by Proposition \ref{lem:epsgammafacedim0-infini} $\eps_\gamma = 0$.
			\item  for any face $\gamma$ which is not horizontal, not vertical, and not in $\N_{(0,\infty)}$, $\N_{(\infty,0)}$, and 
				$\N_{(\infty,\infty)}$, the dual cone $C_\gamma$ does not intersect $\Omega$, then by Proposition \ref{lem:epsgammafacedim0-infini} $\eps_\gamma = 0$.
			\item  As $(0,0)$ belongs to the convex polygon $\overline{\N}(f-{c})$, using Remark \ref{annulationesp} we have $\eps_\gamma = 0$ for any zero dimensional face of {$\overline{\N}(f-{c})$}.
				
		\end{itemize}
		Then, we have 
		$$ { S_{f,c}^{\infty}=\sum _{\underset{\dim \gamma=1, \eta_\gamma \in \Omega}{\gamma \in \overline{\mathcal{N}}(f-c)} }\sum_{\mu \in R_\gamma} 
		(S_{(f-c)_{\sigma_{(p,q,\mu)}},v\neq0})_{((0,0),0)}}.$$
		\begin{itemize}
			\item  If $\gamma$ is a dicritical face at infinity, then {the face polynomial $(f-{c})_\gamma$ is equal to $P_{\gamma}(x^{-q}y^{p})-{c}$ 
				(or $P_{\gamma}(x^{q}y^{-p})-{c}$) with $P_\gamma$ in $\k[s]$. As ${c}$ is a Newton generic value, all the roots of the polynomial $P_\gamma(s)-{{c}}$ have multiplicity one, then for any root $\mu$, the Newton transform  $(f-{c})_{\sigma_{(p,q,\mu)}}$ belongs to $\k[x,y]$, and is a base case of type \ref{ex:monomial} or \ref{ex:casdebase-f} with $M=0$ and $m=1$}. 
				Then {$(S_{(f-{c})_{\sigma_{(p,q,\mu)}},v\neq0})_{((0,0),0)}=0$}. 
			\item If $\gamma$ is not a dicritical face, then the face polynomial $(f-c)_\gamma$ is equal to the face polynomial $f_\gamma$. Furthermore it is supported by a line of equation $\alpha p + \beta q = N_\gamma$ with $(p,q)$ the exterior normal to the Newton polygon $\N(f-c)$ equal to $\N_\infty(f))$ and $N_\gamma >0$. Then, for each root $\mu$ of the face polynomial $(f-c)_{\gamma} = f_{\gamma}$, the Newton transform we have 
				$(f-c)_{\sigma_{(p,q,\mu)}} = f_{\sigma_{(p,q,\mu)}}-c$ and $f_{\sigma_{(p,q,\mu)}}$ belongs to $\k[x^{-1},x,y] \setminus \k[x,y]$. Furthermore, as $c$ does not belong to $B_f^{\text{Newton}}$, $c$ does not belong to $B_{f_\sigma}^{\text{Newton}}$ then applying Proposition \ref{Shcnulle} we get $(S_{(f-c)_{\sigma_{(p,q,\mu)}},v\neq0})_{((0,0),0)}=0$.
		\end{itemize}
		Then we conclude that $S_{f,c}^{\infty}=0$.

	\item Assume $c=f(0,0)$. In that case, as $c$ does not belong to the discriminant of $f$, the point $(0,0)$ is not a critical point of $f-c$, then $f-c$ has a monomial $x$ or $y$. 
		Assume for instance $x$ is a monomial of $f-c$ (the case $y$ is a monomial is similar). 
		We denote by $\gamma$ the dicritical face of $f$ not contained in the coordinate axes $\mathbb R_{>0}(1,0)$.
		\begin{itemize}
			\item Assume $\gamma$ not contained in the coordinate axes $\mathbb R_{>0}(0,1)$.
				As $c$ is a generic Newton value for $f$, this face is smooth, the face polynomial $(f-c)_\gamma(x,y)$ is equal to $P(x^{-q}y)-c$ with $P(s)-c$ a polynomial with simple roots. Then, $f-c$ has a monomial $x^{-q}y$ . Considering, the monomials $x$ and $x^{-q}y$ we observe that each one dimensional face $\gamma'\neq \gamma$ of $\overline{\N}(f-c)\setminus \N(f-c)$ 
				is supported by a line of equation 
				$\alpha p' + \beta q' = N_{\gamma'}$ with $(p',q')$ the exterior normal vector and $N_{\gamma'}>0$. Then, for each root $\mu$ of the face polynomial $(f-c)_{\gamma'}$, for each Newton transformation $\sigma_{(p',q',\mu)}$, the Newton transform $(f-c)_{\sigma_{(p',q',\mu)}}$ belongs to $\k[x^{-1},x,y]\setminus \k[x,y]$ and $c$ does not belong to $B_{(f-c)_{\sigma_{(p',q',\mu)}}}^{\text{Newton}}$ because $c$ does not belong to 
				$B_{f}^{\text{Newton}}$. Then applying Proposition 
				\ref{Shcnulle} 	we get $(S_{(f-c)_{\sigma_{(p',q',\mu)}},v\neq0})_{((0,0),0)}=0$.
				Furthermore as the face polynomial $(f-c)_\gamma(x,y)$ has simple roots we also have 
                                $(S_{(f-c)_{\sigma_{(p,q,\mu)}},v\neq0})_{((0,0),0)}=0$
				for each root $\mu$.
				Then, we conclude that $S_{f,c}^{\infty}=0$.
			\item Assume $\gamma$ contained in the coordinate axes $\mathbb R_{>0}(0,1)$. The polynomial $f$ has a face $(0,b)$. Using the monomial $y^b$ and $x$ we observe that each one dimensional face $\gamma'$ of ${\overline{\N}}(f-c)\setminus \N(f-c)$ 
				is supported by a line of equation 
				$\alpha p' + \beta q' = N_{\gamma'}$ with $(p',q')$ the exterior normal vector and $N_{\gamma'}>0$. Using the same ideas as in the previous point, we obtain  $S_{f,c}^{\infty}=0$.
		\end{itemize}
\end{itemize}

\end{proof}

\subsection{Complement of the proof of theorems \ref{thm:thmSfinfini} and \ref{thmcyclesprochesinfini}} \label{section:preuveformules} 
Let $f$ be a polynomial in $\k[x,y]$ {not in $\k[x]$ or $\k[y]$, we denote by $d_x$ and $d_y$ the degrees of $f$ in variables $x$ and $y$.
\subsubsection{Compactification} \label{compactification}
In the following, we consider the compactification $(X,i,\hat{f})$ of $f$ 
with $X$ the algebraic variety 
$$X=\left\{
	\begin{array}{c|c} 
		\left([x_0:x_1],[y_0:y_1],[z_0:z_1]\right)\in \left(\mathbb P^{1}_{\k}\right)^3  
		&  z_0 x_{0}^{d_x}y_0^{d_y}f\left(\frac{x_1}{x_0}, \frac{y_1}{y_0} \right) = z_1 x_0^{d_x} y_0^{d_y} 
	\end{array} 
\right\},
$$ 
$i$ and $j$ are the following open dominant immersions and $\hat{f}$ is the following projection which is proper 
$$ \begin{array}{ccccc} 
	i & : & \mathbb A^2_{\k} & \ra & X \\ &   & (x,y) & \mapsto &
	\left([1:x],[1:y],[1:f(x,y)]\right) 
\end{array},\: 
\begin{array}{ccccc} 
	j & : & \mathbb A^1_{\k} & \ra & \mathbb P^{1}_{\k} \\ 
	&   & a & \mapsto & [1:a] 
\end{array},\:
\begin{array}{ccccc} 
	\hat{f} & : & X & \ra & \mathbb P^1_{\k} \\ &   & ([x_0:x_1],[y_0:y_1],[z_0:z_1])
	& \mapsto & [z_0:z_1].  
\end{array} 
$$ 
\begin{rem} \label{rem:cartes-fibre-infini} {We use notations of Definition \ref{def:compactification} of the closed subset at infinity $X_\infty$ and values $\infty$ and $a$.} 
	{The fiber}
\begin{equation}
	\hat{f}^{-1}(\infty) = \left\{([0:1],[y_0:y_1],[0:1]) {\in X} \mid [y_0:y_1] \in \mathbb P^1 \right\} 
	\cup \left\{([x_0:x_1],[0:1],[0:1]) {\in X} \mid [x_0:x_1] \in \mathbb P^1 \right\}.
\end{equation}
 is covered by three charts:
	$\hat{f}^{-1}(\infty) = 
	     (\hat{f}^{-1}(\infty) \cap \{x_1 \neq 0, y_1 \neq 0\}) 
	\cup (\hat{f}^{-1}(\infty) \cap \{x_0 \neq 0, y_1 \neq 0\})
	\cup (\hat{f}^{-1}(\infty) \cap \{x_1 \neq 0, y_0 \neq 0\}).$	
\end{rem}

 \subsubsection{The motivic zeta function $Z^{\delta}_{\hat{f}^\eps,i(\mathbb A_{\k}^2)}(T)$}

 Let $\eps$ be in $\{\pm\}$. If $``\eps=+"$, we simply write $\hat{f}$ for $\hat{f}^{+}$, this function extends $f$ and is equal to $z_1/z_0$. 
 If $``\eps=-"$, we simply write $1/\hat{f}$ for $\hat{f}^{-}$, this function extends $1/f$ 
 and is equal to $z_0/z_1$.
 For any $n\geq 1$, for any $\delta \geq 1$, we consider 
 $$ X_{n}^{\delta}(\hat{f}^\eps) = 
 \left\{ 
	 \begin{array}{c|c} 
		 \varphi(t) \in \mathcal L(X) & 
		 \begin{array}{l} 
			 \varphi(0) \in X_\infty,\:
			 \ord \varphi^{*}(\mathcal I_{X_\infty}) \leq n \delta, \:
			 \ord \hat{f}^\eps(\varphi(t)) = n 
		 \end{array} 
	 \end{array}
 \right\} $$ 
 endowed with the map $\ac \hat{f}^\eps$ to $\Gm$. Following Definition \ref{deffctzetamodifiee} we denote 
 $$Z^{\delta}_{\hat{f}^\eps,i(\mathbb A^2_\k)}(T) = \sum_{n \geq 1} \mes(X_{n}^{\delta}(\hat{f}^\eps)) T^n \in 
 \mathcal M_{\Gm}^{\Gm}[[T]]$$
 and by subsections \ref{section:Sfinfini} and \ref{section:Sfainfini}, we have 
 $S_{f,\infty} =  - \lim_{T \ra \infty} Z^{\delta}_{1/\hat{f},i(\mathbb A^2_\k)}(T) \:\:\text{and}\:\:
 S_{f,0}^{\infty} =  - \lim_{T \ra \infty} Z^{\delta}_{\hat{f},i(\mathbb A^2_k)}(T).$

\subsubsection{Description of arcs of $X_{n}^{\delta}(\hat{f}^\eps)$}

\begin{notation} \label{notationxphiyphi} An arc $\varphi$ in $\mathcal L(X)$ has the form 
\begin{equation} \label{notation:arc} 
	\varphi(t)=([x_0(t):x_1(t)],[y_0(t):y_1(t)],[z_0(t):z_1(t)])
\end{equation}
where coordinates are formal series in $k[[t]]$ satisfying the equation
\begin{equation} \label{eqarc}
	z_0(t) x_{0}(t)^{d_x}y_0(t)^{d_y}f\left( \frac{x_1(t)}{x_0(t)},\frac{y_1(t)}{y_0(t)} \right) = z_1(t) x_0(t)^{d_x} y_0(t)^{d_y}
\end{equation}
and such that none of the couples $(x_0(0),x_1(0))$, $(y_0(0),y_1(0))$ and $(z_0(0),z_1(0))$ is equal to $(0,0)$.
In the following, we only work with arcs not contained in the closed subset $X \setminus i(\Gm^2)$, namely such that the orders of $x_0(t)$, $x_1(t)$, $y_0(t)$, $y_1(t)$ are finite. The set of arcs contained in $X \setminus i(\Gm^2)$ has motivic measure zero.
For an arc $\varphi$ as above, we will denote 
$$x(\varphi) = x_1(t)/x_0(t) \:\: \text{and} \:\: y(\varphi) = y_1(t)/y_0(t).$$
With these notations, we have 
$\ord \hat{f}^{\eps}(\varphi) = \ord f^{\eps}(x(\varphi),y(\varphi)) = \eps \ord(z_1(t)/z_0(t)).$
\end{notation}

\begin{rem} \label{rem:ecriture-carte-origine} Let $\delta>0$ and $n\geq 1$. Let $\varphi$ be an arc {in} $X_{n}^{\delta}(\hat{f}^{\eps})$. Its origin $\varphi(0)$ belongs to $X_\infty$, then the product $x_0(0)y_0(0)$ is equal to $0$ and we consider the following three cases: 
	\begin{enumerate} 
		\item if $x_0(0)=y_0(0)=0$ then the arc can be written as
			\begin{equation} \label{ecriturelocale1}
			   \varphi(t) = ([A(t):1],[B(t):1],[z_0(t):z_1(t)])
		        \end{equation}
			with $\ord A(t)>0$, $\ord B(t)>0$ and $\ord f^{\eps}\left(1/A(t), 1/B(t)\right)=n$.
			Remark that $\ord x(\varphi)=-\ord A(t),\:\ord y(\varphi)=-\ord B(t).$
			The origin $\varphi(0)=\left([0:1],[0:1],[z_0(0):z_1(0)]\right)$ belongs to the chart $x_1y_1 \neq 0$.
			As the equation of $X_\infty$ around $\varphi(0)$ is $x_0 y_0 = 0$, we have the equality
			$\ord \varphi^*\left(\mathcal I_{X_\infty}\right) = \ord A(t) + \ord B(t).$
		
		\item if $x_0(0)\neq 0$ and $y_0(0)=0$, then the arc can be written as 
			\begin{equation} \label{ecriturelocale2} 
				\varphi(t) = ([1:A(t)],[B(t):1],[z_0(t):z_1(t)])
			\end{equation}
			with $\ord A(t)\geq 0$, $\ord B(t)>0$ and $\ord f^{\eps}\left(A(t), 1/B(t)\right)=n$.
			Remark that $\ord x(\varphi)=\ord A(t),\:\ord y(\varphi)=-\ord B(t).$
			The origin $\varphi(0)=([1:{A(0)}],[0:1],[z_0(0):z_1(0)])$ belongs to the chart $x_0y_1 \neq 0$.
			As the equation of $X_\infty$ around $\varphi(0)$ is $y_0 = 0$, we have the equality
			$\ord \varphi^*\left(\mathcal I_{X_\infty}\right) = \ord B(t).$

		\item  if $x_0(0) = 0$ and $y_0(0) \neq 0$, then the arc can be written as 
			\begin{equation} \label{ecriturelocale3}
			    \varphi(t) = ([A(t):1],[1:B(t)],[z_0(t):z_1(t)]) 
		        \end{equation}
			with $\ord A(t)> 0$, $\ord B(t)\geq 0$ and $\ord f^{\eps}\left(1/A(t), B(t)\right)=n$.
			Remark that $\ord x(\varphi)=-\ord A(t),\:\ord y(\varphi)=\ord B(t).$
		The origin $\varphi(0)=([0:1],[1:{B(0)}],[z_0(0):z_1(0)])$ belongs to the chart $x_1y_0 \neq 0$.
			As the equation of $X_\infty$ around $\varphi(0)$ is $x_0 = 0$, we have the equality
			$\ord \varphi^*\left(\mathcal I_{X_\infty}\right) = \ord A(t).$
	\end{enumerate}
\end{rem}

\begin{notation} \label{notation:Omega} We denote by $\Omega$ the set $\mathbb Z^2 \setminus (\mathbb Z_{\leq 0})^2$. \end{notation}

\subsubsection{Decomposition of the motivic zeta function $Z^{\delta}_{\hat{f}^\eps,i(\mathbb A^2_\k)}(T)$} \label{sec:decompositionzetainfininewton} 
 \begin{rem} \label{rem:m} \label{ord1/f} \label{proposition-arc-polyedre}
	We write $f(x,y)=\sum_{(a,b)\in \mathbb N^2} c_{a,b}x^a y^b$. Let $(\alpha,\beta)$ be in $\Omega$.
	By Lemma \ref{prop:dualfan}, the restriction to $\overline{\Delta}(f)$ (Definition \ref{def:Ninfinif}) 
	of the linear form $l_{(\alpha,\beta)}$ equal to 
$\left( (\alpha,\beta) \mid . \right)$ has a maximum {$\m(\alpha,\beta)$} along a face denoted by $\gamma(\alpha, \beta)$ in $\overline{\N}(f)$. 
	Then, we denote 
        $$\tilde{f}_{\gamma(\alpha,\beta)}(x,y,u):=
	f_{\gamma(\alpha,\beta)}(x,y) + 
	\sum_{(a,b)\notin \gamma(\alpha,\beta)}c_{a,b}u^{\m(\alpha,\beta)-(\alpha a + \beta b)}x^ay^b \:\:\text{where}\:\: 
	f_{\gamma(\alpha,\beta)}(x,y) = \sum_{(a,b) \in \gamma(\alpha,\beta)} c_{a,b}x^a y^b. $$ 
	Let $P(t)$ and $Q(t)$ be invertible formal series.  
	Then, we obtain the equalities 
	\begin{equation} \label{ordcompose}
		f\left(\frac{P(t)}{t^\alpha}, \frac{Q(t)}{t^\beta}\right) =  
		          t^{-\m(\alpha,\beta)}\tilde{f}_{\gamma(\alpha,\beta)}(P(t),Q(t),t)
			  \:\:\text{and}\:\: 
	 \eps\ord f^\eps\left(\frac{P(t)}{t^\alpha}, \frac{Q(t)}{t^\beta}\right)  = 
	 -\m(\alpha,\beta) + \ord \tilde{f}_{\gamma(\alpha,\beta)}(P(t),Q(t),t).
        \end{equation}
	In particular, by equation (\ref{eqarc}), for an arc $\varphi$ in $\mathcal L(X)$, writing $x(\varphi)=P(t)/t^\alpha$ and $y(\varphi)=Q(t)/t^\beta$, we have:
	\begin{itemize}
		\item  if $f_{\gamma}(\ac x(\varphi), \ac y(\varphi))\neq 0$ then $\eps\ord \hat{f}^{\eps}(\varphi(t)) 
			= -\m(-\ord x(\varphi),-\ord y(\varphi))$,
		\item  if $f_{\gamma}(\ac x(\varphi), \ac y(\varphi))= 0$ then 
			$\eps\ord \hat{f}^{\eps}(\varphi(t)) > -\m(-\ord x(\varphi),-\ord y(\varphi))$.
	\end{itemize}
	{and by Remark \ref{rem:ecriture-carte-origine}, we have $\ord \varphi^{*}(\mathcal I_{F_{\infty}}) = c(\alpha,\beta)$
	with 
	\begin{equation} \label{c}
		c(\alpha,\beta)=
		\left\{ \begin{array}{ll}
			\alpha & \text{if $\alpha>0$ and $\beta\leq 0$} \\
			\beta  & \text{if $\beta>0$ and $\alpha\leq 0$} \\ 
			\alpha+\beta & \text{if $\alpha>0$ and $\beta> 0$} 
		\end{array}.
		\right.
	\end{equation}
}

\end{rem}

\begin{rem} \label{rem:NfinfinipourSfinfini}
	In this remark we assume ``$\eps = -$''. 
	By {Remark} \ref{assumption-generic} {we can assume that} $(0,0)$ belongs to the support of $f$ and in particular $\N_{\infty}(f)=\overline{\N}(f)$. Then, for any $(\alpha,\beta)$ in $\Omega { = \mathbb Z^2 \setminus (\mathbb Z_{\leq 0})^2}$, $\m(\alpha,\beta)\geq 0$. Furthermore, if $\m(\alpha,\beta)>0$ then the face $\gamma(\alpha,\beta)$ belongs to $\N_\infty(f)^{o}$, the set of faces of $\N_\infty(f)$ which does not contain $0$.
	For instance, for any arc $\varphi$ in $\mathcal L(X)$, with $-(\ord x(\varphi),\ord y(\varphi))$ in $\Omega$, if $\ord 1/\hat{f}(\varphi){>0}$ then the corresponding face $\gamma(-(\ord x(\varphi),\ord y(\varphi)))$ belongs to $\N_\infty(f)^{o}$. 
\end{rem}

\begin{notations}
	Let $\eps$ be in $\{\pm\}$ and $\gamma$ be a face of $\overline{\N}(f)$. 
	By Remarks \ref{rem:ecriture-carte-origine} and \ref{rem:m}, we introduce
	\begin{itemize}
		\item Let $C_\gamma$ be the dual cone of $\gamma$. We consider the following polyhedral rational cones of $\mathbb R^2$ and $\mathbb R^3$ 
			\begin{equation} \label{Cgammadelta=}
				C_{\gamma,\eps}^{\delta, = } := \{(\alpha,\beta) \in C_\gamma \mid 
				c(\alpha,\beta) \leq -\eps\m(\alpha,\beta)\delta \}
			\end{equation}
			\begin{equation} \label{Cgammadelta<}
				C_{\gamma,\eps}^{\delta,<} : = \{(n, (\alpha,\beta))\in \mathbb R_{>0} \times C_\gamma \mid 
				c(\alpha,\beta) \leq n\delta, \;  -\m(\alpha,\beta)<\eps n\}
			\end{equation}
		\item For any integer $n$, we consider 
			\begin{equation} \label{defXngamma}
				X_{n,\gamma}^{\delta}(\hat{f}^\eps)=\{\varphi \in X_{n}^{\delta}(\hat{f}^\eps) \mid -(\ord x(\varphi), \ord y(\varphi)) \in C_\gamma\}
			\end{equation}
			endowed with its structural map $\ac \hat{f}^{\eps}$ to $\Gm$. We consider the motivic zeta function in $\mathcal M^{\Gm}_{\Gm}[[T]]$
                        $$Z^{\delta}_{\gamma,\eps}(T) = 
			\sum_{n\geq 1} \mes(X_{n,\gamma}^{\delta}(\hat{f}^\eps)) T^{n}.$$
			For any face $\gamma$ in $\N(f)\setminus \{\gamma_h,\gamma_v\}$, the dual cone $C_\gamma$ does not intersect $\Omega$, then $Z^{\delta}_{\gamma,\eps}(T)=0$.

		\item We assume that the face $\gamma$ is not contained in a coordinate axis. 
			For any integer $n\geq 1$ and $\delta>0$, for any $(\alpha, \beta)$ in $C_\gamma \cap \Omega$,
		      we consider the arc spaces endowed with their structural map $\ac \hat{f}^{\eps}$ to $\Gm$
			$$X_{n,(\alpha,\beta)}(\hat{f}^\eps) 
			:= \left\{ 
				\begin{array}{c|c} 
					\varphi \in X_{n}^{\delta}(\hat{f}^\eps) & 
					-(\ord x(\varphi), \ord y(\varphi))  = (\alpha, \beta)
				\end{array} 
			\right\},
			$$
			$$
			X_{n,(\alpha,\beta)}^{=}(\hat{f}^\eps) 
			:= \left\{ 
				\begin{array}{c|c} 
					\varphi \in X_{n,(\alpha,\beta)}(\hat{f}^\eps)  & 
					 f_{\gamma}(\ac x(\varphi), \ac y(\varphi)) \neq 0
				\end{array} 
			\right\}, 
			$$
			and
			$$X_{n,(\alpha,\beta)}^{<}(\hat{f}^\eps) 
			:= \left\{ 
				\begin{array}{c|c} 
					\varphi \in X_{n,(\alpha,\beta)}(\hat{f}^\eps)  & 
					f_{\gamma}(\ac x(\varphi), \ac y(\varphi)) = 0 
				\end{array} 
			\right\}. 
			$$
	                The sets
				$\Omega \cap C_{\gamma,\eps}^{\delta,=}$ and $C_{\gamma,\eps}^{\delta,<}\cap (\{n\}\times \Omega)$
			are finite for any integer $n$, and we introduce 
                        $$Z^{\delta,=}_{\gamma,\eps}(T) = 
			\sum_{n\geq 1} 
			\sum_{\text{\tiny{$\begin{array}{c}
					(\alpha,\beta) \in C_{\gamma,\eps}^{\delta,=}\cap \Omega \\
					n=-\eps \m(\alpha, \beta)
				\end{array}$}}}	
				\mes(X^{=}_{\m(\alpha,\beta),(\alpha,\beta)})(\hat{f}^\eps))
			       T^{n} 
			       \:\:\text{and}\:\: Z^{\delta,<}_{\gamma,\eps}(T) = \sum_{n\geq 1}\sum_{
                                \text{\tiny{
				$\begin{array}{c}
					(\alpha,\beta) \in \Omega \\
				        (n,(\alpha,\beta)) \in C_{\gamma,\eps}^{\delta,<}
				\end{array}$}
			}
			}
			\mes(X^{<}_{n,(\alpha,\beta)}(\hat{f}^\eps))
			T^n.$$ 
                        	\end{itemize}
\end{notations}

We deduce from Remark \ref{proposition-arc-polyedre} and from the additivity of the measure the following proposition:

\begin{prop} \label{rem:decomposition} 
	For any $\delta>0$, for any $\eps$ in $\{\pm\}$, we have the decomposition
	$$Z^{\delta}_{\hat{f}^\eps, i(\mathbb A_\k^2)}(T) = 
	\sum_{\gamma \in \overline{\N}(f)} Z^{\delta}_{\gamma,\eps}(T).$$ 
	For any face $\gamma$ which is not contained in a coordinate axes we have
	\begin{equation} \label{decompositionzeta=<}
		Z^{\delta}_{\gamma,\eps}(T) = Z^{\delta,=}_{\gamma,\eps}(T) + Z^{\delta,<}_{\gamma,\eps}(T).
	\end{equation}
	\end{prop}
\begin{rem} If $\gamma$ is the origin, then $Z^{\delta}_{\gamma,\eps}(T)=0$ because for any $n$, $X_{n,\gamma}^{\delta}$ is empty. \end{rem}

\subsubsection{The formal series $Z^{\delta}_{\gamma, \eps}$ for $\gamma$ contained in a coordinate axis}
\begin{prop}[Case $``\eps=-"$] \label{lem:fctzetainfiniaxe} 
	Let $\gamma$ be a face $(0,b_0)$ in $\N_\infty(f)$. There is $\delta_0$ such that for any $\delta \geq \delta_0$, the motivic zeta function $Z^{\delta}_{\gamma,-}(T)$
	is rational and 
	$-\lim Z^{\delta}_{\gamma,-}(T)=[y^{-b_0}:\Gm \to \Gm].$
	More precisely, writing $C_\gamma = \mathbb R_{>0}(-1,0) + \mathbb R_{>0}\eta$ 
	with $\eta$ in $\mathbb Z \times \mathbb N^*$, the dual cone of $\gamma$, we have 
	\begin{itemize}
		\item if $\eta = (p,q)$ with $p \leq 0$ and $q>0$, then {we have}
                        $Z^{\delta}_{\gamma,-}(T) = [1/y^{b_0}:\Gm \to \Gm,\sigma_{\Gm}] R_{\gamma}^{\delta}(T),$
			with 
			\begin{equation} \label{casb01} 
				R_{\gamma}^{\delta}(T) = 
			-1+\frac{1}{1-\mathbb L^{p - q}T^{qb_0 }} \sum_{r=0}^{q-1} \mathbb L^{-r - [-pr/q]}T^{rb_0}.
		        \end{equation}
		\item if $\eta = (p,q)$ with $p > 0$ and $q>0$, then {we have}
                        $Z^{\delta}_{\gamma,-}(T) = [1/y^{b_0}:\Gm \to \Gm,\sigma_{\Gm}] R_{\gamma}^{\delta}(T),$
			with 
			\begin{equation} \label{casb02} 
				R_{\gamma}^{\delta}(T) = \frac{\mathbb L^{-1}T^{b_0}}{1-\mathbb L^{-1}T^{b_0}} + 
				(\mathbb L-1)
				\left(\sum_{(\alpha_0,\beta_0)\in \mathcal P_{\gamma,(>,>)}}\frac{\mathbb L^{-\alpha_0-\beta_0}T^{b_0 \beta_0}}
				{(1-\mathbb L^{-1}T^{b_0})(1-\mathbb L^{-p-q}T^{qb_0})}+\frac{T^{b_0}\mathbb L^{-1}}{1-\mathbb L^{-1}T^{b_0}}
			\right)
		        \end{equation}
			with $\mathcal P_{\gamma,(>,>)} = (]0,1](0,1) + ]0,1]\eta)\cap \mathbb Z^2$.

	\end{itemize}
	Similarly, let $\gamma$ be a face $(a_0,0)$ in $\N_\infty(f)$. There is $\delta_0$ such that for any $\delta \geq \delta_0$, the motivic zeta function $Z^{\delta}_{\gamma,-}(T)$
	is rational and 
	$-\lim Z^{\delta}_{\gamma,-}(T)=[x^{-a_0}:\Gm \to \Gm].$
\end{prop} 
\begin{proof}
	The dual cone $C_\gamma$ of $\gamma$ is equal to $\mathbb R_{>0}(-1,0) + \mathbb R_{>0}\eta$ with $\eta$ a primitive vector in 
	$\mathbb Z \times \mathbb Z_{>0}$.
        
	\begin{itemize}
		\item Assume that $\eta$ belongs to $\mathbb Z_{\leq 0}\times \mathbb Z_{>0}$.
                        Writing $\eta=(p,q)$ with $p \leq 0$ and $q>0$. We observe that 
			$$C_{\gamma}= 
			\{(\alpha,\beta) \in \mathbb R^2 \mid \beta>0,\: \alpha <\beta (p /q)  
			= -\beta \abs{p/q}\}=C_{\gamma} \cap (\mathbb R_{< 0}\times \mathbb R_{>0})$$
			{and we conclude that for any $n\geq 1$,
	                        $X_{n,\gamma}^{\delta}(1/\hat{f})=\{\varphi \in X_{n}^{\delta}(1/\hat{f}) \mid \ord x(\varphi) \geq [-\ord y(\varphi) \abs{p/q}]+1\}.$
			}
			By Remark \ref{rem:m}, for any arc $\varphi$ in 
			$X_{n,\gamma}^{\delta}(1/\hat{f})$ we have
			$\ord 1/\hat{f}(\varphi) = -\ord y(\varphi) b_0$.
		        Assume $\delta >1/b_0$, then we have the inequality 
			$\ord \varphi^{*}(X_\infty)=-\ord y(\varphi) \leq -\delta b_0 \ord y(\varphi)$
			and 
			$Z^{\delta}_{\gamma,-}(T) = \sum_{k\geq 1} \mes(X_{k}^{\gamma}(1/\hat{f}))T^{b_0k}$
			with for any $k\geq 1$, 
			$$X_{k}^{\gamma}(1/\hat{f}) = \left\{\varphi \in \mathcal L(X) \mid -\ord y(\varphi)=k,\: \ord x(\varphi) \geq [k\abs{p/q}]+1 \right\}.$$
			As in the proof of formula (\ref{gamma=(a,0)}) in Proposition \ref{ex:cas-x^N}, we get 
			$Z^{\delta}_{\gamma,-}(T) = [1/y^{b_0}:\Gm \to \Gm,\sigma_{\Gm}] 
			R_{\gamma,(<,>)}^{\delta}(T),$
			with 
			\begin{equation} \label{caseta01} R_{\gamma,(<,>)}^{\delta}(T) = 
				-1+\frac{1}{1-\mathbb L^{-\abs{p} - q}T^{qb_0 }} \sum_{r=0}^{q-1} \mathbb L^{-r - [\abs{pr/q}]}T^{rb_0} \underset{T\to \infty}{\longrightarrow} -1
			\end{equation}
			
		\item Assume that $\eta$ belongs to $\mathbb Z_{>0} \times \mathbb Z_{>0}$.
                        Writing $\eta=(p,q)$ with $p > 0$ and $q>0$. We observe that 
			$$C_{\gamma} = (\mathbb R_{ {<0} }\times \mathbb R_{>0}) \sqcup (\mathbb R_{ {>} 0}(0,1) + \mathbb R_{ {\geq} 0} \eta)$$
			{In particular we denote
				\begin{equation} \label{formuledecomposition}
					C_{\gamma,(<,>)} := C_{\gamma} \cap (\mathbb R_{<0}\times \mathbb R_{>0}) = \mathbb R_{<0} \times \mathbb R_{>0} \:\: \text{and} \:\: 
                                    C_{\gamma,(\geq ,>)} := C_{\gamma} \cap (\mathbb R_{\geq 0}\times \mathbb R_{>0}) = \mathbb R_{>0}(0,1) + \mathbb R_{ \geq 0} \eta.
			    \end{equation}
			 }
				We decompose the zeta function {following these quadrants}
			$Z^{\delta}_{\gamma,-}(T) = 
			Z^{\delta}_{\gamma,(<,>)}(T) + Z^{\delta}_{\gamma,(\geq ,>)}(T)$
			with 
			$$Z^{\delta}_{\gamma,(<,>)}(T) = 
			\sum_{n\geq 1}\mes(X_{n,(<,>)}^{\delta}(1/\hat{f}))T^n \:\:\text{and}\:\:
			Z^{\delta}_{\gamma,(\geq ,>)}(T) = 
			\sum_{n\geq 1}\mes(X_{n,(\geq ,>)}^{\delta}(1/\hat{f}))T^n$$
			with for any $n\geq 1$,
			$X_{n,(<,>)}^{\delta}(1/\hat{f}) = \{\varphi \in X_{n,\gamma}^{\delta}(1/\hat{f}) \mid 
			{-(\ord x(\varphi),\ord y(\varphi)) \in C_{\gamma,(<,>)}}\}$,\:
			and
			$$X_{n,(\geq ,>)}^{\delta}(1/\hat{f}) = \{\varphi \in X_{n,\gamma}^{\delta}(1/\hat{f}) \mid 
			{-(\ord x(\varphi),\ord y(\varphi)) \in C_{\gamma,(\geq,>)} } \}.$$
			\begin{itemize}
				\item  It follows from formula (\ref{caseta01}) that for $\delta > 1/b_0$, we have 
					$Z^{\delta}_{\gamma,(<,>)}(T) = [1/y^{b_0}:\Gm \to \Gm,\sigma_{\Gm}] R_{\gamma,(<,>)}^{\delta}(T)$
					with
					\begin{equation} \label{eq<>} R_{\gamma,(<,>)}^{\delta}(T) = \frac{\mathbb L^{-1}T^{b_0}}{1-\mathbb L^{-1}T^{b_0}}
                                        \underset{T\to \infty}{\longrightarrow} -1.
					\end{equation}
				\item  
					{By Remark \ref{rem:ecriture-carte-origine}}, for any arc $\varphi$ in $X_{n,(\geq ,>)}^{\delta}(1/\hat{f})$ 
					we have the equality $\ord \varphi^*(\mathcal I_{X_\infty}) = -\ord x(\varphi) - \ord y(\varphi).$              
					For any $\delta$ satisfying $\delta > \delta_3 = \max \left(\frac{1}{b_0}, \frac{p+q}{b_0q} \right)${=$\frac{p+q}{b_0q}$}, we have the inequality 
					$\ord \varphi^*(\mathcal I_{X_\infty})\leq \delta \ord (1/\hat{f})(\varphi) 
					= -\delta b_0\: {\ord} y(\varphi),$
                                        for any arc $\varphi$ in $\mathcal L(X)$
					with $-(\ord x(\varphi),\ord y(\varphi))$ in $\mathbb R_{>0}(0,1) + \mathbb R_{ \geq 0} \eta$.
					{Indeed, let $-(\ord x(\varphi),y(\varphi))=k(0,1)+l(p,q)$, we have 
						$-\ord x(\varphi) - \ord y(\varphi) = k + l(p+q) \leq \delta b_0 k + \delta b_0 lq \leq -\delta b_0\: {\ord} y(\varphi).$ 
				        }
					{Then}, we can decompose
					$Z^{\delta}_{\gamma,(\geq ,>)}(T) = \sum_{n\geq 1} 
					\sum_{\underset{n=\beta b_0}{(\alpha,\beta)\in {C_{\gamma,(\geq,>)}}}}
					\mes(X_{(\alpha,\beta)})T^{b_0 \beta},$
					with for any $(\alpha, \beta)$ in {$C_{\gamma,(\geq,>)}$}
					$$X_{(\alpha,\beta)} = \{\varphi \in \mathcal L(X) \mid -(\ord x(\varphi),\ord y(\varphi)) = (\alpha,\beta)\}.$$
					It follows from standard computations of motivic measure that 
					$$\mes(X_{(\alpha,\beta)}) = \mathbb L^{-\alpha-\beta}[y^{-b_0}:\Gm^2 \to \Gm] = 
					\mathbb L^{-\alpha-\beta}(\mathbb L-1)[y^{-b_0}:\Gm \to \Gm].$$
					Then, we have the equality
					$Z^{\delta}_{\gamma,(\geq ,>)}(T) = [y^{-b_0}:\Gm \to \Gm] (\mathbb L-1) R_{\gamma,(\geq, >)}(T)$
					with 
					$$R_{\gamma,(\geq, >)}(T) = \sum_{n\geq 1} 
					\sum_{\underset{n=\beta b_0}{(\alpha,\beta)\in {C_{\gamma,(\geq,>)}}}}
					\mathbb L^{-\alpha-\beta}T^{b_0 \beta}.$$
					In particular applying Lemma \ref{lemmedescones} and using 
                                        $\mathcal P_{\gamma} = (]0,1](0,1) + ]0,1]\eta)\cap \mathbb Z^2$
					we obtain that 
					\begin{equation} \label{eqgeq>} 				
						R_{\gamma,(\geq,>)}(T) = 
						\sum_{(\alpha_0,\beta_0) \in \mathcal P_{\gamma}} 
						\frac{\mathbb L^{-\alpha_0-\beta_0}T^{b_0 \beta_0}}
						{(1-\mathbb L^{-1}T^{b_0})
						(1-\mathbb L^{-p-q}T^{qb_0})}+\frac{T^{b_0}\mathbb L^{-1}}{1-\mathbb L^{-1}T^{b_0}}
						\underset{T\to \infty}{\longrightarrow} 0
					\end{equation}
		\end{itemize}

       	\end{itemize}	
        By symmetry, we obtain a similar result for a zero-dimensional face $(a_0,0)$.
	\end{proof}
	
	\begin{prop}[Case $``\eps=+"$] \label{prop:casfacecontenuedansaxe}
		Let $\gamma$ be a face $(a,0)$ {or $(0,b)$} of $\overline{\N}(f)$. There is $\delta_0$ such that for any $\delta \geq \delta_0$, the formal series $Z_{\gamma,+}^{\delta,=}(T)$ is rational and equal to 
	$$ {Z^{\delta}_{\gamma,+}(T)} = 
	[x^a : \Gm^{2} \ra \Gm,\sigma_{\gamma}] R_{\gamma}^{\delta,=}(T) 
	{\:\:\text{or}\:\:
	Z^{\delta}_{\gamma,+}(T) = 
	[y^b : \Gm^{2} \ra \Gm,\sigma_{\gamma}] R_{\gamma}^{\delta,=}(T)}
	$$ 
	with $R_{\gamma}^{\delta,=}(T)$ expressed in (\ref{eqRgammaa}).
	It admits a limit $-\lim {Z_{\gamma,+}^{\delta}(T)}$ in $\mgg$ with 
	$$ - \lim {Z_{\gamma,+}^{\delta}(T)} = \eps_\gamma[x^a:\Gm^2 \ra \Gm,\sigma_\gamma]
	{\:\:\text{or}\:\:
	- \lim {Z_{\gamma,+}^{\delta}(T)} = \eps_\gamma[y^b:\Gm^2 \ra \Gm,\sigma_\gamma]}	
	$$
	with $\eps_\gamma = - \chi_{c}(C_{\gamma,+}^{\delta,=} \cap \Omega)$ belongs to $\{0,-1\}$.
	More precisely, {in the case $\gamma=(a,0)$ (the case $\gamma=(0,b)$ is similar)},
	let $H_\gamma = \{(\alpha,\beta) \in \mathbb R^2 \mid \alpha<0\}$ and $C_\gamma$ be the dual cone of $\gamma$.
	The cone $C_{\gamma,+}^{\delta,=}$, {defined in formula (\ref{Cgammadelta=})}, is included in $C_\gamma \cap H_\gamma$, and
			\begin{itemize}
				\item if $C_\gamma \cap H_\gamma \cap \Omega = \emptyset$ then 
					$C_{\gamma,+}^{\delta,=} {\cap \Omega}$ is empty and {$\eps_\gamma = 0$},
				\item  otherwise, there is $\eta=(p,q)$ with $p<0$ and $q>0$ such that
					$C_\gamma \cap H_\gamma \cap \Omega = \mathbb R_{>0}(-1,0) + \mathbb R_{>0}\eta$.\\
					There is $\delta_1>0$ such that for any $\delta>\delta_1$, we have
					$C_{\gamma,+}^{\delta,=} {\cap \Omega} = \mathbb R_{>0}(-1,0) + \mathbb R_{>0}\eta
					\:\:\text{and}\:\: {\eps_\gamma=-1}.$
			\end{itemize}
\end{prop}
\begin{proof}
	Let $\gamma$ be a face $(a,0)$ in $\overline{\N}(f)$ with $a>0$. The dual cone of $\gamma$ has dimension 2.
	Remark that, for any integer $n\geq 1$, for any arc $\varphi$ in $X_{n,\gamma}^{\delta}(\hat{f})$, by Remark \ref{rem:m} we have
	$\ord \hat{f}(\varphi) = \ord x(\varphi)a = {n >0}$
        {which implies that $\ord x(\varphi)>0$, namely $-(\ord x(\varphi),\ord y(\varphi))$ belongs to}
	$C_\gamma \cap H_\gamma \cap \Omega$.
        If $C_\gamma \cap H_\gamma \cap \Omega$ is empty then $Z_{\gamma,+}^{\delta}(T)=0$.
	
	Assume $C_\gamma \cap H_\gamma \cap \Omega = \mathbb R_{>0}(-1,0) + \mathbb R_{>0}(p,q)$ with $p<0$ and $q>0$.
	We observe that 
	$$C_\gamma \cap H_\gamma \cap \Omega = \{(\alpha,\beta)\in \mathbb Z^2 \mid \beta>0, \alpha<0, \beta p/q> \alpha\}$$
        and we conclude that for any $n\geq 1$,
		$$X_{n,\gamma}^{\delta}(\hat{f})=\{\varphi \in X_{n}^{\delta}(\hat{f}) \mid \ord y(\varphi)<0,\: \ord x(\varphi)>0,\: \abs{q/p}\ord x(\varphi) >-\ord y(\varphi) \}.$$
        For any arc $\varphi$ in $X_{n,\gamma}^{\delta}(\hat{f})$, we have $\varphi(0) = ([1:0],[0:1],[1:0])$ and by formula (\ref{ecriturelocale2}) in Remark \ref{rem:ecriture-carte-origine}, we have 
	\begin{equation} \label{eqinegalite}
		\ord \varphi^{*}(\mathcal I_{X_\infty}) = -\ord y(\varphi) \leq \delta a \ord x(\varphi).
	\end{equation}
        For any $\delta \geq \abs{q/p}/a$, we have 
        $X_{n,\gamma}^{\delta}(\hat{f})=\{\varphi \in \mathcal L(X) \mid \ord \hat{f}(\varphi)=n,\:
	(-\ord x(\varphi),-\ord y(\varphi))\in C_\gamma\cap H_\gamma \cap \Omega\}$
        and
	$$ 	Z_{\gamma,+}^{\delta}(T) = \sum_{n\geq 1}\mes(X_{n,\gamma}^{\delta})T^n \\
		 =  \sum_{k\geq 1} \sum_{(-k,l)\in C_\gamma\cap H_\gamma \cap \Omega} \mes(X_{k,l})T^{ka}
	$$
	with for any $(-k,l)$ in $C_\gamma\cap H_\gamma \cap \Omega$,
	$X_{k,l}=\{\varphi \in X_{n,\gamma}^{\delta}(\hat{f}) \mid (-\ord x(\varphi),-\ord y(\varphi))=(-k,l)\}.$
	By construction of the motivic measure we have 
	$\mes(X_{k,l})=\mathbb L^{-k-l}[x^a:\Gm^2 \to \Gm, \sigma_{\Gm}].$
        Then, by application of Lemma \ref{lemmedescones} we obtain the expression
	$Z_{\gamma,+}^{\delta}(T)= [x^a:\Gm^2 \to \Gm, \sigma_{\Gm}]R_{\gamma}^{\delta,=}(T)$,
        where denoting $\mathcal P =( ]0,1](-1,0) + ]0,1](p,q) )\cap \mathbb Z^2$, we have 
	\begin{equation} \label{eqRgammaa}
		R_\gamma^{\delta,=}(T) = \sum_{(k_0,l_0) \in \mathcal P} \frac{\mathbb L^{k_0-l_0}T^{-ak_0}}{(1-\mathbb L^{-1}T^{a})(1-\mathbb L^{p-q}T^{-ap})} \underset{T \to \infty}{\longrightarrow}
		1 = \chi_{c}(C_\gamma\cap H_\gamma \cap \Omega).
	\end{equation}
\end{proof}

	\subsubsection{The formal series $Z^{\delta,=}_{\gamma,\eps}(T)$ for $\gamma$ not contained in a coordinate axes}
\begin{prop}[Case $``\eps = -"$] \label{lem:fctzetaegalinfini} 
	We assume $f(0,0)\neq 0$.
	\begin{itemize}
		\item If $f$ can be written as $f(x,y)=P(x^ay^b)$ with $P$ in $\k[s]$ {of degree $d$} with $(a,b)$ in $(\mathbb N^*)^2$. Let $\gamma$ be the face $(ad,bd)$ of $\N_{\infty}(f)$. 
			If $\delta > \max(1/(da),1/(db))$ then we have
			\begin{equation} Z^{\delta,=}_{\gamma,-}(T) = [{1/(x^{a}y^{b})^d}:\Gm^2 \to \Gm, \sigma_{\Gm}] R_{\gamma}^\delta(T) 
				\:\:\text{and}\:\:
				-\lim Z^{\delta,=}_{\gamma,-}(T) = [1/(x^{a}y^{b})^d:\Gm^2 \to \Gm, \sigma_{\Gm}].
			\end{equation}
			where $R_{\gamma,=}^\delta(T)$ is rational and given in formula (\ref{casmonom}).
	\item  Assume $f$ is not of the previous form. Let $\gamma$ be a face in $\N_\infty(f)^o$ and not contained in a coordinate axes. 
			There is $\delta_0>0$ such that for any $\delta > \delta_0$, we have                
			\begin{equation}\label{eqratformZdelta=gamma-}
				Z^{\delta,=}_{\gamma,-}(T) = 
				[1/f_\gamma : \Gm^{2}\setminus f_\gamma^{-1}(0)\ra \Gm,\sigma_{\gamma}]
				R_{\gamma}^{\delta,=}(T)\:\text{and}\:
				-\lim Z^{\delta,=}_{\gamma,-}(T)=
				\eps_\gamma [1/f_\gamma : \Gm^{2}\setminus f_\gamma^{-1}(0)\ra \Gm,\sigma_{\gamma}] 
			\end{equation}
			with {$\eps_\gamma = -\chi_{c}(C_{\gamma,-}^{\delta,=}\cap \Omega)$ 
			equal to $(-1)^{\dim \gamma + 1}$} if
			$\gamma$ is not contained in a face which contains 0 and otherwise {is }equal to 0, and 
			$R_{\gamma}^{\delta,=}(T)$ is rational and given in formula (\ref{eqR}) 
			(with (\ref{eqcasa1}), (\ref{eqcasa2}), (\ref{eqcasb1}), (\ref{eqcasb2}), 
			(\ref{eqcasb3}), (\ref{eqcasb4}), (\ref{eqcasb5})) for zero-dimensional faces and in  formula (\ref{eqcasc1}) for one-dimensional faces.
	\end{itemize}
\end{prop} 
\begin{proof} 
	{We assume first that $f(0,0)\neq 0$ and $f$ is not of the form $f(x,y)=P(x^ay^b)$ with $P$ in $\k[s]$ with $(a,b)$ in $(\mathbb N^*)^2$}. 
	Let $\gamma$ be a face in $\N_{\infty}(f)^o$ not contained in a coordinate axes and $\delta >0$. By assumption $(0,0)$ does not belong to the face $\gamma$, then for any $(\alpha,\beta)$ in the dual cone $C_\gamma$, $\m(\alpha,\beta)>0$. Then, {by homogeneity of the functions $\m$ and $c$ defined in {formula} (\ref{c}), there is $\delta'$ such that for any $\delta>\delta'$} the cone $C_{\gamma,-}^{\delta,=}$ {defined in formula (\ref{Cgammadelta=})} is non empty. 
	{Indeed, let $(\alpha, \beta)$ be in the dual cone $C_\gamma$, we have $\m(\alpha, \beta)>0$, there is $\delta'>0$ such that 
	$c(\alpha,\beta)\leq \delta' \m(\alpha,\beta)$, then $(\alpha,\beta)$ belongs to 
	$C_{\gamma,-}^{\delta',=}$. Then, for any $\delta > \delta'$, $\delta'/\delta(\alpha,\beta)$ belongs to $C_{\gamma,-}^{\delta,=}$.} 
	Let $\delta>\delta'$ and $(\alpha,\beta)$ be an element of $C_{\gamma,-}^{\delta,=}$. 
	By standard arguments on the definition of the motivic measure, it follows from \cite[Lemma 3.16]{Rai11}  
	that the motivic measure of
	$X_{\m(\alpha,\beta),(\alpha,\beta)}(1/\hat{f})$ is equal to 
	$\mathbb L^{-\abs{\alpha}-\abs{\beta}} 
	[1/f_\gamma : \Gm^{2}\setminus f_{\gamma}^{-1}(0)\ra \Gm,\sigma_{\alpha,\beta}].$
	By Remark \ref{notation-identification-action} (see also \cite[Proposition 3.13]{Rai11} for details) this measure does not depend on $(\alpha,\beta)$ in $C_{\gamma}$ and we replace $\sigma_{\alpha,\beta}$ by $\sigma_\gamma$. {Note that, as the face $\gamma$ is not contained in a coordinate axes,
for any integer $n$, the set of $(\alpha,\beta)$ in $C_{\gamma,-}^{\delta,=}\cap \Omega$ with $n=\m(\alpha, \beta)$ is finite.}  {Indeed, as the face $\gamma$ is not contained in a coordinate axes, this face has a point $(a_0,b_0)$ in $(\mathbb N^*)^2$, then for any $(\alpha,\beta)$ in 
		$C_{\gamma,-}^{\delta,=}$ we have $\m(\alpha,\beta)=a_0 \alpha + b_0 \beta$, then the equation and inequation
		$n = \m(\alpha,\beta) = a_0 \alpha + b_0 \beta,\: c(\alpha,\beta)\leq n\delta$
		have finitely many solutions in $C_{\gamma,-}^{\delta,=}\cap \Omega$.}
	Then, we have the equality 
	$Z^{\delta{,=}}_{\gamma,-}(T) = 
	[1/f_\gamma : \Gm^{2}\setminus f_{\gamma}^{-1}(0)\ra \Gm,\sigma_{\gamma}] R_{\gamma}^{\delta,=}(T),$ 
	where
	$$R_{\gamma}^{\delta,=}(T) = \sum_{n\geq 1} 
			\sum_{\text{\tiny{
				$\begin{array}{c}
				  (\alpha,\beta) \in C_{\gamma,-}^{\delta,=}\cap \Omega \\
				  n=\m(\alpha, \beta)
				  \end{array}$
			     }}}\mathbb L^{-\abs{\alpha}-\abs{\beta}}T^{n}.$$
			 {As $\gamma$ belongs to $\N_{\infty}(f)$}, we remark that $C_\gamma$ is included in $\Omega$. 
		The application $c$ is linear on each cone $\mathbb R_{?0} \times \mathbb R_{!0}$ with 
		{symbols} $``?"$ and $``!"$ in $\{>,<,=\}$ and $\mathbb R_{=0} = \{0\}$.
		Then, we {decompose} the cone $C_\gamma$ {along the quadrants of $\mathbb R^2$}, as the disjoint union 
	$$C_\gamma = \bigsqcup_{(?,!)\in\{<,=,>\}^2 \setminus \{<,=\}^2} C_{\gamma} \cap (\mathbb R_{?0}\times \mathbb R_{!0})$$
	and we have
	\begin{equation} \label{eqR} R_{\gamma}^{\delta,=}(T) = \sum_{(?,!)\in\{<,=,>\}^2 \setminus \{<,=\}^2} R_{\gamma,(?,!)}^{\delta,=}(T)
	\:\:\text{with}\:\:
	R_{\gamma,(?,!)}^{\delta,=}(T) = \sum_{n\geq 1} 
	\sum_{
		\text{\tiny{$
		\begin{array}{c}
			(\alpha,\beta) \in C_{\gamma,-}^{\delta,=} \cap (\mathbb R_{?0}\times \mathbb R_{!0}) {\cap \Omega} \\
			n=\m(\alpha, \beta)
		\end{array}$}
	}
	}
	\mathbb L^{-\abs{\alpha}-\abs{\beta}}
	T^{n}.
        \end{equation}
	By Lemma \ref{lemmedescones}, each $R_{\gamma,(?,!)}^{\delta,=}(T)$ is rational and its limit when $T$ goes to infinity is 
	$\chi_{c}(C_{\gamma,-}^{\delta,=} \cap (\mathbb R_{?0}\times \mathbb R_{!0}))$. By additivity, we obtain the rational form of 
	$R_{\gamma}^{\delta,=}(T)$ and its limit is $\chi_{c}(C_{\gamma}^{\delta,=})$. In the following, we study the cones $C_{\gamma}^{\delta,=}$ and 
	$C_{\gamma}^{\delta,=} \cap (\mathbb R_{?0}\times \mathbb R_{!0})$.
	\begin{itemize}
		\item  Assume $\gamma$ is a zero dimensional face {equal to $(a_0,b_0)$ in $(\mathbb N^*)^2$ namely not contained in a coordinate axes}. As {$f$ is not quasi homogeneous,}
			the face $\gamma$ is the intersection of two one-dimensional faces $\gamma_1$ and $\gamma_2$
			with primitive exterior normal vectors $\eta_1$ and $\eta_2$, such that we have the inequality of measure of oriented angles 
			\begin{equation} \label{mesangle}
				\text{mes}((1,0),\eta_1)>\text{mes}((1,0),\eta_2).
			\end{equation}
			The dual cone of $\gamma$ is 
			$C_\gamma = \mathbb R_{>0}\eta_1+ \mathbb R_{>0}\eta_2$
			and for any $\delta > 0$, by definition
			$C_{\gamma,-}^{\delta,=}  = \{(\alpha,\beta) \in C_\gamma \mid c(\alpha,\beta) \leq \m(\alpha,\beta)\delta \}.$

			\begin{itemize} 
				\item Assume the faces $\gamma_1$ and $\gamma_2$ do not contain the origin.
					Hence, the vectors $\eta_1$ and $\eta_2$ belong to $\Omega$ and by convexity of the Newton polygon $(\gamma \mid \eta_1)>0$ and $(\gamma \mid \eta_2)>0$, these vectors belong to the half plane $(\gamma\mid .)>0$. We remark that for any $\delta>0$ larger than 
					$$\delta_0=\max \left( 
					\frac{\abs{(\eta_1 \mid (1,0)}+\abs{(\eta_1 \mid (0,1))}}{(\eta_1 \mid {\gamma})},
					\frac{\abs{(\eta_2 \mid (1,0)}+\abs{(\eta_2 \mid (0,1))}}{(\eta_2 \mid {\gamma})} 
					\right)$$
					we have for any $(\alpha,\beta)$ in $C_\gamma$, the inequalities
						$c(\alpha,\beta)\leq \abs{\alpha} + \abs{\beta}
					< \delta \m(\alpha,\beta)$ inducing
					the equality $C_\gamma = C_{\gamma,-}^{\delta,=}$ and $\chi_{c}(C_{\gamma,-}^{\delta,=}) = 1$.
Then, by Lemma \ref{lemmedescones} we obtain for any $``?"$ and $``!"$ in $\{<,>\}$
\begin{itemize}
\item if $C_\gamma \cap (\mathbb R_{?0}\times \mathbb R_{!0})=\emptyset$ then $R_{\gamma,(?,!)}^{\delta,=}(T) = 0$,
\item {if $C_\gamma \cap (\mathbb R_{?0}\times \mathbb R_{!0})$ is generated by two vectors $\omega_1$ and $\omega_2$} then 
	\begin{equation}  \label{eqcasa1}
		R_{\gamma,(?,!)}^{\delta,=}(T) = 
		\sum_{(\alpha_0,\beta_0) \in \mathcal P_{\gamma,(?,!)}} 
		\frac{\mathbb L^{-((\eps_1,\eps_2)\mid (\alpha_0,\beta_0))}T^{((\alpha_0,\beta_0)\mid \gamma)}}
		{(1-\mathbb L^{-((\eps_1,\eps_2)\mid \omega_1)}T^{(\omega_1\mid \gamma)})
		(1-\mathbb L^{-((\eps_1,\eps_2)\mid \omega_2)}T^{(\omega_2\mid \gamma)})},
	\end{equation}
with $\mathcal P_{\gamma,(?,!)} = (]0,1]\omega_1 + ]0,1]\omega_2)\cap \mathbb Z^2$ 
and 
$$
\eps_1 = \left\{ 
	\begin{array}{ll}
		1 & \text{if $``?"$ is $``>"$} \\
		-1 & \text{if $``?"$ is $``<"$}
	\end{array}
	\right.\:\: \text{and} \:\:
	\eps_2 = \left\{ 
		\begin{array}{ll}
			1 & \text{if $``!"$ is $``>"$} \\
			-1 & \text{if $``!"$ is $``<"$}
		\end{array}
		\right. .
		$$
\end{itemize}
If  $``?"$ is $``="$ and $``!"$ is $``>"$ (the case $``?"$ is $``>"$ and $``!"$ is $``="$ is similar), we obtain also
\begin{itemize}
	\item if $C_\gamma \cap (\{0\}\times \mathbb R_{>0})=\emptyset$ then 
		$R_{\gamma,(?,!)}^{\delta,=}(T) = 0$,
	\item if $C_\gamma \cap (\{0\}\times \mathbb R_{>0})=\mathbb R_{>0}(0,1)$ then 
		\begin{equation} \label{eqcasa2}
			R_{\gamma,(?,!)}^{\delta,=}(T) = \frac{\mathbb L^{-1}T^{((0,1)\mid \gamma)}}
			{1-\mathbb L^{-1}T^{((0,1)\mid \gamma)}}.
		\end{equation}
\end{itemize}

\item  Assume the origin $(0,0)$ is contained in $\gamma_1$ and not in $\gamma_2$ (the case where the origin is contained in $\gamma_2$ and not in $\gamma_1$ is similar). Then, as $b_0 >0$, by convention (\ref{mesangle}) on $\eta_1$ and $\eta_2$ and the fact that these normal vectors are exterior to $\N_{\infty}(f)$, $\eta_1$ belongs to $\mathbb R_{>0}(-b_0,a_0)$ and by convexity of the Newton polygon, we have $(\eta_2 \mid \gamma)>0$.
For any $(\alpha,\beta)$ in $C_\gamma$, there is $x>0$ and $y>0$ such that 
$(\alpha,\beta)=x \eta_1 + y \eta_2$ and $\m(\alpha,\beta)=(\gamma \mid (\alpha, \beta))=y(\gamma \mid \eta_2)$.
In the following, we show that for any $\delta$ large enough, we have the equality 
$C_{\gamma,-}^{\delta,=} = \mathbb R_{>0} \eta_{\delta} + \mathbb R_{\geq 0}\eta_2,$
with $\eta_\delta = (1-\delta b_0, \delta a_0 )$ and $\chi_{c}(C_{\gamma,-}^{\delta,=})=0$.
In order to prove this description, we study below each intersection $C_{\gamma,-}^{\delta,=} \cap (\mathbb R_{?0}\times \mathbb R_{!0})$ for any 
$``?"$ and $``!"$ in $\{>,=,<\}$ and take their union.
\begin{itemize}
\item Necessarily $C_{\gamma} \cap (\mathbb R_{<0} \times \mathbb R_{> 0})$ is non empty, because 
	{$\eta_1 \in \mathbb R_{>0}(-b_0,a_0)$ and $(a_0,b_0) \in (\mathbb N^*)^2$}.
We have the equality 
$C_\gamma \cap (\mathbb R_{<0} \times \mathbb R_{> 0}) = \mathbb R_{>0} \eta_1 + \mathbb R_{>0} \eta$
with $\eta=(0,1)$ or $\eta=\eta_2$. 
In particular, we have the inequalities $(\eta \mid (1,0))\leq 0$, $(\eta\mid (0,1))>0$ and $(\gamma \mid \eta)>0$.
By Equation (\ref{c}), 
for any $(\alpha,\beta)$ in $C_{\gamma} \cap (\mathbb R_{<0} \times \mathbb R_{> 0})$, $c(\alpha,\beta)=\beta$.
We introduce 
$L_\delta = (\alpha,\beta) \mapsto \beta-\delta(\gamma\mid (\alpha,\beta))$
{and recall that $\gamma=(a_0,b_0)$.}
Then, we have 
\begin{equation} \label{cas<0,>0}
C_{\gamma, -}^{\delta,=} \cap (\mathbb R_{<0} \times \mathbb R_{> 0}) = 
\{(\alpha,\beta) \in \mathbb R_{>0} \eta_1 + \mathbb R_{>0} \eta \mid L_{\delta}(\alpha,\beta) \leq 0\}
=
\mathbb R_{>0} \eta_{\delta} + \mathbb R_{\geq 0} \eta.
\end{equation}
Indeed remark that $L_{\delta}(\eta_\delta)=0$, $L_{\delta}(\eta_1)>0$ and 
$L_{\delta}(\eta)=(\eta \mid (0,1))-\delta (\gamma \mid \eta)<0$
for any $\delta >\delta_1 = \frac{(\eta \mid (0,1))}{(\gamma \mid \eta)}$.
And by Lemma \ref{lemmedescones}, using 
$\mathcal P_{\gamma,(<,>)} = (]0,1]\eta_\delta + ]0,1]\eta)\cap \mathbb Z^2$, we have
\begin{equation} \label{eqcasb1}
	R_{\gamma,(<,>)}^{\delta,=}(T) = 
	\sum_{(\alpha_0,\beta_0) \in \mathcal P_{\gamma,(<,>)}} 
	\frac{\mathbb L^{\alpha_0-\beta_0}T^{((\alpha_0,\beta_0)\mid \gamma)}}
	{(1-\mathbb L^{-((-1,1)\mid \eta_\delta)}T^{(\eta_\delta\mid \gamma)})
	(1-\mathbb L^{-((-1,1)\mid \eta)}T^{(\eta\mid \gamma)})} 
	+ \frac{\mathbb L^{-((-1,1)\mid \eta_\delta)}T^{(\eta_\delta\mid \gamma)}}
	{1-\mathbb L^{-((-1,1)\mid \eta_\delta)}T^{(\eta_\delta\mid \gamma)}}.
\end{equation}

	\item The cone $C_\gamma \cap (\{0\}\times \mathbb R_{>0})$ is empty or equal to $\mathbb R_{>0}(0,1)$.
		As $b_0>0$, for any $\delta > \delta_2 = 1/b_0$ we have 
		\begin{equation} \label{cas=0,>0}
			C_{\gamma,-}^{\delta,=} \cap (\{0\} \times \mathbb R_{> 0})  = 
			C_{\gamma} \cap (\{0\} \times \mathbb R_{> 0}) 
		\end{equation}
		and in the non empty case we have by Lemma \ref{lemmedescones}
		\begin{equation} \label{eqcasb2}
			R_{\gamma,(=,>)}^{\delta,=}(T) = \frac{\mathbb L^{-1}T^{((0,1)\mid \gamma)}}{1-\mathbb L^{-1}T^{((0,1)\mid \gamma)}}.
		\end{equation}

	\item  The cone $C_\gamma \cap (\mathbb R_{>0} \times \mathbb R_{>0})$ is empty or equal to $\mathbb R_{>0}(0,1) + \mathbb R_{>0} \eta$ with $\eta$ equal to $(1,0)$ 
		or $\eta_2$. In particular {in the non empty case} 
		$(\eta\mid (1,0))>0$ and $(\eta \mid (0,1))\geq 0$.
		For $\delta > \delta_3 = \max \left(\frac{1}{(\gamma \mid (0,1))}, 
		\frac{(\eta \mid (1,1))}{(\gamma \mid \eta)} \right)$ we have 
		\begin{equation} \label{cas>0,>0}
			C_{\gamma,-}^{\delta,=} \cap (\mathbb R_{>0} \times \mathbb R_{>0}) = 
			C_{\gamma} \cap (\mathbb R_{>0} \times \mathbb R_{>0}) = 
			{\mathbb R_{>0}(0,1) + \mathbb R_{>0} \eta.}
		\end{equation}
		Indeed, for any $(\alpha,\beta)$ in $C_\gamma \cap (\mathbb R_{>0} \times \mathbb R_{>0})$, we have $c(\alpha,\beta)=\alpha+\beta$. Let $\delta > \delta_3$ and $(\alpha, \beta)$ in $C_\gamma$.
		There is $\lambda > 0$ and $\mu > 0$ such that $(\alpha,\beta)=\lambda(0,1) + \mu \eta$ and we conclude that $(\alpha,\beta)$ belongs to $C_{\gamma,-}^{\delta,=}$ by 
		$$c(\alpha,\beta) = 
		\alpha+\beta= \lambda + \mu (\eta \mid (1,1)) \leq 
		\delta[\lambda(\gamma \mid (0,1)) + \mu (\gamma \mid \eta))]$$
		Then, in the non empty case, we have by Lemma \ref{lemmedescones} 
	with $\mathcal P_{\gamma,(>,>)} = (]0,1](0,1) + ]0,1]\eta)\cap \mathbb Z^2$,
		\begin{equation} \label{eqcasb3}
			R_{\gamma,(>,>)}^{\delta,=}(T) = 
			\sum_{(\alpha_0,\beta_0) \in \mathcal P_{\gamma,(>,>)}} 
			\frac{\mathbb L^{-\alpha_0-\beta_0}T^{((\alpha_0,\beta_0)\mid \gamma)}}
			{(1-\mathbb L^{-1}T^{((0,1)\mid \gamma)})
			(1-\mathbb L^{-((1,1)\mid \eta)}T^{(\eta\mid \gamma)})}.
		\end{equation}

\item The cone $C_\gamma \cap (\mathbb R_{>0} \times \{0\})$ is empty or equal to $\mathbb R_{>0}(1,0)$.
	If $a_0=0$ then the cone is empty, otherwise $a_0>0$ and for $\delta > \delta_4 = 1/a_0$, we have 
	\begin{equation} \label{cas>0,=0}
		C_{\gamma,-}^{\delta,=} \cap (\mathbb R_{> 0} \times \{0\})  = 
		C_{\gamma} \cap (\mathbb R_{> 0} \times \{0\} ) 
	\end{equation}
	and by Lemma \ref{lemmedescones} we have
	\begin{equation} \label{eqcasb4}
		R_{\gamma,(>,=)}^{\delta,=}(T) = 
		\frac{\mathbb L^{-1}T^{((1,0)\mid \gamma)}}{1-\mathbb L^{-1}T^{((1,0)\mid \gamma)}}.
	\end{equation}

\item The cone $C_\gamma \cap (\mathbb R_{>0} \times \mathbb R_{<0})$ is empty or equal to 
	$\mathbb R_{>0}(1,0)+\mathbb R_{>0} \eta_2$
with $(\gamma \mid \eta_2)>0$ and $(\gamma \mid (1,0)) = a_0> 0$.
	In the second case, we  necessarily have $(\eta_2 \mid (0,1))\leq 0$.
	For any 
	$\delta \geq \delta_5 = \max\left( \frac{1}{a_0}, \frac{(\eta_2 \mid (1,0))}{(\gamma \mid \eta_2)} \right)$,
	we have 
	\begin{equation} \label{cas>0,<0}
		C_{\gamma,-}^{\delta,=} \cap (\mathbb R_{>0} \times \mathbb R_{<0}) 
		= C_\gamma \cap (\mathbb R_{>0} \times \mathbb R_{<0}) = 
		{\mathbb R_{>0}(1,0)+\mathbb R_{>0} \eta_2.}
	\end{equation}
	Indeed, for any $\delta \geq \delta_5$ and $(\alpha,\beta)$ in 
	$C_\gamma\cap (\mathbb R_{>0} \times \mathbb R_{<0})$, we have $c(\alpha,\beta)=\alpha$,
	and there is $\lambda>0$ and $\mu>0$,
	such that $(\alpha,\beta)=\lambda (1,0) + \mu \eta_2$.
	Then we have,
	$(\alpha,\beta)=(\lambda+\mu (\eta_2 \mid (1,0)), \mu (\eta_2 \mid (0,1)))$
	and as $\delta \geq {\delta_5}$, we have 
	$c(\alpha,\beta) = \alpha = \lambda + \mu  (\eta_2 \mid (1,0)) \leq \delta (\lambda (\gamma \mid (1,0))+ \mu(\gamma \mid \eta_2)) = \delta \m(\alpha,\beta).$
	Then, in the non empty case, by Lemma \ref{lemmedescones} we have
with $\mathcal P_{\gamma,(>,<)} = (]0,1](1,0) + ]0,1]\eta_2)\cap \mathbb Z^2$,
	\begin{equation} \label{eqcasb5}
		R_{\gamma,(>,<)}^{\delta,=}(T) = 
		\sum_{(\alpha_0,\beta_0) \in \mathcal P_{\gamma,(>,<)}} 
		\frac{\mathbb L^{-\alpha_0+\beta_0}T^{((\alpha_0,\beta_0)\mid \gamma)}}
		{(1-\mathbb L^{-((1,-1)\mid \eta_2)}T^{(\eta_2\mid \gamma)})
		(1-\mathbb L^{-1}T^{((1,0) \mid \gamma)})}.
	\end{equation}
\end{itemize}
	The bound $\delta_0$ of the statement can be chosen larger than $\delta_1$, $\delta_2$, $\delta_3$, $\delta_4$ and $\delta_5$.
	\end{itemize}
	\item  	Assume $\gamma$ is a one dimensional face in $\N_{\infty}(f)^o$. {Let $(a_0,b_0)$ be a point of $\gamma$} and $\eta$ be the primitive  normal vector to $\gamma$ exterior to the Newton polygon. The cone $C_\gamma$ is $\mathbb R_{>0}\eta$.
	The face $\gamma$ does not contain $0$, {then $(\eta \mid (a_0,b_0))>0$.}
		We have the equality $C_\gamma =  C_{\gamma,-}^{\delta,=}$ and 
		$\chi_{c}(C_{\gamma,-}^{\delta,=})=-1$, for any $\delta>0$ such that
		$$
		\delta\geq \delta_0:= 
		2\max \left( \frac{\abs{(\eta \mid (1,0))}}{(\eta \mid (a_0,b_0))}, 
		\frac{\abs{(\eta \mid (0,1))}}{(\eta \mid (a_0,b_0))}
		\right)
		$$
		Indeed, for any $(\alpha,\beta)$ in $C_\gamma$, there is $k>0$ such that 
		$(\alpha,\beta)=k\eta$, {by Remark \ref{rem:m}} we have $\m(\alpha,\beta)=k((a_0,b_0)\mid \eta)$ and 
		$$c(\alpha , \beta) \leq \abs{\alpha} + \abs{\beta} = k\abs{(\eta | (1,0))} + k \abs{(\eta | (0,1))}
		\leq 2(k\eta | (a_0,b_0))\frac{\delta}{2} = \m(\alpha,\beta)\delta.$$
		Then, by Lemma \ref{lemmedescones}, and denoting $(a_0,b_0)$ any element of $\gamma$, we have
		\begin{equation} \label{eqcasc1}
			R_{\gamma}^{\delta,=}(T) = 
			\frac{\mathbb L^{-\abs{\eta \mid (1,0)}-\abs{\eta \mid (0,1)}}T^{(\eta \mid (a_0,b_0))}}
			{1-\mathbb L^{-\abs{\eta \mid (1,0)}-\abs{\eta \mid (0,1)}}T^{(\eta \mid (a_0,b_0))}}.
		\end{equation}
      \end{itemize}
      {We assume now  that there is $(a,b)$ in $(\mathbb N^*)^2$ and a polynomial $P$ in $\k[s]$ such that $f(x,y)=P(x^ay^b)$}.  
	We denote by $\gamma$ the face $(ad,bd)$.
	Then {the rationality of $Z_{\gamma,-}^{\delta,=}(T)$ and its rational form} follows from above ideas, using the fact that 
	$$C_{\gamma,-}^{\delta,=} \cap (\mathbb R_{<0} \times \mathbb R_{>0}) = \mathbb R_{>0} \eta_\delta + \mathbb R_{\geq 0}(0,1),\:
        C_{\gamma,-}^{\delta,=} \cap (\mathbb R_{>0} \times \mathbb R_{<0}) = \mathbb R_{>0} \eta_\delta' + \mathbb R_{\geq 0}(1,0),\:
        C_{\gamma,-}^{\delta,=} \cap (\mathbb R_{\geq 0} \times \mathbb R_{\geq 0}) = \mathbb R_{\geq 0}\times \mathbb R_{\geq 0}$$
        with $\eta_\delta = (1-bd\delta,ad\delta)$ and $\eta_\delta'=(bd\delta,1-ad\delta)$. 
	More precisely, we have for any $\delta > \max(1/{(da)},1/{(bd)})$
	\begin{equation} \label{casmonom}
	\begin{array}{ccl}
		R_{\gamma,=}^\delta(T) & = & \sum_{(\alpha_0,\beta_0) \in \mathcal P_{\gamma,(<,>)}} 
								\frac{\mathbb L^{\alpha_0-\beta_0}T^{a\alpha_0+b\beta_0}}
								{(1-\mathbb L^{-((-1,1)\mid \eta_\delta)}T^{(\eta_\delta\mid {\gamma}})
								(1-\mathbb L^{-1}T^{b})} 
								+ \frac{\mathbb L^{-((-1,1)\mid \eta_\delta)}T^{(\eta_\delta\mid {\gamma})}}
								{1-\mathbb L^{-((-1,1)\mid \eta_\delta)}T^{(\eta_\delta\mid {\gamma})}}\\ \\
				 & + & \sum_{(\alpha_0,\beta_0) \in \mathcal P_{\gamma,(>,<)}} 
								\frac{\mathbb L^{-\alpha_0+\beta_0}T^{a\alpha_0+b\beta_0}}
								{(1-\mathbb L^{-((1,-1)\mid \eta_\delta')}T^{(\eta_\delta'\mid {\gamma})})
								(1-\mathbb L^{-1}T^{a})} 
								+ \frac{\mathbb L^{-((1,-1)\mid \eta_\delta')}T^{(\eta_\delta'\mid {\gamma})}}
								{1-\mathbb L^{-((1,-1)\mid \eta_\delta')}T^{(\eta_\delta'\mid {\gamma})}} \\ \\
								& + & 
								{\frac{\mathbb L^{-1}T^{ad}}{1-\mathbb L^{-1}T^{ad}} + 
								\frac{\mathbb L^{-1}T^{bd}}{1-\mathbb L^{-1}T^{bd}}
								+  \frac{\mathbb L^{-2}T^{ad+bd}}{(1-\mathbb L^{-1}T^{ad})(1-\mathbb L^{-1}T^{bd})}
							        },
	\end{array}
        \end{equation}
	with
$\mathcal P_{\gamma,(<,>)} = (]0,1](0,1) + ]0,1]\eta_\delta)\cap \mathbb Z^2\:\:\text{and}\:\: 
\mathcal P_{\gamma,(>,<)} = (]0,1](1,0) + ]0,1]\eta_\delta')\cap \mathbb Z^2.$
In particular $-\lim R_{\gamma,=}^{\delta}(T)=1$.
\end{proof}

\begin{prop}[Case $``\eps = +"$] \label{lem:epsgammafacedim0-infini} 
	Let $\gamma$ be a face of $\overline{\N}(f)$ {not contained in the coordinate axes}. There is $\delta_0$ such that for any $\delta \geq \delta_0$, the formal series $Z_{\gamma,+}^{\delta,=}(T)$
	is rational and equal to 
	$$Z^{\delta{,=}}_{\gamma,+}(T) = 
	[f_\gamma : \Gm^{2}\setminus f_{\gamma}^{-1}(0)\ra \Gm,\sigma_{\gamma}] R_{\gamma}^{\delta,=}(T),$$ 
	with $R_{\gamma}^{\delta,=}(T)$ expressed in \ref{eqdecompozeta+}, \ref{eqcas+a1}, \ref{eqcas+a2} and \ref{eqcas+b1}.
	It admits a limit $-\lim Z_{\gamma,+}^{\delta,=}(T)$ in $\mgg$ with 
	$$ - \lim Z_{\gamma,+}^{\delta,=}(T) = {\eps_\gamma}[f_\gamma:\Gm^2 \setminus f_\gamma^{-1}(0) \to \Gm,\sigma_\gamma],$$
	with {$\eps_\gamma = - \chi_{c}(C_{\gamma,+}^{\delta,=} \cap \Omega)$ belongs to $\{0,-1,-2\}$ if $\gamma$ is zero-dimensional and to $\{0,1\}$ if $\gamma$ is one-dimensional.}
	More precisely:
	\begin{enumerate}
		\item \label{cas1eps}  {Assume} $\gamma$ is the origin then for any $\delta>0$,  
			$Z^{\delta{,=}}_{\gamma,+}(T)=0$ and $\eps_\gamma = 0$.
		\item \label{cas2eps} {Assume} $\gamma$ is zero dimensional equal to $(a_0,b_0)$. 
			Let $H_\gamma = \{(\alpha,\beta) \in \mathbb R^2 \mid ( (\alpha,\beta)\mid \gamma)<0\}$ and $C_\gamma$ be the dual cone of $\gamma$.
			We have $C_{\gamma,+}^{\delta,=}\subset H_\gamma$ and 
			$C_{\gamma,+}^{\delta,=}\cap \Omega = 
			C_{\gamma,+}^{\delta,=} \cap \left( (\mathbb R_{>0}\times \mathbb R_{<0}) \cup (\mathbb R_{<0}\times \mathbb R_{>0}) \right).$
			For {any} $(?,!)$ in $\{(<,>),(>,<)\}$
			\begin{enumerate}
				\item \label{cas2aeps} if $C_\gamma \cap H_\gamma \cap (\mathbb R_{?0}\times \mathbb R_{!0}) = \emptyset$ then 
				$C_{\gamma,+}^{\delta,=} \cap (\mathbb R_{?0}\times \mathbb R_{!0})$ is empty and its Euler characteristic is zero 
			        \item \label{cas2beps} if $C_\gamma \cap H_\gamma \cap (\mathbb R_{?0}\times \mathbb R_{!0}) = \mathbb R_{>0}\omega_1 + \mathbb R_{>0}\omega_2$,
			      with $\omega_1$ and $\omega_2$ in the closure  $\overline{H_\gamma}$ and non colinear then: 
	\begin{enumerate}
        \item  \label{cas2bieps} if $(\gamma \mid \omega_1)<0$ and $(\gamma \mid \omega_2)<0$ then, there is $\delta_1>0$ such that for any $\delta>\delta_1$, the cone $C_{\gamma,+}^{\delta,=} \cap (\mathbb R_{?0}\times \mathbb R_{!0})$ 
		is equal to $\mathbb R_{>0}\omega_1 + \mathbb R_{>0}\omega_2$, with Euler characteristic equal to 1. 
        \item \label{cas2biieps} if $(\gamma \mid \omega_1)=0$ and $(\gamma \mid \omega_2)<0$ 
	then there is $\delta_1>0$ such that for any $\delta>\delta_1$, considering the vector 
	$\omega_\delta =(-b_0-1/\delta,a_0-1/\delta)$, the cone 
	$C_{\gamma,+}^{\delta,=} \cap H_\gamma \cap (\mathbb R_{?0}\times \mathbb R_{!0})$ is equal to
	$\mathbb R_{>0} \omega_\delta + \mathbb R_{\geq 0}\omega_2$,
	with Euler characteristic equal to 0.
					\end{enumerate}
			\end{enumerate}
		\item  \label{cas3eps}{Assume} $\gamma$ is a one-dimensional face supported by a line $ap+bq=N$ with $\eta_\gamma = (p,q)$ the primitive exterior normal vector of the face $\gamma$ in $\overline{\N}(f)$, and $C_\gamma = \mathbb R_{>0}\eta_\gamma$, we have 
			\begin{enumerate}
				\item \label{cas3aeps} if $N\geq 0$ then the cone $C_{\gamma,+}^{\delta,=} \cap \Omega$ is empty and $\eps_\gamma= 0$.
				\item \label{cas3beps} if $N<0$ then there is $\delta_2$ such that for any $\delta>\delta_2$ we have 
					$C_{\gamma,+}^{\delta,=} \cap \Omega = C_{\gamma} \cap \Omega$, in particular
					\begin{enumerate}
						\item \label{cas3bieps} if $\eta_\gamma$ belongs to $\Omega$ then, $C_{\gamma,+}^{\delta,=} \cap \Omega = 
							\mathbb R_{>0}{\eta_\gamma}$ with Euler characteristic $-1$, {then $\eps_\gamma=1$,}
						\item  \label{cas3biieps} otherwise,  $C_{\gamma,+}^{\delta,=} \cap \Omega$ is empty with Euler characteristic 0 {and $\eps_\gamma = 0$}.
					\end{enumerate}
			\end{enumerate}
		\item \label{cas4eps} {Assume} $\gamma$ is a one-dimensional face supported by a line $ap+bq=N$ with $(p,q)$ the primitive normal vector of the face $\gamma$ in $\overline{\N}(f)$, and $C_\gamma = \mathbb R_{>0}(p,q) + \mathbb R_{>0}(-p,-q)$ with $pq<0$ (this case occurs if and only if $\overline{\N}(f)$ is a segment), we have 
			\begin{enumerate}
				\item \label{cas4aeps} if $N = 0$ then the cone $C_{\gamma,+}^{\delta,=} \cap \Omega$ is empty, {and $\eps_\gamma = 0$}.
				\item \label{cas4beps} if $N \neq 0$ then there is $\delta_2$ such that for any $\delta>\delta_2$ we have 
					$C_{\gamma,+}^{\delta,=} \cap \Omega = \mathbb R_{>0}(p,q)\: \text{or} \: \mathbb R_{>0}(-p,-q)$
					and its Euler characteristic is $-1$ and $\eps_\gamma = 1$.
			\end{enumerate}
	\end{enumerate}

\end{prop}

\begin{proof}
	Let $\gamma$ be a face in $\overline{\N}(f)$ and $\delta >0$. If the cone $C_{\gamma}^{\delta,=}$ is empty then the result is immediate. 
	If there is $\delta'>0$ such that the cone $C_{\gamma}^{\delta',=} \cap \Omega$ is non empty then, for any $\delta>\delta'$ the cone $C_{\gamma}^{\delta,=} \cap \Omega$ is non empty.
	In the following of the proof, we work with this assumption.
        Let $\delta>0$ and  $(\alpha,\beta)$ be an element of $C_\gamma^\delta \cap \Omega$, then similarly to the proof of Proposition \ref{lem:fctzetaegalinfini}, the motivic measure of $X_{\m(\alpha,\beta),(\alpha,\beta)}(\hat{f})$ is equal to 
	$\mathbb L^{-\abs{\alpha}-\abs{\beta}} 
	[f_\gamma : \Gm^{2}\setminus f_{\gamma}^{-1}(0)\ra \Gm,\sigma_{\gamma}].$
        Then, we have the equality 
	$Z^{\delta{,=}}_{\gamma,-}(T) = 
	[f_\gamma : \Gm^{2}\setminus f_{\gamma}^{-1}(0)\ra \Gm,\sigma_{\gamma}] R_{\gamma}^{\delta,=}(T),$ 
	where
	$R_{\gamma}^{\delta,=}(T) = \sum_{n\geq 1} 
			\sum_{
				\text{\tiny{$\begin{array}{c}
					(\alpha,\beta) \in C_{\gamma,+}^{\delta,=}\cap \Omega \\
					n=-\m(\alpha, \beta)
				\end{array}$}
			     }
			            }
	                        \mathbb L^{-\abs{\alpha}-\abs{\beta}}
			 T^{n}.$
			 The application $c$, defined in {formula (\ref{c})}, is linear on each cone $\mathbb R_{?0} \times \mathbb R_{!0}$ with $``?"$ and $``!"$ in $\{>,<,=\}$ and $\mathbb R_{=0} = \{0\}$.
	Then, we consider the cone $C_\gamma$ as the disjoint union 
	$$C_{\gamma,+}^{\delta,=} \cap \Omega = \bigsqcup_{(?,!)\in\{<,=,>\}^2 \setminus \{<,=\}^2} C_{\gamma,+}^{\delta,=} \cap (\mathbb R_{?0}\times \mathbb R_{!0})$$
	and we have

	\begin{equation} \label{eqdecompozeta+} 
		R_{\gamma}^{\delta,=}(T) = \sum_{(?,!)\in\{<,=,>\}^2 \setminus \{<,=\}^2} R_{\gamma,(?,!)}^{\delta,=}(T)
	\:\:\	\text{with}\:\:
	R_{\gamma,(?,!)}^{\delta,=}(T) = \sum_{n\geq 1} 
	\sum_{
		\text{\tiny{$\begin{array}{c}
			(\alpha,\beta) \in 
			C_{\gamma,+}^{\delta,=} \cap (\mathbb R_{?0}\times \mathbb R_{!0})\cap \Omega \\
			n=\m(\alpha, \beta)
		\end{array}$}}
	}
	\mathbb L^{-\abs{\alpha}-\abs{\beta}}
	T^{n}.
	\end{equation}
	By Lemma \ref{lemmedescones}, each $R_{\gamma,(?,!)}^{\delta,=}(T)$ is rational and its limit, when $T$ goes to infinity is 
	$\chi_{c}(C_{\gamma,+}^{\delta,=} \cap (\mathbb R_{?0}\times \mathbb R_{!0}))$. By additivity, we obtain the rational form of 
	$R_{\gamma}^{\delta,=}(T)$ and its limit is $\chi_{c}(C_{\gamma,+}^{\delta,=} \cap \Omega)$. In the following, we study the cones $C_{\gamma,+}^{\delta,=} \cap \Omega$ and 
	$C_{\gamma,+}^{\delta,=} \cap (\mathbb R_{?0}\times \mathbb R_{!0})$ for any $(?,!)$ in $\{<,=,>\}^2 \setminus \{<,=\}^2$.
	\begin{itemize}
		\item  Assume $\gamma$ is a zero dimensional face, written as $\gamma = (a_0,b_0)$.
			Recall that 
			$C_{\gamma,+}^{\delta,=}  = 
			\{(\alpha,\beta) \in C_\gamma \mid {0<} c(\alpha,\beta) \leq -\m(\alpha,\beta)\delta \}.$
			We consider $H_\gamma = \{(\alpha,\beta)\in \mathbb R^2 \mid \m(\alpha,\beta) = ((\alpha,\beta) \mid \gamma)<0\}$.
			By definition, $C_{\gamma,+}^{\delta,=}$ is a subset of $H_\gamma$ then we have
			$C_{\gamma,+}^{\delta,=} \cap (\mathbb R_{?0} \times \mathbb R_{!0}) = \emptyset$ for $``?"$ and $``!"$ in $\{=,>\}$.
			Then, we only study the cases $C_{\gamma,+}^{\delta,=} \cap (\mathbb R_{>0} \times \mathbb R_{<0})$ and 
                        $C_{\gamma,+}^{\delta,=} \cap (\mathbb R_{<0} \times \mathbb R_{>0})$ which are similar.
			If $C_{\gamma} \cap H_\gamma \cap (\mathbb R_{<0} \times \mathbb R_{>0})$ is non empty then, there are two non colinear vectors $\omega_1$ and 
			$\omega_2$, such that $$C_{\gamma} \cap H_\gamma  \cap (\mathbb R_{<0} \times \mathbb R_{>0}) = \mathbb R_{>0}\omega_1 + \mathbb R_{>0} \omega_2,$$
			with $(\gamma \mid \omega_1)\leq 0$ and $(\gamma \mid \omega_2)\leq 0$.
			\begin{itemize}
				\item If $(\gamma \mid \omega_1) <0$ and $(\gamma \mid \omega_2)<0$ then, for $\delta\geq 2 \delta_1$ with 
					$\delta_{1} = \max\left( 
					\frac{\abs{(\omega_1 \mid (0,1))}}{\abs{(\gamma \mid \omega_1)}}, 	
					\frac{\abs{(\omega_2 \mid (0,1))}}{\abs{(\gamma \mid \omega_2)}} \right)$
we have
					\begin{equation} \label{eqC1} 
						C_{\gamma,+}^{\delta,=} \cap (\mathbb R_{<0} \times \mathbb R_{>0}) = 
						C_{\gamma} \cap H_{\gamma} \cap (\mathbb R_{<0} \times \mathbb R_{>0}).
					\end{equation}
					Indeed, for $\delta \geq \delta_1$ and for any $(\alpha,\beta)$ in 
					$C_{\gamma} \cap H_{\gamma} \cap (\mathbb R_{<0} \times \mathbb R_{>0})$,
					there are $\lambda>0$ and $\mu>0$ such that, $(\alpha,\beta)=\lambda \omega_1 + \mu \omega_2$.
                                        As $(\omega_1 \mid \gamma)<0$ and $(\omega_2 \mid \gamma)<0$, 
					$(\alpha,\beta)$ belongs to $C_{\gamma,+}^{\delta,=}$ thanks to
					$$ 
					c(\alpha, \beta) = \beta =  
					\lambda (\omega_1 \mid (0,1)) + \mu (\omega_2 \mid (0,1)) 
					\leq \delta_{1}( \lambda \abs{(\omega_1 \mid \gamma)} + \mu \abs{(\omega_2 \mid \gamma)}) 
					= -\delta_{1} ((\alpha,\beta)\mid \gamma) \leq -\delta \m(\alpha, \beta).
					$$

					Then, by equality (\ref{eqC1}), and Lemma \ref{lemmedescones} we conclude that  
					$R_{\gamma,(<,>)}^{\delta,=}(T)$ is rational and equal to
				        \begin{equation} \label{eqcas+a1}
							      R_{\gamma,(<,>)}^{\delta,=}(T) = 
							      \sum_{(\alpha_0,\beta_0) \in \mathcal P_{\gamma,(<,>)}} 
								\frac{\mathbb L^{\alpha_0-\beta_0}T^{-((\alpha_0,\beta_0)\mid \gamma)}}
								{(1-\mathbb L^{-((-1,1)\mid \omega_1)}T^{-(\omega_1\mid \gamma)})
								(1-\mathbb L^{-((-1,1)\mid \omega_2)}T^{-(\omega_2 \mid \gamma)})} 
							\end{equation}
							with $\mathcal P_{\gamma,(<,>)} = (]0,1]\omega_1 + ]0,1]\omega_2)\cap \mathbb Z^2$.
					The limit of $R_{\gamma,(<,>)}^{\delta,=}(T)$ is $\chi_c(C_{\gamma,+}^{\delta,=})=1$.

				\item Assume $(\gamma \mid \omega_1)=0$ and $(\gamma \mid \omega_2)<0$ (the case  $(\gamma \mid {\omega_2})=0$ and 
					$(\gamma \mid {\omega_1})<0$ is similar). Then, for any $\lambda>0$ and $\mu>0$, we have 
					$\m(\lambda \omega_1 + \mu \omega_2) = \mu (\omega_2 \mid \gamma).$
					Assume $\delta > \delta_1 = -\frac{(\omega_2 \mid (0,1))}{(\gamma \mid \omega_2)}$, and 
					denote $\omega_\delta = (-b_0 - 1/\delta, a_0)$. {We recall that $\gamma=(a_0,b_0)$.}
					Denote $L_{\delta}$ the function $(\alpha,\beta) \mapsto \beta + \delta (\gamma \mid (\alpha,\beta))$ on 
					$C_{\gamma,+}^{\delta,=} \cap (\mathbb R_{<0}\times \mathbb R_{>0})$.
					Then,
                                        $L_{\delta}(\omega_2)<0$ and $L_\delta(\omega_\delta)=0$ and we have
					\begin{equation} \label{eqC2}
						C_{\gamma,+}^{\delta,=} \cap (\mathbb R_{<0}\times \mathbb R_{>0})  = 
						\mathbb R_{>0}\omega_\delta + \mathbb R_{\geq 0} \omega_2.
					\end{equation}
	                                Then, by equality (\ref{eqC2}), and Lemma \ref{lemmedescones} we conclude that the zeta function 
					$R_{\gamma,(<,>)}^{\delta,=}(T)$ is rational with 
					 \begin{equation} \label{eqcas+a2}
							      R_{\gamma,(<,>)}^{\delta,=}(T) = 
							      \sum_{(\alpha_0,\beta_0) \in \mathcal P_{\gamma,(<,>)}} 
								\frac{\mathbb L^{\alpha_0-\beta_0}T^{-((\alpha_0,\beta_0)\mid \gamma)}}
								{(1-\mathbb L^{-((-1,1)\mid \omega_\delta)}T^{-(\omega_\delta\mid \gamma)})
								(1-\mathbb L^{-((-1,1)\mid \omega_2)}T^{-(\omega_2\mid \gamma)})} 
								+ \frac{\mathbb L^{-((-1,1)\mid \omega_\delta)}T^{-(\omega_\delta\mid \gamma)}}
								{1-\mathbb L^{-((-1,1)\mid \omega_\delta)}T^{-(\omega_\delta\mid \gamma)}}
							\end{equation}
							with $\mathcal P_{\gamma,(<,>)} = (]0,1]\omega_\delta + ]0,1]\omega_2)\cap \mathbb Z^2$.
					The limit of $R_{\gamma,(<,>)}^{\delta,=}(T)$ is $\chi_c(C_{\gamma,+}^{\delta,=})=0$.
			\end{itemize}

		{\item We consider the case of a {one dimensional face} $\gamma$. {The face $\gamma$ is supported by a line of} equation 
		${a}p+ {b}q=N$ with $(p,q)$ the primitive normal vector to the face $\gamma$ exterior to the Newton polygon ${\GN}(f)$.

		{Assume $C_\gamma = \mathbb R_{>0}(p,q)$.} We have for any $k$ in $\mathbb R^+$, $\m(pk,qk)=kN$,
		and 
		$$C_{\gamma,+}^{\delta,=} = 
		\left\{
			\begin{array}{c|l} 
				(pk,qk) \in C_\gamma & 
				\begin{array}{c} 
					k\in \mathbb R^+,\: 0<c(pk,qk)\leq -Nk\delta
				\end{array}
			\end{array} 
		\right\}.$$ 
                \begin{itemize}
				\item If $N\geq 0$ then the cone $C_{\gamma,+}^{\delta,=} \cap \Omega$ is empty.
				\item If $N<0$ then there is $\delta_2 = -c(p,q)/N $ such that for any $\delta>\delta_2$ we have 
					$C_{\gamma,+}^{\delta,=} \cap \Omega = C_{\gamma} \cap \Omega$, in particular
					\begin{itemize}
						\item  if $(p,q)$ belongs to $\Omega$ then, 
						$C_{\gamma,+}^{\delta,=} \cap \Omega = \mathbb R_{>0}(p,q)$ with Euler characteristic $-1$, with 
							\begin{equation} \label{eqcas+b1}
								{R_{\gamma}^{\delta,=}(T)} = \frac{\mathbb L^{-\abs{p}-\abs{q}}T^{-N}}{1-\mathbb L^{-\abs{p}-\abs{q}}T^{-N}}
							\end{equation}

						\item  otherwise,  $C_{\gamma,+}^{\delta,=} \cap \Omega$ is empty with Euler characteristic 0 and 
							${R_{\gamma}^{\delta,=}}(T)=0$.
					\end{itemize}
			\end{itemize}

			{Assume} $C_\gamma = \mathbb R_{>0}(p,q) + \mathbb R_{>0}(-p,-q)$ with $pq<0$ {(this case only occurs in the case where $\overline{\N}(f)$ is a segment)}.
			\begin{itemize}
				\item If $N = 0$ then the cone $C_{\gamma,+}^{\delta,=} \cap \Omega$ is empty and $R_{\gamma}^{\delta,=}(T) = 0$.
				\item If $N \neq 0$ then there is $\delta_2$ such that for any $\delta>\delta_2$ we have 
				$C_{\gamma,+}^{\delta,=} \cap \Omega = \mathbb R_{>0}(p,q)\: \text{or} \: \mathbb R_{>0}(-p,-q)$
					and its Euler characteristic is $-1$ and $R_{\gamma}^{\delta,=}(T)$ is given by formula (\ref{eqcas+b1}).
			\end{itemize}
		}
	\end{itemize}
	The bound $\delta_0$ in the statement can be chosen larger then the maximum of the bounds $\delta_1$ and $\delta_2$ above.
\end{proof}

\subsubsection{The formal series $Z^{\delta,<}_{\gamma,\eps}(T)$ for a face $\gamma$ not contained in a coordinate axes} \label{sec:casdim1<} 
\begin{rem}[Vanishing $f_\gamma$] \label{annulation-f-gamma} 
	Let $\gamma$ be a face in $\overline{\N}(f)$. Let $(\alpha,\beta)$ be in $C_\gamma$ and $\varphi$ be in an arc in $\mathcal L(X)$ with 
	$\ord x(\varphi) = -\alpha$ and $\ord y(\varphi)=-\beta$. Then, by Remark \ref{rem:m} 
	$\eps\ord \hat{f^{\eps}}(\varphi) > -\m(\alpha,\beta)$ if and only if 
	$f_{\gamma}(\ac x(\varphi), \ac y(\varphi))=0$. In particular in that case $\gamma$ is a one-dimensional face {of $\overline{\N}(f)$}.
\end{rem}

	\begin{rem} We only consider one dimensional faces with dual cone $C_\gamma$ in $\Omega$, then there are five cases to study: the face $\gamma$ belongs to $\N_{\infty,\infty}(f)$, $\N_{0,\infty}(f)$, $\N_{\infty,0}(f)$ or is horizontal or vertical. \end{rem}

\paragraph{The face $\gamma$ belongs to $\N_{\infty,\infty}(f)$.}
{The dual cone of the one dimensional face $\gamma$} is the cone $C_\gamma = \mathbb R_{>0}(p,q)$ 
{with $(p,q)$ the primitive normal vector to $\gamma$ exterior to $\overline{\N}(f)$} and $\gamma$ is supported by a line of equation $p\alpha+q\beta = N$. {As $\gamma$ belongs to $\N_{\infty,\infty}(f)$, we have $p>0$ and $q>0$.}
We write $f_\gamma(x,y)=x^{a_\gamma}y^{b_\gamma}\prod_{\mu_\in R_\gamma} (x^q-\mu_i y^p)^{\nu_{\mu}}$.
Let $\mu$ be a root in $R_\gamma$. Using Notations \ref{notationxphiyphi} {and equation (\ref{Cgammadelta<})}, for any $(n,(\alpha,\beta))$ in $C_{\gamma,\eps}^{\delta,<}$, we denote 
$$X_{n,(\alpha,\beta),\mu}(\hat{f}^\eps) = \{ \varphi \in X_{n,(\alpha,\beta)}(\hat{f}^\eps) \mid 
-(\ord x(\varphi), \ord y(\varphi))=(\alpha,\beta),\: \ac x(\varphi)^q=\mu \ac y(\varphi)^p \}.$$ 
endowed with the induced structural map to $\Gm$.

{\begin{rem} \label{remecriturelocale}
 The origin of any arc of $X_{n,(\alpha,\beta),\mu}(\hat{f}^\eps)$ is the point $([0:1],[0:1],[0:1])$ in $X_{\infty}$ and by Remark \ref{rem:ecriture-carte-origine} (formula \ref{ecriturelocale1}) the arc space $X_{n,(\alpha,\beta),\mu}(\hat{f}^\eps)$ is isomorphic to the arc space 

 $$
		\left\{
			\begin{array}{c|l}
				(A(t),B(t))\in \mathcal L(\mathbb A^2_{\k}) & 
			\begin{array}{l}
				\ord A(t)=\alpha, \: \ord B(t)=\beta,\:\ord f(1/A(t),1/B(t))=n, \\
				\ac A(t)^{-q}-\mu \ac B(t)^{-p} =0 \:\:\text{namely}\:\: \ac B(t)^p = \mu \ac A(t)^q\\
			\end{array}
		        \end{array}
		\right\}.
	$$
\end{rem}
}

\begin{prop} \label{prop:mesXnmu-infini-infini} 
	Let $\mu$ be a root {in $R_{\gamma}$}
	and ${\sigma_{(p,q,\mu)}}$ the induced Newton transform 
	$$\begin{array}{ccccl}
		\sigma_{(p,q,\mu)}& : & \k[x,y] & \longrightarrow & \k[v^{-1},v,w] \\ 
		                  &   &  g(x,y) & \mapsto         & g_{{\sigma_{(p,q,\mu)}}}(v,w)=g(\mu^{q'}v^{-p}, v^{-q}(w+\mu ^{p'}))\\
	   \end{array}
	$$
	defined in Definition \ref{def:lestransformationsdeNewton} with $qq'-pp'=1$. 
			For any $k>0$, for any $(n,(\alpha,\beta))$ in $C_{\gamma}^{\delta,<}$ with 
		$\alpha = pk$ and $\beta = qk$, we have
		$$\mes(X_{n,(\alpha,\beta),\mu}({\hat{f}^{\eps}}))=\mathbb L^{-(p+q-1)k}
		\mes(Y_{(n,k)}({(f_{{\sigma_{(p,q,\mu)}}})^{\eps}}))
	\in \mathcal M_{\Gm}^{\Gm}$$
	with 
	$Y_{(n,k)}\left({\left(f_{{\sigma_{(p,q,\mu)}}}\right)^{\eps}}\right) = 
	\left\{ 
		\begin{array}{c|c}
			\left(v(t),w(t)\right) \in \mathcal L(\mathbb A^2_k) &
			\begin{array}{l} 
				\ord v(t)=k,\:\ord w(t)>0,\:
				\ord {\left(f_{{\sigma_{(p,q,\mu)}}}\right)^{\eps}}(v(t),w(t)) = n 
			\end{array}
		\end{array} 
	\right\},
	$ 
	endowed with the structural map {$\ac \left(f_{{\sigma_{(p,q,\mu)}}}\right)^{\eps}$} to $\Gm$.
\end{prop} 
\begin{proof}
	The proof is inspired from that of \cite[Lemma 3.3]{CassouVeys13}. Let $L$ be an integer bigger than $\alpha$, $\beta$ and $n$. We consider
	{$$ X_{(n,\alpha,\beta),\mu}^{(L)} = 
	\left\{ 
		\begin{array}{c|l}
			(A(t),B(t)) \in 
			\left(\k[[t]]/(t^{L+1})\right)^2 & 
			\begin{array}{l}
				\ord A(t)=\alpha,\: \ord B(t)=\beta,\:
				\ac B(t)^p = \mu \ac A(t)^q \\ 
				\ord f^{\eps}(1/A(t),1/B(t))=n 
			\end{array} 
		\end{array}
	\right\},
	$$ 
         }
	{which is isomorphic to the jet-spaces $\pi_{L}( X_{n,(\alpha,\beta),\mu}(\hat{f}^\eps))$ 
	by Remark \ref{remecriturelocale}}, 
	$$\overline{X}_{(n,\alpha,\beta),\mu}^{(L)} = 
	\{
		\begin{array}{c|l} 
			(\varphi_1(t),\varphi_2(t))\in \left(\k[[t]]/t^{L+1}\right)^2 & \begin{array}{l}
				\varphi_1(0)\neq 0, \varphi_2(0)\neq 0, \:
				{ \varphi_2(0)^p = \mu \varphi_1(0)^q},\:
				{\ord f^{\eps}\left(1/(t^{\alpha} \varphi_1),1/(t^{\beta} \varphi_2)\right)=n}
			\end{array} 
		\end{array} 
	\},
	$$
	$$\overline{Y}_{(n,k)}^{(L)} = 
	\{ 
		\begin{array}{c|l}
			(\psi_1(t),\psi_2(t))\in \left(\k[[t]]/t^{L+1}\right)^2 &
			\begin{array}{l} 
				{\ord \left(f_{{\sigma_{(p,q,\mu)}}}\right)^{\eps}(t^k \psi_1(t),\psi_2(t))=n},\:
				\ord \psi_1(t)=0,\: \ord \psi_2(t)\geq 1 
			\end{array} 
		\end{array} 
	\},
	$$
	$$Y_{(n,k)}^{(L)} = 
	\{ 
		\begin{array}{c|l} 
			(v'(t),w'(t))\in \left(\k[[t]]/t^{L+1}\right)^2 & 
			\begin{array}{l} 
				{\ord \left(f_{{\sigma_{(p,q,\mu)}}}\right)^{\eps}(v'(t),w'(t))=n},\:
				\ord v'(t)=k,\: \ord w'(t)\geq 1 
			\end{array} 
		\end{array} 
	\},
	$$
	which is  
	{$\pi_{L}\left( Y_{(n,k)}\left(\left(f_{{\sigma_{(p,q,\mu)}}}\right)^{\eps}\right)\right)$.}
	The application
	$(\varphi_1,\varphi_2)\mapsto (t^\alpha \varphi_1 \mod t^{L+1}, t^\beta \varphi_2 \mod t^{L+1})$ 
	induces a structure of bundle on $\overline{X}_{(n,\alpha,\beta),\mu}^{(L)}$
	over $X_{(n,\alpha,\beta),\mu}^{(L)}$ with fiber $\mathbb A^{\alpha+\beta}$. 
	Also, the application 
	$(\psi_1,\psi_2)\mapsto (t^k \psi_1 \mod t^{L+1}, \psi_2 \mod t^{L+1})$ 
	induces a structure of bundle on $\overline{Y}_{(n,k)}^{(L)}$ over $Y_{(n,k)}^{(L)}$ with fiber $\mathbb A^{k}$.
	We deduce the equalities 
	$[X_{(n,\alpha,\beta),\mu}^{(L)}] = \mathbb L^{-\alpha-\beta} [\overline{X}_{(n,\alpha,\beta),\mu}^{(L)}]$
	and 
	$[\overline{Y}_{(n,k)}^{(L)}]= \mathbb L^{k}[Y_{(n,k)}^{(L)}].$
	We consider the application
	{$$\begin{array}{ccccl}
		\Phi^{{\sigma_{(p,q,\mu)}}} & : &
		\overline{Y}_{(n,k)}^{(L)} & \ra &
		\overline{X}_{(n,\alpha,\beta),\mu}^{(L)} \\
		& & (\psi_1,\psi_2) & \mapsto & 
		\left(\mu^{-q'} \psi_1^{p}, 
		\psi_1^{q}(\psi_2 +\mu^{p'})^{-1}
		\right)=:(\varphi_1,\varphi_2).
	\end{array}
	$$ 
        }
	Using the relation { $qq'-pp'=1$}, we can check that  $\Phi^{{\sigma_{(p,q,\mu)}}}(\overline{Y}_{(n,k)}^{(L)}) \subset  \overline{X}_{(n,\alpha,\beta),\mu}^{(L)}$. Indeed, if 
	$\left(\varphi_1(t),\varphi_2(t) \right)$ is equal to $\Phi^{{\sigma_{(p,q,\mu)}}}(\psi_1(t),\psi_2(t))$ then
	{$ \varphi_2(0)^p - \mu \varphi_1(0)^q = (\psi_1(0)^{q}\mu^{-p'})^{p} - \mu(\mu^{-q'}\psi_1(0)^{p})^{q} = 0$}
	and using the relations $\alpha=pk$, $\beta=qk$, and the definitions we deduce the equality 
	{$f^{\eps}\left(1/(t^{\alpha} \varphi_1(t)),1/(t^{\beta} \varphi_2(t)) \right) = 
	\left((f_{{\sigma_{(p,q,\mu)}}}(t^k\psi_1(t)),\psi_2(t))\right)^{\eps}.$}

	We prove that $\Phi^{{\sigma_{(p,q,\mu)}}}$ is an isomorphism building the inverse application.
	Consider $\varphi(t)=\left(\varphi_1(t),\varphi_2(t) \right)$ in $\overline{X}_{(n,\alpha,\beta),\mu}^{(L)}$.
	Remark that if there is $\psi(t)=(\psi_1(t),\psi_2(t))$ in $\overline{Y}_{(n,k)}^{(L)}$ such that
	$\varphi(t)=\Phi^{{\sigma_{(p,q,\mu)}}}(\psi_1(t),\psi_2(t))$ then we have the equality
	${\varphi_2(0)^{q'}}/{\varphi_1(0)^{p'}} = \psi_1(0).$ 
	Furthermore, denoting $\varphi_1(t) = \varphi_1(0) \tilde{\varphi_1}(t)$, by Hensel lemma, there is a
	unique formal series $a(t)$ such that $a(0)=1\:\:\text{and}\:\:a(t)^p = \tilde{\varphi_1}(t).$ 
	The formal series $a(t)$ is denoted by $\tilde{\varphi_1}(t)^{1/p}$.  Hence, the
	inverse map is given by 
	 $$\left(\varphi_1(t),\varphi_2(t) \right) \mapsto
	\left(
	\frac{\varphi_2(0)^{q'}}{\varphi_1(0)^{p'}}\tilde{\varphi_1}(t)^{1/p} \mod t^{L+1}, 
	-\mu^{p'} + 			
	\left( \frac{\varphi_2(0)^{q'}}{\varphi_1(0)^{p'}}\tilde{\varphi_1}(t)^{1/p}\right)^{q}\varphi_2(t)^{-1} 
	\mod t^{L+1}
	\right).
$$ 
	Thus $\Phi^{{\sigma_{(p,q,\mu)}}}$ is an isomorphism, we have
	$[X_{(n,\alpha,\beta),\mu}^{(L)}]=\mathbb L^{-\alpha-\beta+k} [Y_{(n,k)}^{(L)}],$
	and conclude by definition of the motivic measure.
\end{proof}

\begin{prop}\label{prop:rationalityZ<infini-infini} 
	For $\delta$ large enough, the motivic zeta function
	$Z_{\gamma,\eps}^{\delta,<}$ is rational and can be decomposed as 
\begin{equation} \label{eq:cle-infini-infini}
	Z^{\delta,<}_{\gamma,\eps}(T)= 
	\sum_{\mu \in R_\gamma} 
	\left(
	Z^{\delta/(p+q)}_{\left({f_{{\sigma_{(p,q,\mu)}}}}\right)^\eps, \omega_{p,q}, v\neq 0}(T)\right)_{((0,0),0)}
	\:\text{and}\:
        -\lim_{T \ra \infty} Z^{\delta,<}_{\gamma,\eps}(T)= 
	\sum_{\mu \in R_\gamma} 
	(S_{\left({f_{{\sigma_{(p,q,\mu)}}}}\right)^\eps, v\neq 0})_{((0,0),0)}
\end{equation}
	with the differential form 
	$\omega_{p,q}(v,w) = v^{( p +q -1)}dv\wedge dw.$ 
\end{prop}
	\begin{rem}
		In the statement of the proposition, we can also use the differential  
	$$\omega_{p,q,\mu}(v,w) =v^{( p +q -1)}(w+\mu ^{p'})^{-2}dv\wedge dw.$$
because we work locally at $(0,0)$, in particular $\ord w(t)>0$, then we have equality of orders
$$\ord \omega_{p,q,\mu}(v(t),w(t))= \ord \omega_{p,q}(v(t),w(t)) = (p+q-1)\ord v(t).$$
\end{rem}
	
\begin{proof}
	For any element $(\alpha,\beta)$ in $C_\gamma$ there is $k>0$ such that
	$\alpha =pk$ and $\beta = qk$. In particular, as $\gamma$ belongs to $\N_{\infty,\infty}(f)$, $p>0$ and $q>0$,
	and we have $\alpha>0$ and $\beta>0$. The set of integer points of the cone
	$C_{\gamma}^{\delta,<}$ {defined in equation (\ref{Cgammadelta<}) is}, 
	$$ C_{\gamma}^{\delta,<} \cap \mathbb N^{3} =
	\left\{ \begin{array}{c|l} 
		(n,(\alpha,\beta)) \in {\mathbb R_{>0}} \times C_\gamma & 
		               \begin{array}{l} 
			              -\m(\alpha,\beta)<\eps n,\:
			              1 \leq \alpha + \beta \leq n\delta 
		               \end{array} 
	         \end{array} 
	 \right\} {\cap \mathbb N^3} ,$$ 
is in bijection with the cone
$\overline{C}_{\gamma}^{\delta,<} = 
\left\{ 
	\begin{array}{c|l} 
		(n,k) \in \mathbb N^{*2} & 
		\begin{array}{l} 
			-\m(pk,qk)<\eps n,\:
			1 \leq (p+q)k \leq n\delta 
		\end{array} 
	\end{array}
\right\},
$
by $(n,k)\mapsto (n,(pk,qk))$.
Using this notation we prove the equality \ref{eq:cle-infini-infini}.
\noindent Indeed, using Proposition \ref{prop:mesXnmu-infini-infini} we have 

$$ \begin{array}{ccl} 
	Z^{\delta,<}_{\gamma}(T) & = & 
	\sum_{\mu \in R_\gamma} {\sum_{n\geq 1}} \sum_{(n,(\alpha,\beta)) 
		\in C_{\gamma}^{\delta,<}\cap \mathbb N^3} 
		\mes(X_{(n,(\alpha,\beta)),\mu})T^n \\ \\ 
		& = & \sum_{\mu \in R_\gamma} {\sum_{n\geq 1}}\sum_{(n,k)\in \overline{C}_{\gamma}^{\delta,<}}  \mathbb L^{-(p+q-1)k} 
		\mes { \left( Y_{(n,k)}\left(\left({f_{{\sigma_{(p,q,\mu)}}}}\right)^{\eps}\right) \right)}T^{n},
	\end{array} 
	$$ 
	but using the definition of the zeta function  in subsection \ref{section:mot-zeta-omega} 
	(see also equation (\ref{zeta})) 
	we have  
	$$ \begin{array}{ccl}
		\left(
		Z_{\left(f_{{\sigma_{(p,q,\mu)}}}\right)^{\eps},\omega_{p,q}, v\neq 0}^{\delta/(p+q)}(T)
		\right)_{((0,0),0)}
		& = & 
		\sum_{n \geq 1} 
		\left[
		\sum_{k \geq 1} \mathbb L^{-(p+q-1)k} 
		\mes\left(
		Y_{(n,k)}^{\delta/(p+q)}
		\left(
		      \left(f_{{\sigma_{(p,q,\mu)}}}\right)^{\eps}
		\right)
		\right)
	        \right]
		T^{n} \\
		& = & \sum_{n\geq 1}\sum_{(n,k)\in \overline{C}_{\gamma}^{\delta,<}} 
		\mathbb L^{-( p+q-1)k}
		\mes\left(Y_{(n,k)}\left(\left({f_{{\sigma_{(p,q,\mu)}}}}\right)^{\eps}\right)\right)T^{n}
	\end{array}
	$$ 
	with  
	$Y_{(n,k)}^{\delta/(p+q)}
		\left(\left({f_{{\sigma_{(p,q,\mu)}}}}\right)^{\eps}\right)= 
	\left\{
		\begin{array}{c|c} 
			\left(v(t),w(t)\right) \in \mathcal L(\mathbb A^2_k) &
			\begin{array}{l} 
				(v(0),w(0))=0,\:
				\ord v(t) = k \leq n\delta /(p+q) \\
				\ord \left({f_{{\sigma_{(p,q,\mu)}}}}\right)^{\eps}(v(t),w(t)) = n \\ 
			\end{array}
		\end{array} 
	\right\}.
	$

	In particular we can conclude thanks to  section \ref{section:mot-zeta-omega},
	for $\delta$ large enough, that the motivic zeta function is rational and
	has a limit independent on $\delta$ when $T$ goes to infinity
	$$-\lim_{T \ra \infty}	
	\left(
	Z_{\left({f_{{\sigma_{(p,q,\mu)}}}}\right)^{\eps}, \omega_{p,q}, v\neq 0}^{\delta/(p+q)}(T)
	\right)_{((0,0),0)} = 
	\left(S_{\left({f_{{\sigma_{(p,q,\mu)}}}}\right)^{\eps}, \omega_{p,q}, v\neq 0}\right)_{((0,0),0)} \in \mgg.
	$$
\end{proof}

\paragraph{The face $\gamma$ belongs to $\N_{\infty,0}(f)$ or $\N_{0,\infty}(f)$.}
Let $\gamma$ be a face in $\N_{i(\gamma),j(\gamma)}$ with $(i(\gamma),j(\gamma))$ equal to $(\infty,0)$ or $(0,\infty)$
with $C_\gamma = \mathbb R_{>0}(p,q)$  with $p>0$ and $q<0$, or, $p<0$ and $q>0$. We denote by $R_\gamma$ the roots of $f_\gamma$ (see Notations \ref{rem:factorisationinfini}).
\begin{prop}\label{prop:rationalityZ<infini-zero} 
	For $\delta$ large enough, the motivic zeta function
	$Z_{\gamma,\eps}^{\delta,<}$ is rational with
	$$Z^{\delta,<}_{\gamma,\eps}(T) = 
	\sum_{\mu \in R_\gamma} 
	\left(Z^{\delta/c(p,q)}_{\left({f_{{\sigma_{(p,q,\mu)}}}}\right)^{\eps}, \omega_{p,q}, v\neq 0}\right)_{((0,0),0)}
	\:\text{and}\:
	-\lim_{T \ra \infty} Z^{\delta,<}_{\gamma,\eps}(T) = 
	\sum_{\mu \in R_\gamma} 
	\left(S_{\left({f_{{\sigma_{(p,q,\mu)}}}}\right)^{\eps}, v\neq 0}\right)_{((0,0),0)}$$ 
	with 
	$\omega_{p,q}(v,w) = v^{( \abs{p} + \abs{q} -1)}dv\wedge dw$ and the convenient Newton transformations defined in definition \ref{def:lestransformationsdeNewton} in (\ref{transformation-Newton-infini-zero}) and (\ref{transformation-Newton-zero-infini}).
\end{prop}
 
\begin{proof}
	The proof is similar to the proof of Proposition \ref{prop:rationalityZ<infini-infini}. \\
\end{proof}

\begin{rem} Assume $\eps=+$. In the case where $C_\gamma = \mathbb R_{>0}(p,q)+\mathbb R_{>0}(-p,-q)$ with $pq<0$ then applying twice the previous proposition we obtain, that for $\delta$ large enough, the motivic zeta function
	$Z_{\gamma,\eps}^{\delta,<}$ is rational with
	$$Z^{\delta,<}_{\gamma,\eps}(T) = 
	\sum_{\mu \in R_\gamma} 
	\left(Z^{\delta/c(p,q)}_{{f_{\sigma(\abs{p},-\abs{q},\mu)}}, \omega_{p,q}, v\neq 0}\right)_{((0,0),0)}
	+ \left(Z^{\delta/c(p,q)}_{{f_{\sigma(-\abs{p},\abs{q},\mu)}}, \omega_{p,q}, v\neq 0}\right)_{((0,0),0)}
	$$ 
	It has a limit (independent from $\delta$)
	$-\lim_{T \ra \infty} Z^{\delta,<}_{\gamma,\eps}(T) = 
	\sum_{\mu \in R_\gamma} 
	(S^{\delta/c(p,q)}_{{f_{\sigma(\abs{p},-\abs{q},\mu)}}, v\neq 0})_{((0,0),0)}
	+ (S^{\delta/c(p,q)}_{{f_{\sigma(-\abs{p},\abs{q},\mu)}}, v\neq 0})_{((0,0),0)}.
	$ 
\end{rem}

\paragraph{The face $\gamma$ is horizontal.} \label{sec:cashorizontalinfini}

There are at most two one-dimensional horizontal faces. There is only one with exterior normal vector in $\Omega$, we denote it $\gamma_H$.
The  face polynomial  $f_{\gamma_H}$ has the 
form $y^MT(x)$ where $M\geq 1$ and $T$ is a polynomial in $\k[x]$. We denote by $R_{\gamma_H}$ the set of roots of $T$ in $\Gm$ and call it set of roots of $f_{\gamma_H}$. 
In that case, we have $C_{\gamma_H} = \mathbb R_{>0}(0,1)$.

\begin{rem} \label{annulationdefgamma}
Let $(n,(\alpha,\beta))$ in $C_{\gamma_H,\eps}^{\delta,<}$ necessarily $\alpha=0$ and $\beta>0$. 
By Remark \ref{rem:ecriture-carte-origine}, any arc $\varphi$ in 
$X_{n,(\alpha,\beta)}$ can be written as 
$\varphi(t)= ([1:A(t)],[B(t):1],[z_0(t):z_1(t)])$
with $\ord A(t)=0$ and $\ord B(t)>0$ and 
$\ord \hat{f}^{\eps}(\varphi(t)) = \ord f^{\eps}\left( A(t), 1/B(t) \right) = n.$
Furthermore, as $n<-\eps \m(\alpha,\beta)$, $A(0)$ is a root of $f_\gamma$.
\end{rem}
\begin{prop} \label{prop:face-hozitontale}
For $\delta$ large enough, the motivic zeta function
	$Z_{\gamma_H,\eps}^{\delta,<}$ is rational with
        $$Z_{\gamma_H,\eps}^{\delta,<}(T) = 
	\sum_{\mu \in R_{\gamma_H}} 
	(Z^{\delta}_{f^{\eps}_{0,\infty,\mu},y\neq 0})_{((0,0),0)}
	\:\text{and}\:
        -\lim_{T \ra \infty} Z_{\gamma_H,\eps}^{\delta,<}(T) = 
	\sum_{\mu \in R_{\gamma_H}} 
	\left(S_{f^{\eps}_{0,\infty,\mu},y\neq 0}\right)_{((0,0),0)},
	$$
	where for any root $\mu$ of $f_{\gamma_H}$, we consider {$f_{0,\infty,\mu}(x,y) = f(x+\mu,1/y).$}
\end{prop}
\begin{rem} The motives $(S_{f^{\eps}_{0,\infty,\mu},y\neq 0})_{((0,0),0)}$ are computed in subsection \ref{Sfeps}.
\end{rem}

	\begin{proof} 
		For any $n \geq 1$, $\delta\geq 1$ and $\mu \in R_{\gamma_H}$ we consider
		$$X_{n,\eps,\gamma_H,\mu}^{\delta} =  
		\left\{
			\begin{array}{c|c}
			(A(t),B(t)) & 
			  \begin{array}{l}
				A(0)=\mu,\: B(0)\neq 0,\:
				0<\ord B(t) \leq n\delta,\:
				\ord f^{\eps}(A(t),1/B(t))=n
			  \end{array}
		        \end{array}
		\right\}
		$$
		and using remark \ref{annulationdefgamma} we obtain the decomposition of the zeta function 
		$Z_{\gamma_H,\eps}^{\delta,<}(T)= 
		\sum_{\mu \in R_{f_{\gamma_H}}} 
		Z_{\hat{f}^{\eps},\gamma_H,\mu}^{\delta}(T)
		$
		with 
		 for any root $\mu$ in $R_{\gamma_H}$,
		$Z_{\hat{f}^{\eps},\gamma_H,\mu}^{\delta}(T) = 
		\sum_{n\geq 1} \mes\left(X_{n,\eps,\gamma_H,\mu}^{\delta}\right)T^n.$
					We have the isomorphism
			$$ \begin{array}{ccc}
				X_{n, y\neq 0,(0,0)}^{\delta}(f^{\eps}_{0,\infty,\mu}) & \to & 
				X_{n,\gamma_H,\mu}^{\delta} \\
				(w(t),B(t)) & \mapsto & (w(t)+\mu,B(t))
			\end{array}
			$$
			where 
			$ X_{n, y\neq 0,(0,0)}^{\delta}(f^{\eps}_{0,\infty,\mu}) = 
			\left\{
				\begin{array}{c|c}
				(w(t),t^k B(t)) & 
				\begin{array}{l}
					\ord w(t)>0,\:
					0<\ord B(t) \leq n\delta,\:
					\ord f^{\eps}_{0,\infty,\mu}(w(t), B(t)) = n
				\end{array}
				\end{array}
			\right\}.
			$
			Then, we conclude that we have the equality
			$Z_{\hat{f}^{\eps},\gamma_H,\mu}^{\delta}(T) = 
			(Z_{f^{\eps}_{0,\infty,\mu},y\neq 0}^{\delta}(T))_{((0,0),0)},$
		which induces the result.
		\end{proof}
			
\paragraph{The face $\gamma$ is vertical.} \label{sec:casverticalinfini}
There are at most two one-dimensional vertical faces. There is only one with exterior normal vector in $\Omega$, we denote it $\gamma_V$.
The  face polynomial  $f_\gamma$ has the 
form $x^MT(y)$ where $M\geq 1$ and $T$ is a polynomial in $\k[y]$. We denote by $R_{\gamma_V}$ the set of roots of $T$ in $\Gm$. We call this set, set of roots of $f_{\gamma_V}$. 
In that case, we have $C_{\gamma_V} = \mathbb R_{>0}(1,0)$.

\begin{prop} \label{prop:face-verticale}
For $\delta$ large enough, the motivic zeta function $Z_{\gamma_V,\eps}^{\delta,<}$ is rational with
$$Z_{\gamma_V,\eps}^{\delta,<} = \sum_{\mu \in R_{f_{\gamma_V}}} \left(Z^{\delta}_{f^{\eps}_{\infty,0,\mu},x\neq 0}\right)_{((0,0),0)}
\:\text{and}\:-\lim_{T \ra \infty} Z_{\gamma_V,\eps}^{\delta,<}(T) = \sum_{\mu \in R_{f_{\gamma_V}}} 
	\left(S_{f^{\eps}_{\infty,0,\mu},x\neq 0}\right)_{((0,0),0)}.
$$ 
where for any root $\mu$ of $f_{\gamma_V}$ we consider {$f_{\infty,0,\mu}(x,y) = f(1/x,y+\mu).$}
\end{prop}
\section*{Acknowledgement}
The first author is partially supported by the spanish grants: MTM2016-76868-C2-2-P (directors Pedro Gonzalez Perez and Alejandro Melle), MTM2016-76868-C2-1-P (directors Enrique Artal and Jose Ignacio Cogolludo) and the group {\it Geometria, Topologia, Algebra y Cryptografia en singularidades y sus aplicaciones}.
The second author is partially supported by ANR-15-CE40-0008 (Defigeo)
\bibliographystyle{plain}
\bibliography{biblio_complete}

\begin{thebibliography}{10}

\bibitem{ArtLueMel00b}
E.~{Artal Bartolo}, I.~Luengo, and A.~Melle-Hern{{\'a}}ndez.
\newblock {Milnor number at infinity, topology and {N}ewton boundary of a
  polynomial function}.
\newblock {\em Math. Z.}, 233(4):679--696, 2000.

\bibitem{ArtCasLue05a}
Enrique {Artal Bartolo}, Pierrette Cassou-Nogu{{\`e}}s, Ignacio Luengo, and
  Alejandro {Melle Hern{{\'a}}ndez}.
\newblock {Quasi-ordinary power series and their zeta functions}.
\newblock {\em Mem. Amer. Math. Soc.}, 178(841):vi+85, 2005.

\bibitem{Bit05a}
Franziska Bittner.
\newblock {On motivic zeta functions and the motivic nearby fiber}.
\newblock {\em Math. Z.}, 249(1):63--83, 2005.

\bibitem{Bro88a}
S.~A. Broughton.
\newblock {Milnor numbers and the topology of polynomial hypersurfaces}.
\newblock {\em Invent. Math.}, 92(2):217--241, 1988.

\bibitem{Cas96}
Pierrette Cassou-Nogu\`es.
\newblock Sur la g\'en\'eralisation d'un th\'eor\`eme de {K}ouchnirenko.
\newblock {\em Compositio Math.}, 103(1):95--121, 1996.

\bibitem{Cassou11}
Pierrette Cassou-Nogu{\`e}s.
\newblock Newton trees at infinity of algebraic curves.
\newblock In {\em Affine algebraic geometry}, volume~54 of {\em CRM Proc.
  Lecture Notes}, pages 1--19. Amer. Math. Soc., Providence, RI, 2011.

\bibitem{carai-antonio}
Pierrette Cassou-Nogu\`es and Michel Raibaut.
\newblock Newton transformations and the motivic {M}ilnor fiber of a plane
  curve.
\newblock In {\em Singularities, algebraic geometry, commutative algebra, and
  related topics}, pages 145--189. Springer, Cham, 2018.

\bibitem{CassouVeys13}
Pierrette Cassou-Nogu\`es and Willem Veys.
\newblock Newton trees for ideals in two variables and applications.
\newblock {\em Proc. Lond. Math. Soc. (3)}, 108(4):869--910, 2014.

\bibitem{Cassou-Nogues-Veys-15}
Pierrette Cassou-Nogu\`es and Willem Veys.
\newblock The {N}ewton tree: geometric interpretation and applications to the
  motivic zeta function and the log canonical threshold.
\newblock {\em Math. Proc. Cambridge Philos. Soc.}, 159(3):481--515, 2015.

\bibitem{CNS}
Antoine Chambert-Loir, Johannes Nicaise, and Julien Sebag.
\newblock {\em Motivic integration}, volume 325 of {\em Progress in
  Mathematics}.
\newblock Birkh\"{a}user/Springer, New York, 2018.

\bibitem{DenLoe98b}
Jan Denef and Fran\c{c}ois Loeser.
\newblock {Motivic {I}gusa zeta functions}.
\newblock {\em J. Algebraic Geom.}, 7(3):505--537, 1998.

\bibitem{DenLoe99a}
Jan Denef and Fran\c{c}ois Loeser.
\newblock {Germs of arcs on singular algebraic varieties and motivic
  integration}.
\newblock {\em Invent. Math.}, 135(1):201--232, 1999.

\bibitem{DenLoe01b}
Jan Denef and Fran\c{c}ois Loeser.
\newblock {Geometry on arc spaces of algebraic varieties}.
\newblock In {\em {European {C}ongress of {M}athematics, {V}ol. {I}
  ({B}arcelona, 2000)}}, volume 201 of {\em {Progr. Math.}}, pages 327--348.
  Birkh{\"a}user, Basel, 2001.

\bibitem{DenLoe02a}
Jan Denef and Fran\c{c}ois Loeser.
\newblock {Lefschetz numbers of iterates of the monodromy and truncated arcs}.
\newblock {\em Topology}, 41(5):1031--1040, 2002.

\bibitem{FR}
Lorenzo Fantini and Michel Raibaut.
\newblock Motivic and analytic nearby fibers at infinity and bifurcation sets.
\newblock {\em To appear in Arc schemes and singularities, World scientific
  publishing, D. Bourqui, J. Nicaise, J. Sebag editors}.

\bibitem{Gui02a}
Gil Guibert.
\newblock {Espaces d'arcs et invariants d'{A}lexander}.
\newblock {\em Comment. Math. Helv.}, 77(4):783--820, 2002.

\bibitem{GuiLoeMer05a}
Gil Guibert, Fran\c{c}ois Loeser, and Michel Merle.
\newblock {Nearby cycles and composition with a nondegenerate polynomial}.
\newblock {\em Int. Math. Res. Not.}, (31):1873--1888, 2005.

\bibitem{GuiLoeMer06a}
Gil Guibert, Fran\c{c}ois Loeser, and Michel Merle.
\newblock {Iterated vanishing cycles, convolution, and a motivic analogue of a
  conjecture of {S}teenbrink}.
\newblock {\em Duke Math. J.}, 132(3):409--457, 2006.

\bibitem{GuiLoeMer09a}
Gil Guibert, Fran\c{c}ois Loeser, and Michel Merle.
\newblock {Composition with a two variable function}.
\newblock {\em Math. Res. Lett.}, 16(3):439--448, 2009.

\bibitem{Kon95a}
Kontsevich.
\newblock {Lecture at Orsay}.
\newblock D{\'e}cembre 7, 1995.

\bibitem{Kou76a}
A.~G. Kouchnirenko.
\newblock {Poly{\`e}dres de {N}ewton et nombres de {M}ilnor}.
\newblock {\em Invent. Math.}, 32(1):1--31, 1976.

\bibitem{Loo02a}
Eduard Looijenga.
\newblock {Motivic measures}.
\newblock {\em Ast{\'e}risque}, (276):267--297, 2002.
\newblock S{\'e}minaire Bourbaki, Vol.\ 1999/2000.

\bibitem{Matsui-Takeuchi-13}
Yutaka Matsui and Kiyoshi Takeuchi.
\newblock Monodromy at infinity of polynomial maps and {N}ewton polyhedra (with
  an appendix by {C}. {S}abbah), announced in arxiv:0912.5144v5.
\newblock {\em Int. Math. Res. Not. IMRN}, (8):1691--1746, 2013.

\bibitem{Matsui-Takeuchi-14}
Yutaka Matsui and Kiyoshi Takeuchi.
\newblock Motivic {M}ilnor fibers and {J}ordan normal forms of {M}ilnor
  monodromies.
\newblock {\em Publ. Res. Inst. Math. Sci.}, 50(2):207--226, 2014.

\bibitem{NemZah90a}
Andr{\'a}s N{\'e}methi and Alexandru Zaharia.
\newblock {On the bifurcation set of a polynomial function and {N}ewton
  boundary}.
\newblock {\em Publ. Res. Inst. Math. Sci.}, 26(4):681--689, 1990.

\bibitem{Pha83a}
Fr{\'e}d{\'e}ric Pham.
\newblock {Vanishing homologies and the {$n$} variable saddlepoint method}.
\newblock In {\em {Singularities, {P}art 2 ({A}rcata, {C}alif., 1981)}},
  volume~40 of {\em {Proc. Sympos. Pure Math.}}, pages 319--333. Amer. Math.
  Soc., Providence, RI, 1983.

\bibitem{Rai10a}
Michel Raibaut.
\newblock {Fibre de {M}ilnor motivique {\`a} l'infini}.
\newblock {\em C. R. Math. Acad. Sci. Paris}, 348(7-8):419--422, 2010.

\bibitem{Rai12}
Michel Raibaut.
\newblock Fibre de {M}ilnor motivique \`a l'infini et composition avec un
  polyn\^{o}me non d\'{e}g\'{e}n\'{e}r\'{e}.
\newblock {\em Ann. Inst. Fourier (Grenoble)}, 62(5):1943--1981, 2012.

\bibitem{Rai11}
Michel Raibaut.
\newblock {Singularit{\'e}s {\`a} l'infini et int{\'e}gration motivique}.
\newblock {\em Bull. Soc. Math. France}, 140(1):51--100, 2012.

\bibitem{Suz74}
Masakazu Suzuki.
\newblock Propri\'et\'es topologiques des polyn\^omes de deux variables
  complexes, et automorphismes alg\'ebriques de l'espace {${\bf C}^{2}$}.
\newblock {\em J. Math. Soc. Japan}, 26:241--257, 1974.

\bibitem{Takeuchi-Tibar-16}
Kiyoshi Takeuchi and Mihai Tib\u{a}r.
\newblock Monodromies at infinity of non-tame polynomials.
\newblock {\em Bull. Soc. Math. France}, 144(3):477--506, 2016.

\bibitem{Tibar}
Mihai Tib\u{a}r.
\newblock {\em Polynomials and vanishing cycles}, volume 170 of {\em Cambridge
  Tracts in Mathematics}.
\newblock Cambridge University Press, Cambridge, 2007.

\bibitem{Veys01}
Willem Veys.
\newblock Zeta functions and ``{K}ontsevich invariants'' on singular varieties.
\newblock {\em Canad. J. Math.}, 53(4):834--865, 2001.

\bibitem{VuiTra84a}
H{\`a}~Huy Vui and L{\^e}~D{\~u}ng Tr{\'a}ng.
\newblock {Sur la topologie des polyn{\^o}mes complexes}.
\newblock {\em Acta Math. Vietnam.}, 9(1):21--32 (1985), 1984.

\end{thebibliography}
\end{document}